\documentclass[11pt, reqno]{amsart}
\usepackage{amsmath, amssymb, latexsym, enumerate,  graphicx,tikz, stmaryrd}
\usetikzlibrary{arrows,decorations,decorations.pathmorphing,snakes,positioning}
\usepackage{mdframed}
\usepackage{mdwlist}
\usepackage[pdftex]{hyperref}
\usepackage{amscd,mathrsfs,epic,empheq,float}
\usepackage{bbm}
\usepackage[all]{xy}
 \usepackage{todonotes}
\usepackage{tikz}
\usepackage{graphicx}
\usepackage[vcentermath]{youngtab}
\newcommand{\renc}{\renewcommand}

 \newlength{\baseunit}               
 \newcount{\numlines}                
\newtheorem{theorem}{Theorem}[section]
\newtheorem{lemma}[theorem]{Lemma}
\newtheorem{remark}[theorem]{Remark}
\newtheorem{prop}[theorem]{Proposition}
\newtheorem{corollary}[theorem]{Corollary}
\newtheorem{conjecture}[theorem]{Conjecture}
\newtheorem{definition}[theorem]{Definition}

\newtheorem{ex}[theorem]{Example}

\newcommand{\cA}{{\mathcal A}}
\newcommand{\cB}{{\mathcal B}}
\newcommand{\cC}{{\mathcal C}}

\newcommand{\cE}{{\mathcal E}}
\newcommand{\cF}{{\mathcal F}}

\newcommand{\cM}{{\mathcal M}}

\newcommand{\cO}{{\mathcal O}}

\newcommand{\cS}{{\mathcal S}}
\newcommand{\cT}{{\mathcal T}}

\newcommand{\ov}{\overline}

\newcommand{\B}{\mathbb{B}_{\Lambda}}

\newcommand{\D}{\mathbb{D}_{\Lambda}}

\def\down{\vee}
\def\up{\wedge}

\newcommand{\MOD}{\text{-}\op{mod}}
\newcommand{\gMOD}{\text{-}\op{gmod}}
\newcommand{\mg}{\mathfrak{g}}

\newcommand{\mh}{\mathfrak{h}}

\newcommand{\mR}{\mathbb{R}}

\newcommand{\mC}{\mathbb{C}}

\newcommand{\mZ}{\mathbb{Z}}
\newcommand{\mX}{\mathbb{X}}
\newcommand{\mD}{\mathbb{D}}

\newcommand{\pos}{\operatorname{pos}}

\newcommand{\la}{\lambda}
\newcommand{\La}{\Lambda}
\newcommand{\Lap}{\Lambda_k^{\ov{p}}}

\newcommand{\p}{\mathfrak{p}}

\newcommand{\op}{\operatorname}

\newcommand{\Ssym}{{\mathbb{S}^{\operatorname{sym}}_k}}

\renewcommand{\S}{\mathbb{S}_k}

\DeclareMathOperator{\End}{End}   \DeclareMathOperator{\Id}{Id}
  
\DeclareMathOperator{\Mod}{\text{-}mod}

\DeclareMathOperator{\Perv}{Perv}

\newcommand{\CupConnectGraph}[1]{%
\begin{tikzpicture}[#1]%
\draw[thick] (0,0.6ex) -- (2.75ex,0.6ex);%
\draw[opacity=0] (0,0) -- (3ex,0);%
\end{tikzpicture}%
}

\newcommand{\DCupConnectGraph}[1]{%
\begin{tikzpicture}[#1]%
\draw[thick] (0,0.6ex) -- (2.75ex,0.6ex);%
\fill (1.375ex,0.6ex) circle(0.4ex);%
\draw[opacity=0] (0,0) -- (3ex,0);%
\end{tikzpicture}%
}

\newcommand{\RayConnectGraph}[1]{%
\begin{tikzpicture}[#1]%
\draw[opacity=0] (0,0) -- (3ex,0);%
\draw[thick] (0,0.6ex) -- (2.75ex,0.6ex);%
\draw[thick] (2.75ex,1.2ex) -- (2.75ex,0);%
\end{tikzpicture}%
}

\newcommand{\DRayConnectGraph}[1]{%
\begin{tikzpicture}[#1]%
\draw[opacity=0] (0,0) -- (3ex,0);%
\draw[thick] (0,0.6ex) -- (2.75ex,0.6ex);%
\fill (1.375ex,0.6ex) circle(0.4ex);%
\draw[thick] (2.75ex,1.2ex) -- (2.75ex,0);%
\end{tikzpicture}%
}

\newcommand{\CupConnect}{\CupConnectGraph{}}
\newcommand{\DCupConnect}{\DCupConnectGraph{}}
\newcommand{\RayConnect}{\RayConnectGraph{}}
\newcommand{\DRayConnect}{\DRayConnectGraph{}}

\newcommand{\un}{\underline}

\begin{document}
\title[Diagrams for isotropic Grassmannians]{Diagrams for perverse sheaves on isotropic Grassmannians and the supergroup $SOSP(m|2n)$.}
\author{Michael Ehrig}
\author{Catharina Stroppel}
\thanks{M.E. was financed by the DFG Priority program 1388. This material is based on work supported by the National Science Foundation under Grant No. 0932078 000, while the authors were in residence at the MSRI in Berkeley, California.}

\begin{abstract}
We describe diagrammatically a positively graded Koszul algebra $\mathbb{D}_k$ such that the category of finite dimensional
$\mathbb{D}_k$-modules is equivalent to the category of perverse sheaves on the isotropic Grassmannian of type $D_k$ constructible with
respect to the Schubert stratification. The connection is given by an explicit isomorphism to the endomorphism algebra of a projective
generator described in \cite{Braden}. The algebra is obtained by a "folding" procedure from the generalized Khovanov arc algebras. We relate this algebra to the category of finite dimensional representations of the orthosymplectic supergroups. The proposed equivalence of categories gives a concrete description of the categories of finite dimensional $\op{SOSP}(m|2n)$-modules.
\end{abstract}

\maketitle

\tableofcontents
\section*{Introduction}
\renc{\thetheorem}{\Alph{theorem}}
This is the first of three articles
studying some generalizations $\mathbb{D}_k$
of Khovanov's diagram algebra from \cite{KhJones} to type $D$,
as well as their quasi-hereditary covers.
In this article we introduce these algebras and prove that the category of $\mathbb{D}_k$-modules is equivalent to the category $\op{Perv}_k$ of perverse sheaves on the isotropic Grassmannian $X$ of type $D_k$ constructible with respect to the Schubert stratification. The equivalence will be established by giving an explicit isomorphism to the endomorphism algebra of a projective
generator described in \cite{Braden}. Since we have $X=G/P$ for $G=SO(2k,\mC)$ and $P$ the parabolic subgroup of type $A_{k-1}$ these categories are then also equivalent to the Bernstein-Gelfand-Gelfand parabolic category $\cO_0^\p(\mathfrak{so}(2k,\mC))$ of type $D_k$ with parabolic of type $A_{k-1}$.

The algebra $\mathbb{D}_k$ is constructed elementary, purely in terms of diagrams; it comes naturally equipped with a grading. As a vector space it has an explicit homogeneous basis given by certain oriented circle diagrams similar to \cite{BS1}. Under our isomorphism, Braden's algebra inherits a grading which agrees with the geometric Koszul grading established in \cite{BGS}.

To be more specific, recall that Schubert varieties, and hence the simple objects in $\op{Perv}_k$ are labelled by symmetric partitions fitting into a box of size $k\times k$ or equivalently by the representatives of minimal length for the cosets $S_k\backslash W(D_k)$ of the Weyl group $W(D_k)$ of type $D_k$ modulo a parabolic subgroup isomorphic to the symmetric group $S_k$. We first identify in Section~\ref{section21} these cosets with {\it diagrammatical weights} $\la$ (ie. with $\{\up, \down\}$-sequences of length $k$ with an even number of $\up$'s) and then associate to each such diagrammatical weight a {\it (decorated) cup diagram} $\underline{\la}$ on $k$ points, see Definition~\ref{decoratedcups}. For instance, the eight possible decorated cup diagrams for $k=4$ are displayed in Figure~\ref{fig:quiv} below. Such a cup diagram $\underline{\la}$ can be paired with a second cup diagram  $\underline{\mu}$ by putting $\underline\mu$ upside down on top of $\underline{\la}$ to obtain a {\it (decorated) circle diagram} $\underline{\la}\overline{\mu}$. Adding additionally a compatible weight diagram in between gives us an {\it oriented circle diagram}, see the right hand side of \eqref{exmul} for two examples. Let $\mathbb{D}_k$ be the vector space spanned by these decorated circle diagrams for fixed $k$.

\begin{theorem}
\label{intro}
\begin{enumerate}
\item The vector space  $\mathbb{D}_k$ can be equipped with a diagrammatically defined associative algebra structure such that  $\mathbb{D}_k$ is isomorphic to the endomorphism algebra of a minimal projective generator of $\op{Perv}_k$.
\item The multiplication is compatible with the grading (Proposition~\ref{forgotgraded}), such that $\mathbb{D}_k$ becomes a graded algebra.
\item The underlying graph of the  $\op{Ext}$-quiver of $\op{Perv}_k$ is a finite graph with vertices labelled by the cup diagrams on $k$ points. Each cup $C$ in a cup diagram $\la$ defines an ingoing and outgoing arrow from resp. to a cup diagram $\mu$, where $\mu$ is obtained from $\la$ by swapping $\up$ with $\down$ at the points connected by $C$.
    For $k=4$ we have for instance the graph from Figure~\ref{fig:quiv}.
\item The graded Cartan matrix $C_k$ of $\mathbb{D}_k$ is indexed by the cup diagrams $\underline{\la}$ and the entries count the number of possible orientations of the circle diagram $\overline{\la}\mu$ with their gradings, see Figure~\ref{fig:CM} for $k=4$.
\end{enumerate}
\end{theorem}

\begin{figure}
\begin{equation*}
\begin{tikzpicture}[thick,scale=.8]
\begin{scope}[xshift=3.1cm,yshift=-2cm]
\draw (.4,0) to +(0,-.5);
\draw (.8,0) to +(0,-.5);
\draw (1.2,0) to +(0,-.5);
\draw (1.6,0) to +(0,-.5);
\end{scope}
\draw[<-] (4.25,-.8) to node[right]{$b_1$} +(0,-.9);
\draw[->] (4,-.8) to node[left]{$a_1$} +(0,-.9);

\begin{scope}[xshift=3.1cm]
\draw (.4,0) .. controls +(0,-.5) and +(0,-.5) .. +(.5,0);
\fill (.65,-.35) circle(2.5pt);
\draw (1.2,0) to +(0,-.5);
\draw (1.6,0) to +(0,-.5);
\end{scope}
\draw[->] (5.25,-.2) to node[above]{$a_2$} +(1,0);
\draw[<-] (5.25,-.4) to node[below]{$b_2$} +(1,0);

\begin{scope}[xshift=6.2cm]
\draw (.4,0) to +(0,-.5);
\fill (.4,-.25) circle(2.5pt);
\draw (.8,0) .. controls +(0,-.5) and +(0,-.5) .. +(.5,0);
\draw (1.6,0) to +(0,-.5);
\end{scope}
\draw[->] (8.25,.5) to node[above,pos=0.4]{$a_3$} +(1,.6);
\draw[<-] (8.25,.3) to node[below,pos=0.8]{$b_3$} +(1,.6);
\draw[<-] (8.25,-.6) to node[above,pos=0.7]{$b_4$} +(1,-.6);
\draw[->] (8.25,-.8) to node[below,pos=0.4]{$a_4$} +(1,-.6);

\begin{scope}[xshift=9.2cm,yshift=1.5cm]
\draw (.4,0) to +(0,-.5);
\fill (.4,-.25) circle(2.5pt);
\draw (.8,0) to +(0,-.5);
\draw (1.2,0) .. controls +(0,-.5) and +(0,-.5) .. +(.5,0);
\end{scope}

\begin{scope}[xshift=9.2cm,yshift=-1.5cm]
\draw (.4,0) .. controls +(0,-.5) and +(0,-.5) .. +(.5,0);
\draw (1.2,0) to +(0,-.5);
\fill (1.2,-.25) circle(2.5pt);
\draw (1.6,0) to +(0,-.5);
\end{scope}
\draw[->] (11.25,1.1) to node[above,pos=.65]{$a_5$} +(1,-.6);
\draw[<-] (11.25,.9) to node[below,pos=.4]{$b_5$} +(1,-.6);
\draw[<-] (11.25,-1.2) to node[above,pos=.3]{$b_6$} +(1,.6);
\draw[->] (11.25,-1.4) to node[below,pos=.7]{$a_6$} +(1,.6);

\begin{scope}[xshift=12.2cm]
\draw (.4,0) .. controls +(0,-.5) and +(0,-.5) .. +(.5,0);
\draw (1.2,0) .. controls +(0,-.5) and +(0,-.5) .. +(.5,0);
\end{scope}
\draw[->] (14.25,-.1) to node[above]{$a_7$} +(1,0);
\draw[<-] (14.25,-.3) to node[below]{$b_7$} +(1,0);

\begin{scope}[xshift=15cm]
\draw (.5,0) .. controls +(0,-1) and +(0,-1) .. +(1.5,0);
\draw (1,0) .. controls +(0,-.5) and +(0,-.5) .. +(.5,0);
\end{scope}
\draw[->] (16.15,.2) to node[left]{$a_8$} +(0,1.3);
\draw[<-] (16.35,.2) to node[right]{$b_8$} +(0,1.3);
\draw[<-] (16.15,-1) -- +(0,-1.5) -- node[above]{$a_{10}$} +(-8.8,-1.5) -- +(-8.8,.3);
\draw[->] (16.35,-1) -- +(0,-1.7) -- node[below]{$b_{10}$} +(-9.2,-1.7) -- +(-9.2,.3);

\begin{scope}[xshift=15.2cm,yshift=2.1cm]
\draw (.4,0) .. controls +(0,-.5) and +(0,-.5) .. +(.5,0);
\fill (.65,-.35) circle(2.5pt);
\draw (1.2,0) .. controls +(0,-.5) and +(0,-.5) .. +(.5,0);
\fill (1.45,-.35) circle(2.5pt);
\end{scope}
\draw[<-] (15.3,2) -- +(-11.3,0) -- node[left]{$a_9$} +(-11.3,-1.7);
\draw[->] (15.3,1.8) -- +(-11.1,0) -- node[right]{$b_9$} +(-11.1,-1.5);
\end{tikzpicture}
\end{equation*}
\caption{The Ext-graph for $k=4$.}
\label{fig:quiv}
\end{figure}
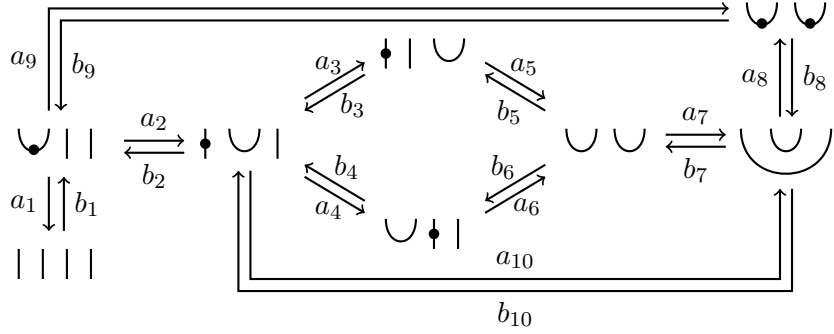

The definition of the multiplication on $\mathbb{D}_k$ is similar to Khovanov's original definition in type $A$ which used the fact that the Frobenius algebra $R=\mC[x]/(x^2)$ defines a 2-dimensional TQFT, hence assigns in a functorial way to a union $(S^1)^n$ of $n$ circles the tensor product $\cT((S^1)^n)=R^{\otimes n}$ of $n$ copies of $R$ and to a cobordisms $(S^1)^n\rightarrow (S^1)^m$ between two finite union of circles a linear map $\cT((S^1)^n)\rightarrow \cT((S^1)^m)$.  A basis vector in our algebra $\mathbb{D}_k$ corresponds to a fixed orientation of some circle diagram with say $n$ circles and so can be identified canonically with a standard basis vector of $\cT((S^1)^n)$, see Proposition~\ref{coho}. This allows to use the linear maps attached to the cobordisms to define the multiplication of $\mathbb{D}_k$. In contrast to the type $A$ case, this construction needs however some additional signs incorporated in a subtle way (encoded by the decorations on the diagrams). Depending on the viewpoint these signs destroy either the locality or force to consider the circles to be embedded in the plane. Therefore general topological arguments using the TQFT-structure cannot be applied (for instance to deduce the associativity of the multiplication).

\begin{figure}
\begin{equation*}
\label{Cartanintro}
\begin{pmatrix}
1&q&0&0&0&0&0&q^2\\
q&1+q^2&q&0&0&0&q^2&q+q^3\\
0&q&1+q^2&q&q&q^2&q+q^3&q^2\\
0&0&q&1+q^2&q^2&q+q^3&^2&0\\
0&0&q&q^2&1+q^2&q+q^3&q^2&0\\
0&0&q^2&q+q^3&q+q^3&1+2q^2+q^4&q+q^3&0\\
0&q^2&q+q^3&q^2&q^2&q+q^3&1+2q^2+q^4&q+q^3\\
q^2&q+q^3&q^2&0&0&0&q+q^3&1+2q^2+q^4
\end{pmatrix}
\end{equation*}
\caption{The Cartan matrix for $k=4$.}
\label{fig:CM}
\normalsize
\end{figure}

We therefore first ignore these difficulties and explain the construction over $\mathbb{F}_2$ in Section~\ref{F2}, but have to work hard afterwards to show the associativity of the multiplication over $\mZ$ or $\mC$ by algebro-combinatorial methods, see Proposition~\ref{prop:multiplication_associative}.

In Section~\ref{sec:cell} we study in detail the structure of the associative algebra $\mathbb{D}_k$. We establish its cellularity in the sense of Graham and Lehrer \cite{GL} in the graded version of \cite{MH} and determine explicitly the decomposition matrices $M_k$ by an easy counting formula in terms of parabolic Kazhdan-Lusztig polynomials of type $(D_k, A_{k-1})$. This allows to identify the Cartan matrix $C_k=M^t_kM_k$ with the Cartan matrix of Braden's algebra $A(D_k)$.
 This is then used to show that the explicit assignment in Theorem~\ref{mainthm} defines indeed an isomorphism $\Phi_k$. Note that our generators are obtained from Braden's by some formal logarithm. This fact is responsible for two major advantages of our presentation in contrast to Braden's which will play an important role in the subsequent papers:
 \begin{itemize}
 \item the nice positive (Koszul) grading of $\mathbb{D}_k$  becomes invisible under the exponential map $\Phi_k^{-1}$; and
\item the Cartan matrix and decomposition numbers which are totally explicit in our setup are impossible to compute in terms of Braden's generators.
\end{itemize}

To emphasize that our results should not just be seen as straight-forward generalizations of known results, let us indicate some  applications and connections which will appear in detail in the subsequent papers of the series.

\subsection*{The orthosymplectic supergroup}
In Section~\ref{super} of this paper we relate our diagram algebra to the category $\cF(\op{SOSP}(m|2n))$ of finite dimensional representations of the orthosymplectic Lie supergroup $\op{SOSP}(m|2n)$.

In analogy to \cite{BS1}, certain natural idempotent truncations $\mathbb{H}_k$ of $\mathbb{D}_k$ have a limiting version $\mathbb{H}_k^\infty$ built up from infinite cup diagrams with $k$ cups. Our combinatorics attaches to each indecomposable projective module $P(\la)$ in $\cF(\op{SOSP}(m|2n))$ a cup diagram $\underline{\la}$. In contrast to \cite{GS2} our assignment is injective and moreover allows to compute Jordan-H\"older multiplicities by an easy {\it positive counting formula} involving parabolic Kazhdan-Lusztig polynomials of type $D$. We will give a dictionary relating our algebras to the combinatorics developed in \cite{GS2}.

\begin{theorem}
 Given a block $\cB$ of $\cF(\op{SOSP}(m|2n))$  of atypicality $k$ and $P:=\oplus_{\la\in\La}P(\la)$ a minimal projective generator.
 For any finite subset $S$ of $\Lambda$ let $\mathbb{D}_S$ be the vector space with basis all oriented circle diagrams of shape  $\underline{\la}\overline{\mu}$, $\la,\mu\in\La$. Let $\mathbb{I}$ be the subspace spanned by all of those which contain a non-propagating line.
 Then there is an isomorphism of vector spaces
 \begin{eqnarray}
\label{isoPS1}
\Psi_S:\quad \op{End}_\mg(P_S)\cong \mathbb{D}_{S}/\mathbb{I}.
\end{eqnarray}
Moreover, as vector space, $\mathbb{D}_{S}$ is isomorphic to an idempotent truncation of our diagram algebra $\mathbb{D}_{k}$ containing $\mathbb{I}$ as ideal.
\end{theorem}

In particular, $\mathbb{D}_{S}/\mathbb{I}$ inherits a (positively graded) algebra structure and we conjecture that the isomorphisms $\Psi_S$ are compatible isomorphisms of algebras. After taking the limit
we obtain a full description of the blocks:

\begin{conjecture}
\label{conj}
Let $m,n$ be natural numbers and fix a block $\cB$ of atypicality $k$ in $\cF(\op{SOSP}(m|2n))$. Then $\cB$ is equivalent to the category of finite dimensional $\mathbb{H}_k^\infty/\mathbb{I}$-modules, where $\mathbb{I}$ is the subspace spanned by all those diagrams which contain a non-propagating line.
\end{conjecture}

If the conjecture is true, $\cB$ could be equipped with a grading induced from a positive grading on $\mathbb{H}_k^\infty$. Moreover, $\mathbb{H}_k^\infty$ is a cellular algebra in the sense of Graham and Lehrer, \cite{GL}, hence comes equipped with a class of {\it cell modules}. Unfortunately this property doesn't carry over to the quotient $\mathbb{H}_k^\infty/\mathbb{I}$ and $\cF(\op{SOSP}(m|2n))$ is not cellular in general. This phenomenon explains and is responsible for serious extra difficulties in comparison to the type $A$ case, where cell modules (usually called Kac modules) exist. Moreover, the arguments in \cite{BoeNakano} in connection with our quiver description of the algebra $\mathbb{D}_k$ involving diamonds \eqref{diamonds2} imply that all blocks of atypicality greater than $1$ are wild, which agrees with the classification result in \cite{Germoni}.

\begin{ex}
For $\cF(\op{SOSP}(3|2))$ our conjecture can be checked by hands. All blocks are semisimple or equivalent to the principal block $\cB$ (of atypicality $1$) containing the trivial representation.  The block $\cB$ contains the simple modules $L(i)$, $j=0,1,2\ldots$ of highest weight $\la$, where $\la+\rho=(i-\frac{1}{2}, i-\frac{1}{2})$ for $i>0$ and  $\la+\rho=(-\frac{1}{2},\frac{1}{2})$ for $i=0$, see \eqref{laab} and Remark \ref{superlink}. To $P(i)$ we assign the cup diagrams shown below (with infinitely many rays to the right). The socle, radical and grading filtration of the corresponding projective $\mathbb{H}_1^\infty/\mathbb{I}$-modules agree and are displayed just underneath the cell module filtrations of the corresponding lifts $\tilde{P}(i)$ in $\mathbb{H}_1^\infty$.

\special{em:linewidth 0.7pt} \unitlength 0.80mm
\linethickness{1pt}
\begin{eqnarray*}
\begin{array}[t]{c|c|c|c|c|c}
P(0)&P(1)&P(2)&P(3)&P(4)&\cdots\\
\hline
&&&&\vspace{-4mm}\\
\hline
&&&&\vspace{-2mm}\\
\begin{tikzpicture}[scale=0.5]
\draw (-5,0) .. controls +(0,-.5) and +(0,-.5) .. +(.5,0);
\fill (-4.75,-0.35) circle(3pt);
\draw (-4,0) -- +(0,-.8);
\fill (-4,-.4) circle(3pt);
\draw (-3.5,0) -- +(0,-.8);
\draw (-3,0) -- +(0,-.8);
\draw (-2.5,0) -- +(0,-.8) node[midway,right] {$\cdots$};
\end{tikzpicture}
&
\begin{tikzpicture}[scale=0.5]
\draw (-.5,0) .. controls +(0,-.5) and +(0,-.5) .. +(.5,0);
\draw (.5,0) -- +(0,-.8);
\draw (1,0) -- +(0,-.8);
\draw (1.5,0) -- +(0,-.8);
\draw (2,0) -- +(0,-.8) node[midway,right] {$\cdots$};
\end{tikzpicture}
&
\begin{tikzpicture}[scale=0.5]
\draw (4,0) -- +(0,-.8);
\draw (4.5,0) .. controls +(0,-.5) and +(0,-.5) .. +(.5,0);
\draw (5.5,0) -- +(0,-.8);
\draw (6,0) -- +(0,-.8);
\draw (6.5,0) -- +(0,-.8) node[midway,right] {$\cdots$};
\end{tikzpicture}
&
\begin{tikzpicture}[scale=0.5]
\draw (8.5,0) -- +(0,-.8);
\draw (9,0) -- +(0,-.8);
\draw (9.5,0) .. controls +(0,-.5) and +(0,-.5) .. +(.5,0);
\draw (10.5,0) -- +(0,-.8);
\draw (11,0) -- +(0,-.8) node[midway,right] {$\cdots$};
\end{tikzpicture}
&
\begin{tikzpicture}[scale=0.5]
\draw (13,0) -- +(0,-.8);
\draw (13.5,0) -- +(0,-.8);
\draw (14,0) -- +(0,-.8);
\draw (14.5,0) .. controls +(0,-.5) and +(0,-.5) .. +(.5,0) ;
\draw (15.5,0) -- +(0,-.8);
\node at (16.3,-.42) {$\cdots$};
\end{tikzpicture}
&\cdots
\\
\hline
&&&&\\
\begin{picture}(20.00,8.00)
\special{em:linewidth 0.4pt} \unitlength 0.80mm
\linethickness{1pt}
\put(7.50,10.00){\makebox(0,0)[cc]{\tiny{$0$}}}
\put(7.50,5.00){\makebox(0,0)[cc]{\tiny{$2$}}}
\put(5.00,0){\makebox(0,0)[cc]{\tiny{$3$}}}
\put(10.00,0){\makebox(0,0)[cc]{\tiny{$0$}}}
\linethickness{0.3pt}
\dashline{1}(5.50,11.50)(5.50,8.50)
\dashline{1}(9.50,11.50)(9.50,8.50)
\dashline{1}(5.50,11.50)(9.50,11.50)
\dashline{1}(5.50,8.50)(9.50,8.50)
\dashline{1}(2.50,-2.00)(12.50,-2.00)
\dashline{1}(2.50,2.00)(2.50,-2.00)
\dashline{1}(12.50,2.00)(12.50,-2.00)
\dashline{1}(12.50,2.00)(7.50,7.50)
\dashline{1}(2.50,2.00)(7.50,7.50)
\end{picture}
&
\begin{picture}(20.00,0.00)
\special{em:linewidth 0.4pt} \unitlength 0.80mm
\linethickness{1pt}
\put(7.50,10.00){\makebox(0,0)[cc]{\tiny{$1$}}}
\put(7.50,5.00){\makebox(0,0)[cc]{\tiny{$2$}}}
\put(7.50,0){\makebox(0,0)[cc]{\tiny{$1$}}}
\linethickness{0.3pt}
\dashline{1}(5.50,11.50)(5.50,3.00)
\dashline{1}(9.50,11.50)(9.50,3.00)
\dashline{1}(5.50,11.50)(9.50,11.50)
\dashline{1}(5.50,3.00)(9.50,3.00)
\dashline{1}(5.50,2.00)(5.50,-2.00)
\dashline{1}(9.50,2.00)(9.50,-2.00)
\dashline{1}(5.50,2.00)(9.50,2.00)
\dashline{1}(5.50,-2.00)(9.50,-2.00)
\end{picture}
&
\begin{picture}(20.00,0.00)
\special{em:linewidth 0.4pt} \unitlength 0.80mm
\linethickness{1pt}
\put(7.50,10.00){\makebox(0,0)[cc]{\tiny{$2$}}}
\put(2.50,5.00){\makebox(0,0)[cc]{\tiny{$0$}}}
\put(7.50,5.00){\makebox(0,0)[cc]{\tiny{$1$}}}
\put(12.50,5.00){\makebox(0,0)[cc]{\tiny{$3$}}}
\put(7.50,0){\makebox(0,0)[cc]{\tiny{$2$}}}
\linethickness{0.3pt}
\dashline{1}(5.50,6.50)(5.50,-2.00)
\dashline{1}(9.50,6.50)(9.50,-2.00)
\dashline{1}(5.50,6.50)(9.50,6.50)
\dashline{1}(5.50,-2.00)(9.50,-2.00)
\dashline{1}(0.50,7.50)(0.50,2.50)
\dashline{1}(4.50,7.50)(4.50,2.50)
\dashline{1}(0.50,2.50)(4.50,2.50)
\dashline{1}(10.50,7.50)(10.50,2.50)
\dashline{1}(14.50,7.50)(14.50,2.50)
\dashline{1}(10.50,2.50)(14.50,2.50)
\dashline{1}(5.50,11.50)(9.50,11.50)
\dashline{1}(4.50,7.50)(10.50,7.50)
\dashline{1}(0.50,7.50)(5.50,11.50)
\dashline{1}(14.50,7.50)(9.50,11.50)
\end{picture}&
\begin{picture}(20.00,8.00)
\special{em:linewidth 0.4pt} \unitlength 0.80mm
\linethickness{1pt}

\put(7.50,10.00){\makebox(0,0)[cc]{\tiny{$3$}}}
\put(5.50,5.00){\makebox(0,0)[cc]{\tiny{$2$}}}
\put(12.50,5.00){\makebox(0,0)[cc]{\tiny{$4$}}}
\put(3.00,0){\makebox(0,0)[cc]{\tiny{$3$}}}
\put(8.00,0){\makebox(0,0)[cc]{\tiny{$0$}}}
\linethickness{0.3pt}
\dashline{1}(5.50,11.50)(5.50,8.50) 
\dashline{1}(5.50,11.50)(9.50,11.50) 
\dashline{1}(5.50,8.50)(10.50,3.00)  %
\dashline{1}(9.50,11.50)(14.50,6.00) %
\dashline{1}(11.50,3.00)(14.50,3.00) %
\dashline{1}(14.50,3.00)(14.50,6.00)

\dashline{1}(0.50,-2.00)(10.50,-2.00)
\dashline{1}(0.50,2.00)(0.50,-2.00)
\dashline{1}(10.50,2.00)(10.50,-2.00)
\dashline{1}(10.50,2.00)(5.50,7.50)
\dashline{1}(0.50,2.00)(5.50,7.50)
\end{picture}
&
\begin{picture}(20.00,8.00)
\special{em:linewidth 0.4pt} \unitlength 0.80mm
\linethickness{1pt}
\put(7.50,10.00){\makebox(0,0)[cc]{\tiny{$4$}}}
\put(2.50,5.00){\makebox(0,0)[cc]{\tiny{$3$}}}
\put(12.50,5.00){\makebox(0,0)[cc]{\tiny{$5$}}}
\put(7.50,0){\makebox(0,0)[cc]{\tiny{$4$}}}
\linethickness{0.3pt}
\dashline{1}(5.50,11.50)(5.50,8.50) 
\dashline{1}(5.50,11.50)(9.50,11.50) 
\dashline{1}(5.50,8.50)(10.50,3.00)  %
\dashline{1}(9.50,11.50)(14.50,6.00) %
\dashline{1}(11.50,3.00)(14.50,3.00) %
\dashline{1}(14.50,3.00)(14.50,6.00)
%
\dashline{1}(0.50,6.50)(0.50,3.50) 
\dashline{1}(0.50,6.50)(4.50,6.50) 
\dashline{1}(0.50,3.50)(5.50,-2.00)  %
\dashline{1}(4.50,6.50)(9.50,1.00) %
\dashline{1}(6.50,-2.00)(9.50,-2.00) %
\dashline{1}(9.50,-2.00)(9.50,1.00)
\end{picture}
&\cdots\\
&&&&&\vspace{-0.2cm}\\
\hline
&&&&&\\

\begin{picture}(20.00,8.00)
\special{em:linewidth 0.4pt} \unitlength 0.80mm
\linethickness{1pt}
\put(7.50,10.00){\makebox(0,0)[cc]{\tiny{$0$}}}
\put(7.50,5.00){\makebox(0,0)[cc]{\tiny{$2$}}}
\put(7.50,0){\makebox(0,0)[cc]{\tiny{$0$}}}
\end{picture}
&
\begin{picture}(20.00,0.00)
\special{em:linewidth 0.4pt} \unitlength 0.80mm
\linethickness{1pt}
\put(7.50,10.00){\makebox(0,0)[cc]{\tiny{$1$}}}
\put(7.50,5.00){\makebox(0,0)[cc]{\tiny{$2$}}}
\put(7.50,0){\makebox(0,0)[cc]{\tiny{$1$}}}
\linethickness{0.3pt}
\end{picture}
&
\begin{picture}(20.00,0.00)
\special{em:linewidth 0.4pt} \unitlength 0.80mm
\linethickness{1pt}
\put(7.50,10.00){\makebox(0,0)[cc]{\tiny{$2$}}}
\put(2.50,5.00){\makebox(0,0)[cc]{\tiny{$0$}}}
\put(7.50,5.00){\makebox(0,0)[cc]{\tiny{$1$}}}
\put(12.50,5.00){\makebox(0,0)[cc]{\tiny{$3$}}}
\put(7.50,0){\makebox(0,0)[cc]{\tiny{$2$}}}
\end{picture}&
\begin{picture}(20.00,8.00)
\special{em:linewidth 0.4pt} \unitlength 0.80mm
\linethickness{1pt}
\put(7.50,10.00){\makebox(0,0)[cc]{\tiny{$3$}}}
\put(2.50,5.00){\makebox(0,0)[cc]{\tiny{$2$}}}
\put(12.50,5.00){\makebox(0,0)[cc]{\tiny{$4$}}}
\put(7.50,0){\makebox(0,0)[cc]{\tiny{$3$}}}
\linethickness{0.3pt}
\end{picture}
&
\begin{picture}(20.00,8.00)
\special{em:linewidth 0.4pt} \unitlength 0.80mm
\linethickness{1pt}
\put(7.50,10.00){\makebox(0,0)[cc]{\tiny{$4$}}}
\put(2.50,5.00){\makebox(0,0)[cc]{\tiny{$3$}}}
\put(12.50,5.00){\makebox(0,0)[cc]{\tiny{$5$}}}
\put(7.50,0){\makebox(0,0)[cc]{\tiny{$4$}}}
\end{picture}
&\cdots
\\
\vspace{-0.2cm}
&&&&\\
\hline
\end{array}
\\
\end{eqnarray*}

\normalsize

Note that there are two cell modules labelled by $1$, since the truncation $\mathbb{H}_k$ of our quasi-hereditary algebra $\mathbb{D}_k$ is not compatible with the quasi-hereditary ordering. Altogether $\cB$ is equivalent to the category of finite dimensional modules over the path algebra of the quiver
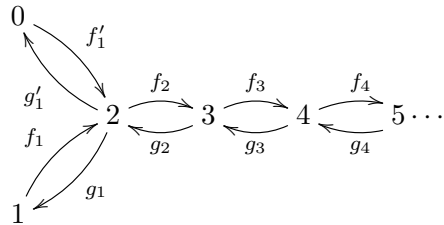
\begin{eqnarray*}
\begin{xy}
  \xymatrix{
  0\ar@/^/[dr]^{f_1'}\\
  &2\ar@/^/[ul]^{g_1'}\ar@/^/[dl]^{g_1}\ar@/^/[r]^{f_2}
  &3\ar@/^/[l]^{g_2}\ar@/^/[r]^{f_3}
  &4\ar@/^/[l]^{g_3}\ar@/^/[r]^{f_4}
  &5\ar@/^/[l]^{g_4}\cdots\\
 1\ar@/^/[ur]^{f_1}}
\end{xy}
\end{eqnarray*}
\normalsize
with relations $f_{i+1}\circ f_i=0=g_{i}\circ g_{i+1}$, $g_{i+1}\circ f_{i+1}=f_{i}\circ g_{i}$, $g_1'\circ f_1=0=g_1\circ f_1'$, $f_1'\circ g_1'=g_1\circ f_1$ and  $f_2\circ f_1'=0=g_1'\circ g_2$. The last two relations are the relations from $I$ and force the projective modules to become injective, a property which is well-known to hold in $\cB$, see \cite{BKN}. This category was studied in detail by Gruson \cite{Gruson}, who also showed that it is representation finite and classified the indecomposable modules.
\end{ex}

We want to stress that our approach indicates a new relationship between the representation theory of the classical orthogonal Lie algebra and the representation theory of the orthosymplectic Lie superalgebras. It is for instance not a special case of the general super duality of \cite{CLW}, where more general parabolic categories would be involved.

Conjecture~\ref{conj} with the expected equivalence of categories will be studied in detail in \cite{SS} by connecting it to the combinatorics of Fock space.

\subsection*{Cyclotomic VW-algebras}
The strategy of the proof will be built on a super higher Schur-Weyl duality relating the orthosymplectic Lie algebra with level 2 cyclotomic Nazarov-Wenzl algebras or short VW-algebras from \cite{AMR}. These are certain degenerate versions of cyclotomic Birman-Murakawi-Wenzl algebras.

In part II of this series \cite{ES2} we develop in detail the graded versions of level 2 cyclotomic Nazarov-Wenzl algebras and show their blocks are isomorphic to certain idempotent truncations of the algebras $\mathbb{D}_k$. In particular, they inherit from our description here a positive Koszul grading and a geometric interpretation in terms of perverse sheaves on isotropic Grassmannians. In the same paper we will also relate the non-semisimple versions of Brauer algebras \cite{CMD1}, \cite{CMD2} to idempotent truncations of our algebras $\mathbb{D}_k$.
This gives in particular a conceptual explanation, why {\it all} the Brauer algebras are determined by the combinatorics of Weyl groups of type $D$, an amazing phenomenon first observed by Martin, Cox and DeVisscher.

\subsection*{Categorified coideal subalgebras}
In the same paper we will also study the action of translation functors in detail. Here a completely new phenomenon appears. Besides their natural categorifications of Hecke algebras, translation functors were usually used in type $A$ to categorify  actions of (super) quantum groups, see for instance \cite{BS3}, \cite{FKS}, \cite{Antonio}, or \cite{Mazcat} for an overview. In particular they give standard examples of $2$-categorifications of Kac-Moody algebras in the sense of \cite{Rou2KM} and \cite{KL} and the combinatorics of the quantum group and the Hecke algebras are directly related, see e.g. \cite{FKK}, \cite{BS3} and \cite{Antonio}. This coincidence fails outside type $A$. Instead we obtain the action of a quantized fixed point Lie algebra $\mathfrak{gl}(n)^\sigma$ for an involution $\sigma$. Although these algebras were studied in detail by algebraic methods under the name {\it coideal subalgebras}, \cite{Letzter}, it seems a geometric approach is completely missing.  Our main Theorem~\ref{intro} will then finally provide for the first time such a geometric approach and also gives a first instance of a categorifications of modules for these algebras. For these coideal subalgebras the graded versions of cyclotomic VW-algebras play the analogous role to cyclotomic {\it quiver Hecke algebra} or {\it Khovanov-Lauda algebras} attached to the universal enveloping algebra $\mathcal{U}(\mathfrak{g})$ in \cite{Rou2KM}, \cite{KL}.

\subsection*{Hermitian symmetric pairs}
Our focus on the type $D$ case might look artificial. Instead one could consider all $G/P$ for classical pseudo hermitian symmetric cases, \cite{BHKostant}, \cite{BoeCollingwood}. In fact, the natural inclusions of algebraic groups $SO(2n+1,\mC)\hookrightarrow SO(2n+2,\mC)$ and
$Sp(2n,\mC)\hookrightarrow SL(2n,\mC)$ induce isomorphisms of the partial flag varieties for the pairs $(G,P)$ of type $(B_n, A_{n-1})$ and $(D_{n+1}, A_{n})$ respectively $(C_n,C_{n-1})$ and $(A_{2n-1}, A_{2n-2})$ which are compatible with the Schubert stratification,  \cite[3.1]{BrionPolo}. Hence the corresponding categories of perverse sheaves are naturally isomorphic. Therefore, the algebras from \cite{BS1} and our algebras $\mathbb{D}_{n+1}$ introduced here govern these cases as well. In particular, the non-simply laced case can be "unfolded" to the simply laced case for larger rank. For the remaining classical symmetric pairs, apart from type $(D_n,D_{n-1})$ one might use the fact that the principal blocks of category $\cO$ for type $B_n$ and $C_n$ are equivalent via Soergel's combinatorial description, \cite{Sperv}. Using our connection of the algebra $\mathbb{D}_k$ with the geometry of Springer fibres and Slodowy slices, there is also a Langlands's dual picture of this unfolding relating different Springer fibers and Slodowy slices. Details which will appear in a subsequent paper.

\subsection*{Acknowledgements} {We like to thank Antonio Sartori and in particular Vera Serganova for many helpful discussions.}

\renc{\thetheorem}{\arabic{section}.\arabic{theorem}}

\section{The isotropic Grassmannian, weight and cup diagrams}

\subsection{The isotropic Grassmannian}
\label{section21}
We fix a natural number $k$. Let $n=2k$ and consider the group $\op{SL}(n,\mC)$ with standard Borel $B$ given by all upper triangular
matrices. Let $P_{k,k}$ be the standard parabolic given by all matrices $A=(a_{i,j})_{1\leq i,j\leq n}$ where $a_{i,j}=0$ for $i>k, j\leq
k$ and $\mathcal{X}_{k,k}=\op{Gr}_k(k)=\op{SL}(n,\mC)/P_{k,k}$ the corresponding Grassmannian variety of $k$-dimensional subspaces in $\mC^k$.\\

We fix the non-degenerate quadratic form
$$Q(x_1, \ldots, x_{2k}) = x_1x_{2k} + x_2x_{2k-1} + \cdots + x_kx_{k+1}$$
on $\mC^n$. The space $\{V \in \mathcal{X}_{k,k} \mid Q_{|V} = 0\}$ has two connected components (namely the two equivalence classes for the
relation $V\sim V'$ if $\dim{V\cap V'}-k$ is even). Let $Y_k$
be the component containing the point $\{x_{k+1}=\cdots =x_n = 0\}$. The subgroup $\op{SO}(n,\mC)$ of
$\op{SL}(n,\mC)$ of transformations preserving $Q$ acts on $Y_k$. Moreover, $$Y_k = \op{SO}(n,\mC)/(SO(n,\mC)\cap P_{k,k}),$$ the
isotropic {\it Grassmannian of type $D_k$}. The group $B_D=\op{SO}(n,\mC)\cap B$ is a Borel subgroup of $\op{SO}(n,\mC)$. We fix the
stratification of $Y_k$ by $B_D$-orbits. Note that the latter are just the intersections $Y_\la=\Omega_\la\cap Y_k$ of the Schubert cells
$\Omega_\la$ of $\mathcal{X}_{k,k}$ with $Y_k$.
Taking the natural labelling of Schubert cells of $\mathcal{X}_{k,k}$ by partitions of Young diagrams fitting into a $k\times k$-box (such that the dimension of the cell equals the number of boxes), then
$\Omega_\la\not=\emptyset$ precisely if its Young diagram is {\it symmetric} in the sense that it is fixed under reflection about the
diagonal and the number of boxes on the diagonal is even. Let $\Omega_k$ be the set of symmetric Young diagrams fitting into an
$n\times n$-box.
For instance, the set $\Omega_4$ labelling the strata of $Y_4$ is the following
\begin{eqnarray}
\label{Young}
{\tiny\Yvcentermath1\yng(4,4,4,4)\quad\quad\yng(4,4,2,2)\quad\quad\yng(4,3,2,1)\quad\quad\yng(4,2,1,1)\quad\quad\yng(3,3,2)\quad\quad\yng(3,2,1)
\quad\quad\yng(2,2)\quad\quad\emptyset}
\end{eqnarray}

To such a Young diagram, we associate a $\{+,-\}$-sequence of length $n=2k$ as follows: starting at the bottom left corner of the $n\times n$-square walk
along the right boundary of the diagram to the top right corner of the  $n\times n$-square - each step either to the right, encoded by a $+$, or up, encoded
by a $-$. If the diagram labels a $B_D$-orbit $Y_\la$ then the resulting sequence $s$ is automatically antisymmetric of length $n=2k$. We denote by
$\Ssym$ the set of all antisymmetric sequences obtained in this way, and by $\S$ those given by the second halves of any $s\in \Ssym$. Hence $\S$ consists of sequences of length $k$ with an even number of $+$'s.
In the example \eqref{Young} we get the following set $\mathbb{S}_4$ of sequences of length $k=4$
\small
\begin{equation*}
{----}\;,\;{++--}\;,\;{+-+-}\;,\;{-++-}\;,\;
{+--+}\;,\;{-+-+}\;,\;{--++}\;,\;
{++++}
\end{equation*}
\normalsize
\subsection{Weights and linkage}
\label{linkage}
In the following we identify $\mZ_{> 0}$ with integral points, called {\it vertices}, on the positive numberline. A {\it (diagrammatic) weight} $\la$ is a diagram obtained by labelling each of the vertices by  $\circ$ (nought), $\times$ (cross), $\down$ (down) or
$\up$ (up) such that the vertex $i$ is labelled $\circ$ for $i\gg 0$. For instance,
\begin{eqnarray}
\label{lla}
\la&=&\overset{1}{\down}\overset{2}\up\overset{3}\up\overset{4}\up\overset{5}\up\overset{6}\down\overset{7}\down\overset{8}\down\overset{9}\down\overset{10}\up\overset{11}\circ\overset{12}\circ \, \cdots,\\
\label{lla2}
\la&=&{\up}\times\times\up\circ\down\down\circ\up\up\circ\circ\, \cdots,
\end{eqnarray}
are examples of weights. The $\cdots$'s indicate that all further entries of the sequence are equal to $\circ$ and the label above the symbols indicates the positions. As in the second example we will usually omit these numbers. We denote by $\mX$ the set of diagrammatic weights. To $\lambda \in \mX$ we attach the sets
$$
P_*(\lambda) = \{ i \in \mZ_{> 0} \mid \lambda_i = * \} \text{ if } \lambda \in \mX,$$
where $* \in \{\up, \down, \times, \circ \}$.

For the weights $\eqref{lla}$ and \eqref{lla2}
we have $P_\up(\lambda)=\{2,3,4,5,10\}$, $P_\down(\lambda)=\{1,6,7,8,9\}$, $P_\times(\lambda) = \emptyset$, and $P_\circ(\lambda) = \{ i \in \mZ_{> 0} \mid i \geq 11 \}$ and
$P_\up(\lambda)=\{1,4,9,10\}$, $P_\down(\lambda)=\{6,7\}$, $P_\times(\lambda) = \{2,3\}$, and $P_\circ(\lambda) = \{5,8, i\mid i \geq 11 \}$ respectively.
\

Two diagrammatic weights $\la,\mu\in\mX$ are {\it linked} if $\mu$ is obtained from $\la$ by finitely many combinations of {\it basic linkage moves} of swapping neighboured labels $\up$ and $\down$ (ignoring symbols $\circ$ and $\times$) or swapping the pairs  $\up\up$ and $\down\down$ at the first two positions (again ignoring symbols $\circ$ and $\times$) in the sequence. The equivalence classes of the linkage relation are called {\it blocks}.

\begin{ex}
The weight $\la$ from \eqref{lla}
is linked to
$\mu={\down}\down\up\down\up\down\down\down\down\up\circ\circ \, \cdots$
but not to
$\mu'={\down}\down\up\up\up\down\down\down\down\up\circ\circ \, \cdots.$
\end{ex}

Obviously a block $\Lambda$ is then given by fixing the positions of all $\times$ and $\circ$ and the parity $\ov{|P_\up(\lambda)|}$ of the number of $\up$'s. Formally a block $\Lambda$ is the data of a parity, either $\ov{0}$ or $\ov{1}$, and a \emph{block diagram} $\Theta$, that is a sequence $\Theta_i$ of symbols $\times$, $\circ$, $\bullet$, indexed by $i\in\mZ_{>0}$ with the same conditions as for diagrammatic weights, ie. all but finitely many $\Theta_i$ are equal to $\circ$. Attached to a block $\Lambda$ or a block diagram $\Theta$ we have the sets $P_\times(\Lambda)=P_\times(\Theta)$, $P_\circ(\Lambda)=P_\circ(\Theta)$, and $P_\bullet(\Lambda)=P_\bullet(\Theta)$ defined analogously to the diagrammatic weights. The block $\Lambda$ is then the equivalence class of weights
\begin{eqnarray}
\label{defblock}
\Lambda=\Lambda_\Theta^{\ov{\epsilon}} = \left\lbrace \begin{array}{l|c}
& P_\up(\lambda) \cup P_\down(\lambda) = P_\bullet(\Theta),\\
\lambda \in \mX  & P_\times(\lambda) = P_\times(\Theta)\text{, } P_\circ(\lambda) = P_\circ(\Theta),\\
& \ov{|P_\up(\lambda)|} = \ov{\epsilon}
\end{array} \right\rbrace
\end{eqnarray}

 A block $\Lambda$ is called {\it regular} if $P_\times(\Lambda) = \emptyset$ and {\it principal} of type $D_k$ if additionally $P_\bullet(\Lambda) = \{1,\ldots,k\}$. A weight $\lambda$ is called {\it regular} if it belongs to a regular block. For instance $\eqref{lla}$ is a principal weight, and thus also a regular weight.

For fixed $k$ the set of principal weights of type $D_k$ decomposes into exactly two blocks, the {\it even} block $\Lambda_k^{\ov{0}}$, where each weight has an even number of $\up$'s, and the {\it odd} block $\Lambda_k^{\ov{1}}$, where the weight has an odd number of $\up$'s. Both blocks correspond to the same sequence
$$\Lambda = \underset{k}{\underbrace{{\bullet}\, \cdots \, \bullet}}\circ\circ \, \cdots.$$ Example~\ref{exn4} shows the two blocks for $k=4$.

\begin{lemma} \label{bijection_S_Lambda}
Assigning to $s=(x_1,\ldots, x_k)\in\S$ the diagrammatic weight $\la=\la(s)$ supported on the integers with $\la(s)_i = \up$ if $x_i = +$, $\la(s)_i = \down$ if $x_i=-$, and $\la(s)_i = \circ$ otherwise defines a bijection between $\S$ and $\Lambda_k^{\ov{0}}$.
\end{lemma}
\begin{proof}
Obviously $\la(s)$ is principal and the assignment $s\mapsto \la(s)$ is injective. By cardinality count it is also surjective.
\end{proof}
The basic linkage moves induces a partial ordering on each block, the {\it reversed Bruhat order} by declaring that changing a pair of labels $\up\down$ to $\down\up$ or a pair
$\down\down$ to $\up\up$ makes a weight smaller in this reversed Bruhat order. Repeatedly applying the basic moves implies

\begin{lemma}
\label{lem:Bruhat}
Changing a (not necessarily neighboured) pair of labels $\up\down$ to $\down\up$ or a pair
$\down\down$ to $\up\up$ makes a weight smaller in the Bruhat order.
\end{lemma}

For $\mathbb{S}_4$ as above the associated weights via Lemma~\ref{bijection_S_Lambda} (omitting $\circ$'s) are
\begin{eqnarray*}
&{\down\down\down\down}\;,\;{\up\up\down\down}\;,\;{\up\down\up\down}\;,\;{\down\up\up\down}\;,\;
{\up\down\down\up}\;,\;{\down\up\down\up}\;,\;{\down\down\up\up}\;,\;
{\up\up\up\up}&
\end{eqnarray*}
If we label them $\la_1, \ldots, \la_8$ then $\la_1>\la_2>\la_3>\la_4,\la_5>\la_6>\la_7>\la_8$.

\begin{lemma}
Lemma~\ref{bijection_S_Lambda} defines an isomorphism of posets between $\S$, equipped with the partial order given by containment of the corresponding Young diagrams with the smaller diagram corresponding to the smaller element, and $\Lambda_k^{\ov{0}}$ equipped with the Bruhat order.
\end{lemma}
\begin{proof}
This follows directly from the definitions.
\end{proof}

\subsection{The Weyl group of type $D$}
\label{typeD}
To make the connection with the Bruhat order on Coxeter groups let $W=W(k)$ be the Weyl group of type $D_k$. It is generated by simple reflections $s_i$,  $0\leq i\leq k-1$ subject to the relations $s_i^2=e$ for
all $i$, and for $i,j\not=0$  $s_is_j=s_js_i$ if $|i-j|>1$ and $s_is_js_i=s_js_is_j$ if $|i-j|=1$, and additionally $s_0s_2s_0=s_2s_0s_2$
and $s_0s_j=s_js_0$ for $j\not=2$.
 It has two parabolic subgroups isomorphic to the symmetric group $S_k$, namely $W_0$ generated by $s_i$, $i\not=0$ and $W_1$ generated
 by $s_i$, $i\not=1$. Let $W^0$ and $W^1$ be the set of shortest coset representatives for $W_0\backslash W$ and $W_1\backslash W$
 respectively.

The group $W$ acts from the right on the sets $\Lambda_k^{\ov{p}}$ for $p=\ov{0}, \ov{1}$ as follows: $s_i$ for $i>1$ swaps the $i$th and $(i+1)$th label, whereas
$s_1$ swaps the first two in case they are $\up\down$ or $\down\up$ and is the identity otherwise whereas $s_0$ turns $\up\up$ into $\down\down$ and vice versa at the first two positions and is the identity otherwise. The sequence consisting of $k$ $\down$'s is
fixed by the parabolic $W_0$, whereas the sequence consisting of one $\up$ followed by $k-1$ $\down$'s is fixed by $W_1$. Sending the identity element $e\in W$ to either of these
sequences $s$ and $w$ to $s.w$  defines canonical bijections $\Phi^0$ and $\Phi^1$ between the sets $W^0$ (resp. $W^1$) and $\Lambda_k^{\ov{0}}$
(resp. $\Lambda_k^{\ov{1}}$). In the above example the elements from $W$ are the following in the respective cases:
\begin{eqnarray}
{e}\quad {s_0}\quad {s_0s_2}\quad {s_0s_2s_1}\quad {s_0s_2s_3}\quad {s_0s_2s_3s_1}\quad {s_0s_2s_1s_3s_2}\quad {s_0s_2s_1s_3s_2s_0}\\
{e}\quad {s_1}\quad {s_1s_2}\quad {s_1s_2s_0}\quad {s_1s_2s_3}\quad {s_1s_2s_3s_0}\quad {s_1s_2s_3s_0s_2}\quad {s_1s_2s_3s_0s_2s_1}
\end{eqnarray}
More generally, blocks $\Lambda$ with $|P_\bullet(\Lambda)|=k$ can be identified with $W$-orbits, where $W$ acts on the vertices not labelled $\circ$ or $\times$. Moreover, our combinatorially defined partial ordering on weights agrees with the ordering induced from the reversed Bruhat order on $W$ or reversed inclusion order on
Schubert cells.

\begin{remark}
{\rm Note that flipping the first symbol in $\la(s)$ from Lemma~\ref{bijection_S_Lambda} would alternatively give us a bijection between $\mathbb{S}_4$ and $\Lambda_k^{\ov{1}}$. At the moment we will stick with the even choice, but one should think of the second choice as in fact passing to the second natural choice of parabolic of Dynkin type $A_k$ in the complex semisimple Lie algebra $\mg$ of type $D_k$. It corresponds to different, but canonically equivalent parabolic category $\cO(\mg)$'s related by a {\it parity switch functor}. This functor plays an important role in \cite{ES2} where we relate our setup to cyclotomic versions of VW-algebras (degenerate versions of BMW-algebras). The name parity switch functor should also indicate that it plays the analogous role to parity switch functors in the theory of Lie superalgebras. It should not be confused with the slightly mysterious extra translation functors, called switch functors, appearing in the finite dimensional representation theory of the orthosymplectic Lie superalgebras, \cite{GS1}, \cite{GS2}. In \cite{ES2} the switch functors are used to categorify a skew Howe duality for quantum symmetric pairs using special cases of Letzter's coideal subalgebras \cite{Letzter}.}
\end{remark}

\section{Cup diagrams and the quiver}
The goal of this section is to introduce the required combinatorics to define the type $D$ generalized Khovanov algebra and describe the graph of the quiver (Definition \eqref{Extquivdef} and Corollary~\ref{CorCartan}).
Throughout this section we fix a block $\Lambda$ with its diagram (also denoted $\Lambda$). We abbreviate
$$ P_\bullet = P_\bullet(\Lambda), \qquad P_\times = P_\times(\Lambda), \qquad P_\circ = P_\circ(\Lambda).$$

Consider the semi-infinite strip
$$R = \{ (x,y) \in \mR_{\geq 0} \times \mR \mid 1 \geq y \geq -1 \} \subset \mR_{\geq 0} \times \mR.$$
as the union of the two smaller strips $R^- = \{ (x,y) \in R \mid 0 \geq y \geq -1\}$ and $R^+ = \{ (x,y) \in R \mid 1 \geq y \geq 0 \}$.
For $* \in \{ \bullet, \circ, \times \}$ and $t \in \{-1,0,1\}$ we denote $$ P_*^t = \{(p,t) \mid p \in P_* \}.$$ Later on we will need the following statistic for $p \in P_\bullet$ and $(p,t) \in P_\bullet^t$
$$\pos((p,t))=\pos(p):=\#\{ q \in P_\bullet \mid q \leq p\} \in \mZ_{>0}.$$
For principal blocks $\pos(p)$ is nothing else than the position of $p$.  For $r \in \mR$ we denote $T_r:\mR_{\geq 0} \times \mR \longrightarrow \mR_{\geq 0} \times \mR$, $T_r(x,y) = (x,y+r)$ the vertical translations by $r$.

\subsection{Cup diagrams}
\begin{definition}
\label{defarcs}
A subset $\gamma \subset R^-$ is called an \emph{arc} if there exists an injective continuous map $\tilde{\gamma}:[0,1] \rightarrow R$ with image $\gamma$ such that
\begin{enumerate}
\item $\widetilde{\gamma}(0) \in P_\bullet^0$,
\item $\widetilde{\gamma}(1) \in P_\bullet^0$ or $\widetilde{\gamma}(1) \in \mR_{\geq 0} \times \{-1\}$,
\item $\widetilde{\gamma}( (0,1) ) \subset (R^-)^\circ$, the interior of $R^-$.
\end{enumerate}
If the arc satisfies the first condition in (2) we call it a \emph{cup}, if it satisfies the second we call it a \emph{ray}.
\end{definition}

\begin{definition} \label{cup_diagram}
An \emph{undecorated cup diagram} $c$ (with respect to the block diagram $\Lambda$) is a finite union of arcs $\{ \gamma_1,\ldots,\gamma_r \}$, such that
\begin{enumerate}
\item $\gamma_i \cap \gamma_j = \emptyset$ for $i \neq j$,
\item for every $(p,0) \in P_\bullet^0$ there exists an arc $\gamma_{i(p)}$ such that $(p,0) \in \gamma_{i(p)}$.
\end{enumerate}
\end{definition}

In the following we will always consider two diagrams as the same if they differ only by an isotopy of $R$ fixing the axis of $\mR_{\geq 0} \times \{0\}$ pointwise.

Following \cite{ES1}, we have the notion of a decorated cup diagram, a generalization of the setup from \cite{BS1}:

\begin{definition}
Let $c$ be a cup diagram with set of arcs $\{\gamma_1,\ldots, \gamma_r\}$.Then a \emph{decoration} of $c$ is a map
$${\rm deco}_c:\quad\{\gamma_1,\ldots, \gamma_r\} \rightarrow \{0,1\},$$
such that ${\rm deco}_c(\gamma_i)=1$ implies that an arbitrary point in the interior of $\gamma_i$ can be connected with the left boundary of $R$ by a path not intersecting any other arc. \footnote{This condition is very well-known in the context of blob algebras as studied for instance in \cite{Martinblob}, and generalized Temperley-Lieb algebras \cite{Green}. It arises naturally from the theory in \cite{BS1} by a `folding' procedure, see \cite{LS}.}
If the value for $\gamma_i$ is $1$ we call the arc \emph{dotted}, otherwise we call it \emph{undotted}. A cup diagram together with a decoration is called a \emph{(decorated) cup diagram}. We visualize this by putting a "$\bullet$" at an arbitrary interior point of the arc in the underlying undecorated cup diagram.
\end{definition}

\begin{ex}
\label{excup}
The following shows two examples of possible decorations for the same undecorated cup diagram with the block diagram listed on top.

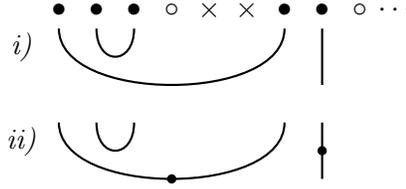
\begin{eqnarray*}
\begin{tikzpicture}[thick,scale=0.50]
\node at (0,.5) {$\bullet$};
\node at (1,.5) {$\bullet$};
\node at (2,.5) {$\bullet$};
\node at (3,.5) {$\circ$};
\node at (4,.5) {$\times$};
\node at (5,.5) {$\times$};
\node at (6,.5) {$\bullet$};
\node at (7,.5) {$\bullet$};
\node at (8,.5) {$\circ$};
\node at (9,.5) {$\cdots$};

\node at (-1,-.5) {i)};
\draw (0,0) .. controls +(0,-2) and +(0,-2) .. +(6,0);
\draw (1,0) .. controls +(0,-1) and +(0,-1) .. +(1,0);
\draw (7,0) -- (7,-1.5);

\begin{scope}[yshift=-2.5cm]
\node at (-1,-.5) {ii)};
\draw (0,0) .. controls +(0,-2) and +(0,-2) .. +(6,0);
\fill (3,-1.5) circle(3.5pt);
\draw (1,0) .. controls +(0,-1) and +(0,-1) .. +(1,0);
\draw (7,0) -- (7,-1.5);
\fill (7,-.75) circle(3.5pt);
\end{scope}

\end{tikzpicture}
\end{eqnarray*}
Note that the interior cup is not allowed to carry a dot.
\end{ex}

\begin{definition}
\label{decoratedcups}
For a weight $\la \in \Lambda$, the associated \emph{(decorated) cup diagram}, denoted $\underline \la$ is obtained by the following procedure.
\begin{enumerate}[(C1)]
\item First connect neighboured vertices in $P_\bullet$ labelled $\down\up$ successively by a cup (ignoring already joint vertices and vertices not in $P_\bullet$) as long as possible. (Note that the result is independent of the order the connections are made).
\item Attach to each remaining $\down$ a vertical ray.
\item Connect from left to right pairs of two neighboured $\up$'s by cups.
\item If a single $\up$ remains, attach a vertical ray.
\item Finally define the decoration ${\rm deco}_c$ such that all arcs created in step (1) and (2) are undotted whereas the ones created in steps (3) and (4) are dotted.
\end{enumerate}
\end{definition}

We denote by $\op{C}_\Lambda$ the set of cup diagrams associated with a block $\Lambda$. For instance the cup diagram (i) from Example~\ref{excup} corresponds to the weight $\down\down\up\circ\times\times\up\down\circ\cdots$, whereas the diagram (ii) corresponds to the weight $\up\down\up\circ\times\times\up\up\circ\cdots$.
\begin{ex}
\label{exn4}
For $k=4$ we get the following set of even regular weights with the set  $\op{C}_k^{\ov{0}}$ of corresponding cup diagrams

\begin{eqnarray} \label{sequ1}
\begin{tikzpicture}[thick,scale=0.5]
\begin{scope}[xshift=-3cm]
\node at (.8,.8) {\small $\down\down\down\down$};
\draw (0,0) -- +(0,-.8);
\draw (.5,0) -- +(0,-.8);
\draw (1,0) -- +(0,-.8);
\draw (1.5,0) -- +(0,-.8);

\node at (3.8,.8) {\small $\up\up\down\down$};
\draw (3,0) .. controls +(0,-.5) and +(0,-.5) .. +(.5,0);
\fill (3.25,-.35) circle(2.5pt);
\draw (4,0) -- +(0,-.8);
\draw (4.5,0) -- +(0,-.8);

\node at (6.8,.8) {\small $\up\down\up\down$};
\draw (6,0) -- +(0,-.8);
\fill (6,-.4) circle(2.5pt);
\draw (6.5,0) .. controls +(0,-.5) and +(0,-.5) .. +(.5,0);
\draw (7.5,0) -- +(0,-.8);

\node at (9.8,.8) {\small $\down\up\up\down$};
\draw (9,0) .. controls +(0,-.5) and +(0,-.5) .. +(.5,0);
\draw (10,0) -- +(0,-.8);
\fill (10,-.4) circle(2.5pt);
\draw (10.5,0) -- +(0,-.8);
\end{scope}

\begin{scope}[xshift=9cm, yshift=2cm]
\node at (.8,-1.2) {\small $\up\down\down\up$};
\draw (0,-2) -- +(0,-.8);
\fill (0,-2.4) circle(2.5pt);
\draw (0.5,-2) -- +(0,-.8);
\draw (1,-2) .. controls +(0,-.5) and +(0,-.5) .. +(.5,0);

\node at (3.8,-1.2) {\small $\down\up\down\up$};
\draw (3,-2) .. controls +(0,-.5) and +(0,-.5) .. +(.5,0);
\draw (4,-2) .. controls +(0,-.5) and +(0,-.5) .. +(.5,0);

\node at (6.8,-1.2) {\small $\down\down\up\up$};
\draw (6,-2) .. controls +(0,-1) and +(0,-1) .. +(1.5,0);
\draw (6.5,-2) .. controls +(0,-.5) and +(0,-.5) .. +(.5,0);

\node at (9.8,-1.2) {\small $\up\up\up\up$};
\draw (9,-2) .. controls +(0,-.5) and +(0,-.5) .. +(.5,0);
\fill (9.25,-2.35) circle(2.5pt);
\draw (10,-2) .. controls +(0,-.5) and +(0,-.5) .. +(.5,0);
\fill (10.25,-2.35) circle(2.5pt);
\end{scope}
\end{tikzpicture}
\end{eqnarray}
and odd regular weights with the set  $\op{C}_k^{\ov{1}}$ of corresponding cup diagrams:
\begin{eqnarray} \label{sequ2}
\begin{tikzpicture}[thick,scale=.5]
\begin{scope}[xshift=-3cm]
\node at (.8,.8) {\small $\up\down\down\down$};
\draw (0,0) -- +(0,-.8);
\fill (0,-.4) circle(2.5pt);
\draw (.5,0) -- +(0,-.8);
\draw (1,0) -- +(0,-.8);
\draw (1.5,0) -- +(0,-.8);

\node at (3.8,.8) {\small $\down\up\down\down$};
\draw (3,0) .. controls +(0,-.5) and +(0,-.5) .. +(.5,0);
\draw (4,0) -- +(0,-.8);
\draw (4.5,0) -- +(0,-.8);

\node at (6.8,.8) {\small $\down\down\up\down$};
\draw (6,0) -- +(0,-.8);
\draw (6.5,0) .. controls +(0,-.5) and +(0,-.5) .. +(.5,0);
\draw (7.5,0) -- +(0,-.8);

\node at (9.8,.8) {\small $\up\up\up\down$};
\draw (9,0) .. controls +(0,-.5) and +(0,-.5) .. +(.5,0);
\fill (9.25,-0.35) circle(2.5pt);
\draw (10,0) -- +(0,-.8);
\fill (10,-.4) circle(2.5pt);
\draw (10.5,0) -- +(0,-.8);
\end{scope}

\begin{scope}[yshift=2cm, xshift=9cm]
\node at (.8,-1.2) {\small $\down\down\down\up$};
\draw (0,-2) -- +(0,-.8);
\draw (0.5,-2) -- +(0,-.8);
\draw (1,-2) .. controls +(0,-.5) and +(0,-.5) .. +(.5,0);

\node at (3.8,-1.2) {\small $\up\up\down\up$};
\draw (3,-2) .. controls +(0,-.5) and +(0,-.5) .. +(.5,0);
\fill (3.25,-2.35) circle(2.5pt);
\draw (4,-2) .. controls +(0,-.5) and +(0,-.5) .. +(.5,0);

\node at (6.8,-1.2) {\small $\up\down\up\up$};
\draw (6,-2) .. controls +(0,-1) and +(0,-1) .. +(1.5,0);
\draw (6.5,-2) .. controls +(0,-.5) and +(0,-.5) .. +(.5,0);
\fill (6.75,-2.75) circle(2.5pt);

\node at (9.8,-1.2) {\small $\down\up\up\up$};
\draw (9,-2) .. controls +(0,-.5) and +(0,-.5) .. +(.5,0);
\draw (10,-2) .. controls +(0,-.5) and +(0,-.5) .. +(.5,0);
\fill (10.25,-2.35) circle(2.5pt);
\end{scope}
\end{tikzpicture}
\end{eqnarray}
\end{ex}

Note that  $\op{C}_k^{\ov{p}}$ is precisely the set of diagrams, where the number of dotted rays plus undotted cups is even (resp. odd), depending if $p=0$ or $p=1$.

\subsection{$\la$-pairs and the arrows in the quiver}
Recall from the introduction that we like to describe the quiver of the category of perverse sheaves on isotropic Grassmannians. The vertices are labelled by Schubert varieties or their corresponding cup diagram $\underline{\la}$ introduced above. The arrows attached to a vertex $\la$ will correspond to so-called $\la$-pairs, a notion introduced in \cite{Braden} which we now adapt to our setting.

\begin{definition} \label{lapair}
Fix a block $\Lambda$. Let $\la \in \Lambda$ be a weight and $\gamma$ a cup in $\underline{\la}$ connecting the vertices $l_\gamma < r_\gamma$ in $P_\bullet$.
With $\gamma$ we associate a pair of integers $(\alpha,\beta)\in\mathbb{Z} \times \mathbb{N}$ called a $\la$-\emph{pair}. If $\gamma$ is undotted we set  $(\alpha,\beta)=(\pos(l_\gamma),\pos(r_\gamma))$ and if $\gamma$ is dotted we set $(\alpha,\beta)=(-\pos(l_\gamma),\pos(r_\gamma))$. Given a $\la$-pair let $\la'$ be the weight obtained by changing $\down$ into $\up$ and $\up$ into $\down$ at the entries $l_\gamma$ and $r_\gamma$ of $\lambda$. In this case we write $\la\stackrel{(\alpha,\beta)}{\longrightarrow}\la'$ or short $\la\rightarrow\la'$ or equivalently $\la'\leftarrow \la$. Note that this implies $\la'>\la$. We abbreviate $\la\leftrightarrow\mu$ in case $\la\rightarrow\mu$ or $\mu\rightarrow\la$.
\end{definition}

The relation $\mu\leftarrow \la$ has a nice interpretation in terms of the associated cup diagrams $\underline\la\leftarrow\underline\mu$:

\begin{lemma}
\label{lapairscup}
Given weights $\la,\mu$. We have $\mu\leftarrow \la$ if and only if the corresponding cup diagrams $\underline\la$ and $\underline\mu$ differ precisely by one of the following local moves  $\underline\mu\leftarrow\underline\la$. In particular $\la$ and $\mu$ are from the same block:\footnote{the attached $+,-$-sequences will be used later; it encodes the weights $\la$ and $\mu$ when replacing $+$ with $\up$ and $-$ with $\down$ as in Lemma~\ref{bijection_S_Lambda}.}

\begin{eqnarray}
\begin{tikzpicture}[thick,scale=0.7]
\draw (0,0) node[above]{$-$} .. controls +(0,-.5) and +(0,-.5) .. +(.5,0) node[above]{+};
\draw (1,0) node[above]{$-$} .. controls +(0,-.5) and +(0,-.5) .. +(.5,0) node[above]{+};
\draw[<-] (1.75,-.2) -- +(1,0);
\draw (3,0) node[above]{$-$} .. controls +(0,-1) and +(0,-1) .. +(1.5,0) node[above]{+};
\draw (3.5,0) node[above]{$-$} .. controls +(0,-.5) and +(0,-.5) .. +(.5,0) node[above]{+};

\begin{scope}[xshift=7cm]
\draw (0,0) node[above]{+} -- +(0,-.8);
\fill (0,-.4) circle(2.5pt);
\draw (.5,0) node[above]{$-$} .. controls +(0,-.5) and +(0,-.5) .. +(.5,0) node[above]{+};
\draw[<-] (1.25,-.2) -- +(1,0);
\draw (2.5,0) node[above]{$-$} .. controls +(0,-.5) and +(0,-.5) .. +(.5,0) node[above]{+};
\draw (3.5,0) node[above]{+} -- +(0,-.8);
\fill (3.5,-.4) circle(2.5pt);
\end{scope}

\begin{scope}[xshift=12cm]
\draw (0,0) node[above]{+} .. controls +(0,-.5) and +(0,-.5) .. +(.5,0) node[above]{+};
\fill (.25,-.36) circle(2.5pt);
\draw (1,0) node[above]{$-$} .. controls +(0,-.5) and +(0,-.5) .. +(.5,0) node[above]{+};
\draw[<-] (1.75,-.2) -- +(1,0);
\draw (3,0) node[above]{+} .. controls +(0,-1) and +(0,-1) .. +(1.5,0) node[above]{+};
\fill (3.75,-.74) circle(2.5pt);
\draw (3.5,0) node[above]{$-$} .. controls +(0,-.5) and +(0,-.5) .. +(.5,0) node[above]{+};
\end{scope}

\begin{scope}[yshift=-2cm, xshift=1cm]
\draw (0,0) node[above]{+} -- +(0,-.8);
\fill (0,-.4) circle(2.5pt);
\draw (.5,0) node[above]{$-$} -- +(0,-.8);
\draw[<-] (.75,-.2) -- +(1,0);
\draw (2,0) node[above]{$-$} .. controls +(0,-.5) and +(0,-.5) .. +(.5,0) node[above]{+};

\begin{scope}[xshift=6cm]
\draw (0,0) node[above]{+} .. controls +(0,-.5) and +(0,-.5) .. +(.5,0) node[above]{+};
\fill (.25,-.36) circle(2.5pt);
\draw (1,0) node[above]{$-$} -- +(0,-.8);
\draw[<-] (1.25,-.2) -- +(1,0);
\draw (2.5,0) node[above]{+} -- +(0,-.8);
\fill (2.5,-.4) circle(2.5pt);
\draw (3,0) node[above]{$-$} .. controls +(0,-.5) and +(0,-.5) .. +(.5,0) node[above]{+};
\end{scope}

\begin{scope}[xshift=11.5cm]
\draw (0,0) node[above]{$-$} .. controls +(0,-.5) and +(0,-.5) .. +(.5,0) node[above]{+};
\draw (1,0) node[above]{$-$} -- +(0,-.8);
\draw[<-] (1.25,-.2) -- +(1,0);
\draw (2.5,0) node[above]{$-$} -- +(0,-.8);
\draw (3,0) node[above]{$-$} .. controls +(0,-.5) and +(0,-.5) .. +(.5,0) node[above]{+};
\end{scope}
\end{scope}

\begin{scope}[yshift=-4cm]
\draw (0,0) node[above]{$-$} .. controls +(0,-1) and +(0,-1) .. +(1.5,0) node[above]{+};
\draw (.5,0) node[above]{$-$} .. controls +(0,-.5) and +(0,-.5) .. +(.5,0) node[above]{+};
\draw[<-] (1.75,-.2) -- +(1,0);
\draw (3,0) node[above]{+} .. controls +(0,-.5) and +(0,-.5) .. +(.5,0) node[above]{+};
\fill (3.25,-.36) circle(2.5pt);
\draw (4,0) node[above]{+} .. controls +(0,-.5) and +(0,-.5) .. +(.5,0) node[above]{+};
\fill (4.25,-.36) circle(2.5pt);

\begin{scope}[xshift=7cm]
\draw (0,0) node[above]{$-$} -- +(0,-.8);
\draw (.5,0) node[above]{$-$} .. controls +(0,-.5) and +(0,-.5) .. +(.5,0) node[above]{+};
\draw[<-] (1.25,-.2) -- +(1,0);
\draw (2.5,0) node[above]{+} .. controls +(0,-.5) and +(0,-.5) .. +(.5,0) node[above]{+};
\fill (2.75,-.36) circle(2.5pt);
\draw (3.5,0) node[above]{+} -- +(0,-.8);
\fill (3.5,-.4) circle(2.5pt);
\end{scope}

\begin{scope}[xshift=13cm]
\draw (0,0) node[above]{$-$} -- +(0,-.8);
\draw (.5,0) node[above]{$-$} -- +(0,-.8);
\draw[<-] (.75,-.2) -- +(1,0);
\draw (2,0) node[above]{+} .. controls +(0,-.5) and +(0,-.5) .. +(.5,0) node[above]{+};
\fill (2.25,-.36) circle(2.5pt);
\end{scope}
\end{scope}

\begin{scope}[xshift=6.5cm, yshift=-6cm]
\draw (0,0) node[above]{$+$} .. controls +(0,-1) and +(0,-1) .. +(1.5,0) node[above]{+};
\draw (.5,0) node[above]{$-$} .. controls +(0,-.5) and +(0,-.5) .. +(.5,0) node[above]{+};
\fill (0.75,-.74) circle(2.5pt);
\draw[<-] (1.75,-.2) -- +(1,0);
\draw (3,0) node[above]{$-$} .. controls +(0,-.5) and +(0,-.5) .. +(.5,0) node[above]{+};
\draw (4,0) node[above]{+} .. controls +(0,-.5) and +(0,-.5) .. +(.5,0) node[above]{+};
\fill (4.25,-.36) circle(2.5pt);
\end{scope}

\end{tikzpicture}
\end{eqnarray}
\end{lemma}
\begin{proof}
This follows directly from the definitions.
\end{proof}

\begin{definition}
\label{Extquivdef}
The \emph{quiver associated with the block $\Lambda$} is the oriented graph $\mathcal{Q}(\Lambda)$ with vertex set the weights $\la\in\Lambda$ and arrows $\la\rightarrow \mu$ whenever $\la\leftrightarrow\mu$ in the sense of Definition~\ref{lapair}. (For an example we refer to Figure~\ref{fig:quiv}, where the arrows labelled $b$ correspond to the $\la$-pairs.).
\end{definition}

\begin{definition}
An \emph{undecorated cap diagram} (respectively {\emph (decorated) cap diagram} $c'$ is obtained by reflecting an undecorated (respectively decorated) cup diagram $c$ along the horizontal line $\mR_{\geq 0} \times \{0\}$ inside $R$. The image of a cup in such a diagram is called a \emph{cap}. It is \emph{dotted} if the cup was dotted.
\end{definition}

We will denote the reflection at the horizontal line by ${}^*$ and use the notation $c^*$ to denote the cap diagram that we obtain from a given cup diagram $c$. In particular, we obtain a cap diagram associated with $\la$, denoted $\ov{\la} = \underline{\la}^*$ for any $\la\in\Lambda$. Alternatively one can also define cap diagrams in the same way as cup diagrams by substituting $R^+$ for $R^-$ and $\mR_{\geq 0} \times \{1\}$ for $\mR_{\geq 0} \times \{-1\}$ in Definition~\ref{defarcs}.

\begin{definition}
A \emph{(decorated) circle diagram} $d=c c'$ is the union of the arcs contained in a decorated cup diagram $c$ and the ones in a decorated cap diagram $c'$, ie., visually we stack the cap diagram $c'$ on top of the cup diagram $c$.
The decorations of $c$ and $c'$ give the decoration for $d$.
\end{definition}

A \emph{ connected component} of a circle diagram $d$ with set of arcs $\{\gamma_1,\ldots,\gamma_r\}$ is a connected component of the subset $\bigcup_{1 \leq i \leq r} \gamma_i \subset R$; and is always a circle or a line.

\begin{definition}
A \emph{line} in a circle diagram $d$ is a connected component of $d$ that intersects the boundary of the strip $R$. A \emph{circle} in a circle diagram $d$ is a connected component of $d$ that does not intersect the boundary of the strip $R$. A line is \emph{propagating}, if one endpoint is at the top and one at the bottom of the diagram.
\end{definition}

\begin{remark}
\label{decorated}
{\rm
By definition, the following does not occur in a cup or cap diagram
\begin{itemize}
\item
In the lower half-plane:
\begin{itemize}
\item
a dotted cup nested inside another cup, or
\item a dotted cup or dotted ray to the right of a (dotted or undotted) ray.
\end{itemize}
\item In the upper half-plane:
\begin{itemize}
\item a dotted cap nested inside another cap, or
\item a dotted cap or dotted ray to the right of a (dotted or undotted) ray.
\end{itemize}
\end{itemize}
}
\end{remark}

\subsection{Orientations, gradings and the space $\D$}

An \emph{oriented cup diagram}, denoted $c\mu$ is obtained from a cup diagram $c \in C_\Lambda$ by putting a weight $\mu \in \Lambda$ on top of $c$ such that all arcs of $c$ are `oriented' in one of the ways displayed in Figure~\ref{oriented}.  We allow $P_\circ$ or $P_\times$ to be nonempty, but then we additionally require that the vertices in $\mu$ labelled $\circ$ agree with $P_\circ$ and the vertices labelled $\times$ agree with $P_\times$.

Similarly, an \emph{oriented cap diagram} $\mu c'$ is obtained by putting such a weight $\mu$ below a cap diagram with $k$ vertices such that all arcs of $c'$ become `oriented' in one of the following ways:
\begin{eqnarray}
\label{oriented}
\begin{tikzpicture}[thick,>=angle 60,scale=0.8]
\node at (-1,-1.2) {degree};
\draw [>->] (0,0) .. controls +(0,-1) and +(0,-1) .. +(1,0) node at
+(0.5,-1.2) {0};
\draw [<-<] (2,0) .. controls +(0,-1) and +(0,-1) .. +(1,0) node at
+(0.5,-1.2) {1};
\draw [<-<] (4,-0.7) .. controls +(0,1) and +(0,1) .. +(1,0) node at
+(0.5,-0.5) {0};
\draw [>->] (6,-0.7) .. controls +(0,1) and +(0,1) .. +(1,0)  node at
+(0.5,-0.5) {1};
\draw [>-] (8,0) -- +(0,-0.7) node at +(0,-1.2) {0};
\draw [->] (9,0) -- +(0,-0.7) node at +(0,-1.2) {0};

\node at (-1,-3.2) {degree};
\draw [<->] (0,-2) .. controls +(0,-1) and +(0,-1) .. +(1,0) node at
+(0.5,-1.2) {0};
\fill (0.5,-2.75) circle(2.5pt);
\draw [>-<] (2,-2) .. controls +(0,-1) and +(0,-1) .. +(1,0) node at
+(0.5,-1.2) {1};
\fill (2.5,-2.76) circle(2.5pt);
\draw [>-<] (4,-2.7) .. controls +(0,1) and +(0,1) .. +(1,0) node at
+(0.5,-0.5) {0};
\fill (4.5,-1.93) circle(2.5pt);
\draw [<->] (6,-2.7) .. controls +(0,1) and +(0,1) .. +(1,0)  node at
+(0.5,-0.5) {1};
\fill (6.5,-1.95) circle(2.5pt);
\draw [-<] (8,-2) -- +(0,-0.7) node at +(0,-1.2) {0};
\fill (8,-2.35) circle(2.5pt);
\draw [<-] (9,-2) -- +(0,-0.7) node at +(0,-1.2) {0};
\fill (9,-2.35) circle(2.5pt);
\end{tikzpicture}
\end{eqnarray}

Additionally, as shown in \eqref{oriented} we assign a degree to each arc of an oriented cup/cap diagram. The cups/caps of degree 1 are called {\it clockwise} and of degree 0 are called {\it anti-clockwise}. To make sense of the word orientation the "$\bullet$" on a cup/cap should be seen as an orientation reversing point.

\begin{definition}
The \emph{degree} of an oriented cup diagram $c\mu $ (resp. cap diagram $c'\mu$) is the sum of the degrees of all arcs contained in $\mu c$ (resp. $c'\mu$), ie., the number of clockwise oriented cups in $\mu c$ (resp. clockwise oriented caps in $c'\mu$).
\end{definition}

Note that $\underline{\la}$ (resp.  $\ov{\la}$) denotes the unique cup (resp. cap diagram) such that $\underline{\la}\la$ (resp. $\la\ov{\la}$) is an oriented cup/cap diagram of degree 0, ie. all its cups/caps are anti-clockwise. Conversely we have the following

\begin{lemma} \label{lem:orient}
Given a cup diagram $c$ with block sequence $\Theta$ one can find a unique weight $\la$ such that $c=\underline{\la}$. The corresponding parity is unique; and if $\mu$ is a weight with the same block sequence such that $\underline{\la}\mu$ is oriented then $\mu$ has the same parity, hence $\la$ and $\mu$ are in the same block $\Lambda$.
The parity of $\Lambda$ equals the number of dotted rays plus undotted cups in $c$ modulo two.
\end{lemma}
\begin{proof}
We choose $\la$ to be the unique weight which induces the orientation on $c$ of degree zero. If $\underline{\la}\mu$ is oriented, then $\mu$ differs from $\la$ by swapping pairs of $\down\up$ to $\up\down$ or $\up\up$ to $\down\down$. In any case $\mu$ has the same  parity, hence is in the same block as $\la$. The number of $\up$'s, hence the parity equals the number of dotted rays plus undotted cups in $c$ modulo two.
\end{proof}

Write $\la\subset\mu$ if $\underline{\la}\mu$ is an oriented cup diagram, then our combinatorics is compatible with the Bruhat order in the following sense:
\begin{lemma}\label{lem:Bruhator} Fix a block $\Lambda$.
\begin{enumerate}
\item If $\la\subset\mu$ then $\la\leq \mu$ in the Bruhat order.
\item If $a \la b$ is an oriented circle diagram
then $a = \underline{\alpha}$ and $b = \overline{\beta}$ for unique weights
$\alpha,\beta$ with $\alpha \subset \la \supset \beta$.
\end{enumerate}
\end{lemma}

\begin{proof}
 For part (2) we more precisely claim: if $c\la$ is an oriented cup diagram then $c = \underline{\alpha}$ for a unique weight $\alpha$ with $\alpha \subset\la$; if $\la c'$ is an oriented cap diagram then $c'= \overline{\beta}$
for a unique weight $\beta$ with $\la \subset \beta$. Indeed, $\alpha$ is the unique weight obtained by reversing the labels at each clockwise cup into an anticlockwise cup. Clearly, $\underline{\alpha}=c$ and hence $\alpha\subset\la$. Similarly for $\beta$.
If $\la\subset\mu$ then $\underline{\la}\mu$ is oriented, hence $\mu$ is obtained from $\la$ by (possibly) changing some of the anti-clockwise cups into clockwise. This however means we either change a pair $\down\up$ into $\up\down$ or $\up\up$ into $\down\down$. In either case the weight gets bigger in the Bruhat order by Lemma~\ref{lem:Bruhat} and (1) holds.
\end{proof}

As in \cite{BS1} we call the number of cups in $\underline{\la}$ the {\it defect} or {\it degree of atypicality} of $\la$, denoted $\op{def}(\la)$, and the maximum of all the defects
of elements in a block the defect of a block, (which equals $\frac{k}{2}$ or $\frac{k-1}{2}$ depending on $k$ even or odd).
Note that in contrast to \cite{BS1} we work with finite blocks only, since infinite blocks would have automatically infinite defect. Lemma~\ref{lem:orient} implies that the defect counts the number of possible orientations of a cup diagram $\underline{\la}$, since each cup has precisely two orientations:
\begin{eqnarray}
\label{defectcount}
|\{\mu\mid \la\subset\mu\}|&=&2^{\op{def}(\la)}.
\end{eqnarray}

The following connects the notion of $\la$-pairs with the degree:

\begin{lemma}
\label{degone}
If $\la\rightarrow \mu$ then $\la$ and $\mu$ are in the same block. Moreover, $\la\rightarrow \mu$ if and only if $\underline{\la}\mu$ has degree $1$
and $\la<\mu$ in the Bruhat order.
\end{lemma}

\begin{proof}
Assume $\la\rightarrow \mu$ then by definition of a $\la$-pair we have a cup $C$ connecting vertices $i$ and $j$ ($i<j$). If $C$ is
undotted (resp. dotted) the weight $\mu$ differs from $\la$ only at these places with an $\up$ at position $i$ and $\down$ at position
$j$ (resp. a $\down$ at position $i$ and $\down$ at position $j$). Hence in $\underline{\la}\mu$ the cup $C$ becomes
clockwise, while all other cups stay anti-clockwise and thus it has degree $1$ and $\mu>\la$.

Assuming $\underline{\la}\mu$ has degree $1$ and $\la<\mu$. If $C$ is the unique clockwise cup, then $\mu$ equals $\la$ except at the vertices of $C$ where $\down\up$ gets replaced by $\up\down$ or $\up\up$ gets turned into $\down\down$. Hence the diagrams $\underline{\la}$ and $\underline{\mu}$ look locally as in Lemma~\ref{lapairscup}.
\end{proof}

\section{The generalized Khovanov algebra of type $D$}

We introduce now a type $D$ analogue of the algebras studied in \cite{Str09}, \cite{BS1}.  It is a diagrammatical algebra which finally will describe as a special case the endomorphism ring of a minimal projective generator in the category of perverse sheaves on isotropic Grassmannians.

\subsection{The underlying graded vector space $\D$}
Let in the following $\Lambda$ be a fixed block with (using the notation from \eqref{defblock}) $P_\bullet(\Lambda)=\{i_1,\cdots, i_r\}$, $i_1<\ldots< i_r$. Set $S=\mC[X_{i_1},\ldots, X_{i_r}]/(X_{i_s}^2\mid 1\leq s\leq r)$. In the following we will not distinguish in notation between a polynomial $f$ and its canonical image in a quotient of a polynomial ring.
\begin{definition}
Denote by
\begin{eqnarray}
\label{DefB}
\B&:=&\left\{ \underline{\la}\nu\ov{\mu}\mid \la,\mu,\nu\in\Lambda,\quad\ov{\mu}\nu \text{ and $\underline{\la}\nu$ are
oriented}\right\},
\end{eqnarray}
the set of {\it oriented (decorated) circle diagrams} associated with $\Lambda$. The {\it degree} of an element in $\B$
is the sum of the degrees of the cups and caps as displayed in \eqref{oriented}. The degree of a connected component in an oriented
circle diagram is the sum of the degrees of the cups and caps contained in the connected component. We denote by $\D$ the graded complex vector space on the basis $\B$.
\end{definition}
Note that the $\underline{\la}\la\ov{\la}$ for $\la\in\Lambda$ span the degree zero part of $\D$.  For $\mu,\la\in \Lambda$ let ${}_\mu(\D)_\la{}$ be the subspace spanned by all $\underline{\la}\nu\ov{\mu}$, $\nu\in\Lambda$.

The most important special case will be the graded vector spaces $\mathbb{D}_k=\D$ associated with the principal blocks $\Lambda=\La_k^{\overline{0}}$ equipped with its distinct homogeneous basis $\mathbb{B}_k=\mathbb{B}_{\La_k^{\overline{0}}}$.

\begin{ex}
For $k=4$, the graded vector space $\mathbb{D}_k$ is of  dimension $66$ with graded Poincare polynomial $p_{\mD_4}(q)=8+20q+24q^2+12q^3+2q^4$, and Cartan matrix given by Figure~\ref{fig:CM}.
\end{ex}

The vector space  $\D$ has the following alternative description which will be useful to describe the multiplication and make the connection to geometry:
\begin{definition}
\label{relations}
For a cup or cap diagram $c$ and $i,j \in \{i_1,\ldots, i_r\}$ we write
\begin{enumerate}[i)]
\item $i \CupConnect j$ if $c$ contains a cup/cap connecting $i$ and $j$,
\item $i \DCupConnect j$ if $c$ contains a dotted cup/dotted cap connecting $i$ and $j$,
\item $i \RayConnect$ if $c$ contains a ray that ends at $i$,
\item $i \DRayConnect$ if $c$ contains a dotted ray that ends at $i$.
\end{enumerate}
\end{definition}

\begin{prop}
\label{coho}
Let $\la,\mu$ weights in $\Lambda$ and consider the circle diagram $\un\la\ov{\mu}$. Let $I\triangleleft S$ be the ideal generated by the relations $X_i+X_j$ if $i \CupConnect j$, and $X_i-X_j$ if $i \DCupConnect j$, and finally $X_i$ if $i \RayConnect$ or $i \DRayConnect$ in $\ov\mu$ or $\underline{\la}$.
\begin{enumerate}
\item There is an isomorphism of graded vector spaces
\begin{eqnarray*}
\Psi:\quad {}_\mu(\D)_\la{}&\cong&\cM(\un\la\ov{\mu}):=S/I\;\langle r\rangle\\
\underline{\la}\nu\ov{\mu}&\mapsto&\prod_{a\in J(\underline{\la}\nu\ov{\mu})}X_a
\end{eqnarray*}
where $J(\underline{\la}\nu\ov{\mu})$ is the set of rightmost vertices of the clockwise circles in $\underline{\la}\nu\ov{\mu}$. Here $\langle r\rangle$ denotes the grading shift up by $r=k-|J|$, see Section~\ref{graded}, and the $X_i$'s are of degree $2$.
\item The monomials $y_1\cdots y_k\in S/I$ where $y_s\in\{1,X_{i_s}\}$ are linearly independent and correspond under $\Psi$ to the standard basis vectors from $\B$ in ${}_\mu(\D)_\la{}$.
    \end{enumerate}
\end{prop}

The proof will be given at the end of the section.

\begin{ex}
\label{ps}
Consider the principal blocks $\Lap$ for $k=4$ with $P_\bullet(\Lap)=\{1,2,3,4\}$. The following circle diagrams $\un\la\ov{\mu}$ have always exactly two possible orientations, anticlockwise or clockwise which are sent respectively to the elements $1, X_4\in \cM(\un\la\ov{\mu})$ under the isomorphism $\Psi$. Similar for $\un\mu\ov{\la}$ with $\cM(\un\mu\ov{\la})=\cM(\un\la\ov{\mu})$.
\small
\begin{eqnarray*}
\label{1}
\begin{tikzpicture}[thick, scale=0.5]
\draw (0,-.1) .. controls +(0,1) and +(0,1) .. +(1,0);
\draw (2,-.1) .. controls +(0,1) and +(0,1) .. +(1,0);
\draw (1,-.1) .. controls +(0,-1) and +(0,-1) .. +(1,0);
\draw (0,-.1) .. controls +(0,-2.2) and +(0,-2.2) .. +(3,0);
\node at (14.5,-0.8) {$\rightsquigarrow \cM(\un\la\ov{\mu})=S/(X_1+X_2,X_2+X_3,X_3+X_4,X_1+X_4)\cong \mC[X_4]/(X_4^2)$};
\end{tikzpicture}\\
\label{2}
\begin{tikzpicture}[thick, scale=0.5]
\draw (0,-.1) .. controls +(0,1) and +(0,1) .. +(1,0);
\fill (0.5,.63) circle(4pt);
\draw (2,-.1) .. controls +(0,1) and +(0,1) .. +(1,0);
\fill (2.5,.63) circle(4pt);
\draw (1,-.1) .. controls +(0,-1) and +(0,-1) .. +(1,0);
\draw (0,-.1) .. controls +(0,-2.2) and +(0,-2.2) .. +(3,0);
\node at (14.5,-0.8){$\rightsquigarrow
\cM(\un\la\ov{\mu})=S/(X_1-X_2,X_2+X_3,X_3-X_4,X_1+X_4)\cong \mC[X_4]/(X_4^2)$};
\end{tikzpicture}\\
\label{3}
\begin{tikzpicture}[thick, scale=0.5]
\draw (0,-.1) .. controls +(0,1) and +(0,1) .. +(1,0);
\fill (0.5,.63) circle(4pt);
\draw (2,-.1) .. controls +(0,1) and +(0,1) .. +(1,0);
\draw (1,-.1) .. controls +(0,-1) and +(0,-1) .. +(1,0);
\draw (0,-.1) .. controls +(0,-2.2) and +(0,-2.2) .. +(3,0);
\fill (1.5,-1.75) circle(4pt);
\node at (14.5,-0.8){$\rightsquigarrow
\cM(\un\la\ov{\mu})=S/(X_1-X_2,X_2+X_3,X_3+X_4,X_1-X_4)\cong \mC[X_4]/(X_4^2)$};
\end{tikzpicture}\\
\label{4}
\begin{tikzpicture}[thick, scale=0.5]
\draw (0,-.1) .. controls +(0,1) and +(0,1) .. +(1,0);
\draw (2,-.1) .. controls +(0,1) and +(0,1) .. +(1,0);
\fill (2.5,.63) circle(4pt);
\draw (1,-.1) .. controls +(0,-1) and +(0,-1) .. +(1,0);
\draw (0,-.1) .. controls +(0,-2.2) and +(0,-2.2) .. +(3,0);
\fill (1.5,-1.75) circle(4pt);
\node at (14.5,-0.8){$\rightsquigarrow
\cM(\un\la\ov{\mu})=S/(X_1+X_2,X_2+X_3,X_3-X_4,X_1-X_4)\cong \mC[X_4]/(X_4^2)$};
\end{tikzpicture}
\end{eqnarray*}
\end{ex}
\normalsize

\begin{remark}
{\rm
In \cite{ES1}, the special case $\Lambda=\Lambda_k^{\ov{p}}$ was studied from a geometric point of view. Weights $\la$, $\mu$ are identified with fixed point of the natural $\mathbb{C}^*$-action on the (topological) Springer fibre of type $D_k$ for the principal nilpotent of type $A_k$, and the cup diagrams $\un{\la}$, $\un\mu$ are canonically identified with the cohomology of the closure of the corresponding attracting cells $\mathcal{A}_\la$, $\mathcal{A}_\mu$. The vector space $S/I$ is then the cohomology of the intersection $\mathcal{A}_\la\cap \mathcal{A}_\mu$.
}
\end{remark}

\subsection{Properties of oriented circle diagrams}

The following is a crucial technical result which allows us to mimic \cite{BS1}:
\begin{prop}
\label{lem:stupidlemma}
Fix a block $\Lambda$.
\begin{enumerate}
\item Each closed circle in an oriented (decorated) circle diagram has an even number of dots.
\item A closed circle in a circle diagram can either not be oriented at all (by inserting a weight from
    $\Lambda$) or allows precisely two different orientations. In that case the degrees of this component differ by $2$ and equal the number of its caps plus or minus $1$ respectively.
\suspend{enumerate}
{If the circle is oriented with the bigger degree we call it \emph{clockwise} or \emph{of type $X$} and \emph{anticlockwise} or \emph{of
type $1$} if it is oriented with the smaller degree.}
\resume{enumerate}
\item A circle is oriented anticlockwise if its rightmost label is $\up$ and clockwise if its rightmost label is $\down$.
\item A line in a circle diagram can be oriented in at most one way.
\end{enumerate}
\end{prop}

\begin{proof}
The first statement is obvious, since each dot is a orientation reversing point. For the second statement note that choosing an arbitrary vertex on the circle, the orientation of a
circle is determined by the label ($\up$ or $\down$) at that vertex. Hence there are at most two orientations. Moreover, the number of
vertices involved in a circle is even. Swapping the orientation means that in the weight
$\la\in\Lambda$ we change certain $\up$'s into $\down$'s and vice versa, with an even total number of symbols changed. In
particular, the parity of the number of $\down$'s does not change and so the new sequence is again an element in $\Lambda$.
Therefore both orientations can be realized. In case the circle has no dots, the remaining part of the second statement follows by the
same arguments as in \cite[Lemma 2.1]{BS1}: Namely either the circle consists out of precisely one cup and one cap, ie.

\begin{center}
\begin{tikzpicture}[scale=.5,thick,>=angle 60]
\node at (-1,0) {$\deg\bigg($};
\draw (-.25,0) to +(2,0);
\draw [->] (0,0) .. controls +(0,-1) and +(0,-1) .. +(1.5,0);
\draw [<-] (0,0) .. controls +(0,1) and +(0,1) .. +(1.5,0);
\node at (2.7,0) {$\bigg) = 0$,};
\end{tikzpicture}
\hspace{1cm}
\begin{tikzpicture}[scale=.5,thick,>=angle 60]
\node at (-1,0) {$\deg\bigg($};
\draw (-.25,0) to +(2,0);
\draw [<-] (0,0) .. controls +(0,-1) and +(0,-1) .. +(1.5,0);
\draw [->] (0,0) .. controls +(0,1) and +(0,1) .. +(1.5,0);
\node at (2.7,0) {$\bigg) = 2$,};
\end{tikzpicture}
\end{center}
where the statement is clear, or it contains a kink which
can be removed using one of the following straightening rules:
\begin{equation}
\label{kink}
\begin{tikzpicture}[thick, snake=zigzag, line before snake = 2mm, line
after snake = 2mm]
\draw[thin] (-.25,0) -- +(1.5,0);
\begin{scope}[yscale=-1]
\draw (0,0) -- +(0,.3);
\draw[dotted, snake] (0,.3) -- +(0,.8);
\draw (0,0) .. controls +(0,-.5) and +(0,-.5) .. +(.5,0);
\draw (.5,0) .. controls +(0,.5) and +(0,.5) .. +(.5,0);
\draw (1,0) -- +(0,-.3);
\draw[dotted, snake] (1,-.3) -- +(0,-.8);
\end{scope}
\node at (0,-.1) {$\up$};
\node at (1.01,-.1) {$\up$};
\node at (.51,.08) {$\down$};

\draw [->,line join=round,decorate, decoration={zigzag,segment
length=4,amplitude=.9,post=lineto,post length=2pt}]  (1.5,0) -- +(.5,0);

\draw[thin] (2.25,0) -- +(.5,0);
\draw (2.5,0) -- +(0,.3);
\draw[dotted, snake] (2.5,.3) -- +(0,.8);
\draw (2.5,0) -- +(0,-.3);
\draw[dotted, snake] (2.5,-.3) -- +(0,-.8);
\node at (2.51,-.1) {$\up$};

\begin{scope}[xshift=5cm]
\draw[thin] (-.25,0) -- +(1.5,0);
\begin{scope}[yscale=-1]
\draw (0,0) -- +(0,.3);
\draw[dotted, snake] (0,.3) -- +(0,.8);
\draw (0,0) .. controls +(0,-.5) and +(0,-.5) .. +(.5,0);
\draw (.5,0) .. controls +(0,.5) and +(0,.5) .. +(.5,0);
\draw (1,0) -- +(0,-.3);
\draw[dotted, snake] (1,-.3) -- +(0,-.8);
\end{scope}
\node at (0,.1) {$\down$};
\node at (1.01,.1) {$\down$};
\node at (.51,-.1) {$\up$};

\draw [->,line join=round,decorate, decoration={zigzag,segment
length=4,amplitude=.9,post=lineto,post length=2pt}]  (1.5,0) -- +(.5,0);

\draw[thin] (2.25,0) -- +(.5,0);
\draw (2.5,0) -- +(0,.3);
\draw[dotted, snake] (2.5,.3) -- +(0,.8);
\draw (2.5,0) -- +(0,-.3);
\draw[dotted, snake] (2.5,-.3) -- +(0,-.8);
\node at (2.51,.1) {$\down$};
\end{scope}
\end{tikzpicture}
\end{equation}

The result is a new oriented circle diagram with two
fewer vertices than before. It is oriented in the same way as the
original circle,
but has either one clockwise cup and one anti-clockwise
cap less, or one anti-clockwise cup and one clockwise
cap less. Hence the claims follow inductively for undotted circles.

Assume now the circle has dots. The claim is obviously true for the
circles including two vertices:
\begin{center}
\begin{tikzpicture}[scale=.5,thick,>=angle 60]
\node at (-1,0) {$\deg\bigg($};
\draw (-.25,0) to +(2,0);
\draw [<->] (0,0) .. controls +(0,-1) and +(0,-1) .. +(1.5,0);
\fill (.75,-.76) circle(4pt);
\draw (0,0) .. controls +(0,1) and +(0,1) .. +(1.5,0);
\fill (.75,.76) circle(4pt);
\node at (2.7,0) {$\bigg) = 0$,};
\end{tikzpicture}
\hspace{1cm}
\begin{tikzpicture}[scale=.5,thick,>=angle 60]
\node at (-1,0) {$\deg\bigg($};
\draw (-.25,0) to +(2,0);
\draw (0,0) .. controls +(0,-1) and +(0,-1) .. +(1.5,0);
\fill (.75,-.76) circle(4pt);
\draw [<->] (0,0) .. controls +(0,1) and +(0,1) .. +(1.5,0);
\fill (.75,.76) circle(4pt);
\node at (2.7,0) {$\bigg) = 2$,};
\end{tikzpicture}
\end{center}

If there is a kink without dots as above then we remove it and argue by induction. We therefore assume there are no such kinks,
ie. each occurring kink has precisely one dot (either on the cap or the cup), since having a dot on both contradicts Remark
~\ref{decorated} if the component is a closed circle. Using Remark~\ref{decorated} again one easily verifies that a circle with more than
$2$ vertices and
no undotted kink contains at least one of the following diagrams as subdiagram
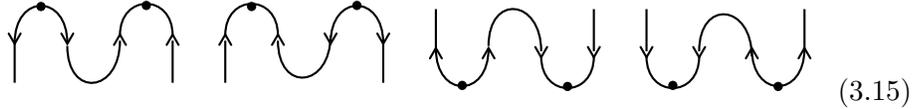
\begin{eqnarray}
\label{snake}
\usetikzlibrary{arrows}
\begin{tikzpicture}[thick,>=angle 60, scale=0.7]

\draw (0,0.1) to +(0,-0.8);
\draw [<->] (0,0) .. controls +(0,1) and +(0,1) .. +(1,0);
\fill (0.5,0.75) circle(2.5pt);
\draw (1,0) .. controls +(0,-1) and +(0,-1) .. +(1,0.1);
\draw [>-<] (2,0) .. controls +(0,1) and +(0,1) .. +(1,0);
\fill (2.5,0.77) circle(2.5pt);
\draw (3,0.1) to +(0,-0.8);
\begin{scope}[xshift=4cm]
\draw (0,0.1) to +(0,-0.8);
\draw [>-<] (0,0) .. controls +(0,1) and +(0,1) .. +(1,0);
\fill (0.5,0.75) circle(2.5pt);
\draw (1,0.1) .. controls +(0,-1) and +(0,-1) .. +(1,0.1);
\draw [<->] (2,0) .. controls +(0,1) and +(0,1) .. +(1,0);
\fill (2.5,0.77) circle(2.5pt);
\draw (3,0.1) to +(0,-0.8);
\end{scope}

\begin{scope}[xshift=8cm,y=-1cm]
\draw (0,0.1) to +(0,-0.8);
\draw [<->] (0,0) .. controls +(0,1) and +(0,1) .. +(1,0);
\fill (0.5,0.75) circle(2.5pt);
\draw (1,0) .. controls +(0,-1) and +(0,-1) .. +(1,0.1);
\draw [>-<] (2,0) .. controls +(0,1) and +(0,1) .. +(1,0);
\fill (2.5,0.77) circle(2.5pt);
\draw (3,0.1) to +(0,-0.8);

\begin{scope}[xshift=4cm]
\draw (0,0.1) to +(0,-0.8);
\draw [>-<] (0,0) .. controls +(0,1) and +(0,1) .. +(1,0);
\fill (0.5,0.75) circle(2.5pt);
\draw (1,0.1) .. controls +(0,-1) and +(0,-1) .. +(1,0.1);
\draw [<->] (2,0) .. controls +(0,1) and +(0,1) .. +(1,0);
\fill (2.5,0.77) circle(2.5pt);
\draw (3,0.1) to +(0,-0.8);
\end{scope}
\end{scope}
\end{tikzpicture}
\end{eqnarray}
We claim that the degree equals the degree of the circle obtained by removing the pair of dots and adjusting the orientations between
them.
In the first case we remove two dotted  caps, one of degree $1$ and one of degree $0$ and a cup of degree $0$. Removing the dots and
adjusting the orientations gives two caps of degree $0$ and a cup of degree $1$. If we take the opposite orientation then we remove two
dotted caps, one of degree $0$ and one of degree $1$ and a cup of degree $1$. Removing dots and adjusting orientations gives only one cup
of degree $0$ with two caps of degree $1$. In the third case we remove one undotted cap of degree $1$ and two dotted  cups, one of degree
$0$ and one of degree $1$. Removing the dots and adjusting the orientations gives two clockwise cups of degree $1$ and one anticlockwise
cap of degree zero. Swapping the orientation gives us one undotted cap of degree $0$ and two dotted  cups, one of degree $1$ and one of
degree $0$. Removing the dots and adjusting the orientations gives two anticlockwise cups of degree $0$ and one clockwise cap of degree
$1$. Hence, by induction, we transferred our statements to the undotted case where it is known by \cite[Lemma 2.1]{BS1}.

The third statement is clear for circles without dots by \cite[Lemma 2.1]{BS1}. For circles with dots we argue by the same induction as
above.
In case the circle has two vertices only, the statement is clear. Otherwise there is an undotted kink we remove it not changing the
orientations on the other parts. In case there is no such kink there is a subdiagram of the form \eqref{snake} in which case we remove the
dots on this part and adjust the two middle labels such that the circle becomes oriented. The statement follows then by induction noting that we never changed the label at the rightmost point of a circle.
For the last statement note that the orientation at the end of each line is determined by whether there is a dot or not.
\end{proof}

\begin{lemma}
\label{signmove}
Let $C$ be a circle in an oriented circle diagram $(a\la b)$.
\begin{enumerate}
\item The total number $m$ of undotted cups and caps in $C$ is even.
\item The relations from Proposition ~\ref{coho} involving $C$ are equivalent to
    \begin{eqnarray}
    \label{newrel}
    X_i&=&(-1)^rX_j
    \end{eqnarray}
    for any two vertices $i$ and $j$ on $C$, where $r$ denotes the total number of undotted cups and caps on any of the two halves of $C$ connecting $i$ with $j$.
\end{enumerate}
\end{lemma}

\begin{proof}
If $C$ consists of precisely one cup and one cap then $m=0$ or $m=2$. Otherwise we can find an undotted kink, see \eqref{kink}, or a subdiagram as in \eqref{snake}. In the first case we remove the kink without changing the parity of $m$, in the second case we remove the two dots and reorient without changing the parity of $m$. Then the first part of the lemma follows by induction. Moreover it shows that the formula \eqref{newrel} does not depend on which half we chose. The rest follows from the definitions.
\end{proof}

\begin{proof}[Proof of Proposition~\ref{coho}]
The map is clearly well-defined. The relations from Definition~\ref{relations} only couple indices in the same component of the diagram, hence the $X_i$'s where the indices lie on different circles are linearly independent if they are not zero. If the diagram can be oriented, then the definition of the relations together with Lemma~\ref{signmove} imply that the $\prod_{a}X_a$ for the different basis vectors $(\underline{\la}\nu\ov{\mu})$ are linearly independent.
\end{proof}

\subsection{The multiplication over $\mathbb{F}_2$}
\label{multF2Def}
Generalizing the approach of \cite{BS1} we define now a multiplication
which turns the graded vector space $\D$ into a
positively graded associative algebra. To simplify the setup we first work over $\mathbb{F}_2$ to show the analogy to the type $A$ case. Let $\D^{{\mathbb F}_2}$ be the $\mathbb{F}_2$-vector space spanned by \eqref{DefB}. We want to define the product of two basis elements $(a \la b), (c \mu d)\in\D^{{\mathbb F}_2}$.

If $b^* \neq c$, we simply define the product
$( a \la b )( c \mu d )$ to be zero.
In case  $b^* = c$ draw the diagram $a \la b$ underneath the diagram
$c \mu d$ to obtain a new sort of diagram
with two sequences of weights separated by a symmetric middle section containing $b$ under $c$.
Then iterate the generalized {\em surgery procedure} for diagrams explained in the next paragraph
to convert this diagram into either zero or a disjoint union of diagrams none of which
have any cups or caps left in their
middle sections.
Finally, collapse each of the resulting diagrams by
identifying their two weight sequences
to obtain a disjoint union of some new oriented circle diagrams of the form $(a\nu d)$ for some weight sequence $\nu$.
The product $( a \la b ) ( c \mu d )$
is then defined to be the sum of the corresponding basis vectors of $\D$.

\subsection*{The (decorated) generalized surgery procedure}
Put the oriented circle-line diagram $(b^* \mu d)$  on top of $(a \la b)$ and stitch the $i$-th lower vertical ray of  $(b^*\la d)$ together
with the $i$-th upper vertical ray of $(a \la b)$ in case they have the same orientation. Otherwise the product is zero.
So let us assume the product is non-zero and let $(a \la b\mu d)$ be the diagram obtained by stitching and removing any dots on the vertical rays stitched together. (Note that dots get always removed in pairs, since either both rays or no ray in a pair contain a
dot as they are oriented in the same way. In particular, it is compatible with orientations).

Then pick the {\bf rightmost} symmetric pair of a
cup and a cap (possibly dotted) in the middle section of $( a \la b )( c \mu d )$
that can be connected without crossing rays or arcs. Forgetting orientations for a while, cut open the cup and the cap, then stitch the loose ends together
to form
a pair of vertical line segments (always without dots even if there were dots on the cup/cap pair before):
\begin{eqnarray} \label{surgery}
\begin{tikzpicture}[thick, snake=zigzag, line before snake = 2mm, line after snake = 2mm]
\draw[thin] (-.25,0) -- +(1,0);
\draw[thin] (-.25,1.5) -- +(1,0);

\draw[thin, dashed] (.25,.38) -- +(0,.74);

\draw (0,1.5) .. controls +(0,-.5) and +(0,-.5) .. +(.5,0);
\draw (0,0) .. controls +(0,.5) and +(0,.5) .. +(.5,0);
\draw[->, snake] (1,.75) -- +(1,0);
\draw[thin] (2.25,0) -- +(1,0);
\draw (2.5,0) -- +(0,1.5);
\draw (3,0) -- +(0,1.5);
\draw[thin] (2.25,1.5) -- +(1,0);

\begin{scope}[xshift=5.5cm]
\draw[thin] (-.25,0) -- +(1,0);
\draw[thin] (-.25,1.5) -- +(1,0);

\draw[thin, dashed] (.25,.38) -- +(0,.74);

\draw (0,1.5) .. controls +(0,-.5) and +(0,-.5) .. +(.5,0);
\fill (0.25,1.125) circle(2.5pt);
\draw (0,0) .. controls +(0,.5) and +(0,.5) .. +(.5,0);
\fill (0.25,.375) circle(2.5pt);
\draw[->, snake] (1,.75) -- +(1,0);
\draw[thin] (2.25,0) -- +(1,0);
\draw (2.5,0) -- +(0,1.5);
\draw (3,0) -- +(0,1.5);
\draw[thin] (2.25,1.5) -- +(1,0);
\end{scope}
\end{tikzpicture}
\end{eqnarray}

This surgery either
\begin{enumerate}[(i)]
\item merges two circles into one (with possibly removing two dots), or
\item splits one circle into two (with possible removing two dots), or
\item at least one of the involved cup and
cap belonged to a line, not a circle, before the cut was made.
\end{enumerate}
In the first two cases re-orient the new (decorated) circle(s) according to the orientation of
the old (decorated) circle(s) using the following rules (remembering
$1=\text{anti-clockwise}$, $x=\text{clockwise}$\footnote{The notion of anti-clockwise and clockwise circle makes sense by the same arguments as in Lemma~\ref{lem:stupidlemma} for any circle appearing in the diagram after a surgery procedure.})
\begin{align}
1 \otimes 1 \mapsto 1,\label{mult}
\qquad
1 \otimes x \mapsto x,
\qquad
&x \otimes 1 \mapsto x,
\qquad
x \otimes x \mapsto 0,\\
1 \mapsto 1 \otimes x + x \otimes 1,
\qquad
&x \mapsto x \otimes x,\label{comult}
\end{align}
to obtain a disjoint union of zero, one or two new oriented diagrams
replacing the old diagram.
For instance, the rule $1 \otimes 1 \mapsto
1$ here indicates that two anti-clockwise circles transform to one
anti-clockwise circle. The rule $1 \mapsto 1 \otimes x + x \otimes 1$ indicates
that one anti-clockwise circle transforms to give a disjoint union of two
different diagrams each having one anti-clockwise circle and one clockwise
circle. Finally the rule $x \otimes x \mapsto 0$ indicates that two clockwise
circles produce no diagram at all at the end of the surgery procedure. In contrast to \cite{BS1} it could happen that the resulting diagrams can not be oriented at all in which case we declare the result of the surgery procedure also to be zero.

Consider the third case where we have a circle and a line. Let $y$ denote a line and $1$ and $x$
denote anti-clockwise and clockwise circles as before then we use the rules
\begin{align}
\label{circline}
1 \otimes y &\mapsto y,\quad
x \otimes y \mapsto 0,\quad
y \mapsto y\otimes x.
\end{align}
So, the rule $1 \otimes y\mapsto y$ indicates that an anti-clockwise circle
and a line convert to a single line (as in the first two pictures below),
whereas a clockwise circle and a line produce zero (as displayed in the third and fourth picture below),
and the rule $y \mapsto x \otimes y$ indicates that a single line converts
to a line and a clockwise circle:

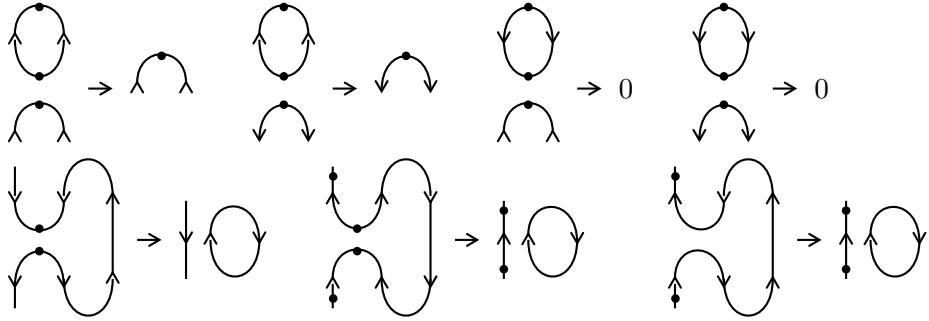
\begin{tikzpicture}[thick,>=angle 60, scale=0.65]

\draw[>-<] (0,-1.1) .. controls +(0,1) and +(0,1) .. +(1,0);
\fill (0.5,-.33) circle(2.5pt);
\draw (0,-1) .. controls +(0,-1) and +(0,-1) .. +(1,0);
\fill (0.5,-1.75) circle(2.5pt);
\draw[>-<] (0,-3.1) .. controls +(0,1) and +(0,1) .. +(1,0);
\fill (0.5,-2.33) circle(2.5pt);
\draw[->] (1.5,-2) -- +(.5,0);
\draw[>-<] (2.5,-2.1) .. controls +(0,1) and +(0,1) .. +(1,0);
\fill (3,-1.33) circle(2.5pt);

\draw[>-<] (5,-1.1) .. controls +(0,1) and +(0,1) .. +(1,0);
\fill (5.5,-.33) circle(2.5pt);
\draw (5,-1) .. controls +(0,-1) and +(0,-1) .. +(1,0);
\fill (5.5,-1.75) circle(2.5pt);
\draw[<->] (5,-3.1) .. controls +(0,1) and +(0,1) .. +(1,0);
\fill (5.5,-2.33) circle(2.5pt);
\draw[->] (6.5,-2) -- +(.5,0);
\draw[<->] (7.5,-2.1) .. controls +(0,1) and +(0,1) .. +(1,0);
\fill (8,-1.33) circle(2.5pt);

\draw[<->] (10,-1.1) .. controls +(0,1) and +(0,1) .. +(1,0);
\fill (10.5,-.33) circle(2.5pt);
\draw (10,-1) .. controls +(0,-1) and +(0,-1) .. +(1,0);
\fill (10.5,-1.75) circle(2.5pt);
\draw[>-<] (10,-3.1) .. controls +(0,1) and +(0,1) .. +(1,0);
\fill (10.5,-2.33) circle(2.5pt);
\draw[->] (11.5,-2) -- +(.5,0);
\node at (12.5,-2) {0};

\draw[<->] (14,-1.1) .. controls +(0,1) and +(0,1) .. +(1,0);
\fill (14.5,-.33) circle(2.5pt);
\draw (14,-1) .. controls +(0,-1) and +(0,-1) .. +(1,0);
\fill (14.5,-1.75) circle(2.5pt);
\draw[<->] (14,-3.1) .. controls +(0,1) and +(0,1) .. +(1,0);
\fill (14.5,-2.33) circle(2.5pt);
\draw[->] (15.5,-2) -- +(.5,0);
\node at (16.5,-2) {0};
\end{tikzpicture}

\begin{tikzpicture}[thick,>=angle 60, scale=0.65]
\draw (0,-.1) -- +(.0,.6);
\draw [>-<] (0,0) .. controls +(0,-1) and +(0,-1) .. +(1,0);
\fill (0.5,-.77) circle(2.5pt);
\draw (0,-1.9) -- +(.0,-.5);
\draw [<->] (0,-2) .. controls +(0,1) and +(0,1) .. +(1,0);
\fill (0.5,-1.23) circle(2.5pt);
\draw (1,-.1) .. controls +(0,1) and +(0,1) .. +(1,0);
\draw [>->] (2,-1.9) -- +(.0,1.9);
\draw (1,-1.8) .. controls +(0,-1) and +(0,-1) .. +(1,0);
\draw[->] (2.5,-1) -- +(.5,0);
\draw[->] (3.5,-.2) -- +(0,-.9);\draw (3.5,-.9) -- +(0,-.9);
\draw[>->] (4,-1.1) .. controls +(0,1) and +(0,1) .. +(1,0);
\draw (4,-1) .. controls +(0,-1) and +(0,-1) .. +(1,0);

\begin{scope}[xshift=6.5cm]
\draw (0,-.1) -- +(.0,.6);
\fill (0.01,.3) circle(2.5pt);
\draw [<->] (0,0) .. controls +(0,-1) and +(0,-1) .. +(1,0);
\fill (0.5,-.77) circle(2.5pt);
\draw (0,-1.9) -- +(.0,-.5);
\fill (0.01,-2.2) circle(2.5pt);
\draw [>-<] (0,-2) .. controls +(0,1) and +(0,1) .. +(1,0);
\fill (0.5,-1.23) circle(2.5pt);
\draw (1,-.1) .. controls +(0,1) and +(0,1) .. +(1,0);
\draw [<-<] (2,-2) -- +(.0,2);
\draw (1,-1.8) .. controls +(0,-1) and +(0,-1) .. +(1,0);
\draw[->] (2.5,-1) -- +(.5,0);
\draw[-<] (3.5,-.2) -- +(0,-.9);\draw (3.5,-.9) -- +(0,-.9);
\fill (3.5,-.4) circle(2.5pt);
\fill (3.5,-1.6) circle(2.5pt);
\draw[>->] (4,-1.11) .. controls +(0,1) and +(0,1) .. +(1,0);
\draw (4,-1) .. controls +(0,-1) and +(0,-1) .. +(1,0);
\end{scope}

\begin{scope}[xshift=13.5cm]
\draw (0,-.1) -- +(.0,.6);
\fill (0.01,.3) circle(2.5pt);
\draw [<-<] (0,0) .. controls +(0,-1) and +(0,-1) .. +(1,0);
\draw (0,-1.9) -- +(.0,-.5);
\fill (0.01,-2.2) circle(2.5pt);
\draw [>->] (0,-2) .. controls +(0,1) and +(0,1) .. +(1,0);
\draw (1,-.1) .. controls +(0,1) and +(0,1) .. +(1,0);
\draw [>->] (2,-2) -- +(.0,2);
\draw (1,-1.8) .. controls +(0,-1) and +(0,-1) .. +(1,0);
\draw[->] (2.5,-1) -- +(.5,0);
\draw[-<] (3.5,-.2) -- +(0,-.9);\draw (3.5,-.9) -- +(0,-.9);
\fill (3.5,-.4) circle(2.5pt);
\fill (3.5,-1.6) circle(2.5pt);
\draw[>->] (4,-1.11) .. controls +(0,1) and +(0,1) .. +(1,0);
\draw (4,-1) .. controls +(0,-1) and +(0,-1) .. +(1,0);
\end{scope}
\end{tikzpicture}

Finally we could have two lines, in which case
\begin{eqnarray}
\label{yuk}
y \otimes y &\mapsto&
\left\{
\begin{array}{ll}
y\otimes y&\text{if both lines are propagating and reconn\-ecting}\\
&\text{as in \eqref{surgery} gives an oriented diagram;}\\
0&\text{otherwise.}
\end{array}
\right.
\end{eqnarray}
Note that the cup-cap pair is never dotted in this case.


We execute this generalized surgery procedure until the middle part contains only vertical lines (without dots).
Then we remove this middle part and identify the two weight sequences to obtain a disjoint union of circle diagrams. The result of the
multiplication $( a \la b )( c \mu d )$ is then the sum of all these diagrams interpreted as elements of $\D$.

\begin{ex}
\label{examplemult}
Here is an example for the multiplication of two basis vectors, $\underline{\mu}\nu\overline{\la}$ and $\underline{\la}\nu\overline{\mu}$: The first surgery map merges two anticlockwise circles into one anticlockwise circle. The second surgery map is a split which sends the anticlockwise circle to the sum $(\underline{\la}\eta\overline{\la})+ (\underline{\la}\eta'\overline{\la})$ of the two possibilities of a pair of an anticlockwise circle and a clockwise circle as displayed:
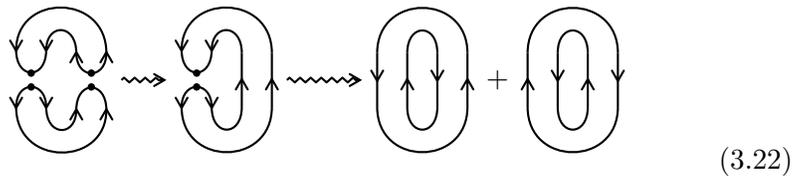
\begin{equation}
\label{exmul}
\begin{tikzpicture}[thick,>=angle 60,scale=.4]

\begin{scope}[yshift=1cm]
\draw [<-] (0,-.1) .. controls +(0,2) and +(0,2) .. +(3,0);
\draw (0,0) .. controls +(0,-1) and +(0,-1) .. +(1,0);
\fill (0.5,-.77) circle(3.5pt);
\draw [<-] (1,-.1) .. controls +(0,1) and +(0,1) .. +(1,0);
\draw [<->] (2,0) .. controls +(0,-1) and +(0,-1) .. +(1,0);
\fill (2.5,-.77) circle(3.5pt);
\end{scope}

\begin{scope}[yshift=-1cm,yscale=-1]
\draw [->] (0,-.1) .. controls +(0,2) and +(0,2) .. +(3,0);
\draw [<->] (0,0) .. controls +(0,-1) and +(0,-1) .. +(1,0);
\fill (0.5,-.77) circle(3.5pt);
\draw [->] (1,-.1) .. controls +(0,1) and +(0,1) .. +(1,0);
\draw (2,0) .. controls +(0,-1) and +(0,-1) .. +(1,0);
\fill (2.5,-.77) circle(3.5pt);
\end{scope}
\draw[->, line join=round, decorate, decoration={zigzag,segment length=4,amplitude=.9,post=lineto,
    post length=2pt}] (3.5,0) -- (5,0);

\begin{scope}[xshift=5.5cm]
\begin{scope}[yshift=1cm]
\draw [<-] (0,-.1) .. controls +(0,2) and +(0,2) .. +(3,0);
\draw (0,0) .. controls +(0,-1) and +(0,-1) .. +(1,0);
\fill (0.5,-.77) circle(3.5pt);
\draw [<-] (1,-.1) .. controls +(0,1) and +(0,1) .. +(1,0);
\draw (2,-.1) -- +(0,-1);
\draw (3,-.1) -- +(0,-1);
\end{scope}

\begin{scope}[yshift=-1cm,yscale=-1]
\draw (0,-.1) .. controls +(0,2) and +(0,2) .. +(3,0);
\draw [<->] (0,0) .. controls +(0,-1) and +(0,-1) .. +(1,0);
\fill (0.5,-.77) circle(3.5pt);
\draw (1,-.1) .. controls +(0,1) and +(0,1) .. +(1,0);
\draw [->] (2,-.1) -- +(0,-1);
\draw [->] (3,-.1) -- +(0,-1);
\end{scope}
\draw[->, line join=round, decorate, decoration={zigzag,segment length=4,amplitude=.9,post=lineto,
    post length=2pt}] (3.5,0) -- (6,0);
\end{scope}

\begin{scope}[xshift=12cm]
\begin{scope}[yshift=1cm]
\draw (0,-.1) .. controls +(0,2) and +(0,2) .. +(3,0);
\draw [->] (0,-.1) -- +(0,-1);
\draw (1,-.1) -- +(0,-1);
\draw (1,-.1) .. controls +(0,1) and +(0,1) .. +(1,0);
\draw [->](2,-.1) -- +(0,-1);
\draw (3,-.1) -- +(0,-1);
\end{scope}

\begin{scope}[yshift=-1cm,yscale=-1]
\draw (0,-.1) .. controls +(0,2) and +(0,2) .. +(3,0);
\draw (0,-.1) -- +(0,-1);
\draw [->] (1,-.1) -- +(0,-1);
\draw (1,-.1) .. controls +(0,1) and +(0,1) .. +(1,0);
\draw (2,-.1) -- +(0,-1);
\draw [->] (3,-.1) -- +(0,-1);
\end{scope}
\node at (4,0) {${\large +}$};
\end{scope}

\begin{scope}[xshift=17cm]
\begin{scope}[yshift=1cm]
\draw (0,-.1) .. controls +(0,2) and +(0,2) .. +(3,0);
\draw (0,-.1) -- +(0,-1);
\draw [->] (1,-.1) -- +(0,-1);
\draw (1,-.1) .. controls +(0,1) and +(0,1) .. +(1,0);
\draw (2,-.1) -- +(0,-1);
\draw [->] (3,-.1) -- +(0,-1);
\end{scope}

\begin{scope}[yshift=-1cm,yscale=-1]
\draw (0,-.1) .. controls +(0,2) and +(0,2) .. +(3,0);
\draw [->] (0,-.1) -- +(0,-1);
\draw (1,-.1) -- +(0,-1);
\draw (1,-.1) .. controls +(0,1) and +(0,1) .. +(1,0);
\draw [->] (2,-.1) -- +(0,-1);
\draw (3,-.1) -- +(0,-1);
\end{scope}
\end{scope}

\end{tikzpicture}
\end{equation}
\end{ex}

\begin{remark}
{\rm
The maybe artificially looking choice of the rightmost dotted cup/cap pair for the surgery procedure becomes transparent and consistent with \cite{BS1} if we use the approach from \cite{LS}. Rewriting the dotted cup diagrams in terms of symmetric cup diagrams as in \cite{LS} might turn neighboured cups into nested cups and then our choice of ordering makes sure that we always chose an outer pair precisely as in \cite{BS1}.
}
\end{remark}

It is important to note that there is only ever one admissible choice for the orientation of vertices lying on lines; the orientations of the vertices at the ends of rays never change in the generalized surgery procedure; the same holds for the cup and cap diagram at the bottom and the top. Up to isotopy and mirror image the nonzero surgery moves of two lines are
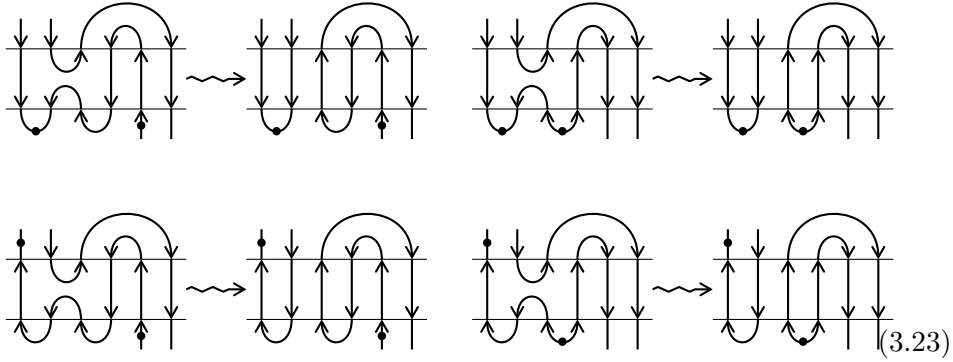
\begin{eqnarray}
\label{twoprop}
\label{lineline}
\usetikzlibrary{arrows}
\usetikzlibrary{decorations.pathmorphing}
\begin{tikzpicture}[thick,>=angle 60,scale=.4]
\draw[thin] (-.5,2) -- +(6,0);
\draw[thin] (-.5,0) -- +(6,0);
\draw [<-] (0,2) -- +(0,1);
\draw [<-] (1,2) -- +(0,1);
\draw [->] (2,2) .. controls +(0,2) and +(0,2) .. +(3,0);
\draw [<-] (3,2) .. controls +(0,1) and +(0,1) .. +(1,0);
\draw [<-] (0,0) -- +(0,2);
\draw [<-] (3,0) -- +(0,2);
\draw [->] (4,0) -- +(0,2);
\draw [->] (1,2) .. controls +(0,-1) and +(0,-1) .. +(1,0);
\draw [<-] (5,0) -- +(0,2);
\draw [<-] (1,0) .. controls +(0,1) and +(0,1) .. +(1,0);
\draw (0,0) .. controls +(0,-1) and +(0,-1) .. +(1,0);
\fill (0.5,-.74) circle(4pt);
\draw [<-] (2,0) .. controls +(0,-1) and +(0,-1) .. +(1,0);
\draw [->] (4,-1) -- +(0,1);
\fill (4,-.55) circle(4pt);
\draw (5,-1) -- +(0,1);

\draw [->,line join=round,decorate, decoration={zigzag,segment length=8,amplitude=.9,post=lineto,post length=4pt}]  (5.5,1) -- +(2,0);

\begin{scope}[xshift=8cm]
\draw[thin] (-.5,2) -- +(6,0);
\draw[thin] (-.5,0) -- +(6,0);
\draw [<-] (0,2) -- +(0,1);
\draw [<-] (1,2) -- +(0,1);

\draw [<-] (0,0) -- +(0,2);
\draw [<-] (1,0) -- +(0,2);
\draw [->] (2,0) -- +(0,2);
\draw [<-] (3,0) -- +(0,2);
\draw [->] (4,0) -- +(0,2);
\draw [<-] (5,0) -- +(0,2);

\draw [->] (2,2) .. controls +(0,2) and +(0,2) .. +(3,0);
\draw [<-] (3,2) .. controls +(0,1) and +(0,1) .. +(1,0);
\draw (0,0) .. controls +(0,-1) and +(0,-1) .. +(1,0);
\fill (0.5,-.74) circle(4pt);
\draw [<-] (2,0) .. controls +(0,-1) and +(0,-1) .. +(1,0);
\draw [->] (4,-1) -- +(0,1);
\fill (4,-.55) circle(4pt);
\draw (5,-1) -- +(0,1);
\end{scope}

\begin{scope}[xshift=15.5cm]
\draw[thin] (-.5,2) -- +(6,0);
\draw[thin] (-.5,0) -- +(6,0);
\draw [<-] (0,2) -- +(0,1);
\draw [<-] (1,2) -- +(0,1);
\draw [->] (2,2) .. controls +(0,2) and +(0,2) .. +(3,0);
\draw [->] (3,2) .. controls +(0,1) and +(0,1) .. +(1,0);

\draw [<-] (0,0) -- +(0,2);
\draw [->] (3,0) -- +(0,2);
\draw [<-] (4,0) -- +(0,2);
\draw [->] (1,2) .. controls +(0,-1) and +(0,-1) .. +(1,0);
\draw [<-] (5,0) -- +(0,2);
\draw [<-] (1,0) .. controls +(0,1) and +(0,1) .. +(1,0);

\draw (0,0) .. controls +(0,-1) and +(0,-1) .. +(1,0);
\fill (0.5,-.74) circle(4pt);
\draw [<->] (2,0) .. controls +(0,-1) and +(0,-1) .. +(1,0);
\fill (2.5,-.74) circle(4pt);
\draw (4,-1) -- +(0,1);
\draw (5,-1) -- +(0,1);

\draw [->,line join=round,decorate, decoration={zigzag,segment length=8,amplitude=.9,post=lineto,post length=4pt}]  (5.5,1) -- +(2,0);

\begin{scope}[xshift=8cm]
\draw[thin] (-.5,2) -- +(6,0);
\draw[thin] (-.5,0) -- +(6,0);
\draw [<-] (0,2) -- +(0,1);
\draw [<-] (1,2) -- +(0,1);

\draw [<-] (0,0) -- +(0,2);
\draw [<-] (1,0) -- +(0,2);
\draw [->] (2,0) -- +(0,2);
\draw [->] (3,0) -- +(0,2);
\draw [<-] (4,0) -- +(0,2);
\draw [<-] (5,0) -- +(0,2);

\draw [->] (2,2) .. controls +(0,2) and +(0,2) .. +(3,0);
\draw [->] (3,2) .. controls +(0,1) and +(0,1) .. +(1,0);
\draw (0,0) .. controls +(0,-1) and +(0,-1) .. +(1,0);
\fill (0.5,-.74) circle(4pt);
\draw [<->] (2,0) .. controls +(0,-1) and +(0,-1) .. +(1,0);
\fill (2.5,-.74) circle(4pt);
\draw (4,-1) -- +(0,1);
\draw (5,-1) -- +(0,1);
\end{scope}
\end{scope}

\begin{scope}[yshift=-7cm]
\draw[thin] (-.5,2) -- +(6,0);
\draw[thin] (-.5,0) -- +(6,0);
\draw (0,2) -- +(0,1);
\fill (0,2.55) circle(4pt);
\draw [<-] (1,2) -- +(0,1);
\draw [->] (2,2) .. controls +(0,2) and +(0,2) .. +(3,0);
\draw [<-] (3,2) .. controls +(0,1) and +(0,1) .. +(1,0);
\draw [->] (0,0) -- +(0,2);
\draw [<-] (3,0) -- +(0,2);
\draw [->] (4,0) -- +(0,2);
\draw [->] (1,2) .. controls +(0,-1) and +(0,-1) .. +(1,0);
\draw [<-] (5,0) -- +(0,2);
\draw [<-] (1,0) .. controls +(0,1) and +(0,1) .. +(1,0);
\draw [<-] (0,0) .. controls +(0,-1) and +(0,-1) .. +(1,0);
\draw [<-] (2,0) .. controls +(0,-1) and +(0,-1) .. +(1,0);
\draw [->] (4,-1) -- +(0,1);
\fill (4,-.55) circle(4pt);
\draw (5,-1) -- +(0,1);

\draw [->,line join=round,decorate, decoration={zigzag,segment length=8,amplitude=.9,post=lineto,post length=4pt}]  (5.5,1) -- +(2,0);

\begin{scope}[xshift=8cm]
\draw[thin] (-.5,2) -- +(6,0);
\draw[thin] (-.5,0) -- +(6,0);
\draw (0,2) -- +(0,1);
\fill (0,2.55) circle(4pt);
\draw [<-] (1,2) -- +(0,1);
\draw [->] (0,0) -- +(0,2);
\draw [<-] (1,0) -- +(0,2);
\draw [->] (2,0) -- +(0,2);
\draw [<-] (3,0) -- +(0,2);
\draw [->] (4,0) -- +(0,2);
\draw [<-] (5,0) -- +(0,2);

\draw [->] (2,2) .. controls +(0,2) and +(0,2) .. +(3,0);
\draw [<-] (3,2) .. controls +(0,1) and +(0,1) .. +(1,0);
\draw [<-] (0,0) .. controls +(0,-1) and +(0,-1) .. +(1,0);
\draw [<-] (2,0) .. controls +(0,-1) and +(0,-1) .. +(1,0);
\draw [->] (4,-1) -- +(0,1);
\fill (4,-.55) circle(4pt);
\draw (5,-1) -- +(0,1);
\end{scope}
\end{scope}

\begin{scope}[xshift=15.5cm,yshift=-7cm]
\draw[thin] (-.5,2) -- +(6,0);
\draw[thin] (-.5,0) -- +(6,0);
\draw (0,2) -- +(0,1);
\fill (0,2.55) circle(4pt);
\draw [<-] (1,2) -- +(0,1);
\draw [->] (2,2) .. controls +(0,2) and +(0,2) .. +(3,0);
\draw [->] (3,2) .. controls +(0,1) and +(0,1) .. +(1,0);

\draw [->] (0,0) -- +(0,2);
\draw [->] (3,0) -- +(0,2);
\draw [<-] (4,0) -- +(0,2);
\draw [->] (1,2) .. controls +(0,-1) and +(0,-1) .. +(1,0);
\draw [<-] (5,0) -- +(0,2);
\draw [<-] (1,0) .. controls +(0,1) and +(0,1) .. +(1,0);

\draw [<-] (0,0) .. controls +(0,-1) and +(0,-1) .. +(1,0);
\draw [<->] (2,0) .. controls +(0,-1) and +(0,-1) .. +(1,0);
\fill (2.5,-.74) circle(4pt);
\draw (4,-1) -- +(0,1);
\draw (5,-1) -- +(0,1);

\draw [->,line join=round,decorate, decoration={zigzag,segment length=8,amplitude=.9,post=lineto,post length=4pt}]  (5.5,1) -- +(2,0);

\begin{scope}[xshift=8cm]
\draw[thin] (-.5,2) -- +(6,0);
\draw[thin] (-.5,0) -- +(6,0);
\draw (0,2) -- +(0,1);
\fill (0,2.55) circle(4pt);
\draw [<-] (1,2) -- +(0,1);

\draw [->] (0,0) -- +(0,2);
\draw [<-] (1,0) -- +(0,2);
\draw [->] (2,0) -- +(0,2);
\draw [->] (3,0) -- +(0,2);
\draw [<-] (4,0) -- +(0,2);
\draw [<-] (5,0) -- +(0,2);

\draw [->] (2,2) .. controls +(0,2) and +(0,2) .. +(3,0);
\draw [->] (3,2) .. controls +(0,1) and +(0,1) .. +(1,0);
\draw [<-] (0,0) .. controls +(0,-1) and +(0,-1) .. +(1,0);
\draw [<->] (2,0) .. controls +(0,-1) and +(0,-1) .. +(1,0);
\fill (2.5,-.74) circle(4pt);
\draw (4,-1) -- +(0,1);
\draw (5,-1) -- +(0,1);
\end{scope}
\end{scope}

\end{tikzpicture}
\end{eqnarray}
where a dot/no dot on a segment means an odd resp. even number of dots.

\begin{remark}{\rm
The surgery rules involving lines are in fact induced from surgery rules involving closed components only by the following {\it extension rule}: To compute $(\un\la\nu\ov\mu)(\un\mu\nu'\ov\eta)$ we extend all involved weights by $r$ $\up$'s to the right, where $r$ is the maximum of the occurring rays at the top and bottom and consider the corresponding extended circle diagram (which contains now only closed components). Then we apply the normal surgery rules with the additional condition that a resulting diagram is zero if one of the added $\up$'s gets changed into a $\down$, ie. a circle containing an added $\up$ is turned clockwise. To see this note that a non-propagating line to the right of a propagating line is not orientable, hence the surgery involving a propagating and a non-propagating line must be zero as claimed in \eqref{yuk}. Every propagating line gets closed by a direct connection to the next free added $\up$ and one easily verifies that the only non-zero products are the ones displayed in \eqref{twoprop} in agreement with \eqref{yuk}. Now assume we have two non-propagating lines. In case they are in the same component, the surgery is a split which creates clockwise circles and the result is zero. Otherwise we have two possible scenarios (and their mirror images):
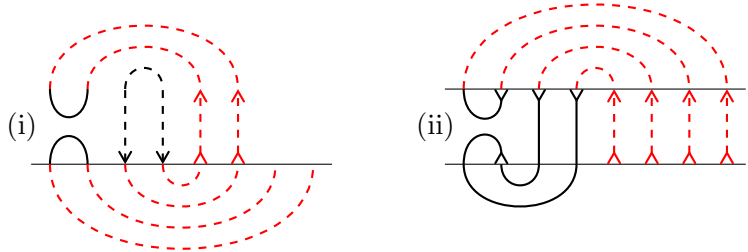
\vspace{-0.5cm}
\begin{eqnarray*}
\begin{tikzpicture}[thick,>=angle 60,scale=.5]
\node at (-.75,1) {(i)};
\draw[thin] (-.5,0) -- +(8,0);
\draw[-] (0,2) .. controls +(0,-1) and +(0,-1) .. +(1,0);
\draw[dashed] (2,2) .. controls +(0,.75) and +(0,.75) .. +(1,0);
\draw[dashed, red] (0,2) .. controls +(0,2.25) and +(0,2.25) .. +(5,0);
\draw[dashed, red] (1,2) .. controls +(0,1.5) and +(0,1.5) .. +(3,0);
\draw[dashed,<-] (2,0) -- +(0,2);
\draw[dashed,<-] (3,0) -- +(0,2);
\draw[dashed,red,>->] (4,0) -- +(0,2);
\draw[dashed,red,>->] (5,0) -- +(0,2);
\draw[dashed, red] (2,0) .. controls +(0,-1.5) and +(0,-1.5) .. +(3,0);
\draw[dashed, red] (3,0) .. controls +(0,-.75) and +(0,-.75) .. +(1,0);
\draw[-] (0,0) .. controls +(0,1) and +(0,1) .. +(1,0);
\draw[dashed, red] (1,0) .. controls +(0,-2.25) and +(0,-2.25) .. +(5,0);
\draw[dashed, red] (0,0) .. controls +(0,-3) and +(0,-3) .. +(7,0);

\begin{scope}[xshift=11cm]
\node at (-.75,1) {(ii)};
\draw[thin] (-.5,2) -- +(8,0);
\draw[thin] (-.5,0) -- +(8,0);
\draw[dashed, red] (2,2) .. controls +(0,1.5) and +(0,1.5) .. +(3,0);
\draw[dashed, red] (3,2) .. controls +(0,.75) and +(0,.75) .. +(1,0);
\draw[dashed, red] (1,2) .. controls +(0,2.25) and +(0,2.25) .. +(5,0);
\draw[dashed, red] (0,2) .. controls +(0,3) and +(0,3) .. +(7,0);
\draw[-<] (0,2) .. controls +(0,-1) and +(0,-1) .. +(1,0);
\draw[-<] (0,0) .. controls +(0,1) and +(0,1) .. +(1,0);
\draw[-<] (2,0) -- +(0,2);
\draw[-<] (3,0) -- +(0,2);
\draw[dashed,red,>->] (4,0) -- +(0,2);
\draw[dashed,red,>->] (5,0) -- +(0,2);
\draw[dashed,red,>->] (6,0) -- +(0,2);
\draw[dashed,red,>->] (7,0) -- +(0,2);
\draw[-] (1,0) .. controls +(0,-.75) and +(0,-.75) .. +(1,0);
\draw[-] (0,0) .. controls +(0,-1.5) and +(0,-1.5) .. +(3,0);

\end{scope}
\end{tikzpicture}
\end{eqnarray*}

where in case (i) the black dashed part of the diagram can't have any dots and is therefore not orientable and in the (ii) the merge forces to reorient one of the lines, hence the result is again zero.
}
\end{remark}

The following will be a consequence of Theorem~\ref{algebra_structure} below
\begin{theorem}
The multiplication $((a\la b), (c\mu d))\mapsto (a\la b)(c\mu d)$ defines an associative graded algebra structure on $\D^{{\mathbb F}_2}$.
\end{theorem}

\subsection{The multiplication over $\mZ$ or $\mC$}
\label{F2}
We want to define the product of two basis elements $(a \la b), (c \mu d)\in\D$ and proceed exactly as in Section~\ref{multF2Def} except that the case (ii) and the last case in \eqref{circline} have to be adjusted in a subtle way with appropriate signs. For simplicity, we work here over $\mC$, but all statements are valid over $\mZ$. We start by identifying all the possible diagrams in the surgery procedure from the last section as elements of some complex vector space.

We identify
$(a \la b)$ and $(c \mu d)$ with basis vectors in $\mC[X_{1,1},\ldots, X_{1,k}]/I(a,b)$ and in $\mC[X_{-1,1},\ldots, X_{-1,k}]/I(c,d)$ respectively via Proposition~\ref{coho} (ignoring the grading for the moment). The indices $1$ and $-1$ indicate whether they belong to the top or bottom diagram in the surgery procedure.

Then the pair $(a \la b)$, $(c \mu d)$ corresponds to a basis vector $v$ in
\begin{eqnarray}
\cM:=\mC[X_{\mp1,1},\ldots, X_{\mp1,k}]/\left(\langle X_{\mp 1,i}^2\mid 1\leq i\leq k\rangle +I(a,b)+I(c,d)\right)
\end{eqnarray}
We identify then the diagram $(a \la b\mu d)$ with the image of $v$ in the quotient
\begin{eqnarray}
\cM/(X_{1,i}-X_{-1,i}\mid i\in I),
\end{eqnarray}
 where $I\subset \{1,\ldots ,k\}$ denotes the vertices where rays were stitched.

More generally, the diagrams appearing during the multiplication rule in Section~\ref{multF2Def} are all obtained from  $(a \la b\mu d)$ by applying some surgeries of the form \eqref{surgery} to the middle part. Given such a diagram we first ignore the orientation to obtain a {\it double circle diagram} $G$. To $G$ we assign the vector space $\cM(G)$, where $\cM(G)=\{0\}$ if $G$ cannot be equipped with an orientation and otherwise
\begin{eqnarray}
\label{MG}
\cM(G)=\mC[X_{\mp1,1},\ldots, X_{\mp1,k}]/\left(\langle X_{(\mp 1,i)}^2 \mid 1\leq i\leq k\rangle +I\right)
\end{eqnarray}
where the ideal $I$ is generated by the following relations with $t\in\{\mp1\}$ and $1\leq j\leq i\leq k$
\begin{eqnarray*}
\cM(G)= \left\lbrace
\begin{array}{cl}
 X_{(t,i)} + X_{(t,j)} & \text{ if $(t,i)$ connects $(t,j)$ by an undotted arc}\\
 X_{(t,i)} - X_{(t,j)} & \text{ if $(t,i)$ connects $(t,j)$ by a dotted arc}\\
 X_{(t,i)} & \text{ if $(t,i)$ is the endpoint of a ray in $a$ or $d$},\\
 X_{(-1,i)} - X_{(1,i)} & \text{ if $(-1,i)$ connects $(1,i)$ by a straight line.}
\end{array} \right\rbrace.
\end{eqnarray*}

By an {\it orientation} of a double circle diagram $G$ we mean a consistent labelling of the vertices $\{(1,i)\mid 1\leq i\leq \}$ and the vertices $\{(-1,i)\mid 1\leq i\leq \}$ by weights $\nu$ and $\nu'$ respectively such that the $i$th vertex is labelled by the $i$th label of $\nu$ and $\nu'$ respectively and all components are oriented (again $\bullet$'s should be seen as orientation reversing points), the diagrams to the left and in the middle of \eqref{exmul} are examples of oriented double circle diagrams, the ones to the right are two circle diagrams obtained from double circle diagrams after applying he collapse map from Lemma~\ref{lem:collapsing} below.


\begin{lemma}
We have $\cM(G)\not=\{0\}$ iff $G$ can be orientated.
\end{lemma}

\begin{proof}
Since the defining relations \eqref{MG} of $\cM(G)$ only couple variables $X_{t,i}$, where the indices $(t,i)$ label vertices on the same component, the  $X_{t,i}$'s from different components are linearly independent or zero. Then apply the second part of Lemma~\ref{signmove}.
\end{proof}

%

\begin{lemma}
\label{lem:collapsing}
Assume $G$ is a double circle diagram obtained from $(a \la b \mu d)$ at the end of all surgeries (ie. all middle cups and caps in $b$ are turned into rays). Let $a=\ov{\alpha}$ and $d=\un\beta$.
Then there exists an isomorphism of graded vector spaces
\begin{eqnarray}
{\rm collapse}:\quad\cM(G) &\longrightarrow &{}_\alpha(\D)_\beta{},\nonumber\\
X_{\mp1,i}&\mapsto&X_i.
\end{eqnarray}
\end{lemma}
\begin{proof}
This follows directly from the definitions.
\end{proof}

\begin{lemma}
\label{join}
Assume that $G$ and $G'$ are double circle diagrams obtained from  $(a \la b \mu d)$ such that $G'$ is obtained from $G$ by one additional surgery which is a merge. Then there is a canonical surjection of vector spaces
\begin{eqnarray}
\op{surj}:\quad\cM(G)&\longrightarrow&\cM(G')\\
X_{\mp1,i}&\longmapsto&X_{\mp1,i}
\end{eqnarray}
\end{lemma}
As an example take for $G$ the first diagram in \eqref{exmul} and for $G'$ the second. We call $\op{surj}$ the associated {\it surgery surjection map}.

\begin{proof}
Assume the surgery merges two circles with the relevant cup-cap pair connecting vertices $(\mp1,i)$ and $(\mp1,j)$. The pair implies the relations $X_{\mp,i}=-X_{\mp,j}$ if it is undotted and $X_{\mp,i}=X_{\mp,j}$ if it is dotted. However, the same relations are already given by the remaining parts of the circles which are not changed under the surgery move and moreover all other relations stay valid when applying the move. If the surgery merges a line and a circle then again the previous relations stay valid, but the move annihilates the variables $X_{\mp,i}$ corresponding to the circle. In case two lines get merged either $\cM(G')=\{0\}$ or $\cM(G)\cong \cM(G'$). The claim follows.
\end{proof}

\begin{definition}
\label{surg}
Assume that $G$ and $G'$ are double circle diagrams obtained from  $(a \la b \mu d)$ such that $G'$ is obtained from $G$ by one additional surgery which is a split. Assume that the cup-cap pair involved in the surgery connects vertex $(t,i)$ with vertex $(t,j)$ where $i<j$ and $t\in\{\mp1\}$. Then the {\it surgery map} is defined to be the linear map
\begin{eqnarray*}
\op{surg}:\quad\cM(G)&\longrightarrow&\cM(G')
\end{eqnarray*}
\begin{eqnarray*}
f&\longmapsto&
\begin{cases}
0&\text{if $G'$ can't be oriented and otherwise}\\
(-1)^{\op{pos}(i)}{(X_{t,j}-X_{t,i})f}&\text{if the cup-cap pair is undotted,}\\
(-1)^{\op{pos}(i)}{(X_{t,j}+X_{t,i})f}&\text{if the cup-cap pair is dotted}.
\end{cases}
\end{eqnarray*}
 where $f\in \cA=\mC[X_{\mp1,1},\ldots X_{\mp1,1}]$. Note that in the second and third case the choice of $t$ is irrelevant, since the two variables get identified in $\cM(G')$.
 In the following we will usually not distinguish in notation between $g\in \cA$ and its canonical image $g$ in the respective quotients.
\end{definition}

\subsubsection*{Signed surgery rules}
We summarize our signed surgery rules:\\

\begin{mdframed}
\begin{enumerate}[({\bf Surg}-1)]
\item Merge: the surgery surjection map $X_{\mp1,r}\longmapsto X_{\mp1,r}$ for all $r$.
\item
Split: \\
\begin{tikzpicture}[thick,scale=0.4]
\draw (0,0) node[above]{$i$} .. controls +(0,-1) and +(0,-1) ..
+(1.5,0) node[above]{$j$};
\draw (0,-2) .. controls +(0,1) and +(0,1) .. +(1.5,0);
\node at (13,-1) {multiplication with $(-1)^{\op{pos}(i)}(X_{j}-X_{i})$};
\end{tikzpicture}
\\
\begin{tikzpicture}[thick,scale=0.4]
\draw (0,0) node[above]{$i$} .. controls +(0,-1) and +(0,-1) ..
+(1.5,0) node[above]{$j$};
\fill (0.75,-.75) circle(4pt);
\draw (0,-2) .. controls +(0,1) and +(0,1) .. +(1.5,0);
\fill (0.75,-1.25) circle(4pt);
\node at (13,-1) {multiplication with
$(-1)^{\op{pos}(i)}(X_{j}+X_{i})$};
\end{tikzpicture}
\item the rules \eqref{circline} and \eqref{lineline} involving lines.
\end{enumerate}
\end{mdframed}

\begin{ex}
\label{signedmult}
Going back to Example~\ref{examplemult}. The product of the two basis vectors
$(\underline{\la}\nu\overline{\la})(\underline{\la}\nu\overline{\mu})$ corresponds to $(-1)^2(X_2+X_1)$ in $\cM( \overline{\la}\underline{\la})$ which equals $-X_3-X_4$ in $\cM(\overline{\la}\underline{\la})$. Under the identification with standard basis vectors it gets identified with the negative of the sum $(\underline{\la}\nu\overline{\la})+(\underline{\la}\nu'\overline{\la})$ of basis vectors displayed in Example~\ref{examplemult}.
\end{ex}

\begin{theorem}
\label{algebra_structure}
The generalized surgery procedure with the signed rules ({Surg}-1)-({Surg}-3) defines a graded associative unitary algebra structure on $\D$ with pairwise orthogonal primitive
idempotents $e_\la=(\underline{\la} \la \overline{\la})$ for $\la\in\La$ spanning the semisimple part of degree zero and unit $1=\sum_{\la\in \La} e_\la$.
\end{theorem}

\begin{proof}
It is clear from the definitions that the only circle diagrams of degree zero in $\B$ are the diagrams of the form $e_\la$ and
\begin{eqnarray*}
e_\la ( a \mu b )=
\begin{cases}
( a \mu b )&\text{if $\underline{\la}=a$,}\\
0&\text{otherwise,}
\end{cases}
&\quad&
( a \mu b ) e_\la
=
\begin{cases}
( a \mu b ) &\text{if $b = \overline{\la}$,}\\
0&\text{otherwise,}
\end{cases}
\end{eqnarray*}
for $\la \in \La$ and any basis vector $(a \mu b)\in \D$.
This implies that the vectors
$\{e_\la\:|\:\la \in \La\}$ are mutually orthogonal
idempotents whose sum is the identity in  $\D$. The multiplication is graded by Proposition~\ref{forgotgraded} below. The highly non-trivial statement here is the associativity of the multiplication. The proof of this will occupy the next section; the claim follows then from Proposition~\ref{prop:multiplication_associative} below.
\end{proof}

\begin{corollary}
\label{atypicality}
The algebra $\D$ depends up to canonical isomorphism only on the atypicality or defect of the block, not on the block itself.
\end{corollary}

\begin{proof}
Given two blocks $\La$ and $\La'$ with the same atypicality then $|P_\bullet(\Lambda)|=|P_\bullet(\Lambda')|$ and hence the cup diagrams and basis vectors in $\D$ and $\mathbb{D}_{\La'}$ are the same up to some vertices labelled $\circ$ or $\times$. If $P_\bullet(\Lambda)=\{i_1<\ldots<i_r\}$ and  $P_\bullet(\Lambda')=\{i'_1<\ldots<i'_k\}$ then the identification $X_{i_s}\mapsto X_{i'_s}$ for $1\leq s\leq r$ defines the canonical isomorphism.
\end{proof}

\begin{prop}
\label{forgotgraded}
The multiplication from Theorem~\ref{algebra_structure} is of degree zero.
\end{prop}

\begin{proof}
It is enough to verify that the surgery moves are degree preserving. Assume first that we merge two circles, say $C$ and $D$ containing the cup respectively cap. If the involved cup-cap pair has no dots we have one of the following situation:
\begin{equation}
(I)\quad
\usetikzlibrary{arrows}
\begin{tikzpicture}[thick,>=angle 60, scale=0.8]
\draw [>->] (0,0) .. controls +(0,-1) and +(0,-1) .. +(1,0);
\draw [<-<] (0,-2) .. controls +(0,1) and +(0,1) .. +(1,0);
\end{tikzpicture}
\quad\quad
(II)
\quad
\usetikzlibrary{arrows}
\begin{tikzpicture}[thick,>=angle 60, scale=0.8]
\draw [<-<] (0,0) .. controls +(0,-1) and +(0,-1) .. +(1,0);
\draw [>->] (0,-2) .. controls +(0,1) and +(0,1) .. +(1,0);
\end{tikzpicture}
\quad\quad
(III)
\quad
\usetikzlibrary{arrows}
\begin{tikzpicture}[thick,>=angle 60, scale=0.8]
\draw [>->] (0,0) .. controls +(0,-1) and +(0,-1) .. +(1,0);
\draw [>->] (0,-2) .. controls +(0,1) and +(0,1) .. +(1,0);
\end{tikzpicture}
\quad\quad
(IV)
\quad
\usetikzlibrary{arrows}
\begin{tikzpicture}[thick,>=angle 60, scale=0.8]
\draw [<-<] (0,0) .. controls +(0,-1) and +(0,-1) .. +(1,0);
\draw [<-<] (0,-2) .. controls +(0,1) and +(0,1) .. +(1,0);
\end{tikzpicture}
\end{equation}
If $C$ and $D$ are clockwise, then nothing is to check.
Otherwise start with Case (I). If both are anticlockwise, then surgery sticks them together and the degrees are obviously preserved. If one of them is clockwise, say $C$, then it must satisfy the following property (R): it must involve a dotted cap above the upper weight line and to the right of the cup-cap pair.

Then $C$ contains all rightmost points and so again the circles are just stitched together; similar if $D$ is the clockwise circle, but then there is a dotted cup below the lower weight line to the right of our cup-cap pair. If in case (II) a circle is anticlockwise then it must satisfy property (R), hence not both of them can be anticlockwise. If one is anticlockwise, then surgery is the same as stitching the two circles together which looses a clockwise cup and a clockwise cap and then revert the orientation of the resulting circle which increases the degree by $2$ and we are done. If in case (III) both circles are anticlockwise then $C$ satisfies $(R)$ and hence applying surgery is the same as reverting the orientation of $D$ which adds $2$ to the degree and then stitching together which looses a clockwise cap and a newly created clockwise cup, hence in total the degree is preserved. If $C$ is clockwise and $D$ is anticlockwise then we again can first revert the orientation of $D$ and then stitch which keeps the degree. For $C$ anticlockwise and $D$ clockwise both circles must satisfy $(R)$ which is impossible. If in case (IV) both are anticlockwise then $D$ satisfies (R); if one circle is clockwise then it must be $D$. In either case we can first swap the orientation of $C$ and then stitch which first adds and then subtracts $2$ in the degree.

If the cup-cap pair is dotted, then we have the cases \eqref{4cases} and the arguments are similar and therefore omitted.

If the surgery splits two circles then the circle can have one of the shapes in \eqref{kidneys} or the rightmost in \eqref{H}. Let us start with \eqref{kidneys} on the right. Then by our decoration rules the circle is anticlockwise in Case (I) and gets split into the two combinations of one clockwise and one anticlockwise. One of these combinations is obtained by just stitching together hence the degree is preserved. In case (II) the circle is clockwise and created two clockwise circles, hence increases the degree by two, but two clockwise cups gets lost as well. Case (III) and (IV) are impossible. Now consider the left configuration of \eqref{kidneys} (the arguments for \eqref{H} are the same). In Case (I) the circle is anticlockwise if there are no dots and again we just stitch to get one of the two summands in the result. If there is a dot on the top and bottom, then the circle is clockwise and again we just stitch. In general we argue similarly but taking the number of dots modulo two. In Case (II) the circle is clockwise if there are no dots. Surgery is the same as stitching together and swapping the orientation of the inner circle; hence we loose the cup-cap pair of degree $2$ and increase afterwards the degree by $2$. The circle is anticlockwise if it contains a dot on the top and bottom. Then stitching produces two nested anticlockwise circles, swapping the orientation of one increases the degree by $2$, but we also lost the cup-cap pair of degree $2$. Again Cases (III) and (IV) are impossible. Again we omit the dotted cases \eqref{4cases}.
\end{proof}

\subsection{Proof of associativity}
To prove associativity we will need the notion of a \emph{triple circle diagram}. It is defined in the same way as a double circle diagram except that we put three circle diagrams on top of each other instead of two. The triple circle diagram associated to the three circle diagrams $ab$, $b^*c$, and $c^*d$ will be denoted by $(abb^*cc^*d)$ and as in the case of a double circle diagram we assume that neither $b$ nor $c$ have any dotted rays. To $G=(abb^*cc^*d)$ we associate the vector space
$$\cT(G) := \frac{\mC[X_{t,1},\ldots,X_{t,k}\mid t \in \{0, \pm 2\}]}{\left(\left\langle X^2_{t,i} \mid 1 \leq i \leq k, t \in \{0, \pm 2\}\right\rangle + I(G)\right)},$$
where the ideal $I(G)$ is generated, in analogy to the double circle diagram case, by the following relations with $t,s \in \{0,\pm 2\}$ and $1 \leq i, j \leq k$
$$
\left\lbrace
\begin{array}{cl}
X_{t,i} + X_{t,j} & \text{if } (t,i) \text{ is connected to } (t,j) \text{ by an undotted arc} \\
X_{t,i} - X_{t,j} & \text{if } (t,i) \text{ is connected to } (t,j) \text{ by a dotted arc} \\
X_{t,i} & \text{if } (t,i) \text{ is the endpoint of a ray in $a$ or $d$} \\
X_{t,i} - X_{s,i} & \text{if } (t,i) \text{ is connected to } (s,i) \text{ by a straight line}\\
& \text{in } (b^*b) \text{ or } (c^*c).
\end{array}
\right\rbrace.
$$
As in the case of double circle diagrams we have collapsing maps, but now of three different types, depending on whether $b$, $c$ or both only contain rays.
Let $G^{\rm up}$ be the double circle diagram $(acc^*d)$. If $b$ contains only rays then we have an isomorphism
\begin{eqnarray}
{\rm collapse}^{\rm up}:\quad\cT(G) &\longrightarrow & \cM(G^{\rm up}) \\
X_{2,i} & \mapsto & X_{1,i} \nonumber\\
X_{0,i} & \mapsto & X_{1,i} \nonumber\\
X_{-2,i} & \mapsto & X_{-1,i}.\nonumber
\end{eqnarray}
Similarly, given a double circle diagram $G^{\rm low} = (abb^*d)$. If $c$ contains only rays, we have a collapsing isomorphism ${\rm collapse}^{\rm low}$ from $\cT(G)$ to $\cM(G^{\rm low})$; and finally if both $b$ and $c$ contain only rays we have the total collapsing isomorphism
\begin{eqnarray}
{\rm collapse}:\cT(G) &\longrightarrow & {}_a(\D)_d \nonumber\\
X_{t,i} & \mapsto & X_i \text{ for } t \in \{0,\pm 2\}.
\end{eqnarray}

\begin{prop} \label{prop:multiplication_associative}
There is an equality of maps $\D \otimes \D \otimes \D \longrightarrow \D$:
$${\rm mult} \circ ({\rm mult} \otimes \Id)={\rm mult} \circ (\Id \otimes {\rm mult}).$$
\end{prop}
\begin{proof}
We restrict ourself to a summand ${}_a(\D)_b \otimes_\mC {}_c(\D)_d \otimes_\mC {}_e(\D)_f$ in the triple tensor product of $\D$ with itself. If $b \neq c^*$ or $d \neq e^*$ then both maps are zero, so we can assume that $b=c^*$ and $d = e^*$. Let $b_0$ and $d_0$ denote the cup diagrams obtained from $b$ and $d$ by eliminating all dots on rays. Then we can form the double circle diagram $G^0_0 = (ab_0b_0^*d_0d_0^*f)$.

In analogy to the double circle case we have an embedding (again denoted by ${\rm glue}$) of ${}_a(\D)_b \otimes_\mC {}_{b^*}(\D)_d \otimes_\mC {}_{d^*}(\D)_f$ into $\cT(G_0^0)$ by identifying the triple tensor product with the space
$$\cT := \frac{\mC[X_{t,1},\ldots,X_{t,k}\mid t \in \{0, \pm 2\}]}{\left(\left\langle X^2_{t,i} \mid 1 \leq i \leq k, t \in \{0, \pm 2\}\right\rangle + I(a,b) + I(b^*,c) + I(c^*,d) \right)},$$
where we use the $X_{2,i}$'s as generators of $I(a,b)$, the $X_{0,i}$'s for $I(b^*,d)$, and finally the $X_{-2,i}$'s for $I(d^*,f)$.

Now let $G^i_j$ be the triple circle diagram obtained from $G^0_0$ by changing the first $i$ (counted from the left) cup-cap pairs of $b_0$  into pairs of rays and the first $j$ (counted from the right) cup-cap pairs of $c_0$ into pairs of rays.

The surgery maps from Lemma~\ref{join} and Definition~\ref{surg} can be applied to these vector spaces as well; for simplicity we always multiply with the $X_{0,i}$'s when a split occurs. Henceforth we have surgery maps
\begin{eqnarray}
\label{G1}
{\rm surg^i}:\quad\cT(G^i_j) \longrightarrow \cT(G^{i+1}_j),
\end{eqnarray}
and
\begin{eqnarray}
\label{G2}
{\rm surg_j}:\quad\cT(G^i_j) \longrightarrow \cT(G^i_{j+1}).
\end{eqnarray}
Let $n({\rm up})$ be the number of cup-cap pairs in $b_0$ and $n({\rm low})$ the number of cup-cap pairs in $d_0$. Then the above surgery maps fit into a diagram as in Figure~\ref{fig:huge}.

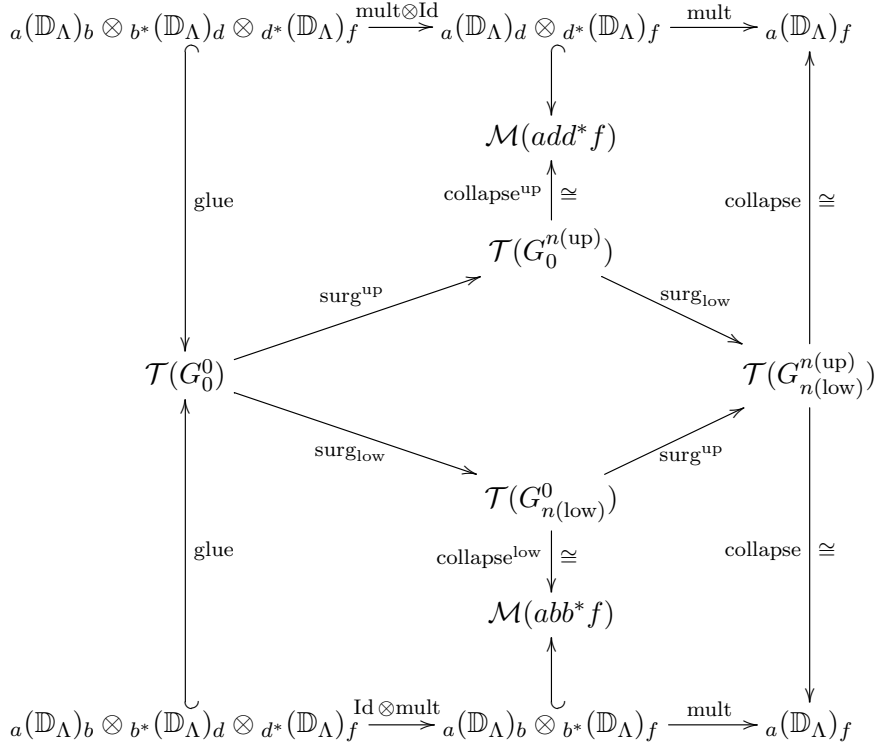
\begin{figure}
\begin{eqnarray*}
\begin{xy}
  \xymatrix{
{}_a(\D)_b \otimes {}_{b^*}(\D)_d \otimes {}_{d^*}(\D)_f \ar[r]^{\qquad{\rm mult} \otimes \Id} \ar@{^{(}->}[ddd]^{\rm glue} & {}_a(\D)_d \otimes {}_{d^*}(\D)_f \ar[r]^{\qquad{\rm mult}} \ar@{^{(}->}[d]& {}_{a}(\D)_f\\
& \cM(add^*f) & \\
& \cT(G_0^{n({\rm up})}) \ar[u]^{\rm collapse^{\rm up}}_\cong \ar[rd]^{\!\! {\rm surg}_{\rm low}}\\
\cT(G_0^0) \ar[ru]^{\!\! {\rm surg}^{\rm up}} \ar[rd]_{\!\! {\rm surg}_{\rm low}}& & \cT(G_{n({\rm low})}^{n({\rm up})}) \ar[ddd]_{\rm collapse}^\cong \ar[uuu]^{\rm collapse}_\cong\\
& \cT(G_{n({\rm low})}^0) \ar[d]_{\rm collapse^{\rm low}}^\cong \ar[ru]_{\!\!\! {\rm surg}^{\rm up}}\\
& \cM(abb^*f) & \\
{}_a(\D)_b \otimes {}_{b^*}(\D)_d \otimes {}_{d^*}(\D)_f \ar[r]^{\qquad \Id \otimes {\rm mult}} \ar@{_{(}->}[uuu]_{\rm glue}&  {}_a(\D)_b \otimes {}_{b^*}(\D)_f \ar[r]^{\qquad{\rm mult}} \ar@{_{(}->}[u]& {}_{a}(\D)_f
  }
\end{xy}
\end{eqnarray*}
\caption{The associativity}
\label{fig:huge}
\end{figure}

Here the four squares, the two at the top and the two at the bottom, commute by construction. The non-trivial claim is that the inner rhombus commutes which we will deduce by successive application of Lemma~\ref{lemlow} below. Thus the two maps from the proposition coincide.
\end{proof}

 It remains to show that the inner rhombus in Proposition~\ref{prop:multiplication_associative} commutes, which will require a more involved combinatorial and topological argument. We keep the setup from Proposition~\ref{prop:multiplication_associative} with the following additional notation.

Fix the diagram $G_j^i$ with rightmost outer cup-cap pair $\beta^{\rm up}$ in the top row and rightmost outer cup-cap pair $\beta^{\rm low}$ in the lower row. Let $a<b$ such that $\beta^{\rm up}$ connects $(a,2)$ and $(b,2)$ (and then also $(a,0)$ and $(b,0)$). Analogously let $c<d$ be such that $(c,0)$ and $(d,0)$, resp. $(c,-2)$ and $(d,-2)$, are connected by $\beta^{\rm low}$. Set $Q^{\rm up} = \{(a,2), (b,2), (a,0), (b,0)\}$ and $Q^{\rm low} = \{(c,-2), (d,-2), (c,0), (d,0)\}$, the set of vertices of the upper respective lower cup-cap pair, and let $Q = Q^{\rm up} \cup Q^{\rm low}$ be the union.
We say that two distinct vertices in $Q$ are \emph{connected} (outside $\beta^{\rm up}$ and $\beta^{\rm low}$) if there exists a sequence of arcs connecting them, not involving the arcs from our distinguished cup-cap pairs in $G^i_j$. In particular the connection does not contain any other element from $Q$.

\begin{lemma} \label{lem:nodots_on_inner}
\begin{enumerate}[1.)]
\item
A point of $Q^{\rm up}$ (resp. $Q^{\rm low}$) can not be connected to its diagonally opposite point (e.g. $(a,2)$ cannot be connected to $(b,0)$).
\label{rem:no_diagonal_connection}
\item Assume $C$ is an inner sequence of arcs that connects two points in $Q$. Then $C$ does not contain any dotted cups or caps.
\end{enumerate}
\end{lemma}
\begin{proof}
Assume that such a diagonal connection exists, say from $(a,2)$ to $(b,0)$. Then $(a,0)$ cannot be connected to any remaining point without crossing this connection or passing between the cup-cap pair neither of which is allowed. The second part follows from the definitions noting that to the right of $\beta^{\rm up}$ and $\beta^{\rm low}$ there are no cups or caps.
\end{proof}

\begin{lemma}
The maps \eqref{G1} and \eqref{G2} fit into a commutative diagram
\begin{eqnarray*}
\begin{xy}
\xymatrix{
\cT(G_j^i) \ar[r]^{{\rm surg}^i} \ar[d]_{{\rm surg}_j}& \cT(G_j^{i+1}) \ar[d]^{{\rm surg}_j}\\
\cT(G_{j+1}^i) \ar[r]_{{\rm surg}^i} & \cT(G_{j+1}^{i+1}).
}
\end{xy}
\end{eqnarray*}
\end{lemma}

For the proof we distinguish three possible scenarios, depending on the number of points in $Q^{\rm up}$ that are connected to points in $Q^{\rm low}$, we denote the number of these connections by $C(i,j)$.

The case $C(i,j)=4$ will be most involved and requires some additional preparations. For this let $\widetilde{G_j^i} \in \mR_{\geq 0} \times \mR$ denote the diagram that we obtain by eliminating all arcs except for the ones that connect points in $Q$ either to each other or to infinity, ie. we are only left with those lines and circles that include one or more points from $Q$. A sequence $C$ of arcs that connect points in $Q$ is called \emph{inner} if it cannot be connected with the left boundary of $\mR_{\geq 0} \times \mR$ by a path not crossing any arc in $\widetilde{G_j^i} \in \mR_{\geq 0} \times \mR$ except by one that goes either through $\beta^{\rm up}$ or $\beta^{\rm low}$.

\begin{lemma} \label{lem:exist_inner_connections}
Assume $C(i,j)=4$ and denote the sequences of arcs that connect the four pairs of points by $C_1$, $C_2$, $C_3$ and $C_4$. Then two of these sequences of arcs are inner.
\end{lemma}
\begin{proof}
We may assume that $C_1$ is the arc connecting $(a,2)$ with some point $p_1 \in Q^{\rm low}$ and $C_2$ the one that connects $(a,0)$ to $p_2$ (otherwise we relabel). Since $p_1$ and $p_2$ cannot be diagonally opposite of each other in $Q^{\rm low}$ (see Lemma~\ref{lem:nodots_on_inner}), they can be connected by the cup or cap from $\beta^{\rm low}$ or by a horizontal edge. Similarly we can connect $(a,2)$ and $(a,0)$ by a horizontal edge and obtain a circle $D$. Since arcs do not intersect there are two possible cases. Either the remaining four points are all in the interior of this circle $D$ or they are all outside of $D$. If they are in the interior then $C_3$ and $C_4$ are inner. If they are outside of the circle $D$ we can connect them to obtain a second circle $D'$. If $D$ lies in the interior of $D'$ then $C_1$ and $C_2$ are inner. Otherwise, if none of the circles lies in the interior of the other, we reconnect the points complementary, ie. a point that was formerly connected by a cup is now connected by a horizontal edge:
\begin{eqnarray} \label{reconnect}
\begin{tikzpicture}[thick, snake=zigzag, line before snake = 2mm, line after snake = 2mm, scale=0.7]
\node at (0,-.05) {$\bullet$};
\node at (.5,-.05) {$\bullet$};
\node at (0,1.45) {$\bullet$};
\node at (0.5,1.45) {$\bullet$};
\draw (0,1.5) .. controls +(0,-.5) and +(0,-.5) .. +(.5,0);
\draw (0,0) .. controls +(0,.5) and +(0,.5) .. +(.5,0);
\draw[->, snake] (1,.75) -- +(1,0);

\node at (2.5,-.05) {$\bullet$};
\node at (3,-.05) {$\bullet$};
\node at (2.5,1.45) {$\bullet$};
\node at (3,1.45) {$\bullet$};
\draw (2.5,0) -- +(0,1.5);
\draw (3,0) -- +(0,1.5);

\begin{scope}[xshift=5.5cm]
\node at (0.5,-.05) {$\bullet$};
\node at (1,-.05) {$\bullet$};
\node at (0.5,1.45) {$\bullet$};
\node at (1,1.45) {$\bullet$};

\draw (0.5,0) -- +(0,1.5);
\draw (1,0) -- +(0,1.5);
\draw[->, snake] (1.5,.75) -- +(1,0);

\node at (3,-.05) {$\bullet$};
\node at (3.5,-.05) {$\bullet$};
\node at (3,1.45) {$\bullet$};
\node at (3.5,1.45) {$\bullet$};
\draw (3,1.5) .. controls +(0,-.5) and +(0,-.5) .. +(.5,0);
\draw (3,0) .. controls +(0,.5) and +(0,.5) .. +(.5,0);

\end{scope}
\end{tikzpicture}
\end{eqnarray}

This will create two new circles nested inside each other. We can find $k,l$ such that $C_k$ and $C_l$ lie on the inner circle, then obviously $C_k$ and $C_l$ are inner.
\end{proof}

\begin{proof}[Proof of Lemma~\ref{lemlow}]
We will use notation as above for the rightmost cup-cap pairs in $G_j^i$ and the corresponding points that they connect.

\textbf{Case $C(i,j)=0$:} Here the two cup-cap pairs are completely independent of each other since the points of $Q^{\rm up}$ and $Q^{\rm low}$ lie on distinct circles and lines. The surgery moves are always the same no matter in which order they are performed, thus the diagram commutes.

\textbf{Case $C(i,j)=2$:} For this situation let $p_1$ and $p_2$ be the two points in $Q^{\rm up}$ (resp. $Q^{\rm low}$) that are connected to each other. If they are directly connected via $\beta^{\rm up}$ (resp. $\beta^{\rm low}$) then they lie on a circle that does not contain any of the other six points. The surgery at $\beta^{\rm up}$ (resp. $\beta^{\rm low}$) will always merge two circles and thus the map is induced from the identity on the polynomial ring in three sets of variables. If on the other hand $p_1$ and $p_2$ are not directly connected, then they need to be directly above each other by Remark~\ref{rem:no_diagonal_connection}, hence connected by a vertical edge after the surgery has been performed. In this case the surgery at this pair will always split a circle into two and thus the map is induced from the multiplication with the same element that only depends on $a$ and $b$ (resp. $c$ and $d$). Therefore, the diagram commutes in this case as well.

\textbf{Case $C(i,j)=4$:} This is the most involved case and a lengthy case by case argument. The number of possibilities that need to be checked is quite large, but the arguments are always very similar. We will therefore just give the full proof for one of these.
\begin{enumerate}
\item We first examine the possibilities how the points of $Q^{\rm up}$ can be connected to points in $Q^{\rm low}$.  The number of possible configurations is very limited due to the fact that both $\beta^{\rm up}$ and $\beta^{\rm low}$ are not contained in another cup-cap pair and no arcs intersect. This yields the followings four possibilities (with the arrows depicting which points are connected):
\begin{equation*}
\begin{array}{cllll}
1.) & (a,2) \leftrightarrow (d,0) & (b,2) \leftrightarrow (c,0)
& (a,0) \leftrightarrow (d,-2) & (b,0) \leftrightarrow (c,-2) \\
2.) & (a,2) \leftrightarrow (c,0) & (b,2) \leftrightarrow (c,-2)
& (a,0) \leftrightarrow (d,0), & (b,0) \leftrightarrow (d,-2) \\
3.) & (a,2) \leftrightarrow (d,-2) & (b,2) \leftrightarrow (c,0)
& (a,0) \leftrightarrow (c,-2) & (b,0) \leftrightarrow (d,0) \\
4.) & (a,2) \leftrightarrow (c,-2) & (b,2) \leftrightarrow (d,-2)
& (a,0) \leftrightarrow (c,0), & (b,0) \leftrightarrow (d,0)
\end{array}
\end{equation*}
In the first two cases all points are contained in a single circle and the first surgery move will be a split (no matter which of the two is performed) while the second will be a merge. In the latter two cases we start with two distinct circles which get merged by the first surgery move and split again by the second surgery move.

\item In each of these cases we need to pick two pairs of connected points and assume that their connecting sequences of arcs $C_1$ and $C_2$ are inner. One can convince oneself that points from $Q^{\rm up}$, resp. $Q^{\rm low}$, of $C_1$ and $C_2$ cannot be diagonally opposite of each other, which reduces the number of possible choices of two inner connections to four for each case.

\item Finally we have to consider all possible decorations on $\beta^{\rm up}$ and $\beta^{\rm low}$. Since they could be dotted or undotted this would give in principle an additional number of four cases each, but one can check that if the two leftmost points on a cup-cap pair lie on an inner circle, then the cup-cap pair cannot be dotted.
\end{enumerate}

We consider now case (i):
$$\begin{array}{ll}
(a,2) \leftrightarrow (d,0), & (b,2) \leftrightarrow (c,0), \\
(a,0) \leftrightarrow (d,-2), & (b,0) \leftrightarrow (c,-2)
\end{array}$$
and assume that $(b,2) \leftrightarrow (c,0)$ and $(b,0) \leftrightarrow (c,-2)$ are inner. Moreover, we assume that $\beta^{\rm up}$ is dotted. (Remember as mentioned in part (3) above that $\beta^{\rm low}$ cannot be dotted since $(c,0)$ and $(c,-2)$ lie on inner connections.) A possible picture for such a configuration may look as follows (although we will not use its exact shape):
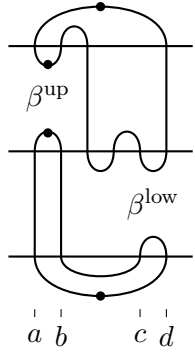
\begin{eqnarray*}
\begin{tikzpicture}[thick, scale=0.7]

\node at (.75,1) {$\beta^{\rm up}$};
\node at (2.75,-1) {$\beta^{\rm low}$};

\draw (0,2) -- (3.5,2);
\draw (0,0) -- (3.5,0);
\draw (0,-2) -- (3.5,-2);

\draw (1,2) .. controls +(0,.5) and +(0,.5) .. +(.5,0);
\draw (.5,2) .. controls +(0,1) and +(0,1) .. +(2.5,0);
\draw (.5,0) .. controls +(0,.5) and +(0,.5) .. +(.5,0);
\draw (2,0) .. controls +(0,.5) and +(0,.5) .. +(.5,0);
\draw (2.5,-2) .. controls +(0,.5) and +(0,.5) .. +(.5,0);

\draw (2.5,0) .. controls +(0,-.5) and +(0,-.5) .. +(.5,0);
\draw (1.5,0) .. controls +(0,-.5) and +(0,-.5) .. +(.5,0);
\draw (.5,2) .. controls +(0,-.5) and +(0,-.5) .. +(.5,0);

\draw (1,-2) .. controls +(0,-.5) and +(0,-.5) .. +(1.5,0);
\draw (.5,-2) .. controls +(0,-1) and +(0,-1) .. +(2.5,0);

\draw (1.5,2) -- (1.5,0);
\draw (3,2) -- (3,0);

\draw (.5,0) -- (.5,-2);
\draw (1,0) -- (1,-2);

\fill (.75,1.65) circle(2.5pt);
\fill (.75,.35) circle(2.5pt);
\fill (1.75,2.75) circle(2.5pt);
\fill (1.75,-2.75) circle(2.5pt);

\draw[dashed,thin] (.5,-3) -- (.5,-3.3);
\draw[dashed,thin] (1,-3) -- (1,-3.3);
\draw[dashed,thin] (2.5,-3) -- (2.5,-3.3);
\draw[dashed,thin] (3,-3) -- (3,-3.3);
\node at (.5,-3.5) {$a$};
\node at (1,-3.5) {$b$};
\node at (2.5,-3.5) {$c$};
\node at (3,-3.5) {$d$};

\end{tikzpicture}
\end{eqnarray*}

Applying surgery first at $\beta^{\rm up}$ (a merge) and then at $\beta^{\rm low}$ (a split) yields the map induced from the multiplication with $(-1)^c(X_{d,0}-X_{c,0})$. Performing the surgeries in the opposite order yields the map induced by multiplication with $(-1)^a(X_{b,0}+X_{a,0})$, since we split at a dotted cup. In the quotient $\cT(G_{j+1}^{i+1})$ we have have following equalities
\begin{eqnarray*}
&(-1)^c(X_{d,0}-X_{c,0}) \overset{(1)}{=} (-1)^c(X_{d,0}+X_{b,2})\overset{(2)}{=} (-1)^c(X_{a,2}+X_{b,2})&\\
&\overset{(3)}{=} (-1)^a(X_{b,2}+X_{a,2})\overset{(4)}{=} (-1)^a(X_{b,0}+X_{a,0}).&\\
\end{eqnarray*}
To see this note that for (1) the positions of $(c,0)$ and $(b,2)$ imply that they are connected by an odd number of cups/caps. Since the connection is inner it contains no dots and hence in $\cT(G_{j+1}^{i+1})$ we have $X_{c,0}=-X_{b,2}$. For (2) consider the circle containing the connection $(a,2) \leftrightarrow (d,0)$. Since $\beta^{\rm up}$ is dotted whereas $\beta^{\rm low}$ is not dotted we know that the connection $(a,2) \leftrightarrow (d,0)$ must contain an odd number of dotted cups/caps for the circle to be oriented and again due to the positions of the points it must contain an odd number of cups/caps. Thus in $\cT(G_{j+1}^{i+1})$ we have $X_{d,0}=X_{a,2}$. Since by the above arguments $(c,0)$ and $(a,2)$ were connected by an even number of cups/caps we know that $a$ and $c$ must have the same parity, thus (3) follows. Finally (4) is just one of the relations that hold in $\cT(G_{j+1}^{i+1})$ since we performed the surgery at positions $a$ and $b$.

Very similar arguments apply to all the other cases and are therefore omitted. The square commutes.
\end{proof}

\section{Cellularity and decomposition numbers diagrammatically}
\label{sec:cell}
The following theorem is an analogue of \cite[Theorem 3.1]{BS1} and the proof follows in principle the proof there. However, because of non-locality and signs the arguments are slightly more involved.
\subsection{Cellularity}
We still fix a block $\Lambda$.
\begin{theorem}\label{cellular}
Let $( a \la b )$ and
$( c \mu d )$ be basis vectors of $\D$.
Then,
$$
( a \la b )
( c \mu d ) =
\left\{
\begin{array}{ll}
0&\text{if $b \neq c^*$,}\\
s_{a \la b}(\mu) ( a \mu d ) + (\dagger)&\text{if $b = c^*$
and $a \mu$ is oriented,}\\
(\dagger)&\text{otherwise,}
\end{array}
\right.
$$
where
\begin{enumerate}[(i)]
\item
$(\dagger)$ denotes a linear combination of basis vectors from $\B$ of the form $( a \nu d)$
for $\nu > \mu$;
\item
the scalar $s_{a \la b}(\mu) \in \{0,1,-1\}$ depends only on
$a \la b$ and $\mu$ (but not on $d$).
\end{enumerate}
\end{theorem}

We first state an eays fact and deduce a few consequences of the theorem:

\begin{lemma}
\label{antiaut}
The assignment $a\la b\mapsto (a \la b)^*:= b^* \la a^*$, on elements in $\B$
defines a graded algebra anti-isomorphism of $\D$.
\end{lemma}

\begin{proof}
Note that $b^* \la a^*$ is the diagram $a\la b$ reflected in the real line, but with the weight $\la$ fixed. Then the surgery procedures which compute the product $(d^* \la c^*)(b^* \la a^*)$ produce precisely the mirror picture of the once produced when computing $(a^* \la b^*)(c^* \la d^*)$ except that weights are kept, ie. not reflected. Then the claim is obvious.
\end{proof}

\begin{corollary}\label{ideal}
The product
$( a \la b )( c \mu d )$ of two basis vectors of $\D$
is a linear combination of vectors of the form
$( a \nu d)\in\B$ with $\la \leq \nu \geq \mu$.
\end{corollary}

\begin{proof}
By Theorem~\ref{cellular}(i),
$( a \la b ) ( c \mu d )$ is a linear combination
of $( a \nu d)$'s for various $\nu \geq \mu$
and
$( d^* \mu c^* ) ( b^* \la a^* )$
is a linear combination of $( d^* \nu a^*)$'s for various
$\nu \geq \la$.
Applying the anti-automorphism $*$ from Lemma~\ref{antiaut}
to the latter statement gives
that
$( a \la b ) ( c \mu d )$ is a linear combination
of $( a \nu d)$'s for various $\nu \geq \la$ too.
\end{proof}

\begin{corollary}\label{iscell}
The algebra $\D$ is a cellular algebra in the sense of Graham and
Lehrer \cite{GL} with cell datum $(\La, M, C, *)$ where
\begin{itemize}
\item[(i)]
$M(\la)$ denotes $\left\{\alpha \in \La\:|\:\alpha
\subset \la\right\}$
for each $\la \in \La$;
\item[(ii)]
$C$ is defined by setting
$C^\la_{\alpha,\beta}:=( \underline{\alpha} \la \overline{\beta} )$
for $\la \in \La$ and $\alpha,\beta \in M(\la)$;
\item[(iii)] $*$ is the anti-automorphism from (\ref{antiaut}).
\end{itemize}
\end{corollary}

Before we prove the corollary let us first recall the relevant definitions from
\cite{GL}. A {\em cellular algebra} means an associative unital algebra $H$
together with a {\em cell datum} $(\La, M, C, *)$ such that
\begin{itemize}
\item[(1)]
$\La$ is a partially ordered set and $M(\la)$
is a finite set
for each $\la \in \La$;
\item[(2)]
$C:\dot\bigcup_{\la \in \La} M(\la)
\times M(\la) \rightarrow H,
(\alpha,\beta) \mapsto C^\la_{\alpha,\beta}$ is an injective map
whose image is a basis for $H$;
\item[(3)]
the map $*:H \rightarrow H$
is an algebra anti-automorphism such that
$(C^\la_{\alpha,\beta})^* = C_{\beta,\alpha}^\la$
for all $\la \in \La$ and $\alpha, \beta \in M(\la)$;
\item[(4)]
if $\mu \in \La$ and $\gamma, \delta \in M(\la)$
then for any $x \in H$ we have that
$$
x C_{\gamma,\delta}^\mu \equiv \sum_{\gamma' \in M(\mu)} r_x(\gamma',\gamma) C_{\gamma',\delta}^\mu
\pmod{H(> \mu)}$$
where the scalar $r_x(\gamma',\gamma)$ is independent of $\delta$ and
$H(> \mu)$ denotes the subspace of $H$
generated by $\{C_{\gamma'',\delta''}^\nu\:|\:\nu > \mu,
\gamma'',\delta'' \in M(\nu)\}$.
\end{itemize}

\begin{proof}[Proof of Corollary~\ref{iscell}]
Condition (1) is clear as $\La$ itself is a finite set;
Condition (2) is a consequence of the definition of $\B$; and since $*$ is an anti-automorphism we get (3).
Finally to verify (4) it suffices to consider the case that $x = C^\la_{\alpha,\beta}$ for some $\la \in \La$ and
$\alpha, \beta \in M(\la)$. If $\beta = \gamma$ and $\alpha \subset \mu$, then
Theorem~\ref{cellular}(i)--(ii) shows that
$$
C_{\alpha,\beta}^\la C_{\gamma,\delta}^\mu \equiv
s_{\underline{\alpha} \la \overline{\beta}}(\mu) C_{\alpha, \delta}^\mu
\pmod{{\D}{(> \mu)}}
$$
where $s_{\underline{\alpha} \la \overline{\beta}}(\mu)$
is independent of $\delta$;
otherwise, we have that
$$
C_{\alpha, \beta}^\la C_{\gamma, \delta}^\mu \equiv 0
\pmod{\D(> \mu)}.
$$
Taking
$$
r_x(\gamma', \gamma)
:=
\left\{
\begin{array}{ll}
s_{\underline{\alpha}\la\overline{\beta}}(\mu)&
\text{if $\gamma' = \alpha, \beta = \gamma$ and $\alpha \subset \mu$,}\\
0&\text{otherwise,}
\end{array}\right.
$$
we deduce that (4) holds.
\end{proof}

\begin{remark}
{\rm
Corollary~\ref{iscell} together with the definition of the grading directly implies that our cellular basis is in fact a graded cellular basis in the sense of \cite{MH}.}
\end{remark}

The following algebra is the analogue of the Khovanov algebra from \cite{BS1} and appears in the context of Springer fibres and resolutions of singularities, see \cite{ES1}.

\begin{corollary}
\label{Khovalg}
Let $\Lambda$ be a block and $e:=\sum e_\la$, where $\la$ runs over all weights $\la\in\La$ such that the associated cup diagram has $\op{def}(\La)$ cups. Then the algebra $\mathbb{H}_\La=e\D e$ is again cellular. It is the endomorphism ring of the sum of all indecomposable projective-injective $\mathbb{D}_\Lambda$-modules.
\end{corollary}
\begin{proof}
The first part follows from the cellularity of $\D$, since cellularity behaves well under idempotent truncation, see e.g. \cite[Proposition 4.3]{KoenigXi}. By definition, the algebra is the endomorphism ring of the sum of projective modules which are of maximal Loewy-length. By Lemma~\ref{atypicality} it is enough to consider the principal blocks. Then it will follow from Corollary~\ref{projinj}.
\end{proof}

The proof of Theorem~\ref{cellular} uses several times the following statements

\begin{lemma}
\label{lem:anticlock}
Let $C$ be an anticlockwise circle appearing in an oriented circle diagram (resp. in an oriented double circle diagram after applying a finite number of surgery moves in our multiplication rule). Let $\nu$ be the sequence of labels attached to the intersection points of $C$ with the weight line (resp. with the top weight line).  If $C$ was anticlockwise then swapping the orientation changes $\nu$ into a sequence which is bigger in the Bruhat order (in the sense of Lemma~\ref{lem:Bruhat}).
\end{lemma}

\begin{definition}
\label{smallcirc}
{\rm
A circle appearing in an oriented circle diagram or during a surgery procedure is called {\it small} if it only involves one cup and one cap.
}
\end{definition}
\begin{proof}[Proof of Lemma~\ref{lem:anticlock}]
The statement is clear if the circle is small. Indeed, the weight moves from $\down\up$ to $\up\down$ in the undotted case and from $\up\up$ to $\down\down$ in the dotted case. Otherwise we first remove all undotted kinks on the cost of an extra small anticlockwise circle according to the rules \eqref{kinktocircle} (i)-(ii) below not changing the labels of the weight. Since the circle was anti-clockwise, the result contains either only small circles or at least one kink of the form \eqref{kinktocircle} (iii)-(iv). Again we remove these kinks using one of the rules below.
 
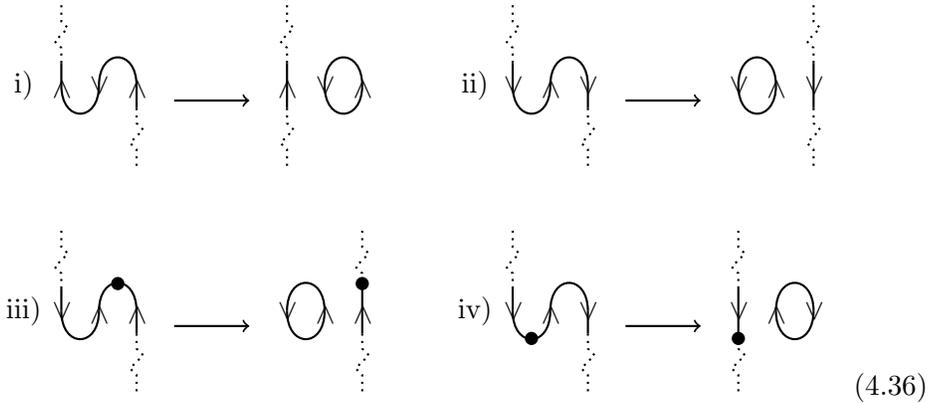
\begin{equation}
 \label{kinktocircle}
 \begin{tikzpicture}[thick, snake=zigzag, line before snake = 2mm, line after snake = 2mm]
\node at (-.5,0) {i)};
\draw (0,0) -- +(0,.3);
\draw[dotted, snake] (0,.3) -- +(0,.8);
\draw (0,0) .. controls +(0,-.5) and +(0,-.5) .. +(.5,0);
\draw (.5,0) .. controls +(0,.5) and +(0,.5) .. +(.5,0);
\draw (1,0) -- +(0,-.3);
\draw[dotted, snake] (1,-.3) -- +(0,-.8);
\node at (.01,-.05) {$\up$};
\node at (1.01,-.05) {$\up$};
\node at (.51,0) {$\down$};

\draw[->] (1.5,-.2) -- +(1,0);

\draw (3,0) -- +(0,.3);
\draw[dotted, snake] (3,.3) -- +(0,.8);
\draw (3.5,0) .. controls +(0,-.5) and +(0,-.5) .. +(.5,0);
\draw (3.5,0) .. controls +(0,.5) and +(0,.5) .. +(.5,0);
\draw (3,0) -- +(0,-.3);
\draw[dotted, snake] (3,-.3) -- +(0,-.8);
\node at (3.01,-.05) {$\up$};
\node at (4.01,-.05) {$\up$};
\node at (3.51,0) {$\down$};

\begin{scope}[xshift=6cm]
\node at (-.5,0) {ii)};
\draw (0,0) -- +(0,.3);
\draw[dotted, snake] (0,.3) -- +(0,.8);
\draw (0,0) .. controls +(0,-.5) and +(0,-.5) .. +(.5,0);
\draw (.5,0) .. controls +(0,.5) and +(0,.5) .. +(.5,0);
\draw (1,0) -- +(0,-.3);
\draw[dotted, snake] (1,-.3) -- +(0,-.8);
\node at (.01,0) {$\down$};
\node at (1.01,0) {$\down$};
\node at (.51,-.05) {$\up$};

\draw[->] (1.5,-.2) -- +(1,0);

\draw (4,0) -- +(0,.3);
\draw[dotted, snake] (4,.3) -- +(0,.8);
\draw (3,0) .. controls +(0,-.5) and +(0,-.5) .. +(.5,0);
\draw (3,0) .. controls +(0,.5) and +(0,.5) .. +(.5,0);
\draw (4,0) -- +(0,-.3);
\draw[dotted, snake] (4,-.3) -- +(0,-.8);
\node at (3.01,0) {$\down$};
\node at (4.01,0) {$\down$};
\node at (3.51,-.05) {$\up$};

\end{scope}

\begin{scope}[yshift=-3cm]
\node at (-.5,0) {iii)};
\draw (0,0) -- +(0,.3);
\draw[dotted, snake] (0,.3) -- +(0,.8);
\draw (0,0) .. controls +(0,-.5) and +(0,-.5) .. +(.5,0);
\draw (.5,0) .. controls +(0,.5) and +(0,.5) .. +(.5,0);
\fill (.75,.365) circle(2.5pt);
\draw (1,0) -- +(0,-.3);
\draw[dotted, snake] (1,-.3) -- +(0,-.8);
\node at (.01,0) {$\down$};
\node at (1.01,-.05) {$\up$};
\node at (.51,-.05) {$\up$};

\draw[->] (1.5,-.2) -- +(1,0);

\draw (4,0) -- +(0,.3);
\draw[dotted, snake] (4,.3) -- +(0,.8);
\fill (4,.365) circle(2.5pt);
\draw (3,0) .. controls +(0,-.5) and +(0,-.5) .. +(.5,0);
\draw (3,0) .. controls +(0,.5) and +(0,.5) .. +(.5,0);
\draw (4,0) -- +(0,-.3);
\draw[dotted, snake] (4,-.3) -- +(0,-.8);
\node at (3.01,0) {$\down$};
\node at (4.01,-.05) {$\up$};
\node at (3.51,-.05) {$\up$};

\end{scope}

\begin{scope}[xshift=6cm, yshift=-3cm]
\node at (-.5,0) {iv)};
\draw (0,0) -- +(0,.3);
\draw[dotted, snake] (0,.3) -- +(0,.8);
\fill (.25,-.365) circle(2.5pt);
\draw (0,0) .. controls +(0,-.5) and +(0,-.5) .. +(.5,0);
\draw (.5,0) .. controls +(0,.5) and +(0,.5) .. +(.5,0);
\draw (1,0) -- +(0,-.3);
\draw[dotted, snake] (1,-.3) -- +(0,-.8);
\node at (.01,0) {$\down$};
\node at (1.01,0) {$\down$};
\node at (.51,-.05) {$\up$};

\draw[->] (1.5,-.2) -- +(1,0);

\draw (3,0) -- +(0,.3);
\draw[dotted, snake] (3,.3) -- +(0,.8);
\fill (3,-.365) circle(2.5pt);
\draw (3.5,0) .. controls +(0,-.5) and +(0,-.5) .. +(.5,0);
\draw (3.5,0) .. controls +(0,.5) and +(0,.5) .. +(.5,0);
\draw (3,0) -- +(0,-.3);
\draw[dotted, snake] (3,-.3) -- +(0,-.8);
\node at (3.01,0) {$\down$};
\node at (4.01,0) {$\down$};
\node at (3.51,-.05) {$\up$};

\end{scope}
\end{tikzpicture}
 \end{equation}
 The result is a collection of small anticlockwise circles. Swapping the orientation of $C$ means swapping the labels in the weight $\nu$ or equivalently swapping the orientation of all the small circles we created by removing the kinks. But this obviously increases the weight.
\end{proof}

\begin{lemma}
\label{lem:correct}
Let $C$ and $D$ be oriented circles inside the diagram appearing after a finite number (possibly zero) of the surgery moves applied to an oriented circle diagram. Assume $D$ is nested inside $C$. Let $\nu$ be the sequence of labels attached to the intersection points of $C$ and $D$ with the top weight line. If $C$ was anticlockwise then swapping the orientation of both circles changes $\nu$ into a sequence which is bigger in the Bruhat order (in the sense of Lemma~\ref{lem:Bruhat}).
\end{lemma}

\begin{proof}
If $C$ and $D$ are both anticlockwise, the statement follows directly from Lemma~\ref{lem:anticlock}. So assume $D$ is clockwise. In case $C$ and $D$ are small we have $\nu=\down\up\down\up$ if the circles are undotted, and $\nu=\up\up\down\up$ in case $C$ has two dots. (Note that since $D$ is nested inside $C$ it can't have any dots at all.) In either case the weight $\nu$ becomes bigger if we swap the orientations of $C$ and $D$ (hence the labels of $\nu$). If the circles are not small, we remove all kinks using the moves \eqref{kinktocircle}. (Note that this is possible, since $C$ is anticlockwise and $D$ has no dots). The result is a collection of small circles of which exactly one is clockwise. This clockwise circle is still nested in at least one anticlockwise circle. Now swapping the labels at these two circles behaves like the case considered at the beginning of the proof, whereas swapping the labels at the remaining anticlockwise circles obviously increases the weight.
\end{proof}

\begin{lemma}
\label{lem:dottedsurgery}
Let $C$ and $D$ be oriented circles appearing in the diagram after a finite number (possibly zero) of the surgery moves applied to an oriented circle diagram. Let $\nu$ be the sequence of labels attached to the intersection points of $C$ and $D$ with the top weight line. Assume that the next surgery move joins the two circles and the relevant cup-cap pair is dotted. Then $\nu$ either doesn't change or gets strictly bigger when applying the surgery move.
\end{lemma}

\begin{proof}
Let $C$ (resp. $D$) be the circle involving the cup (resp. cap) relevant for the surgery move and let $E$ be the circle obtained by applying the surgery move to $C$ an $D$. Let $p$ be a rightmost point amongst the intersection points of $E$ with the two number lines containing the labels for the orientation. (Hence the label at $p$ determines the type of orientation of $E$ by Lemma~\ref{lem:stupidlemma}). Write $p\in C$ or $p\in D$ in case $p$ is lying on $C$ or $D$ respectively.
We consider the four possible orientations if the relevant cup/cap pair is dotted.
\begin{equation}
\label{4cases}
(I)\quad
\usetikzlibrary{arrows}
\begin{tikzpicture}[thick,>=angle 60, scale=0.8]
\draw [>-<] (0,0) .. controls +(0,-1) and +(0,-1) .. +(1,0);
\fill (0.5,-0.77) circle(2.5pt);
\draw [<->] (0,-2) .. controls +(0,1) and +(0,1) .. +(1,0);
\fill (0.5,-1.25) circle(2.5pt);
\end{tikzpicture}
\quad\quad
(II)
\quad
\usetikzlibrary{arrows}
\begin{tikzpicture}[thick,>=angle 60,scale=0.8]
\draw [>-<] (0,0) .. controls +(0,-1) and +(0,-1) .. +(1,0);
\fill (0.5,-0.77) circle(2.5pt);
\draw [>-<] (0,-2) .. controls +(0,1) and +(0,1) .. +(1,0);
\fill (0.5,-1.22) circle(2.5pt);
\end{tikzpicture}
\quad\quad
(III)
\quad
\usetikzlibrary{arrows}
\begin{tikzpicture}[thick,>=angle 60,scale=0.8]
\draw [<->] (0,0) .. controls +(0,-1) and +(0,-1) .. +(1,0);
\fill (0.5,-0.76) circle(2.5pt);
\draw [>-<] (0,-2) .. controls +(0,1) and +(0,1) .. +(1,0);
\fill (0.5,-1.22) circle(2.5pt);
\end{tikzpicture}
\quad\quad
(IV)
\quad
\usetikzlibrary{arrows}
\begin{tikzpicture}[thick,>=angle 60,scale=0.8]
\draw [<->] (0,0) .. controls +(0,-1) and +(0,-1) .. +(1,0);
\fill (0.5,-0.76) circle(2.5pt);
\draw [<->] (0,-2) .. controls +(0,1) and +(0,1) .. +(1,0);
\fill (0.5,-1.24) circle(2.5pt);
\end{tikzpicture}
\end{equation}
Consider Case (I). If both circles are anticlockwise, then we just replace the cup/cap pair by two straight lines and $\nu$ does not change.  Assume $C$ is anticlockwise and $D$ is clockwise. If $p\in C$ then we replace the cup/cap pair by two straight lines to obtain an anticlockwise circle which then has to be reoriented clockwise. Lemma~~\ref{lem:anticlock} implies that $\nu$ strictly increases. If $p\in D$ then we just replace the cup/cap pair by two straight lines and $\nu$ does not change.
Assume $C$ is clockwise and $D$ is anti-clockwise. If $p\in C$ then we just replace the cup/cap pair by two straight lines without changing $\nu$. If $p\in D$ then we replace the cup/cap pair by two straight lines to obtain an anticlockwise circle which then has to be reoriented clockwise. Again, Lemma~\ref{lem:anticlock} implies that $\nu$ strictly increases. In case both, $C$ and $D$, are clockwise, the result is zero and there is nothing to check.

Consider Case (II).  If both circles are anticlockwise and $p\in C$ then the bottom circle gets swapped and hence its intersection with $\nu$, and therefore $\nu$ itself, gets bigger by Lemma~\ref{lem:anticlock}; if instead $p\in D$ then the top circle gets swapped and again $\nu$ increases by Lemma~\ref{lem:anticlock}. If $C$ is anticlockwise and $D$ is clockwise then swapping first the orientation of $C$ and then replacing the cup-cap pair by two straight lines turns $E$ into a clockwise circle as required, independent on where $p$ lies. Again, $\nu$ increases by Lemma~\ref{lem:anticlock}.
If $C$ is clockwise and $D$ is anticlockwise then we can just swap the orientation of $D$ (which increases the part of $\nu$ intersecting $D$) and then join the two circles, hence $\nu$ gets bigger. If $C$ and $D$ are clockwise there is nothing to check.

Consider Case (III). If both circles are anticlockwise then we can just replace the cup/cap pair by two straight lines and $\nu$ doesn't change. If $C$ is anticlockwise, $D$ clockwise and $p\in C$ then we can join the circles and swap the orientation. Thus $\nu$ increases by Lemma~\ref{lem:anticlock}. On the other hand $p\in D$ is impossible, since this would imply in $D$ the existence of some dotted cup or cap to the right of the dotted cup involved in the surgery and attached to the top number line. By definition of the surgery procedure there is no such cup. But if there is such a cap, then it is impossible to close it to a circle $D$. If $C$ is clockwise and $D$ is anticlockwise then $\nu$ does not change if $p\in C$. If instead $p\in D$ then we first replace the cup/cap pair by two straight lines and then reorient so that $\nu$ increases by Lemma~\ref{lem:anticlock}. If $C$ and $D$ are clockwise there is nothing to check.

Consider Case (IV). If both circles are anticlockwise then we just swap the orientation of $D$ (resp. $C$) if $p\in D$ (resp. $p\in C$) before joining. Hence $\nu$ increases by Lemma~\ref{lem:anticlock}. If $C$ is anticlockwise, $D$ clockwise and $p\in C$ or $p\in D$, then we can swap the orientation of $C$ and join the circles. Hence $\nu$ increases by Lemma~\ref{lem:anticlock}.
If $C$ is clockwise and $D$ is anticlockwise then we can just swap the orientation of $D$ and join the circles. If $C$ and $D$ are clockwise there is nothing to check.
\end{proof}

\begin{proof}[Proof of Theorem~\ref{cellular}]
The proof is analogous to \cite{BS1}, but more involved. We already know that $( a \la b ) ( c \mu d ) = 0$ if $b \neq c^*$, so assume
$b=c^*$ from now on. Consider a single iteration of the surgery procedure in the
algorithm computing $( a \la b ) ( c \mu d )$. Let $\tau$ be the {top
weight} (the weight on the top number line) of this diagram
at the start of the
surgery procedure. It is enough to show the following:
\begin{enumerate}[(Cell1)]
\item
the top weight of each diagram obtained at the end of
the surgery procedure is greater than or equal to $\tau$ in the Bruhat order;
\item
the total number of
diagrams produced with top weight equal to $\tau$
is either zero or one, independent of the cap diagram $d$.
\end{enumerate}
Indeed, by applying (Cell1) repeatedly, starting with $\tau = \mu$ at the first step,
it follows that
$(a \la b) (c \mu d)$ is a linear combination of
$(a \nu d)$'s for $\nu \geq \mu$.
Assuming $a \mu$ is oriented,
(Cell2) applied repeatedly
implies that the coefficient
$s_{a \la b}(\mu)$ of the basis vector
$(a \mu d)$ in the product is zero or $\mp1$
independent of the cap diagram $d$.
This proves (i) and (ii).

To verify (Cell1) and (Cell2), we analyse three different situations
depending on the orientations of the cup-cap pair defining the surgery
(both of which clearly do not depend on $d$). We go through the same cases as in \cite{BS1}, but have to incorporate the decorations.
Let us first assume that the diagrams involved do not contain any lines or rays and the relevant cup-cap pair has no dots.

\noindent
{\em Case one: the cap to be cut is clockwise.}
If the circle containing the cap to be cut is anti-clockwise and undotted, we are in one of the three basic situations from \cite{BS1}:
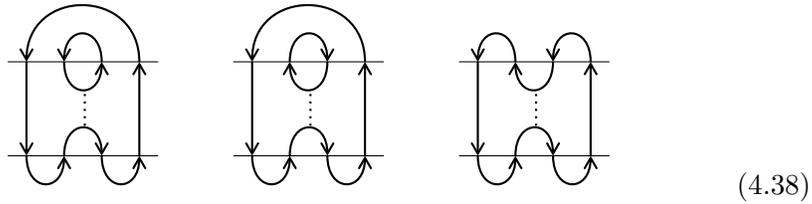
\begin{eqnarray}
\label{H}
\begin{tikzpicture}[thick,>=angle 60,scale=.5]
\draw[thin] (1.5,2) -- +(4,0);
\draw[thin] (1.5,-.5) -- +(4,0);
\draw [<-] (2,2) .. controls +(0,2) and +(0,2) .. +(3,0);
\draw [<-] (3,2) .. controls +(0,1) and +(0,1) .. +(1,0);

\draw [<-] (2,-.5) -- +(0,2.5);
\draw [->] (3,2) .. controls +(0,-1) and +(0,-1) .. +(1,0);
\draw[dotted] (3.55,.3) -- +(0,.95);
\draw [->] (3,-.5) .. controls +(0,1) and +(0,1) .. +(1,0);
\draw [->] (5,-.5) -- +(0,2.5);

\draw [->] (2,-.5) .. controls +(0,-1) and +(0,-1) .. +(1,0);
\draw [->] (4,-.5) .. controls +(0,-1) and +(0,-1) .. +(1,0);

\begin{scope}[xshift=6cm]
\draw[thin] (1.5,2) -- +(4,0);
\draw[thin] (1.5,-.5) -- +(4,0);
\draw [<-] (2,2) .. controls +(0,2) and +(0,2) .. +(3,0);
\draw [->] (3,2) .. controls +(0,1) and +(0,1) .. +(1,0);

\draw [<-] (2,-.5) -- +(0,2.5);
\draw [<-] (3,2) .. controls +(0,-1) and +(0,-1) .. +(1,0);
\draw[dotted] (3.55,.3) -- +(0,.95);
\draw [->] (3,-.5) .. controls +(0,1) and +(0,1) .. +(1,0);
\draw [->] (5,-.5) -- +(0,2.5);

\draw [->] (2,-.5) .. controls +(0,-1) and +(0,-1) .. +(1,0);
\draw [->] (4,-.5) .. controls +(0,-1) and +(0,-1) .. +(1,0);
\end{scope}

\begin{scope}[xshift=12cm]
\draw[thin] (1.5,2) -- +(4,0);
\draw[thin] (1.5,-.5) -- +(4,0);

\draw [<-] (2,2) .. controls +(0,1) and +(0,1) .. +(1,0);
\draw [<-] (4,2) .. controls +(0,1) and +(0,1) .. +(1,0);

\draw [<-] (2,-.5) -- +(0,2.5);
\draw [<-] (3,2) .. controls +(0,-1) and +(0,-1) .. +(1,0);
\draw[dotted] (3.55,.3) -- +(0,.95);
\draw [->] (3,-.5) .. controls +(0,1) and +(0,1) .. +(1,0);
\draw [->] (5,-.5) -- +(0,2.5);

\draw [->] (2,-.5) .. controls +(0,-1) and +(0,-1) .. +(1,0);
\draw [->] (4,-.5) .. controls +(0,-1) and +(0,-1) .. +(1,0);
\end{scope}

\end{tikzpicture}
\end{eqnarray}

with maybe some dots on cups and caps not involved in our pair. In the first two basic diagrams the following configurations are possible
 \begin{enumerate}[(D1)]
 \item none, or
 \item one on the top cap and one on the left bottom cup, or
 \item one on the top cap and one on the right bottom cup, or
 \item one on either of the two bottom cups.
 \end{enumerate}
 In the first diagram, the surgery procedure joins two anti-clockwise circles in cases (D1)-(D2) and one anti-clockwise and one clockwise in cases (D3) and (D4). As a result all the labels of the inner circle get swapped in the surgery procedure. Since the circle was originally anticlockwise we are done with Lemma~\ref{lem:anticlock}: the top weight gets strictly larger in the Bruhat order. In the second diagram we get zero in the cases (D3) and (D4). In the cases (D1) and (D2) the labels of both circles are swapped and we are done by Lemma~\ref{lem:correct}, since the outer circle was anti-clockwise and the inner was clockwise. For the third diagram we get zero except when the dot configurations are of the form (D1)-(D3). In cases (D1) and (D2) we split an anticlockwise circle into two. The reorientation of the surgery procedure swaps one anticlockwise circle into a clockwise and we are done again by Lemma~\ref{lem:anticlock}. In case (D3) we split a clockwise circle into two. The reorientation swaps again an anticlockwise circle into a clockwise and we are done by Lemma~\ref{lem:anticlock}.
Now the above diagrams (and all subsequent pictures) should be interpreted only up to homeomorphism; in particular the circles represented in the pictures may well
cross both number lines many more times than indicated and might have many dots. For the above diagrams on can argue as in the basic cases, by considering the associated segments of the circle instead of just cups and caps and specifying the parity of the number of dots. Since the label on the drawn rightmost vertical line determines the orientation of the circle (because of Lemma~\ref{lem:stupidlemma} and the fact that there are only undotted cups to the right), the arguments are verbatim the same.

Hence in all situations considered so far the top weight of each diagram
obtained at the end of the surgery procedure is strictly greater than $\tau$,
from which (Cell1)--(Cell3) follow immediately.

If instead, the circle containing the cap to be cut is clockwise, then there are four basic situations:
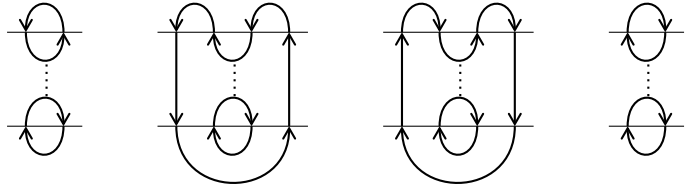
\begin{eqnarray*}
\begin{tikzpicture}[thick,>=angle 60,scale=.5]
\draw[thin] (2.5,2) -- +(2,0);
\draw[thin] (2.5,-.5) -- +(2,0);
\draw [<-] (3,2) .. controls +(0,1) and +(0,1) .. +(1,0);

\draw [->] (3,2) .. controls +(0,-1) and +(0,-1) .. +(1,0);
\draw[dotted] (3.55,.3) -- +(0,.95);
\draw [->] (3,-.5) .. controls +(0,1) and +(0,1) .. +(1,0);

\draw [<-] (3,-.5) .. controls +(0,-1) and +(0,-1) .. +(1,0);

\begin{scope}[xshift=5cm]
\draw[thin] (1.5,2) -- +(4,0);
\draw[thin] (1.5,-.5) -- +(4,0);
\draw [<-] (2,2) .. controls +(0,1) and +(0,1) .. +(1,0);
\draw [<-] (4,2) .. controls +(0,1) and +(0,1) .. +(1,0);

\draw [<-] (2,-.5) -- +(0,2.5);
\draw [<-] (3,2) .. controls +(0,-1) and +(0,-1) .. +(1,0);
\draw[dotted] (3.55,.3) -- +(0,.95);
\draw [->] (3,-.5) .. controls +(0,1) and +(0,1) .. +(1,0);
\draw [->] (5,-.5) -- +(0,2.5);

\draw [->] (2,-.5) .. controls +(0,-2) and +(0,-2) .. +(3,0);
\draw [<-] (3,-.5) .. controls +(0,-1) and +(0,-1) .. +(1,0);
\end{scope}

\begin{scope}[xshift=11cm]
\draw[thin] (1.5,2) -- +(4,0);
\draw[thin] (1.5,-.5) -- +(4,0);
\draw [->] (2,2) .. controls +(0,1) and +(0,1) .. +(1,0);
\draw [->] (4,2) .. controls +(0,1) and +(0,1) .. +(1,0);

\draw [->] (2,-.5) -- +(0,2.5);
\draw [->] (3,2) .. controls +(0,-1) and +(0,-1) .. +(1,0);
\draw[dotted] (3.55,.3) -- +(0,.95);
\draw [->] (3,-.5) .. controls +(0,1) and +(0,1) .. +(1,0);
\draw [<-] (5,-.5) -- +(0,2.5);

\draw [<-] (2,-.5) .. controls +(0,-2) and +(0,-2) .. +(3,0);
\draw [<-] (3,-.5) .. controls +(0,-1) and +(0,-1) .. +(1,0);
\end{scope}

\begin{scope}[xshift=16cm]
\draw[thin] (2.5,2) -- +(2,0);
\draw[thin] (2.5,-.5) -- +(2,0);
\draw [->] (3,2) .. controls +(0,1) and +(0,1) .. +(1,0);

\draw [<-] (3,2) .. controls +(0,-1) and +(0,-1) .. +(1,0);
\draw[dotted] (3.55,.3) -- +(0,.95);
\draw [->] (3,-.5) .. controls +(0,1) and +(0,1) .. +(1,0);

\draw [<-] (3,-.5) .. controls +(0,-1) and +(0,-1) .. +(1,0);
\end{scope}
\end{tikzpicture}
\end{eqnarray*}

In the first two of these, the orientation of
every vertex lying on the anti-clockwise circle containing
the cup to be cut
gets switched, so the top weight gets strictly bigger.
The last two involve the rule $x \otimes x \mapsto 0$,
so there is nothing to check. This works also for all dot configurations which do not  change the type of the circles.

In the first and last case we could arrange dots changing the orientation of the circle (in case they are not small). The surgery procedure joins two circles. If they are both anticlockwise, the top weight does not change, appears with multiplicity one and no other weight appears. If the top circle is anticlockwise and the bottom circle clockwise we apply again Lemma~\ref{lem:anticlock} and deduce that the weight gets strictly bigger. In case the top circle is clockwise and the bottom circle anticlockwise, then the joint circle is clockwise. If its rightmost point is on the top circle, the top weight obviously does not change. Let it be on the bottom circle. Then there is necessarily a dotted cap to the right of it. If the cup involved in the surgery is anticlockwise then there has to be a dotted cup to the right of it. This is a contradiction, since we would have a dotted cup/cap pair which we should have used earlier in the surgery procedure. If the cup involved in the surgery is clockwise then we naively join the circles and obtain an anticlockwise circle. We have to reorient it and hence the weight strictly increases by Lemma~\ref{lem:anticlock}.

Finally let the cup-cap-pair involved in the surgery be dotted. Thanks to Remark~\ref{decorated} we only have to deal with the following basic configurations
\begin{equation}
\label{dottedCase1}
\begin{tikzpicture}[thick]
\node at (-.5,-.5) {};
\draw (0,0) .. controls +(0,1) and +(0,1) .. +(1.5,0);
\draw (0,0) .. controls +(0,-.5) and +(0,-.5) .. +(.5,0);
\fill (.25,-.365) circle(2.5pt);
\draw (.5,0) .. controls +(0,.5) and +(0,.5) .. +(.5,0);
\draw (1,0) -- +(0,-1);
\draw (1.5,0) -- +(0,-1);
\draw (0,-1) .. controls +(0,.5) and +(0,.5) .. +(.5,0);
\fill (.25,-.635) circle(2.5pt);
\draw (.5,-1) .. controls +(0,-.5) and +(0,-.5) .. +(.5,0);
\draw (0,-1) .. controls +(0,-1) and +(0,-1) .. +(1.5,0);
\node at (.01,0) {$\down$};
\node at (.51,0) {$\down$};
\node at (1.01,-.05) {$\up$};
\node at (1.51,-.05) {$\up$};
\node at (.01,-1) {$\down$};
\node at (.51,-1) {$\down$};
\node at (1.01,-1.05) {$\up$};
\node at (1.51,-1.05) {$\up$};

\begin{scope}[xshift=5cm]
\node at (-.5,-.5){};
\draw (0,0) .. controls +(0,1) and +(0,1) .. +(1.5,0);
\fill (.75,.74) circle(2.5pt);
\draw (0,0) .. controls +(0,-.5) and +(0,-.5) .. +(.5,0);
\fill (.25,-.365) circle(2.5pt);
\draw (.5,0) .. controls +(0,.5) and +(0,.5) .. +(.5,0);
\draw (1,0) -- +(0,-1);
\draw (1.5,0) -- +(0,-1);
\draw (0,-1) .. controls +(0,.5) and +(0,.5) .. +(.5,0);
\fill (.25,-.635) circle(2.5pt);
\draw (.5,-1) .. controls +(0,-.5) and +(0,-.5) .. +(.5,0);
\draw (0,-1) .. controls +(0,-1) and +(0,-1) .. +(1.5,0);
\fill (.75,-1.74) circle(2.5pt);
\node at (.01,0) {$\down$};
\node at (.51,0) {$\down$};
\node at (1.01,-.05) {$\up$};
\node at (1.51,0) {$\down$};
\node at (.01,-1) {$\down$};
\node at (.51,-1) {$\down$};
\node at (1.01,-1.05) {$\up$};
\node at (1.51,-1) {$\down$};
\end{scope}
\end{tikzpicture}
\end{equation}
In the first case we split an anticlockwise circle into two circles and we are done with Lemma~\ref{lem:anticlock}. The top weight strictly increases. In the second case surgery produces a sum of two diagrams, one with the same top weight, the other with a strictly bigger top weight.

\noindent
{\em Case two: both the cap and the cup to be cut are anti-clockwise.}
first assume the cup/cap pair is undotted. There are five basic configurations to consider:
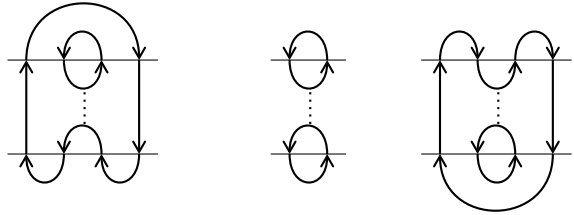
\begin{eqnarray}
\label{Case2partone}
\usetikzlibrary{arrows}
\begin{tikzpicture}[thick,>=angle 60,scale=.5]
\draw[thin] (1.5,2) -- +(4,0);
\draw[thin] (1.5,-.5) -- +(4,0);
\draw [->] (2,2) .. controls +(0,2) and +(0,2) .. +(3,0);
\draw [<-] (3,2) .. controls +(0,1) and +(0,1) .. +(1,0);

\draw [->] (2,-.5) -- +(0,2.5);
\draw [->] (3,2) .. controls +(0,-1) and +(0,-1) .. +(1,0);
\draw[dotted] (3.55,.3) -- +(0,.95);
\draw [<-] (3,-.5) .. controls +(0,1) and +(0,1) .. +(1,0);
\draw [<-] (5,-.5) -- +(0,2.5);

\draw [<-] (2,-.5) .. controls +(0,-1) and +(0,-1) .. +(1,0);
\draw [<-] (4,-.5) .. controls +(0,-1) and +(0,-1) .. +(1,0);

\begin{scope}[xshift=6cm]
\draw[thin] (2.5,2) -- +(2,0);
\draw[thin] (2.5,-.5) -- +(2,0);
\draw [<-] (3,2) .. controls +(0,1) and +(0,1) .. +(1,0);

\draw [->] (3,2) .. controls +(0,-1) and +(0,-1) .. +(1,0);
\draw[dotted] (3.55,.3) -- +(0,.95);
\draw [<-] (3,-.5) .. controls +(0,1) and +(0,1) .. +(1,0);

\draw [->] (3,-.5) .. controls +(0,-1) and +(0,-1) .. +(1,0);
\end{scope}

\begin{scope}[xshift=11cm]
\draw[thin] (1.5,2) -- +(4,0);
\draw[thin] (1.5,-.5) -- +(4,0);
\draw [->] (2,2) .. controls +(0,1) and +(0,1) .. +(1,0);
\draw [->] (4,2) .. controls +(0,1) and +(0,1) .. +(1,0);

\draw [->] (2,-.5) -- +(0,2.5);
\draw [->] (3,2) .. controls +(0,-1) and +(0,-1) .. +(1,0);
\draw[dotted] (3.55,.3) -- +(0,.95);
\draw [<-] (3,-.5) .. controls +(0,1) and +(0,1) .. +(1,0);
\draw [<-] (5,-.5) -- +(0,2.5);

\draw [<-] (2,-.5) .. controls +(0,-2) and +(0,-2) .. +(3,0);
\draw [->] (3,-.5) .. controls +(0,-1) and +(0,-1) .. +(1,0);
\end{scope}
\end{tikzpicture}
\end{eqnarray}
\begin{eqnarray}
\label{kidneys}
\usetikzlibrary{arrows}
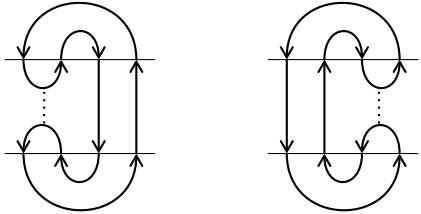
\begin{tikzpicture}[thick,>=angle 60,scale=.5]
\draw[thin] (1.5,2) -- +(4,0);
\draw[thin] (1.5,-.5) -- +(4,0);
\draw [<-] (2,2) .. controls +(0,2) and +(0,2) .. +(3,0);
\draw [->] (3,2) .. controls +(0,1) and +(0,1) .. +(1,0);

\draw [<-] (4,-.5) -- +(0,2.5);
\draw [->] (2,2) .. controls +(0,-
1) and +(0,-1) .. +(1,0);
\draw[dotted] (2.55,.3) -- +(0,.95);
\draw [<-] (2,-.5) .. controls +(0,1) and +(0,1) .. +(1,0);
\draw [->] (5,-.5) -- +(0,2.5);

\draw [->] (2,-.5) .. controls +(0,-2) and +(0,-2) .. +(3,0);
\draw [<-] (3,-.5) .. controls +(0,-1) and +(0,-1) .. +(1,0);

\begin{scope}[xshift=14cm,xscale=-1]
\draw[thin] (1.5,2) -- +(4,0);
\draw[thin] (1.5,-.5) -- +(4,0);
\draw [->] (2,2) .. controls +(0,2) and +(0,2) .. +(3,0);
\draw [<-] (3,2) .. controls +(0,1) and +(0,1) .. +(1,0);

\draw [->] (4,-.5) -- +(0,2.5);
\draw [<-] (2,2) .. controls +(0,-
1) and +(0,-1) .. +(1,0);
\draw[dotted] (2.55,.3) -- +(0,.95);
\draw [->] (2,-.5) .. controls +(0,1) and +(0,1) .. +(1,0);
\draw [<-] (5,-.5) -- +(0,2.5);

\draw [<-] (2,-.5) .. controls +(0,-2) and +(0,-2) .. +(3,0);
\draw [->] (3,-.5) .. controls +(0,-1) and +(0,-1) .. +(1,0);
\end{scope}

\end{tikzpicture}
\end{eqnarray}
In the first two diagrams we always join two circles (both anticlockwise in case we have no dots or in case the bottom cups are dotted and one clockwise and one anticlockwise otherwise) and one easily checks that the top weight gets not changed. In the third diagram we again join two circles (an anticlockwise with a clockwise in case we have no dots or in case the top caps are dotted, and two anticlockwise otherwise) and again the top weight is unchanged. The fourth diagram produces in the undotted case the sum of the two possibilities of two nested circles with opposite orientation. Hence one summand keeps the weight, whereas the other increases it be Lemma~\ref{lem:anticlock}. If the top cap and lower cup are dotted then we split a clockwise circle into two nested clockwise (in which case the top weight stays the same) or split an anticlockwise circle (in which case we get two summands, one keeps the top weight and the other swaps the outer anticlockwise circle into a clockwise circle, hence the weight gets bigger by Lemma~\ref{lem:correct}.

Now consider the case where the cup-cap pair is dotted. Thanks to Remark~\ref{decorated} we only have to consider the basic configurations from \eqref{dottedCase1} but with the orientation reversed. In the first diagram we create two nested clockwise circles and the weight is kept. In the second diagram we create the two possibilities of two nested circles oriented oppositely. One summand keeps the weight, whereas the other increases the weight by Lemma~\ref{lem:correct}.

\noindent
{\it Case three: the cap to be cut is anti-clockwise but the cup to be cut is clockwise.}
By our decoration rules the cup-cap pair is automatically undotted, and then the following basic situations are possible:
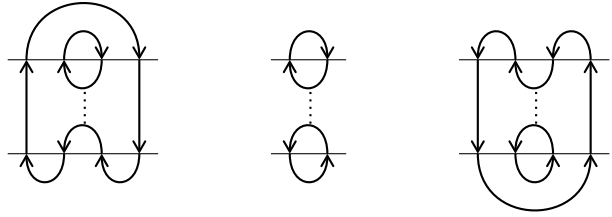
\begin{eqnarray*}
\usetikzlibrary{arrows}
\begin{tikzpicture}[thick,>=angle 60,scale=.5]
\draw[thin] (1.5,2) -- +(4,0);
\draw[thin] (1.5,-.5) -- +(4,0);
\draw [->] (2,2) .. controls +(0,2) and +(0,2) .. +(3,0);
\draw [->] (3,2) .. controls +(0,1) and +(0,1) .. +(1,0);

\draw [->] (2,-.5) -- +(0,2.5);
\draw [<-] (3,2) .. controls +(0,-1) and +(0,-1) .. +(1,0);
\draw[dotted] (3.55,.3) -- +(0,.95);
\draw [<-] (3,-.5) .. controls +(0,1) and +(0,1) .. +(1,0);
\draw [<-] (5,-.5) -- +(0,2.5);

\draw [<-] (2,-.5) .. controls +(0,-1) and +(0,-1) .. +(1,0);
\draw [<-] (4,-.5) .. controls +(0,-1) and +(0,-1) .. +(1,0);

\begin{scope}[xshift=6cm]
\draw[thin] (2.5,2) -- +(2,0);
\draw[thin] (2.5,-.5) -- +(2,0);
\draw [->] (3,2) .. controls +(0,1) and +(0,1) .. +(1,0);

\draw [<-] (3,2) .. controls +(0,-1) and +(0,-1) .. +(1,0);
\draw[dotted] (3.55,.3) -- +(0,.95);
\draw [<-] (3,-.5) .. controls +(0,1) and +(0,1) .. +(1,0);

\draw [->] (3,-.5) .. controls +(0,-1) and +(0,-1) .. +(1,0);
\end{scope}

\begin{scope}[xshift=12cm]
\draw[thin] (1.5,2) -- +(4,0);
\draw[thin] (1.5,-.5) -- +(4,0);
\draw [<-] (2,2) .. controls +(0,1) and +(0,1) .. +(1,0);
\draw [<-] (4,2) .. controls +(0,1) and +(0,1) .. +(1,0);

\draw [<-] (2,-.5) -- +(0,2.5);
\draw [<-] (3,2) .. controls +(0,-1) and +(0,-1) .. +(1,0);
\draw[dotted] (3.55,.3) -- +(0,.95);
\draw [<-] (3,-.5) .. controls +(0,1) and +(0,1) .. +(1,0);
\draw [->] (5,-.5) -- +(0,2.5);

\draw [->] (2,-.5) .. controls +(0,-2) and +(0,-2) .. +(3,0);
\draw [->] (3,-.5) .. controls +(0,-1) and +(0,-1) .. +(1,0);
\end{scope}

\end{tikzpicture}
\end{eqnarray*}
The first diagram can carry decorations as in (D1)-(D4). Note that (D1) and (D2) is a situation when $x \otimes x \mapsto 0$, hence there is nothing to do. In cases (D3) and (D4) the weight strictly increases.
For the second diagram either both circles are clockwise and there is nothing to do, or
both circles are anticlockwise then the top weight is kept or the two circles have not the same type in which case the weight increases. In the third diagram the inner circle is anticlockwise and hence the top weight either increases or stays the same under the surgery move depending on the type of the outer circle.
\end{proof}

\section{The quasi-hereditary structure of $\D$}\label{sdecomposition}

We briefly describe the representation theory of the algebras $\D$ for any block $\La$, their cell modules and explicit closed formulae for $q$-decomposition numbers
which simultaneously describe  composition multiplicities of
cell modules and  cell filtration
multiplicities of projective indecomposable modules.
We will deduce that the category ${\D}\MOD$
of finite dimensional graded $\D$-modules is a graded highest weight
category.

\subsection{Graded modules, projectives and irreducibles}
\label{graded}
If $A$ is a $\mZ$-graded finite dimensional algebra
and $M=\bigoplus_{j \in \mZ} M_j$ is a {\em graded} $A$-module,
ie. $A_i M_j \subseteq M_{i+j}$,
then we write $M\langle j \rangle$ for the same module but with new grading defined by
$M\langle j \rangle_i := M_{i-j}$. For graded modules $M$ and $N$, we define
\begin{equation}\label{homdef}
\hom_{A}(M, N) := \bigoplus_{j \in \mZ} \hom_{A}(M,N)_j
\end{equation}
where $\hom_{A}(M,N)_j$ denotes all homogeneous $A$-module
homomorphisms of degree $j$, meaning that they map $M_i$ into $N_{i+j}$ for
each $i \in \mZ$.
Later we might also work with (not necessarily unital) $\mZ$-graded algebras with many idempotents (which means they come with a given system
$\{e_\la\:|\:\la\in \La'\}$ of mutually orthogonal idempotents
such that
$
A = \bigoplus_{\la,\mu \in \La'} e_\la  A e_\mu.)
$
By an {\em $A$-module} we mean then a left $A$-module
$M$ such that $
M = \bigoplus_{\la \in \La'} e_\la M;
$
and the notions about graded modules generalize. (Note that if $A$ is finite dimensional, our definitions of $A$-modules agree).

Let $A\gMOD$ be the category of all finite dimensional graded $A$-modules $M =
\bigoplus_{j \in \mZ} M_j$. Together with degree zero morphisms this is an abelian category.

Now fix an arbitrary block $\La$ and consider the graded algebra $A=\D$. Let ${\D}^{> 0}$ be the sum of all
components of strictly positive degree, so
\begin{equation}\label{degzero}
\D / {\D}^{> 0} \cong \bigoplus_{\la \in \La} \mC
\end{equation}
as an algebra, with a basis given by the images of
all the idempotents $e_\la$, $\la\in\La$. The image of $e_\la$
spans a one dimensional graded $\D$-module which we denote by $L(\la)$.
Thus, $L(\la)$ is a copy of the field concentrated in degree $0$, and $(a \mu
b)\in \D$ acts on $L(\la)$ as multiplication by $1$ if $a \mu b =
\underline{\la} \la \overline{\la}$, or as zero otherwise. The modules
\begin{equation}\label{simples}
\{L(\la)\langle j \rangle\:|\:\la \in \La, j\in \mZ\}
\end{equation}
give a complete set of isomorphism classes of simple modules in $\D\gMOD$.
For any graded $\D$-module $M$ we let
${M}^\circledast$ denote its graded dual, which means that
$({M}^\circledast)_{j} := \hom_{\mC}(M_{-j},\mC)$ and $x \in \D$ acts on $f
\in {M}^\circledast$ by $(xf)(m) := f(x^* m)$, where $*$ is the antiautomorphism from Lemma~\ref{antiaut}.
Clearly we have for each $\la \in \La$
\begin{equation}
\label{simplesselfdual}
{L(\la)}^\circledast
 \cong L(\la).
\end{equation}

For $\la \in \La$, let $P(\la) := \D e_\la$. This is a graded $\D$-module
with basis $$\left\{(\underline{\nu} \mu \overline\la)\:\big|\:\text{for all
}\nu,\mu \in \La \text{ such that }\nu \subset \mu \supset \la\right\}$$ and is with the natural surjection a
projective cover of $L(\la)$. The modules
\begin{equation}\label{pims}
\{P(\la)\langle j \rangle\:|\:\la \in \La, j\in \mZ\}
\end{equation}
give a full set of indecomposable projective objects in $\D\gMOD$.

\subsection{Grothendieck groups}
The {\it Grothendieck group} of $\D\gMOD$, denoted $K_0(\D)$, is the free $\mZ$-module on isomorphism classes $[M]$ of objects $M$ in $\D\gMOD$ modulo the relation $[M]=[M']+[M'']$ whenever there is a short exact sequence in $\D\gMOD$ with $M$ as middle and $M'$, $M''$ as outer terms; then $K_0(\D)$ has a basis given by the $\{L(\la)\langle j \rangle\:|\:\la \in \La, j\in \mZ\}$. Let $[\op{proj}({\D})]$ be the subgroup generated by the classes of the projective
indecomposable modules from \eqref{pims}. They both carry a free $\mZ[q,q^{-1}]$-module structure by setting
$$
q^j [M] := [M\langle j
\rangle].
$$
In particular, there are $c_{\la,\mu}(q)\in\mZ[q,q^{-1}]$ such that
\begin{equation}\label{decm}
[P(\mu)] = \sum_{\la \in \La} c_{\la,\mu}(q) [L(\la)].
\end{equation}
The  matrix $C_\La(q) = (c_{\la,\mu}(q))_{\la,\mu \in\La}$ is called the {\em $q$-Cartan matrix} of $\D$.

\begin{lemma}
The entries in the Cartan matrix of $\D$ are explicitly given as follows:
\begin{equation}\label{exp}
c_{\la,\mu}(q) =
\sum_{\la \subset \nu \supset \mu}
q^{\deg(\underline{\la}\nu\overline{\mu})}.
\end{equation}
\end{lemma}

\begin{proof}
To determine its entries note first that
\begin{equation}\label{cmat}
c_{\la,\mu}(q) = \sum_{j \in \mZ} q^j \dim \hom_{\D}(P(\la),P(\mu))_j.
\end{equation}
Since $\hom_{\D}(P(\la), P(\mu)) = \hom_{\D} (\D e_\la, \D e_\mu)= e_\la
\D e_\mu
$ and $e_\la \D e_\mu$ has basis
$\left\{(\underline{\la}\nu\overline{\mu})\:\big|\:
\nu \in \La \text{ such that }\la \subset\nu\supset\mu\right\}$ the claim follows.
\end{proof}

Note that $c_{\la,\mu}(q)$ is in fact a polynomial in $\mathbb{N}[q]$ with constant coefficient equal to $1$ if $\la = \mu$ and equal to $0$
otherwise, see Figure~\ref{fig:CM} for an example.


Now we introduce {\em cell modules} in the sense of \cite{GL}. The construction is totally analogous to \cite{BS1}, hence we just recall the results.

\begin{definition}
\label{cellmods}
For $\mu\in \La$, define $V(\mu)$ to be the vector space on homogeneous basis
\begin{equation}
\left\{
(c \mu | \:\big|\:
\text{for all oriented cup diagrams $c \mu$}\right\},
\end{equation}
where the degree of the vector $(c \mu |$ is the degree $\deg(c\mu)$ of the oriented cup diagram.
We make $V(\mu)$ into a graded $\D$-module by declaring
for any basis vector $(a \la b)$ of $\D$ that
\begin{equation}\label{Actby}
(a \la b) (c \mu| :=
\left\{
\begin{array}{ll}
s_{a \la b}(\mu)  (a \mu|
&\text{if $b^* = c$ and $a \mu$ is oriented,}\\
0&\text{otherwise.}
\end{array}\right.
\end{equation}
where $s_{a \la b}(\mu) \in \{0,\mp 1\}$ is the scalar
from Theorem~\ref{cellular}; hence the action is well-defined. We call $V(\mu)$ the {\emph cell module of highest weight $\mu$}.
\end{definition}

\begin{theorem}[Cell module filtration of projectives]
\label{qh1}
For $\la \in \La$,
enumerate the $2^{\op{def}(\la)}$ distinct elements of the set
$\{\mu \in \La\:|\:\mu \supset \la\}$
as $\mu_1,\mu_2,\dots,\mu_n = \la$
so that $\mu_i > \mu_j$ implies $i < j$.
Let $M(0) := \{0\}$ and for $i=1,\dots,n$ define
$M(i)$ to be the subspace of $P(\la)$
generated by $M(i-1)$ and the vectors
$$
\left\{
(c \mu_i \overline{\la} ) \:\big|\:
\text{for all oriented cup diagrams $c \mu_i$}\right\}.
$$
Then
$
\{0\} = M(0) \subset M(1) \subset\cdots\subset M(n) = P(\la)
$
is a filtration of $P(\la)$ as a $\D$-module such that for each $i=1,\dots, n$:
$$
M(i) / M(i-1) \cong V(\mu_i) \langle \deg(\mu_i
\overline{\la})\rangle.
$$
\end{theorem}

\begin{proof}
Just apply the same arguments as in \cite[Theorem 5.1]{BS1}.
\end{proof}

\begin{theorem}[Composition factors of cell modules]
\label{qh2}
For $\mu \in \La$, let $N(j)$ be the submodule of $V(\mu)$ spanned by all
graded pieces of degree $\geq j$. Then we have a filtration
$$
V(\mu) = N(0) \supseteq N(1) \supseteq N(2) \supseteq \cdots
$$
as a $\D$-module with
$N(j) / N(j+1) \cong \displaystyle\bigoplus_{\substack{\la  \subset \mu\text{\,with}\\
\deg(\underline{\la} \mu) = j}}
  L(\la) \langle j \rangle
$ for each $j \geq 0$.
\end{theorem}
\begin{proof}
Analogously to \cite[Theorem 5.2]{BS1}.
\end{proof}


These polynomials and the resulting {\em $q$-decomposition matrix}
\begin{equation}\label{decmat}
M_\La(q) = (d_{\la,\mu}(q))_{\la,\mu \in \La}
\end{equation}
encode by Theorems~\ref{qh1} and~\ref{qh2} the multiplicities of cell modules in projectives and of irreducibles in cell modules; we have
\begin{align}
[V(\mu)] &= \sum_{\la \in \La} d_{\la,\mu}(q) [L(\la)],\label{dmat}&[P(\la)] &= \sum_{\mu \in \La} d_{\la,\mu}(q) [V(\mu)],
\end{align}
in the Grothendieck group $K_0(\D)$.

The $q$-decomposition matrix $M_\La(q)$ is upper
unitriangular when rows and columns are ordered in some way refining the Bruhat order since by Lemma~\ref{lem:Bruhator} $\la \subset \mu$ implies $\la \leq \mu$.
Note that $d_{\la,\mu}(q)=q^d$ if $\underline{\la}\mu$ is oriented of degree $d$ and $d_{\la,\mu}(q)=0$ if $\underline{\la}\mu$ is not oriented.

\begin{ex}
The decomposition matrix for the principal block of type $D_4$ is given as follows (where we again omit the $\circ$'s in the weight diagrams):
\begin{equation}\label{qdec}
\begin{array}{c|cccccccc}
&
\scriptstyle{\down\down\down\down}&
\scriptstyle{\up\up\down\down}&
\scriptstyle{\up\down\up\down}&
\scriptstyle{\down\up\up\down}&
\scriptstyle{\up\down\down\up}&
\scriptstyle{\down\up\down\up}&
\scriptstyle{\down\down\up\up}&
\scriptstyle{\up\up\up\up}\\
\hline
\scriptstyle{\down\down\down\down}&1&q&0&0&0&0&0&q^2\\
\scriptstyle{\up\up\down\down}&0&1&q&0&0&0&q^2&q\\
\scriptstyle{\up\down\up\down}&0&0&1&q&q&q^2&q&0\\
\scriptstyle{\down\up\up\down}&0&0&0&1&0&q&0&0\\
\scriptstyle{\up\down\down\up}&0&0&0&0&1&q&0&0\\
\scriptstyle{\down\up\down\up}&0&0&0&0&0&1&q&0\\
\scriptstyle{\down\down\up\up}&0&0&0&0&0&0&1&q\\
\scriptstyle{\up\up\up\up}&0&0&0&0&0&0&0&1
\end{array}
\end{equation}
\end{ex}

We can restrict ourself to study the principal blocks:
\begin{lemma}
\label{lem:KLpolys}
\begin{enumerate}
\item The q-decomposition matrix $M_\La(q)$ depends (up to relabelling rows and columns) only on the atypicality of the block.
\item In case $\La=\Lap$, the entries of ${M}_\La(q)$ are the parabolic Kazhdan-Lusztig polynomials, denoted $n_{x,y}(q)$ in \cite{SoergelKL} for $x,y\in W^{\ov{p}}$ from Section~\ref{typeD}.
\end{enumerate}
\end{lemma}

\begin{proof}
By Corollary~\ref{atypicality} the algebras are isomorphic and hence the decomposition numbers are of course the same. The second statement is the main result of \cite{LS}.
\end{proof}

As in \cite[Theorem 5.3.]{BS1} we deduce
\begin{theorem}\label{ghw}
The category ${\D}\MOD$ is a positively
graded highest weight category with duality in the sense of Cline, Parshall and
Scott. For $\la \in \La$ (the weight poset), the irreducible, standard, costandard,
indecomposable projective and indecomposable injective objects are respectively the modules
$L(\la), V(\la), {V(\la)}^\circledast$, $P(\la)$ and ${P(\la)}^\circledast$. In particular, the algebra $\D$ is a
positively graded quasi-hereditary algebra.
\end{theorem}

We also like to state the following observation
\begin{corollary}
\label{lem:endo_commutative}
Let $\La$ be a block and $\la\in\La$.
The endomorphism ring $\op{End}_{\La}(P(\la))$ of $P(\lambda)$ is a non-negatively graded commutative algebra.
\end{corollary}

\begin{proof}
As a summand of a positively graded algebra its positively graded. By definition of the multiplication it is commutative (the multiplication is just a composition of merges).
\end{proof}

\section{The categories $\Perv$ and $\cO_0^\p(\mathfrak{so}(2k))$ diagrammatically}
Let $\Perv=\Perv_k$ be the category of perverse sheaves on $Y_k$ constructible with respect to the Schubert stratification, see
\cite{Braden} for details. By the localization theorem and Riemann-Hilbert correspondence this category is equivalent to the principal
block  $\cO_0^\p(\mg)$ of the parabolic category $\cO^\p(\mg)$ for the semisimple Lie algebra $\mg=\mathfrak{so}(2k)$ of type $D_k$, where
$\p$ is one of our maximal parabolic subalgebra associated with $W(1)$ or $W(0)$ respectively. Let us stick to $W(1)$ in the following.
The simple objects in either category are naturally labelled by $W^1$: attached to $w\in W^1$ we have the simple intersection cohomology
complex $\mathcal{I}_w$ corresponding to the Schubert variety labelled $w$ via \eqref{Young} and the simple module $L(w)$ with highest
weight $w\cdot0$. (As usual $W$ acts on weights via the `dot-action' $w\cdot\la=w(\la+\rho)-\rho$). Let $P(w)$ be the projective cover of
$L(w)$. Then $P:=\bigoplus_{w\in W^1} P(w)$ is a minimal projective generator of $\cO_0^\p(\mg)$.

\subsection{The isomorphism theorem}
Then our main theorem is the
following

\begin{theorem}
\label{thm:main}
Let $k\geq 4$. There is an isomorphism of algebras $\End_\mg(P)\cong\mathbb{D}_{\Lambda_k^{\ov{0}}}$, hence equivalences of categories
\begin{align}
\cO_0^\p(\mathfrak{so}(2k))\equiv\Perv_k\equiv\mathbb{D}
_{\Lambda_k^{\ov{0}}}\Mod
\end{align}
which identify the simple objects $L(w)$, $\mathcal{I}_w$ and $\mathcal{L}(w)$ and their projective covers respectively.
\end{theorem}

\begin{remark}
{\rm Note that, although the theorem only deals with the principal block $\Lambda_k^{\ov{0}}$, by Corollary~\ref{atypicality} it gives in fact a geometric description for all blocks $\La$.
}
\end{remark}

\begin{lemma}
\label{projinj}
Let $\La$ be a principal block then the algebra $\mathbb{H}_\La$ from Lemma~\ref{Khovalg} is the endomorphism algebra of the sum of all indecomposable projective-injective $\D$-modules.
\end{lemma}

\begin{proof}
Let $\la\in\La$ be of maximal defect with indecomposable projective module $P(\la)$. Under the equivalence of Theorem~\ref{thm:main}
it corresponds to an indecomposable projective module $P(\la)$ in $\cO_0^\p(\mg)$. Let $\la=w\cdot\mu$ with $\mu$ the dominant weight in the same block and $w\in W^1$. By Irving's characterization of projective-injective modules in parabolic category $\cO$, \cite[4.3]{Irvingselfdual} $P(\la)$ is projective-injective, if and only if $w$ is in the same Kazhdan-Lusztig right cell as the longest element $w_0^1$ in $W^1$. Hence the indecomposable projective-injective modules $P(x\cdot\mu)$ are precisely those appearing as summand in $P(w_0^1\cdot\mu)$ under the Hecke algebra action given by translation functors. By
\cite[Theorem 3.10]{LS} the action of the Hecke algebra factors through the type $D$ Temperley-Lieb algebra and we can identify the Kazhdan-Lusztig basis with our cup diagrams and the action of the Temperley-Lieb algebra purely diagrammatically. Note that $P(w_0^1\cdot\mu)$ corresponds to the cup diagram $C$ of maximal defect with all cups dotted and we ask which cup diagrams can be obtained from it by acting with the Temperley-Lieb algebra. It is obvious that the defect can't decrease. On the other hand one can easily verify that every cup diagram of maximal defect can be obtained, see \cite[Remark 5.24]{LS}. Then the lemma follows.
\end{proof}

\subsection{Braden's algebra $A(D_k)$}
To prove Theorem~\ref{thm:main} we will identify the algebra $\mathbb{D}_{\Lambda_k^{\ov{0}}}$ with Braden's algebra $A=A(D_k)$ defined below using the following result from \cite[Theorem 1.7.2]{Braden}:
\begin{theorem}
\label{thmBraden}
Let $k\geq 4$. There is an equivalence of categories $$A(D_k)\Mod\cong\Perv_k$$ which identifies $A(D_k)$ with the endomorphism ring of a minimal projective generator of $\Perv_k$.
\end{theorem}

\begin{definition}
Let $\Lambda$ be a block and $\la \in \Lambda$ be a weight. Given a $\la$-pair $(\alpha,\beta)$ corresponding to a cup $C$ (see Definition~\ref{lapair}), we say it {\emph has a parent}, if $C$ is nested in another (dotted or undotted) cup or if otherwise there is a dotted cup to the right of $C$. We call then the minimal cup containing $C$ respectively the leftmost dotted cup to the right of $C$ and its associated $\la$-pair the {\it parent} of $C$; we denote them by $C'$ and $(\alpha', \beta')$ respectively.
All the possible examples are shown in \eqref{parent1}-\eqref{parent4}.
\end{definition}

\begin{definition}
\begin{enumerate}
\item A {\it diamond} in $\Lambda_k^{\ov{0}}$ is a quadruple $(\lambda,\lambda',\lambda'',\lambda''')$ of distinct elements in $\Lambda_k^{\ov{0}}$ such that $\lambda \leftrightarrow \lambda' \leftrightarrow \lambda''
\leftrightarrow \lambda''' \leftrightarrow \lambda$ and all elements are pairwise distinct. We depict them as follows
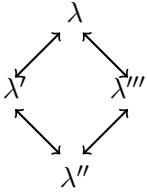
\begin{eqnarray*}
\begin{tikzpicture}[thick,scale=0.6]
\draw[<->] (1,-1) node[above]{\quad $\lambda$} -- +(-1,-1);
\draw[<->] (1.5,-1) -- +(1,-1);
\draw[<->] (0,-2.7) node[above]{$\lambda'$} -- +(1,-1) node[below]{\quad $\lambda''$};
\draw[<->] (2.5,-2.7) node[above]{$\lambda'''$} -- +(-1,-1);
\end{tikzpicture}
\end{eqnarray*} We say that two diamonds are {\it equivalent} if they are obtained one from the other by a sequence of cyclic rotations and
interchanging $\lambda'$ and $\lambda'''$.
\item For fixed $k$, a triple $\la_1 \leftrightarrow \la_2 \leftrightarrow \la_3$ in $\Lambda_{k}^{\ov{0}}$ {\it cannot be extended to a diamond} if there is no diamond $(\la_1,\la_2,\la_3,\la_4)$ in $\Lambda_{k}^{\ov{0}}$.
\item For fixed $k$, a triple $\la_1 \leftrightarrow \la_2 \leftrightarrow \la_3$ in $\Lambda_{k}^{\ov{0}}$ {\it can be enlarged} to a diamond if it cannot be extended but, after extending the weights $\la_i$ to $\tilde{\la}_i$ by adding some fixed number of $\down$'s to the left and $\up$'s to the right, such that the resulting triple can be extended to a diamond $(\tilde{\la}_1,\tilde{\la}_2,\tilde{\la}_3,\tilde{\la}_4)$ in $\Lambda_{m}^{\ov{0}}$ for some $m>k$.
\end{enumerate}
\end{definition}

\begin{ex}
\label{extended}
The triple $(\down\down\down\down\up,\up\up\down\down\up,\up\down\up\down\up)$ can not be extended to a diamond in $\Lambda_{5}^{\ov{0}}$, but it can be enlarged to a diamond $(\down\down\down\down\up\up,\up\up\down\down\up\up,\up\down\up\down\up\up,
\down\down\up\down\up\down)\in\Lambda_{6}^{\ov{0}}$. In other words, the corresponding Young diagrams
\begin{eqnarray}
{\tiny\Yvcentermath1\yng(4,4,4,4)\quad\quad\yng(4,4,2,2)\quad\quad\yng(4,3,2,1)
\quad\quad\yng(6,5,4,4,2,1)}
\end{eqnarray}
fit into a $6\times 6$-box, but only three of them into a $5\times 5$-box.
\end{ex}

\begin{definition}
\label{Bradenalgebra}
The algebra $A=A(D_k)$ is the unitary associative $\mC$-algebra with generators
$$\{e_\la\mid\la\in \Lambda_k^{\ov{0}} \}\cup\{p(\la,\la')\mid \la,\la'\in \Lambda_k^{\ov{0}}, \la\leftrightarrow\la'\}\cup
\{t_{\alpha,\la}\mid \la\in \Lambda_k^{\ov{0}} ,\alpha\in\mZ\}$$
and relations for $\lambda, \mu, \nu \in \Lambda_k^{\ov{0}}$ and $\alpha, \beta \in \mZ$:
\begin{enumerate}[(\text{R}1)]
\item \label{rel1} The $e_\la$'s are pairwise orthogonal idempotents, ie., $e_\la e_\nu=0$ if $\la\not=\nu$ and $e_\la e_\la=e_\la$;
\item \label{rel2} $\sum_{\la\in \Lambda_k^{\ov{0}}} e_\la=1$;
\item \label{rel3a} $e_{\nu}p(\la,\mu)=\delta_{\nu,\la}p(\la,\mu)$ and \\$p(\la,\mu)e_{\nu}=\delta_{\mu,\nu}p(\la,\mu)$;
\item \label{rel3b} $e_{\nu}t_{\la,\alpha}=\delta_{\nu,\la}=t_{\la,\alpha}e_{\nu}$;
\item \label{rel4a} The $t$'s commute with each other;
\item \label{rel4b} $p(\la,\mu)t_{\alpha,\mu} = t_{\alpha,\la}p(\la,\mu)$;
\item \label{rel5}
\begin{enumerate}[(i)]
\item $t_{\alpha,\la}t_{-\alpha,\la} = 1$;
\item $t_{\alpha,\la}=e_\la$ if $\alpha<-k$ or $\alpha>k$;
\item $t_{\alpha,\la}t_{\beta,\la}=e_\la$ if $(\alpha,\beta)$ is a $\la$-pair;
\end{enumerate}
\item \label{rel6} Suppose $\la\stackrel{(\alpha,\beta)}\longrightarrow\la'$. Then
  \begin{eqnarray*}
  \op{m}(\la',\la)^{(-1)^\beta}&=&t_{\alpha,\la'}t_{\zeta,\la'};\\
  \op{m}(\la,\la')^{(-1)^\beta}&=&t_{\alpha,\la}t_{\zeta,\la};
  \end{eqnarray*}
where $\op{m}(\la,\la')=e_\la+p(\la,\la')p(\la',\la)$ and
$$ \zeta = \left\lbrace \begin{array}{ll} -\alpha', & \text{if } \alpha < -\beta' < \beta < -\alpha' \\ \beta' & \text{otherwise },
\end{array} \right.$$
in case the parent $(\alpha',\beta')$ of $(\alpha,\beta)$ exists and $t_{\zeta,\la}=1$ otherwise.

\item \label{rel7}
\begin{enumerate}[(i)]
\item If $(\la_1,\la_2,\la_3,\la_4)$ is a diamond in $\Lambda_{k}^{\ov{0}}$ then
  \begin{eqnarray*}
    p(\la_3,\la_2)p(\la_2,\la_1)&=&p(\la_3,\la_4)p(\la_4,\la_1).
  \end{eqnarray*}
\item
  Given a triple $\la_1 \leftrightarrow \la_2 \leftrightarrow \la_3$ in $\Lambda_{k}^{\ov{0}}$ which cannot be extended but can be enlarged to a diamond then
 \begin{eqnarray}
\label{zero}
    p(\la_3,\la_2)p(\la_2,\la_1)&=0=&p(\la_1,\la_2)p(\la_2,\la_3).
  \end{eqnarray}
\item Given a triple $\la_1 \stackrel{(\alpha,\beta)}{\rightarrow} \la_2 \stackrel{(\gamma,\delta)}{\rightarrow} \la_3$ in $\Lambda_{k}^{\ov{0}}$ such that
    \begin{itemize}
    \item $\alpha<0$ (hence $(\alpha,\beta)$ corresponds to a dotted cup),
    \item $(\gamma,\delta)$ is not a $\la_1$-pair, and
    \item the triple cannot be extended to a diamond, then
    \end{itemize}
      \begin{eqnarray*}
    p(\la_3,\la_2)p(\la_2,\la_1)&=0=&p(\la_1,\la_2)p(\la_2,\la_3).
  \end{eqnarray*}
\end{enumerate}
\end{enumerate}
\end{definition}

\begin{ex}
\label{triples}
{\rm The triple $(\la_1,\la_2,\la_3)=(\up\up\up\up,\down\down\up\up,\down\up\down\up)$ satisfies all conditions from
(R9) (iii), whereas the triple $(\up\up\up\up,\down\down\up\up,\up\down\up\down)$ is of the form  $\mu_1 \rightarrow \mu_2 \stackrel{(\alpha,\beta)}{\rightarrow} \mu_3$ such that $(\alpha,\beta)$ is not a $\mu_1$-pair, but can be extended to the diamond $(\mu_1,\mu_2,\mu_3,\mu_4)=(\up\up\up\up,\down\down\up\up,\up\down\up\down, \up\up\down\down)$.
The corresponding degree one morphisms in $\mathbb{D}_4$ are he following diagrams equipped with a degree one orientation:
\begin{equation*}
\begin{tikzpicture}[thick,yscale=-1, scale=0.5]
\node at (0,-1){$A:=\underline{\la_1}\overline{\la_2}
=\underline{\mu_1}\mu_2\overline{\mu_2}=$};
\begin{scope}[xshift=4.3cm, yshift=-0.5cm]
\draw (0,-.1) .. controls +(0,1) and +(0,1) .. +(1,0);
\fill (0.5,.63) circle(3.5pt);
\draw (2,-.1) .. controls +(0,1) and +(0,1) .. +(1,0);
\fill (2.5,.63) circle(3.5pt);
\draw (1,-.1) .. controls +(0,-1) and +(0,-1) .. +(1,0);
\draw (0,-.1) .. controls +(0,-2.2) and +(0,-2.2) .. +(3,0);
\end{scope}
\node at (10,-1){$B:=\underline{\la_2}\overline{\la_3}=$};
\begin{scope}[xshift=12.6cm,yshift=-1.5cm, yscale=-1]
\draw (0,-.1) .. controls +(0,1) and +(0,1) .. +(1,0);
\draw (2,-.1) .. controls +(0,1) and +(0,1) .. +(1,0);
\draw (1,-.1) .. controls +(0,-1) and +(0,-1) .. +(1,0);
\draw (0,-.1) .. controls +(0,-2.2) and +(0,-2.2) .. +(3,0);
\end{scope}
\node at (17,-1){$\underline{\la_1}\overline{\la_3}=$};
\begin{scope}[xshift=18.3cm,yshift=-1cm]
\draw (0,-.1) .. controls +(0,1) and +(0,1) .. +(1,0);
\fill (0.5,.63) circle(3.5pt);
\draw (1.5,-.1) .. controls +(0,1) and +(0,1) .. +(1,0);
\fill (2,.63) circle(3.5pt);
\draw (0,-.1) .. controls +(0,-1) and +(0,-1) .. +(1,0);
\draw (1.5,-.1) .. controls +(0,-1) and +(0,-1) .. +(1,0);
\end{scope}
\end{tikzpicture}
\end{equation*}
\begin{equation*}
\begin{tikzpicture}[thick,scale=0.5]
\node at (0,-1){$C:=\underline{\mu_2}\overline{\mu_3}=$};
\begin{scope}[xshift=2.7cm, yshift=-0.5cm]
\draw (0,-.1) -- +(0,.7);
\draw (1,-.1) .. controls +(0,1) and +(0,1) .. +(1,0);
\draw (3,-.1) -- +(0,.7);
\fill (0,.13) circle(3.5pt);
\draw (1,-.1) .. controls +(0,-1) and +(0,-1) .. +(1,0);
\draw (0,-.1) .. controls +(0,-2.2) and +(0,-2.2) .. +(3,0);
\end{scope}
\node at (8.6,-1){$D:=\underline{\mu_1}\mu_4\overline{\mu_4}=$};
\begin{scope}[xshift=11.3cm, yshift=-1cm, yscale=-1]
\draw (0,-.1) .. controls +(0,1) and +(0,1) .. +(1,0);
\fill (.5,.63) circle(3.5pt);
\draw (1.5,-.1) .. controls +(0,1) and +(0,1) .. +(1,0);
\fill (2,.63) circle(3.5pt);
\draw (0,-.1) .. controls +(0,-1) and +(0,-1) .. +(1,0);
\fill (.5,-.83) circle(3.5pt);
\draw (1.5,-.8) -- +(0,.7);
\draw (2.5,-.8) -- +(0,.7);
\end{scope}
\node at (16.7,-1){$E=\underline{\mu_4}\mu_4\overline{\mu_3}=$};
\begin{scope}[xshift=19.3cm, yshift=-1cm, yscale=-1]
\draw (1,-.1) .. controls +(0,-1) and +(0,-1) .. +(1,0);
\fill (0,-.43) circle(3.5pt);
\draw (0,-.8) -- +(0,.8);
\draw (2.5,-.8) -- +(0,.8);
\draw (0,-.1) .. controls +(0,1) and +(0,1) .. +(1,0);
\fill (.5,.63) circle(3.5pt);
\draw (2,-.1) -- +(0,.8);
\draw (2.5,-.1) -- +(0,.8);
\end{scope}
\end{tikzpicture}
\end{equation*}
Since $\underline{\la_1}\overline{\la_3}$ cannot be oriented at all we must have $BA=0=AB$ whereas $\underline{\mu_1}\overline{\mu_3}$ has a unique orientation $\underline{\mu_1}\tau\overline{\mu_3}$ with $\tau=\up\up\down\down$ and in fact $CA=\underline{\mu_1}\tau\overline{\mu_3}=ED$ and $AC=\underline{\mu_3}\tau\overline{\mu_1}=DE$.
}
\end{ex}

\subsection{$\la$-pairs and diamonds}
To match Definition~\ref{Bradenalgebra} with Braden's algebra \cite[1.7]{Braden} we need the following key technical lemma which compares our notion of $\la$-pair with Braden's. Given a weight $\la\in \Lambda_k^{\ov{0}}$ recall from Lemma~\ref{bijection_S_Lambda} the associated antisymmetric sequence $s = (s_i)_{-k \leq i \leq k}$ associated to $\la$ and extend it by infinitely many $-$'s
to the left and infinitely many $+$'s to the right to get an infinite sequence $\tilde{\la}$ indexed by half-integers such that the anti-symmetry line $L$ passes between -$\frac{1}{2}$ and $\frac{1}{2}$, in formulas $\tilde{\la}=(X_i)_{i \in\mathbb{Z}+\frac{1}{2}}$, where
$$
X_{i-\frac{1}{2}}= \left\lbrace \begin{array}{ll}
- & \text{if } i < -k, \\
s_i & \text{if } -k \leq i \leq k, \\
+ & \text{if } i > k.
\end{array} \right.
$$

\begin{lemma}
\label{lem:comparelapairs}
Let $\la,\la' \in \Lambda_k^{\ov{0}}$ and $-k\leq i,j\leq k$ then
\begin{eqnarray}
\centerline{$\la\stackrel{(i,j)}{\longrightarrow}\la'$ if and only if $\tilde{\la}\stackrel{(i',j')}{\longrightarrow}\tilde{\la}'$}
\end{eqnarray}
with the right hand side in the terminology of \cite{Braden} where $i'=i-\frac{1}{2}$, $j'=j-\frac{1}{2}$.
\end{lemma}
\begin{proof}
For a weight $\la$ with associated sequence $X_i$ and $a,b\in\mathbb{Z}+\frac{1}{2}$ we set $b(\mp,i,j)=|\{r\mid i< r< j, X_r=\mp\}|$.
A $\la$-pair in the sense of \cite{Braden} is a pair $(\alpha,\beta)\in(\mathbb{Z}+\frac{1}{2})^2$ such that
\begin{enumerate}[(i)]
\item $0<\alpha<\beta$ $X_\alpha=-$ and $X_\beta=+$ and $b(-,\alpha,\beta)=b(+,\alpha,\beta)$, or
\item $0<-\alpha<0<\beta$ with $\alpha+\frac{1}{2}$ even and $X_\alpha=X_\beta=+$ (and then automatically $X_{-\alpha}=X_{-\beta}=-$) and  $b(-,\gamma,\gamma')=b(\mp,\gamma,\gamma')$ for the pairs $(\gamma,\gamma')\in\{(-\alpha, \alpha),(-\beta,\beta)\}$ and $b(+,\alpha,\beta+1)=b(-,\alpha,\beta+1)-1$.
\end{enumerate}
For  $\la,\la' \in \Lambda_k^{\ov{0}}$, the case (i) obviously corresponds precisely to cups of type (C1), hence to our $\la$ pairs $(i,j)$ for $0<i<j$. We claim case (ii) corresponds to cups of type (C3), hence to our $\la$ pairs $(-i,j)$ for $0<-i<j$. The conditions on the pairs $(-\alpha, \alpha),(-\beta,\beta)$ means they correspond to two cups crossing the middle line $L$ in the symmetric cup diagrams from \cite{LS}. Since $\alpha+\frac{1}{2}$ is even, these two cups get turned into a dotted cup using the rules from \cite[5.2]{LS}.
\end{proof}
\begin{remark}
{\rm
One should note that Braden has an infinite number of $\la$-pairs for a single non-truncated weight $\lambda$, but only a finite number with $\la'\in \Lambda_k^{\ov{0}}$ as well. Our definition of $\la$-pairs produces only these relevant pairs.
}
\end{remark}
\begin{corollary}
\label{CorCartan}
Let $k\geq 4$ and $\La={\Lambda_k^{\ov{0}}}$. The algebra $A=A(D_k)$ agrees with the algebra defined in \cite[1.7]{Braden}. The Cartan matrix of $A(D_k)$ agrees with the Cartan matrix of $\mathbb{D}_{\Lambda}.$
\end{corollary}
\begin{proof}
The first part follows from Lemma~\ref{lem:comparelapairs} and the definitions.\footnote{We tried to clarify the misleading formulation of the analogue of part (R9) (iii) in \cite{Braden}.} The second statement follows directly from the fact that both Cartan matrices are of the form $C=D_\La^tD_\La$, where the entries of the decomposition matrix $M_\La$ are parabolic Kazhdan-Lusztig polynomials of type $(D_k,A_{k-1})$, see \cite[Theorem 3.11.4 (i)]{BGS} and Lemma~\ref{lem:KLpolys} respectively.
\end{proof}
In particular, the graph underlying its Ext-quiver is $\mathcal{Q}(\Lambda)$.

\begin{prop}
\label{diamondsclass}
Up to equivalence the following local configurations are the only possible diamonds where the relevant parts do not contain any rays.

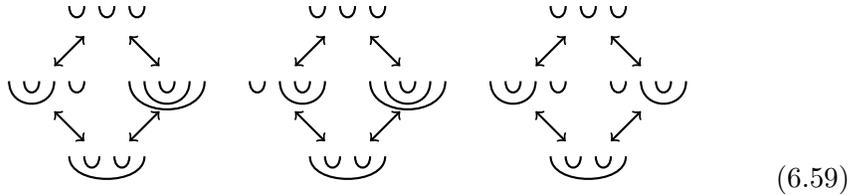
\begin{eqnarray}
\label{diamonds2}
\begin{tikzpicture}[thick,scale=0.4]
\draw (2,0) .. controls +(0,-.5) and +(0,-.5) .. +(.5,0);
\draw (3,0) .. controls +(0,-.5) and +(0,-.5) .. +(.5,0);
\draw (4,0) .. controls +(0,-.5) and +(0,-.5) .. +(.5,0);

\draw[<->] (2.5,-1) -- +(-1,-1);
\draw[<->] (4,-1) -- +(1,-1);

\draw (0,-2.5) .. controls +(0,-1) and +(0,-1) .. +(1.5,0);
\draw (.5,-2.5) .. controls +(0,-.5) and +(0,-.5) .. +(.5,0);
\draw (2,-2.5) .. controls +(0,-.5) and +(0,-.5) .. +(.5,0);

\draw (4,-2.5) .. controls +(0,-1.25) and +(0,-1.25) .. +(2.5,0);
\draw (4.5,-2.5) .. controls +(0,-1) and +(0,-1) .. +(1.5,0);
\draw (5,-2.5) .. controls +(0,-.5) and +(0,-.5) .. +(.5,0);

\draw[<->] (2.5,-4.5) -- +(-1,1);
\draw[<->] (4,-4.5) -- +(1,1);

\draw (2,-5) .. controls +(0,-1) and +(0,-1) .. +(2.5,0);
\draw (2.5,-5) .. controls +(0,-.5) and +(0,-.5) .. +(.5,0);
\draw (3.5,-5) .. controls +(0,-.5) and +(0,-.5) .. +(.5,0);

\begin{scope}[xshift=8cm]
\draw (2,0) .. controls +(0,-.5) and +(0,-.5) .. +(.5,0);
\draw (3,0) .. controls +(0,-.5) and +(0,-.5) .. +(.5,0);
\draw (4,0) .. controls +(0,-.5) and +(0,-.5) .. +(.5,0);

\draw[<->] (2.5,-1) -- +(-1,-1);
\draw[<->] (4,-1) -- +(1,-1);

\draw (1,-2.5) .. controls +(0,-1) and +(0,-1) .. +(1.5,0);
\draw (0,-2.5) .. controls +(0,-.5) and +(0,-.5) .. +(.5,0);
\draw (1.5,-2.5) .. controls +(0,-.5) and +(0,-.5) .. +(.5,0);

\draw (4,-2.5) .. controls +(0,-1.25) and +(0,-1.25) .. +(2.5,0);
\draw (4.5,-2.5) .. controls +(0,-1) and +(0,-1) .. +(1.5,0);
\draw (5,-2.5) .. controls +(0,-.5) and +(0,-.5) .. +(.5,0);

\draw[<->] (2.5,-4.5) -- +(-1,1);
\draw[<->] (4,-4.5) -- +(1,1);

\draw (2,-5) .. controls +(0,-1) and +(0,-1) .. +(2.5,0);
\draw (2.5,-5) .. controls +(0,-.5) and +(0,-.5) .. +(.5,0);
\draw (3.5,-5) .. controls +(0,-.5) and +(0,-.5) .. +(.5,0);
\end{scope}

\begin{scope}[xshift=16cm]
\draw (2,0) .. controls +(0,-.5) and +(0,-.5) .. +(.5,0);
\draw (3,0) .. controls +(0,-.5) and +(0,-.5) .. +(.5,0);
\draw (4,0) .. controls +(0,-.5) and +(0,-.5) .. +(.5,0);

\draw[<->] (2.5,-1) -- +(-1,-1);
\draw[<->] (4,-1) -- +(1,-1);

\draw (0,-2.5) .. controls +(0,-1) and +(0,-1) .. +(1.5,0);
\draw (.5,-2.5) .. controls +(0,-.5) and +(0,-.5) .. +(.5,0);
\draw (2,-2.5) .. controls +(0,-.5) and +(0,-.5) .. +(.5,0);

\draw (4,-2.5) .. controls +(0,-.5) and +(0,-.5) .. +(.5,0);
\draw (5,-2.5) .. controls +(0,-1) and +(0,-1) .. +(1.5,0);
\draw (5.5,-2.5) .. controls +(0,-.5) and +(0,-.5) .. +(.5,0);

\draw[<->] (2.5,-4.5) -- +(-1,1);
\draw[<->] (4,-4.5) -- +(1,1);

\draw (2,-5) .. controls +(0,-1) and +(0,-1) .. +(2.5,0);
\draw (2.5,-5) .. controls +(0,-.5) and +(0,-.5) .. +(.5,0);
\draw (3.5,-5) .. controls +(0,-.5) and +(0,-.5) .. +(.5,0);
\end{scope}
\end{tikzpicture}
\end{eqnarray}
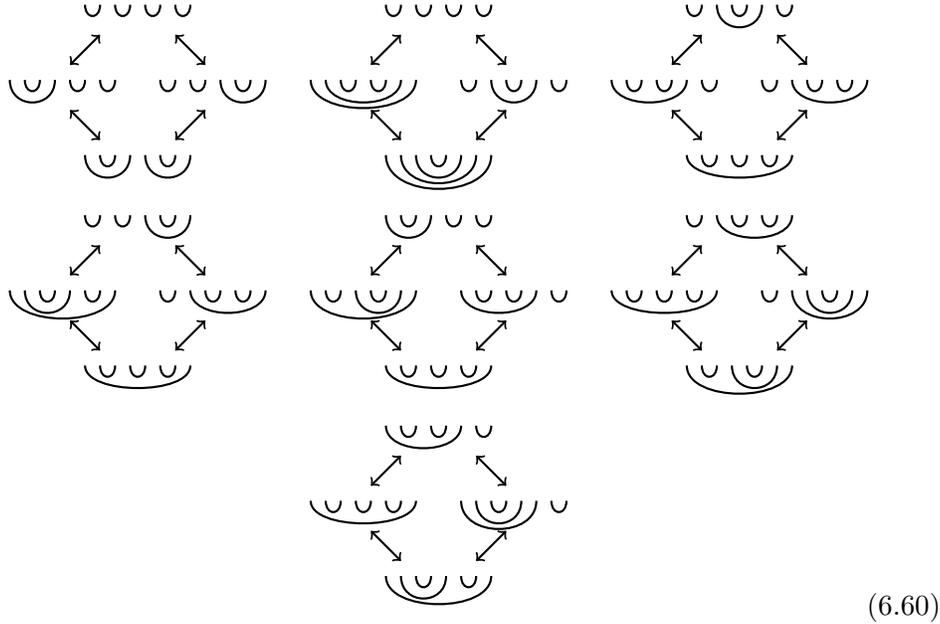
\begin{eqnarray}
\label{diamonds1}
\begin{tikzpicture}[thick,scale=0.4]
\draw (2.5,0) .. controls +(0,-.5) and +(0,-.5) .. +(.5,0);
\draw (3.5,0) .. controls +(0,-.5) and +(0,-.5) .. +(.5,0);
\draw (4.5,0) .. controls +(0,-.5) and +(0,-.5) .. +(.5,0);
\draw (5.5,0) .. controls +(0,-.5) and +(0,-.5) .. +(.5,0);

\draw[<->] (3,-1) -- +(-1,-1);
\draw[<->] (5.5,-1) -- +(1,-1);

\draw (0,-2.5) .. controls +(0,-1) and +(0,-1) .. +(1.5,0);
\draw (.5,-2.5) .. controls +(0,-.5) and +(0,-.5) .. +(.5,0);
\draw (2,-2.5) .. controls +(0,-.5) and +(0,-.5) .. +(.5,0);
\draw (3,-2.5) .. controls +(0,-.5) and +(0,-.5) .. +(.5,0);

\draw (5,-2.5) .. controls +(0,-.5) and +(0,-.5) .. +(.5,0);
\draw (6,-2.5) .. controls +(0,-.5) and +(0,-.5) .. +(.5,0);
\draw (7,-2.5) .. controls +(0,-1) and +(0,-1) .. +(1.5,0);
\draw (7.5,-2.5) .. controls +(0,-.5) and +(0,-.5) .. +(.5,0);

\draw[<->] (3,-4.5) -- +(-1,1);
\draw[<->] (5.5,-4.5) -- +(1,1);

\draw (2.5,-5) .. controls +(0,-1) and +(0,-1) .. +(1.5,0);
\draw (3,-5) .. controls +(0,-.5) and +(0,-.5) .. +(.5,0);
\draw (4.5,-5) .. controls +(0,-1) and +(0,-1) .. +(1.5,0);
\draw (5,-5) .. controls +(0,-.5) and +(0,-.5) .. +(.5,0);

\begin{scope}[xshift=10cm]
\draw (2.5,0) .. controls +(0,-.5) and +(0,-.5) .. +(.5,0);
\draw (3.5,0) .. controls +(0,-.5) and +(0,-.5) .. +(.5,0);
\draw (4.5,0) .. controls +(0,-.5) and +(0,-.5) .. +(.5,0);
\draw (5.5,0) .. controls +(0,-.5) and +(0,-.5) .. +(.5,0);

\draw[<->] (3,-1) -- +(-1,-1);
\draw[<->] (5.5,-1) -- +(1,-1);

\draw (0,-2.5) .. controls +(0,-1.25) and +(0,-1.25) .. +(3.5,0);
\draw (0.5,-2.5) .. controls +(0,-1) and +(0,-1) .. +(2.5,0);
\draw (1,-2.5) .. controls +(0,-.5) and +(0,-.5) .. +(.5,0);
\draw (2,-2.5) .. controls +(0,-.5) and +(0,-.5) .. +(.5,0);

\draw (5,-2.5) .. controls +(0,-.5) and +(0,-.5) .. +(.5,0);
\draw (6,-2.5) .. controls +(0,-1) and +(0,-1) .. +(1.5,0);
\draw (6.5,-2.5) .. controls +(0,-.5) and +(0,-.5) .. +(.5,0);
\draw (8,-2.5) .. controls +(0,-.5) and +(0,-.5) .. +(.5,0);

\draw[<->] (3,-4.5) -- +(-1,1);
\draw[<->] (5.5,-4.5) -- +(1,1);

\draw (2.5,-5) .. controls +(0,-1.5) and +(0,-1.5) .. +(3.5,0);
\draw (3,-5) .. controls +(0,-1.25) and +(0,-1.25) .. +(2.5,0);
\draw (3.5,-5) .. controls +(0,-1) and +(0,-1) .. +(1.5,0);
\draw (4,-5) .. controls +(0,-.5) and +(0,-.5) .. +(.5,0);
\end{scope}

\begin{scope}[xshift=20cm]
\draw (2.5,0) .. controls +(0,-.5) and +(0,-.5) .. +(.5,0);
\draw (3.5,0) .. controls +(0,-1) and +(0,-1) .. +(1.5,0);
\draw (4,0) .. controls +(0,-.5) and +(0,-.5) .. +(.5,0);
\draw (5.5,0) .. controls +(0,-.5) and +(0,-.5) .. +(.5,0);

\draw[<->] (3,-1) -- +(-1,-1);
\draw[<->] (5.5,-1) -- +(1,-1);

\draw (0,-2.5) .. controls +(0,-1) and +(0,-1) .. +(2.5,0);
\draw (0.5,-2.5) .. controls +(0,-.5) and +(0,-.5) .. +(.5,0);
\draw (1.5,-2.5) .. controls +(0,-.5) and +(0,-.5) .. +(.5,0);
\draw (3,-2.5) .. controls +(0,-.5) and +(0,-.5) .. +(.5,0);

\draw (6,-2.5) .. controls +(0,-1) and +(0,-1) .. +(2.5,0);
\draw (5,-2.5) .. controls +(0,-.5) and +(0,-.5) .. +(.5,0);
\draw (6.5,-2.5) .. controls +(0,-.5) and +(0,-.5) .. +(.5,0);
\draw (7.5,-2.5) .. controls +(0,-.5) and +(0,-.5) .. +(.5,0);

\draw[<->] (3,-4.5) -- +(-1,1);
\draw[<->] (5.5,-4.5) -- +(1,1);

\draw (2.5,-5) .. controls +(0,-1) and +(0,-1) .. +(3.5,0);
\draw (3,-5) .. controls +(0,-.5) and +(0,-.5) .. +(.5,0);
\draw (4,-5) .. controls +(0,-.5) and +(0,-.5) .. +(.5,0);
\draw (5,-5) .. controls +(0,-.5) and +(0,-.5) .. +(.5,0);
\end{scope}

\begin{scope}[yshift=-7cm]
\draw (2.5,0) .. controls +(0,-.5) and +(0,-.5) .. +(.5,0);
\draw (3.5,0) .. controls +(0,-.5) and +(0,-.5) .. +(.5,0);
\draw (4.5,0) .. controls +(0,-1) and +(0,-1) .. +(1.5,0);
\draw (5,0) .. controls +(0,-.5) and +(0,-.5) .. +(.5,0);

\draw[<->] (3,-1) -- +(-1,-1);
\draw[<->] (5.5,-1) -- +(1,-1);

\draw (0,-2.5) .. controls +(0,-1.25) and +(0,-1.25) .. +(3.5,0);
\draw (0.5,-2.5) .. controls +(0,-1) and +(0,-1) .. +(1.5,0);
\draw (1,-2.5) .. controls +(0,-.5) and +(0,-.5) .. +(.5,0);
\draw (2.5,-2.5) .. controls +(0,-.5) and +(0,-.5) .. +(.5,0);

\draw (5,-2.5) .. controls +(0,-.5) and +(0,-.5) .. +(.5,0);
\draw (6,-2.5) .. controls +(0,-1) and +(0,-1) .. +(2.5,0);
\draw (6.5,-2.5) .. controls +(0,-.5) and +(0,-.5) .. +(.5,0);
\draw (7.5,-2.5) .. controls +(0,-.5) and +(0,-.5) .. +(.5,0);

\draw[<->] (3,-4.5) -- +(-1,1);
\draw[<->] (5.5,-4.5) -- +(1,1);

\draw (2.5,-5) .. controls +(0,-1) and +(0,-1) .. +(3.5,0);
\draw (3,-5) .. controls +(0,-.5) and +(0,-.5) .. +(.5,0);
\draw (4,-5) .. controls +(0,-.5) and +(0,-.5) .. +(.5,0);
\draw (5,-5) .. controls +(0,-.5) and +(0,-.5) .. +(.5,0);
\end{scope}

\begin{scope}[xshift=10cm,yshift=-7cm]
\draw (4.5,0) .. controls +(0,-.5) and +(0,-.5) .. +(.5,0);
\draw (5.5,0) .. controls +(0,-.5) and +(0,-.5) .. +(.5,0);
\draw (2.5,0) .. controls +(0,-1) and +(0,-1) .. +(1.5,0);
\draw (3,0) .. controls +(0,-.5) and +(0,-.5) .. +(.5,0);

\draw[<->] (3,-1) -- +(-1,-1);
\draw[<->] (5.5,-1) -- +(1,-1);

\draw (0,-2.5) .. controls +(0,-1.25) and +(0,-1.25) .. +(3.5,0);
\draw (1.5,-2.5) .. controls +(0,-1) and +(0,-1) .. +(1.5,0);
\draw (2,-2.5) .. controls +(0,-.5) and +(0,-.5) .. +(.5,0);
\draw (0.5,-2.5) .. controls +(0,-.5) and +(0,-.5) .. +(.5,0);

\draw (5.5,-2.5) .. controls +(0,-.5) and +(0,-.5) .. +(.5,0);
\draw (5,-2.5) .. controls +(0,-1) and +(0,-1) .. +(2.5,0);
\draw (6.5,-2.5) .. controls +(0,-.5) and +(0,-.5) .. +(.5,0);
\draw (8,-2.5) .. controls +(0,-.5) and +(0,-.5) .. +(.5,0);

\draw[<->] (3,-4.5) -- +(-1,1);
\draw[<->] (5.5,-4.5) -- +(1,1);

\draw (2.5,-5) .. controls +(0,-1) and +(0,-1) .. +(3.5,0);
\draw (3,-5) .. controls +(0,-.5) and +(0,-.5) .. +(.5,0);
\draw (4,-5) .. controls +(0,-.5) and +(0,-.5) .. +(.5,0);
\draw (5,-5) .. controls +(0,-.5) and +(0,-.5) .. +(.5,0);
\end{scope}

\begin{scope}[xshift=20cm,yshift=-7cm]
\draw (2.5,0) .. controls +(0,-.5) and +(0,-.5) .. +(.5,0);
\draw (3.5,0) .. controls +(0,-1) and +(0,-1) .. +(2.5,0);
\draw (4,0) .. controls +(0,-.5) and +(0,-.5) .. +(.5,0);
\draw (5,0) .. controls +(0,-.5) and +(0,-.5) .. +(.5,0);

\draw[<->] (3,-1) -- +(-1,-1);
\draw[<->] (5.5,-1) -- +(1,-1);

\draw (0,-2.5) .. controls +(0,-1) and +(0,-1) .. +(3.5,0);
\draw (0.5,-2.5) .. controls +(0,-.5) and +(0,-.5) .. +(.5,0);
\draw (1.5,-2.5) .. controls +(0,-.5) and +(0,-.5) .. +(.5,0);
\draw (2.5,-2.5) .. controls +(0,-.5) and +(0,-.5) .. +(.5,0);

\draw (5,-2.5) .. controls +(0,-.5) and +(0,-.5) .. +(.5,0);
\draw (6,-2.5) .. controls +(0,-1.25) and +(0,-1.25) .. +(2.5,0);
\draw (6.5,-2.5) .. controls +(0,-1) and +(0,-1) .. +(1.5,0);
\draw (7,-2.5) .. controls +(0,-.5) and +(0,-.5) .. +(.5,0);

\draw[<->] (3,-4.5) -- +(-1,1);
\draw[<->] (5.5,-4.5) -- +(1,1);

\draw (2.5,-5) .. controls +(0,-1.25) and +(0,-1.25) .. +(3.5,0);
\draw (3,-5) .. controls +(0,-.5) and +(0,-.5) .. +(.5,0);
\draw (4,-5) .. controls +(0,-1) and +(0,-1) .. +(1.5,0);
\draw (4.5,-5) .. controls +(0,-.5) and +(0,-.5) .. +(.5,0);
\end{scope}

\begin{scope}[xshift=10cm,yshift=-14cm]
\draw (3,0) .. controls +(0,-.5) and +(0,-.5) .. +(.5,0);
\draw (2.5,0) .. controls +(0,-1) and +(0,-1) .. +(2.5,0);
\draw (4,0) .. controls +(0,-.5) and +(0,-.5) .. +(.5,0);
\draw (5.5,0) .. controls +(0,-.5) and +(0,-.5) .. +(.5,0);

\draw[<->] (3,-1) -- +(-1,-1);

\draw[<->] (5.5,-1) -- +(1,-1);

\draw (0,-2.5) .. controls +(0,-1) and +(0,-1) .. +(3.5,0);
\draw (0.5,-2.5) .. controls +(0,-.5) and +(0,-.5) .. +(.5,0);
\draw (1.5,-2.5) .. controls +(0,-.5) and +(0,-.5) .. +(.5,0);
\draw (2.5,-2.5) .. controls +(0,-.5) and +(0,-.5) .. +(.5,0);

\draw (8,-2.5) .. controls +(0,-.5) and +(0,-.5) .. +(.5,0);
\draw (5,-2.5) .. controls +(0,-1.25) and +(0,-1.25) .. +(2.5,0);
\draw (5.5,-2.5) .. controls +(0,-1) and +(0,-1) .. +(1.5,0);
\draw (6,-2.5) .. controls +(0,-.5) and +(0,-.5) .. +(.5,0);

\draw[<->] (3,-4.5) -- +(-1,1);
\draw[<->] (5.5,-4.5) -- +(1,1);

\draw (2.5,-5) .. controls +(0,-1.25) and +(0,-1.25) .. +(3.5,0);
\draw (3.5,-5) .. controls +(0,-.5) and +(0,-.5) .. +(.5,0);
\draw (3,-5) .. controls +(0,-1) and +(0,-1) .. +(1.5,0);
\draw (5,-5) .. controls +(0,-.5) and +(0,-.5) .. +(.5,0);
\end{scope}
\end{tikzpicture}
\end{eqnarray}

The possible diamonds with rays are obtained from these by allowing only vertices in a fixed smaller interval and forgetting cups connecting not allowed vertices only and turn cups which connect an allowed with a not allowed vertex into dotted or undotted rays depending if the vertex which gets removed is to the left or to the right of the allowed interval.
\end{prop}

\begin{proof}
One easily verifies that \eqref{diamonds1} and \eqref{diamonds2} give indeed diamonds and also when we restrict the vertices to an interval. To see that these are all let us first
consider the case where the relevant pieces are cups only, hence the arrow correspond to $\la$-pairs of the form
\begin{eqnarray}
\begin{tikzpicture}[thick,scale=0.7]
\draw (0,0) node[above]{} .. controls +(0,-.5) and +(0,-.5) .. +(.5,0) node[above]{};
\draw (1,0) node[above]{} .. controls +(0,-.5) and +(0,-.5) .. +(.5,0) node[above]{};
\draw[<->] (1.75,-.2) -- +(1,0);
\draw (3,0) node[above]{} .. controls +(0,-1) and +(0,-1) .. +(1.5,0) node[above]{};
\draw (3.5,0) node[above]{} .. controls +(0,-.5) and +(0,-.5) .. +(.5,0) node[above]{};
\end{tikzpicture}
\end{eqnarray}
with all possible configurations of dots. To simplify arguments we first ignore all dots and consider the possible configurations
\begin{eqnarray*}
\begin{tikzpicture}[thick,scale=0.7]
\node (A) at (0,-1.25) {$\gamma = \{ \gamma_1,\gamma_2\}$};
\node (B) at (4.5,-1.25) {$\delta = \{ \delta_1,\delta_2\}$};
\draw[<->] (2,-1) node[above]{\quad $\lambda$} -- +(-1,-1) node[below]{$\lambda'$};
\draw[<->] (2.5,-1) -- +(1,-1) node[below]{\quad $\lambda'''$};
\end{tikzpicture}
\end{eqnarray*}
where we number, from left to right according to their endpoints, the cups in $\underline{\lambda}$ involved in either of the two moves and indicate by the pairs $\gamma$ and $\delta$ the two pairs of vertices where the labels get changed. Then there are the following cases:
\begin{enumerate}
\item $\gamma=\delta$: This would imply $\lambda' = \lambda'''$ which is not allowed.
\item $|\gamma \cap \delta | = 1$: Then one can verify directly that the only possible configurations are the ones listed in the Table
~\ref{table_conf_1} below.
\item $\gamma \cap \delta  = \emptyset$: Then all possible configurations are listed in the Table~\ref{table_conf_2} below.
\end{enumerate}
The first claim follows then by adding all possible configurations of dots to the diamonds. The second claim follows from the definitions and an easy case by case argument which is left to the reader.
\end{proof}

\begin{table}[h]
\caption{Configurations for $|\gamma \cap \delta | = 1$}
\label{table_conf_1}
\begin{tabular}{|c|c|c|c|c|c|c|}
\hline $\lambda$ &
\begin{tikzpicture}[thick,scale=0.4]
\node (B) at (0,.1) {$\quad$};

\draw (0,0) .. controls +(0,-.5) and +(0,-.5) .. +(.5,0);
\draw (1,0) .. controls +(0,-.5) and +(0,-.5) .. +(.5,0);
\draw (2,0) .. controls +(0,-.5) and +(0,-.5) .. +(.5,0);

\node (A) at (0,-1.15) {$\quad$};
\end{tikzpicture}
&
\begin{tikzpicture}[thick,scale=0.4]
\node (B) at (0,.1) {$\quad$};

\draw (0,0) .. controls +(0,-.5) and +(0,-.5) .. +(.5,0);
\draw (1,0) .. controls +(0,-.5) and +(0,-.5) .. +(.5,0);
\draw (2,0) .. controls +(0,-.5) and +(0,-.5) .. +(.5,0);

\node (A) at (0,-1.15) {$\quad$};
\end{tikzpicture}
&
\begin{tikzpicture}[thick,scale=0.4]
\node (B) at (0,.1) {$\quad$};

\draw (0,0) .. controls +(0,-.5) and +(0,-.5) .. +(.5,0);
\draw (1,0) .. controls +(0,-.5) and +(0,-.5) .. +(.5,0);
\draw (2,0) .. controls +(0,-.5) and +(0,-.5) .. +(.5,0);

\node (A) at (0,-1.15) {$\quad$};
\end{tikzpicture}
&
\begin{tikzpicture}[thick,scale=0.4]
\node (B) at (0,.1) {$\quad$};

\draw (0,0) .. controls +(0,-1) and +(0,-1) .. +(2.5,0);
\draw (0.5,0) .. controls +(0,-.5) and +(0,-.5) .. +(.5,0);
\draw (1.5,0) .. controls +(0,-.5) and +(0,-.5) .. +(.5,0);

\node (A) at (0,-1.15) {$\quad$};
\end{tikzpicture}
&
\begin{tikzpicture}[thick,scale=0.4]
\node (B) at (0,.1) {$\quad$};

\draw (0,0) .. controls +(0,-1) and +(0,-1) .. +(2.5,0);
\draw (0.5,0) .. controls +(0,-.5) and +(0,-.5) .. +(.5,0);
\draw (1.5,0) .. controls +(0,-.5) and +(0,-.5) .. +(.5,0);

\node (A) at (0,-1.15) {$\quad$};
\end{tikzpicture}
&
\begin{tikzpicture}[thick,scale=0.4]
\node (B) at (0,.1) {$\quad$};

\draw (0,0) .. controls +(0,-1) and +(0,-1) .. +(2.5,0);
\draw (0.5,0) .. controls +(0,-.5) and +(0,-.5) .. +(.5,0);
\draw (1.5,0) .. controls +(0,-.5) and +(0,-.5) .. +(.5,0);

\node (A) at (0,-1.15) {$\quad$};
\end{tikzpicture}\\
\hline $\gamma$ & $\{1,2 \}$ & $\{1,3 \}$ & $\{1,2 \}$ & $\{1,2 \}$ & $\{1,3 \}$ & $\{1,2 \}$\\
\hline $\delta$ & $\{2,3 \}$ & $\{2,3 \}$ & $\{1,3 \}$ & $\{2,3 \}$ & $\{2,3 \}$ & $\{1,3 \}$\\
\hline $\lambda''$ &
\begin{tikzpicture}[thick,scale=0.4]
\node (B) at (0,.1) {$\quad$};

\draw (0,0) .. controls +(0,-1) and +(0,-1) .. +(2.5,0);
\draw (0.5,0) .. controls +(0,-.5) and +(0,-.5) .. +(.5,0);
\draw (1.5,0) .. controls +(0,-.5) and +(0,-.5) .. +(.5,0);

\node (A) at (0,-1.15) {$\quad$};
\end{tikzpicture}
&
\begin{tikzpicture}[thick,scale=0.4]
\node (B) at (0,.1) {$\quad$};

\draw (0,0) .. controls +(0,-1) and +(0,-1) .. +(2.5,0);
\draw (0.5,0) .. controls +(0,-.5) and +(0,-.5) .. +(.5,0);
\draw (1.5,0) .. controls +(0,-.5) and +(0,-.5) .. +(.5,0);

\node (A) at (0,-1.15) {$\quad$};
\end{tikzpicture}
&
\begin{tikzpicture}[thick,scale=0.4]
\node (B) at (0,.1) {$\quad$};

\draw (0,0) .. controls +(0,-1) and +(0,-1) .. +(2.5,0);
\draw (0.5,0) .. controls +(0,-.5) and +(0,-.5) .. +(.5,0);
\draw (1.5,0) .. controls +(0,-.5) and +(0,-.5) .. +(.5,0);

\node (A) at (0,-1.15) {$\quad$};
\end{tikzpicture}
&
\begin{tikzpicture}[thick,scale=0.4]
\node (B) at (0,.1) {$\quad$};

\draw (0,0) .. controls +(0,-.5) and +(0,-.5) .. +(.5,0);
\draw (1,0) .. controls +(0,-.5) and +(0,-.5) .. +(.5,0);
\draw (2,0) .. controls +(0,-.5) and +(0,-.5) .. +(.5,0);

\node (A) at (0,-1.15) {$\quad$};
\end{tikzpicture}
&
\begin{tikzpicture}[thick,scale=0.4]
\node (B) at (0,.1) {$\quad$};

\draw (0,0) .. controls +(0,-.5) and +(0,-.5) .. +(.5,0);
\draw (1,0) .. controls +(0,-.5) and +(0,-.5) .. +(.5,0);
\draw (2,0) .. controls +(0,-.5) and +(0,-.5) .. +(.5,0);

\node (A) at (0,-1.15) {$\quad$};
\end{tikzpicture}
&
\begin{tikzpicture}[thick,scale=0.4]
\node (B) at (0,.1) {$\quad$};

\draw (0,0) .. controls +(0,-.5) and +(0,-.5) .. +(.5,0);
\draw (1,0) .. controls +(0,-.5) and +(0,-.5) .. +(.5,0);
\draw (2,0) .. controls +(0,-.5) and +(0,-.5) .. +(.5,0);

\node (A) at (0,-1.15) {$\quad$};
\end{tikzpicture} \\
\hline
\end{tabular}

\begin{tabular}{|c|c|c|c|}
\hline $\lambda$ &
\begin{tikzpicture}[thick,scale=0.4]
\node (B) at (0,.1) {$\quad$};

\draw (1,0) .. controls +(0,-1) and +(0,-1) .. +(1.5,0);
\draw (0,0) .. controls +(0,-.5) and +(0,-.5) .. +(.5,0);
\draw (1.5,0) .. controls +(0,-.5) and +(0,-.5) .. +(.5,0);

\node (A) at (0,-1.15) {$\quad$};
\end{tikzpicture}
&
\begin{tikzpicture}[thick,scale=0.4]
\node (B) at (0,.1) {$\quad$};

\draw (0,0) .. controls +(0,-1.25) and +(0,-1.25) .. +(2.5,0);
\draw (0.5,0) .. controls +(0,-1) and +(0,-1) .. +(1.5,0);
\draw (1,0) .. controls +(0,-.5) and +(0,-.5) .. +(.5,0);

\node (A) at (0,-1.15) {$\quad$};
\end{tikzpicture}
&
\begin{tikzpicture}[thick,scale=0.4]
\node (B) at (0,.1) {$\quad$};

\draw (0,0) .. controls +(0,-1) and +(0,-1) .. +(1.5,0);
\draw (.5,0) .. controls +(0,-.5) and +(0,-.5) .. +(.5,0);
\draw (2,0) .. controls +(0,-.5) and +(0,-.5) .. +(.5,0);

\node (A) at (0,-1.15) {$\quad$};
\end{tikzpicture}
\\
\hline $\gamma$ & $\{1,2 \}$ & $\{1,2 \}$ & $\{1,2 \}$\\
\hline $\delta$ & $\{2,3 \}$ & $\{2,3 \}$ & $\{2,3 \}$\\
\hline $\lambda''$ &
\begin{tikzpicture}[thick,scale=0.4]
\node (B) at (0,.1) {$\quad$};

\draw (0,0) .. controls +(0,-1) and +(0,-1) .. +(1.5,0);
\draw (.5,0) .. controls +(0,-.5) and +(0,-.5) .. +(.5,0);
\draw (2,0) .. controls +(0,-.5) and +(0,-.5) .. +(.5,0);

\node (A) at (0,-1.15) {$\quad$};
\end{tikzpicture}
&
\begin{tikzpicture}[thick,scale=0.4]
\node (B) at (0,.1) {$\quad$};

\draw (0,0) .. controls +(0,-.5) and +(0,-.5) .. +(.5,0);
\draw (1,0) .. controls +(0,-.5) and +(0,-.5) .. +(.5,0);
\draw (2,0) .. controls +(0,-.5) and +(0,-.5) .. +(.5,0);

\node (A) at (0,-1.15) {$\quad$};
\end{tikzpicture}
&
\begin{tikzpicture}[thick,scale=0.4]
\node (B) at (0,.1) {$\quad$};

\draw (1,0) .. controls +(0,-1) and +(0,-1) .. +(1.5,0);
\draw (0,0) .. controls +(0,-.5) and +(0,-.5) .. +(.5,0);
\draw (1.5,0) .. controls +(0,-.5) and +(0,-.5) .. +(.5,0);

\node (A) at (0,-1.15) {$\quad$};
\end{tikzpicture}\\
\hline
\end{tabular}
\end{table}
\begin{table}[h]
\caption{Configurations for $\gamma \cap \delta  = \emptyset$}
\label{table_conf_2}
\begin{tabular}{|c|c|c|c|c|c|c|}
\hline $\lambda$ &
\begin{tikzpicture}[thick,scale=0.4]
\node (B) at (0,.1) {$\quad$};

\draw (0,0) .. controls +(0,-.5) and +(0,-.5) .. +(.5,0);
\draw (1,0) .. controls +(0,-.5) and +(0,-.5) .. +(.5,0);
\draw (2,0) .. controls +(0,-1) and +(0,-1) .. +(1.5,0);
\draw (2.5,0) .. controls +(0,-.5) and +(0,-.5) .. +(.5,0);

\node (A) at (0,-1.15) {$\quad$};
\end{tikzpicture}
&
\begin{tikzpicture}[thick,scale=0.4]
\draw (0,0) .. controls +(0,-.5) and +(0,-.5) .. +(.5,0);
\draw (1,0) .. controls +(0,-.5) and +(0,-.5) .. +(.5,0);
\draw (2,0) .. controls +(0,-.5) and +(0,-.5) .. +(.5,0);
\draw (3,0) .. controls +(0,-.5) and +(0,-.5) .. +(.5,0);

\node (A) at (0,-1.15) {$\quad$};
\end{tikzpicture}
&
\begin{tikzpicture}[thick,scale=0.4]
\draw (0,0) .. controls +(0,-.5) and +(0,-.5) .. +(.5,0);
\draw (1,0) .. controls +(0,-.5) and +(0,-.5) .. +(.5,0);
\draw (2,0) .. controls +(0,-.5) and +(0,-.5) .. +(.5,0);
\draw (3,0) .. controls +(0,-.5) and +(0,-.5) .. +(.5,0);

\node (A) at (0,-1.15) {$\quad$};
\end{tikzpicture}
&
\begin{tikzpicture}[thick,scale=0.4]
\draw (0,0) .. controls +(0,-1) and +(0,-1) .. +(1.5,0);
\draw (0.5,0) .. controls +(0,-.5) and +(0,-.5) .. +(.5,0);
\draw (2,0) .. controls +(0,-1) and +(0,-1) .. +(1.5,0);
\draw (2.5,0) .. controls +(0,-.5) and +(0,-.5) .. +(.5,0);

\node (A) at (0,-1.15) {$\quad$};
\end{tikzpicture}
&
\begin{tikzpicture}[thick,scale=0.4]
\draw (0,0) .. controls +(0,-1.25) and +(0,-1.25) .. +(3.5,0);
\draw (0.5,0) .. controls +(0,-1) and +(0,-1) .. +(1.5,0);
\draw (1,0) .. controls +(0,-.5) and +(0,-.5) .. +(.5,0);
\draw (2.5,0) .. controls +(0,-.5) and +(0,-.5) .. +(.5,0);

\node (A) at (0,-1.15) {$\quad$};
\end{tikzpicture}
&
\begin{tikzpicture}[thick,scale=0.4]
\draw (0,0) .. controls +(0,-1.5) and +(0,-1.5) .. +(3.5,0);
\draw (0.5,0) .. controls +(0,-1.25) and +(0,-1.25) .. +(2.5,0);
\draw (1,0) .. controls +(0,-1) and +(0,-1) .. +(1.5,0);
\draw (1.5,0) .. controls +(0,-.5) and +(0,-.5) .. +(.5,0);

\node (A) at (0,-1.15) {$\quad$};
\end{tikzpicture}

\\
\hline $\gamma$ & $\{1,2 \}$ & $\{1,2 \}$ & $\{1,4 \}$ & $\{1,2 \}$ & $\{2,3 \}$ & $\{1,2 \}$\\
\hline $\delta$ & $\{3,4 \}$ & $\{1,2 \}$ & $\{2,3 \}$ & $\{3,4 \}$ & $\{1,4 \}$ & $\{3,4 \}$\\
\hline $\lambda''$ &
\begin{tikzpicture}[thick,scale=0.4]
\draw (0,0) .. controls +(0,-1) and +(0,-1) .. +(1.5,0);
\draw (.5,0) .. controls +(0,-.5) and +(0,-.5) .. +(.5,0);
\draw (2,0) .. controls +(0,-.5) and +(0,-.5) .. +(.5,0);
\draw (3,0) .. controls +(0,-.5) and +(0,-.5) .. +(.5,0);

\node (A) at (0,-1.15) {$\quad$};
\end{tikzpicture}
&
\begin{tikzpicture}[thick,scale=0.4]
\draw (0,0) .. controls +(0,-1) and +(0,-1) .. +(1.5,0);
\draw (0.5,0) .. controls +(0,-.5) and +(0,-.5) .. +(.5,0);
\draw (2,0) .. controls +(0,-1) and +(0,-1) .. +(1.5,0);
\draw (2.5,0) .. controls +(0,-.5) and +(0,-.5) .. +(.5,0);

\node (A) at (0,-1.15) {$\quad$};
\end{tikzpicture}
&
\begin{tikzpicture}[thick,scale=0.4]
\draw (0,0) .. controls +(0,-1.5) and +(0,-1.5) .. +(3.5,0);
\draw (0.5,0) .. controls +(0,-1.25) and +(0,-1.25) .. +(2.5,0);
\draw (1,0) .. controls +(0,-1) and +(0,-1) .. +(1.5,0);
\draw (1.5,0) .. controls +(0,-.5) and +(0,-.5) .. +(.5,0);

\node (A) at (0,-1.15) {$\quad$};
\end{tikzpicture}
&
\begin{tikzpicture}[thick,scale=0.4]
\draw (0,0) .. controls +(0,-.5) and +(0,-.5) .. +(.5,0);
\draw (1,0) .. controls +(0,-.5) and +(0,-.5) .. +(.5,0);
\draw (2,0) .. controls +(0,-.5) and +(0,-.5) .. +(.5,0);
\draw (3,0) .. controls +(0,-.5) and +(0,-.5) .. +(.5,0);

\node (A) at (0,-1.15) {$\quad$};
\end{tikzpicture}
&
\begin{tikzpicture}[thick,scale=0.4]
\draw (0,0) .. controls +(0,-1) and +(0,-1) .. +(2.5,0);
\draw (0.5,0) .. controls +(0,-.5) and +(0,-.5) .. +(.5,0);
\draw (1.5,0) .. controls +(0,-.5) and +(0,-.5) .. +(.5,0);
\draw (3,0) .. controls +(0,-.5) and +(0,-.5) .. +(.5,0);

\node (A) at (0,-1.15) {$\quad$};
\end{tikzpicture}
&
\begin{tikzpicture}[thick,scale=0.4]
\node (B) at (0,.1) {$\quad$};

\draw (0,0) .. controls +(0,-.5) and +(0,-.5) .. +(.5,0);
\draw (1,0) .. controls +(0,-.5) and +(0,-.5) .. +(.5,0);
\draw (2,0) .. controls +(0,-.5) and +(0,-.5) .. +(.5,0);
\draw (3,0) .. controls +(0,-.5) and +(0,-.5) .. +(.5,0);

\node (A) at (0,-1.15) {$\quad$};
\end{tikzpicture}
\\
\hline
\end{tabular}
\begin{tabular}{|c|c|c|c|}
\hline
$\lambda$ &
\begin{tikzpicture}[thick,scale=0.4]
\node (B) at (0,.1) {$\quad$};

\draw (0,0) .. controls +(0,-.5) and +(0,-.5) .. +(.5,0);
\draw (1,0) .. controls +(0,-1) and +(0,-1) .. +(1.5,0);
\draw (1.5,0) .. controls +(0,-.5) and +(0,-.5) .. +(.5,0);
\draw (3,0) .. controls +(0,-.5) and +(0,-.5) .. +(.5,0);

\node (A) at (0,-1.15) {$\quad$};
\end{tikzpicture}
&
\begin{tikzpicture}[thick,scale=0.4]
\draw (0,0) .. controls +(0,-.5) and +(0,-.5) .. +(.5,0);
\draw (1,0) .. controls +(0,-1) and +(0,-1) .. +(2.5,0);
\draw (1.5,0) .. controls +(0,-.5) and +(0,-.5) .. +(.5,0);
\draw (2.5,0) .. controls +(0,-.5) and +(0,-.5) .. +(.5,0);

\node (A) at (0,-1.15) {$\quad$};
\end{tikzpicture}
&
\begin{tikzpicture}[thick,scale=0.4]
\draw (0,0) .. controls +(0,-1.25) and +(0,-1.25) .. +(3.5,0);
\draw (0.5,0) .. controls +(0,-1) and +(0,-1) .. +(2.5,0);
\draw (1,0) .. controls +(0,-.5) and +(0,-.5) .. +(.5,0);
\draw (2,0) .. controls +(0,-.5) and +(0,-.5) .. +(.5,0);

\node (A) at (0,-1.15) {$\quad$};
\end{tikzpicture}

\\
\hline $\alpha$ & $\{1,4 \}$ & $\{1,2 \}$ & $\{3,4 \}$ \\
\hline $\beta$ & $\{2,3 \}$ & $\{3,4 \}$ & $\{1,4 \}$ \\
\hline $\lambda''$ &
\begin{tikzpicture}[thick,scale=0.4]
\node (B) at (0,.1) {$\quad$};

\draw (0,0) .. controls +(0,-1.25) and +(0,-1.25) .. +(3.5,0);
\draw (0.5,0) .. controls +(0,-1) and +(0,-1) .. +(2.5,0);
\draw (1,0) .. controls +(0,-.5) and +(0,-.5) .. +(.5,0);
\draw (2,0) .. controls +(0,-.5) and +(0,-.5) .. +(.5,0);

\node (A) at (0,-1.15) {$\quad$};
\end{tikzpicture}
&
\begin{tikzpicture}[thick,scale=0.4]
\draw (0,0) .. controls +(0,-1.25) and +(0,-1.25) .. +(3.5,0);
\draw (1.5,0) .. controls +(0,-1) and +(0,-1) .. +(1.5,0);
\draw (2,0) .. controls +(0,-.5) and +(0,-.5) .. +(.5,0);
\draw (0.5,0) .. controls +(0,-.5) and +(0,-.5) .. +(.5,0);

\node (A) at (0,-1.15) {$\quad$};
\end{tikzpicture}
&
\begin{tikzpicture}[thick,scale=0.4]
\draw (0,0) .. controls +(0,-.5) and +(0,-.5) .. +(.5,0);
\draw (1,0) .. controls +(0,-1) and +(0,-1) .. +(1.5,0);
\draw (1.5,0) .. controls +(0,-.5) and +(0,-.5) .. +(.5,0);
\draw (3,0) .. controls +(0,-.5) and +(0,-.5) .. +(.5,0);

\node (A) at (0,-1.15) {$\quad$};
\end{tikzpicture}
\\
\hline
\end{tabular}
\end{table}

\subsection{The isomorphism}
To prove Theorem~\ref{thm:main} it is enough to establish an isomorphism of algebras $A(D_k)\cong\mathbb{D}_{\Lambda_k^{\ov{0}}}$. The claim follows then from Theorem~\ref{thmBraden} by the localization theorem and Riemann-Hilbert correspondence \cite{BuchHottaetal}.

\begin{definition}
To shorten the notations we define the following elements
\begin{enumerate}
\item ${}_{\lambda} \mathbbm{1}_\mu := \underline{\lambda}\lambda \overline{\mu}$ for $\la,\mu \in \Lambda_k^{\ov{0}}$ with
    $\mu\stackrel{(i,j)}{\rightarrow}\la$,

\item ${}_{\mu} \mathbbm{1}_{\lambda} := \underline{\mu}\lambda \overline{\la}$ for $\la,\mu \in \Lambda_k^{\ov{0}}$ with
    $\mu\stackrel{(i,j)}{\rightarrow}\la$,

\item ${}_{\lambda} \mathbbm{1}_{\lambda} := \underline{\lambda}\lambda \overline{\lambda}$ for $\la \in \Lambda_k^{\ov{0}}$,

\item For $\la \leftrightarrow \mu$ we let  ${}_{\lambda} X_{\mu}$ be the basis vector obtained from ${}_{\lambda} \mathbbm{1}_\mu$ by reversing the orientation of the circle containing the cup defining the relation $\la\leftrightarrow\mu$.

\item $X_{\alpha,\la}$ the basis vector obtained from ${}_{\lambda} \mathbbm{1}_{\lambda}$ by reversing the orientation of the circle
    containing the point $\alpha$. If there is no such circle we declare  $X_{\alpha,\la}=0.$
\end{enumerate}
\end{definition}

\begin{theorem}
\label{mainthm}
Let $k\geq 4$. The assignments
\begin{eqnarray}
e_\lambda & \mapsto & {}_{\lambda} \mathbbm{1}_{\lambda}\nonumber\\
\label{isops}
p(\lambda,\mu) & \mapsto & {}_{\lambda} \mathbbm{1}_{\mu} + \frac{1}{2} (-1)^j {}_{\lambda} X_{\mu} \text{ if } \lambda \stackrel{(\mp i,j)}\leftrightarrow \mu\\
t_{\alpha,\lambda} & \mapsto &
\begin{cases}
{}_{\lambda} \mathbbm{1}_{\lambda} + X_{\alpha,\lambda} & \text{ if } 1 \leq \alpha \leq k, \\
{}_{\lambda} \mathbbm{1}_{\lambda} - X_{-\alpha,\lambda} & \text{ if } -k \leq \alpha \leq -1, \\
{}_{\lambda} \mathbbm{1}_{\lambda} & \text{ otherwise.}\\
\end{cases}
\nonumber
\end{eqnarray}
define an isomorphism of algebras $\Phi:A(D_k)\rightarrow \mathbb{D}_k$ and induces a nonnegative grading on $A(D_k)$.
\end{theorem}

\begin{proof}
The hardest part is to show that the map is well-defined which will be done in the next subsection. Assuming this we have to show that the map $\Phi$ is an isomorphism. By Corollary~\ref{CorCartan} the two algebras have the same dimension, hence it is enough to prove surjectivity. By definition all idempotents ${}_{\lambda} \mathbbm{1}_{\lambda}$ are in the image. Taking the images of the $t_{\alpha,\la}$ and their squares we also have all $X_{\alpha,\lambda}$ and their products in the image. Note that in the span of the image $p$ of $p(\lambda,\mu)$ and $X_{\alpha,\lambda}p$ we have ${}_{\lambda} \mathbbm{1}_{\mu}$. Hence it is also contained in the image and so is ${}_{\lambda} {X}_{\mu}$. Therefore $\Phi$ is surjective.
\end{proof}

 To show that $\Phi$ defines an algebra homomorphism it remains to verify it is well-defined,i,e, it respect the relations of Braden's algebra.

\subsection{The map $\Phi$ is a well-defined algebra homomorphism}

\begin{prop}
\label{welldefined}
$\Phi$ is welldefined.
\end{prop}
\begin{proof}
Clearly the ${}_{\lambda} \mathbbm{1}_{\lambda}$'s form a set of orthogonal idempotents and $\sum_\lambda {}_{\lambda} \mathbbm{1}_{\lambda} = 1$ by Theorem~~\ref{algebra_structure}, hence $(\text{R}1)$ and $(\text{R}2)$ hold and $(\text{R}3)$ and $(\text{R}4)$ hold by definition.

({\text R}5): We have (with the appropriate sign choice)
\begin{eqnarray*}
\Phi(t_{\alpha,\lambda} t_{\beta,\mu}) &=& \left({}_{\lambda} \mathbbm{1}_{\lambda} \mp X_{\alpha,\lambda}\right)
\left({}_{\mu} \mathbbm{1}_{\mu}\mp X_{\beta,\mu}\right)\\
&=& \begin{cases}
\quad 0 & \text{if } \lambda \neq \mu\\
{}_{\lambda} \mathbbm{1}_{\lambda} \mp\alpha X_{\alpha,\lambda} \mp X_{\beta,\lambda} +
(-1)^{\alpha+\beta}X_{\alpha,\lambda}X_{\beta,\lambda} & \text{otherwise.\quad}
\end{cases}
\end{eqnarray*}
and similarly for $\Phi(t_{\beta,\lambda} t_{\alpha,\mu})$ with analogous signs. Hence it is enough to show
$$X_{\alpha,\lambda}X_{\beta,\lambda} = X_{\beta,\lambda}X_{\alpha,\lambda}.$$
But this is obviously true by Lemma~\ref{lem:endo_commutative}.

({\text R}6):
From the definitions we obtain
$$ \Phi(p(\lambda,\mu)t_{\alpha,\mu}) = {}_{\lambda} \mathbbm{1}_{\mu} \Phi(t_{\alpha,\mu}) + \frac{1}{2}(-1)^j{}_{\lambda} X_{\mu}
\Phi(t_{\alpha,\mu})$$
and
$$ \Phi(t_{\alpha,\lambda}p(\lambda,\mu)) = \Phi(t_{\alpha,\lambda}){}_{\lambda} \mathbbm{1}_{\mu} +
\frac{1}{2}(-1)^j\Phi(t_{\alpha,\lambda}) {}_{\lambda} X_{\mu}.$$
Since ${}_{\lambda} \mathbbm{1}_{\mu}$ and ${}_{\lambda} X_{\mu}$ are basis vectors whose underlying cup diagram looks like ${}_{\mu} \mathbbm{1}_{\mu}$
except for a local change given by the $\lambda$-pair, the circle diagram $\underline{\la}\ov{\mu}$ contains only vertical lines obtained by gluing two rays, small circles and one extra component $C$ containing the defining cup for $\la\stackrel{(\alpha,\beta)}\leftrightarrow\mu$. Assume first that $C$ is a circle. Then it involves 4 vertices lets say on positions $a_1<a_2<a_3<a_4$. We now compare multiplication of ${}_{\lambda} \mathbbm{1}_{\mu}$ and ${}_{\lambda}
    X_{\mu}$ by the image of $t_{\alpha,\lambda}$ from the left or $t_{\alpha,\mu}$ from the right. There are two cases:
\begin{enumerate}[(I)]
\item
$\alpha \not\in \{a_1,a_2,a_3,a_4\}$: then both multiplications change the
    orientation of one small anticlockwise circle and multiply with the same overall sign or they both annihilate the diagram.
\item
$\alpha \in \{a_1,a_2,a_3,a_4\}$: then both multiplications change the orientation of $C$ to clockwise with the same overall sign or both annihilate  the diagram in case $C$ was already clockwise.

If $C$ is not a circle, both multiplications annihilate the diagram.

({\text R}7): We have
$$ \left({}_{\lambda} \mathbbm{1}_{\lambda} + X_{\alpha,\lambda}\right) \left({}_{\lambda} \mathbbm{1}_{\lambda} - X_{\alpha,\lambda}\right) = {}_{\lambda} \mathbbm{1}_{\lambda} - X_{\alpha,\lambda}^2 = {}_{\lambda}
\mathbbm{1}_{\lambda}.$$
Hence the first equality follows. The second is clear by definition.  For the third let us first assume $1 \leq \alpha, \beta \leq k$. Then
$$\Phi(t_{\alpha,\lambda} t_{\beta,\lambda}) = \left({}_{\lambda} \mathbbm{1}_{\lambda} + X_{\alpha,\lambda}\right)\left({}_{\lambda} \mathbbm{1}_{\lambda} + X_{\beta,\lambda}\right)
$$
Since $\alpha$ and $\beta$ are connected by an undotted cup in $\underline\la\ov{\la}$ we have by Proposition~\ref{coho} $X_{\alpha,\lambda} =-X_{\beta,\lambda}$ and thus
$$\Phi(t_{\alpha,\lambda} t_{\beta,\lambda})={}_{\lambda} \mathbbm{1}_{\lambda}-X_{\alpha,\lambda}X_{\alpha,\lambda}={}_{\lambda} \mathbbm{1}_{\lambda}.$$

If instead $1 \leq -\alpha, \beta \leq k$. Then
$$\Phi(t_{-\alpha,\lambda} t_{\beta,\lambda}) = \left({}_{\lambda} \mathbbm{1}_{\lambda} - X_{\alpha,\lambda}\right)\left({}_{\lambda} \mathbbm{1}_{\lambda} + X_{\beta,\lambda}\right)
$$
Since $\alpha$ and $\beta$ are connected by a dotted cup in $\underline\la\ov{\lambda}$ we have by Proposition~\ref{coho} $X_{\alpha,\lambda} =X_{\beta,\lambda}$ and thus again
$$\Phi(t_{\alpha,\lambda} t_{\beta,\lambda})={}_{\lambda} \mathbbm{1}_{\lambda}-X_{\alpha,\lambda}X_{\alpha,\lambda}={}_{\lambda} \mathbbm{1}_{\lambda}.$$
and thus the third equality follows.
\end{enumerate}
({\text R}8):
Consider first the cases where the relevant changes from $\la$ and $\mu$ involve no rays. The possible ${}_{\lambda} \mathbbm{1}_{\mu}{}_{\lambda} X_{\mu}$ and  ${}_{\mu} \mathbbm{1}_{\lambda}$, ${}_{\mu} X_{\lambda}$ are displayed in Example~\ref{ps} when putting clockwise respectively anticlockwise orientation.
We go through all cases $\mu\leftarrow\la$ as displayed in Example~\ref{ps}, see \eqref{parent1}-\eqref{parent4}. Note that $\pos{i}=i$ holds since we are in the principal block. We start with the case
\begin{equation}
\label{parent1}
\begin{tikzpicture}[thick,scale=0.7]
\draw (0,0) node[above]{$$} .. controls +(0,-.5) and +(0,-.5) ..
+(.5,0) node[above]{};
\draw (1,0) node[above]{$$} .. controls +(0,-.5) and +(0,-.5) ..
+(.5,0) node[above]{};
\draw[<-] (1.75,-.2) -- +(1,0);
\draw (3,0) node[above]{$\alpha'$} .. controls +(0,-1) and +(0,-1) .. +(1.5,0)
node[above]{$\beta'$};
\draw (3.5,0) node[above]{$\alpha$} .. controls +(0,-.5) and +(0,-.5) ..
+(.5,0) node[above]{$\beta$};
\end{tikzpicture}
\end{equation}
Then the image of $\op{m}(\mu,\la)=1+p(\mu,\la)p(\la,\mu)$ equals
$$({}_{\mu} \mathbbm{1}_{\lambda}+\frac{1}{2}(-1)^\beta {}_{\mu} X_{\la})
({}_{\mu} \mathbbm{1}_{\mu}+\frac{1}{2}(-1)^\beta {}_{\la} X_{\mu})$$ which equals using our signed surgery rules the following expression in $\cM(\la,\la)$:
\begin{eqnarray*}
&&1+(-1)^\alpha(1+(-1)^\beta X_{\beta'})(X_{\beta}-X_\alpha)\\
&=&1+(-1)^\alpha(X_{\beta}-X_\alpha)-(-1)^{\alpha+\beta}X_\alpha X_{\beta'}\\
&=&1-(-1)^\alpha(X_{\beta'}+X_\alpha)+X_\alpha X_{\beta'}\\
&=&1+(-1)^\beta(X_{\beta'}+X_\alpha)+X_\alpha X_{\beta'}
\end{eqnarray*}
since $\alpha+\beta$ is odd.
On the other hand $\Phi\left((t_{\alpha,\lambda} t_{\beta',\lambda})^{({-1})^\beta}\right)$
corresponds to $(1+(-1)^\beta X_\alpha)(1+(-1)^\beta X_{\beta'})=1+(-1)^\beta(X_{\beta'}+X_\alpha)+X_\alpha X_{\beta'}.$ Hence the required relation for $\op{m}(\mu,\la)$ holds.\\
Similar calculations give $\Phi(e_\la+p(\la,\mu)p(\mu,\la))=\Phi((t_{\alpha,\la} t_{\beta',\la})^{({-1})^\beta})$ corresponding to
\begin{eqnarray*}
1+(-1)^\beta(X_{\beta'}+X_\alpha)+X_\alpha X_{\beta'}
&=&1+(-1)^\beta(X_{\beta'}-X_{\beta})-X_\beta X_{\beta'}
\end{eqnarray*}
and the claim follows in this case.
For the remaining cases we list just the corresponding polynomials in $\cM(\mu,\mu)$ and $\cM(\la,\la)$ respectively, since the calculations are totally analogous:

\begin{equation}
\label{parent2}
\begin{tikzpicture}[thick,scale=0.7]
\draw (0,0) node[above]{} .. controls +(0,-.5) and +(0,-.5) .. +(.5,0)
node[above]{};
\fill (.25,-.36) circle(2.5pt);
\draw (1,0) node[above]{$$} .. controls +(0,-.5) and +(0,-.5) ..
+(.5,0) node[above]{$$};
\draw[<-] (1.75,-.2) -- +(1,0);
\draw (3,0) node[above]{$\alpha'$} .. controls +(0,-1) and +(0,-1) .. +(1.5,0)
node[above]{$\beta'$};
\fill (3.75,-.74) circle(2.5pt);
\draw (3.5,0) node[above]{$\alpha$} .. controls +(0,-.5) and +(0,-.5) ..
+(.5,0) node[above]{$\beta$};
\node at (11,0)
{$\begin{cases}
\; 1+(-1)^\beta(X_{\beta'}-X_\beta)-X_{\beta'}X_\beta\in\cM(\mu,\mu)\\
\; 1+(-1)^\beta(X_\alpha+X_{\beta'}) +X_\alpha X_{\beta'}\in\cM(\la,\la)
\end{cases}$};
\end{tikzpicture}
\end{equation}

\begin{equation}
\label{parent3}
\begin{tikzpicture}[thick,scale=0.7]
\draw (0,0) node[above]{$$} .. controls +(0,-1) and +(0,-1) .. +(1.5,0)
node[above]{};
\draw (.5,0) node[above]{$$} .. controls +(0,-.5) and +(0,-.5) ..
+(.5,0) node[above]{};
\fill (0.75,-.74) circle(2.5pt);
\draw[<-] (1.75,-.2) -- +(1,0);
\draw (3,0) node[above]{$\alpha$} .. controls +(0,-.5) and +(0,-.5) ..
+(.5,0) node[above]{$\beta$};
\draw (4,0) node[above]{$\alpha'$} .. controls +(0,-.5) and +(0,-.5) .. +(.5,0)
node[above]{$\beta'$};
\fill (4.25,-.36) circle(2.5pt);
\node at (11,0)
{$\begin{cases}
\;1+(-1)^\beta(X_{\beta'}+X_{\alpha'})+X_{\beta'}X_{\alpha'}\in\cM(\mu,\mu)\\
\;1+(-1)^\beta (X_{\beta'}-X_\beta)-X_\beta X_{\beta'}\in\cM(\la,\la)
\end{cases}$};
\end{tikzpicture}
\end{equation}
\begin{equation}
\label{parent4}
\begin{tikzpicture}[thick,scale=0.7]
\draw (0,0) node[above]{$$} .. controls +(0,-1) and +(0,-1) .. +(1.5,0)
node[above]{};
\draw (.5,0) node[above]{$$} .. controls +(0,-.5) and +(0,-.5) ..
+(.5,0) node[above]{};
\draw[<-] (1.75,-.2) -- +(1,0);
\draw (3,0) node[above]{$\alpha$} .. controls +(0,-.5) and +(0,-.5) .. +(.5,0)
node[above]{$\beta$};
\fill (3.25,-.36) circle(2.5pt);
\draw (4,0) node[above]{$\alpha'$} .. controls +(0,-.5) and +(0,-.5) .. +(.5,0)
node[above]{$\beta'$};
\fill (4.25,-.36) circle(2.5pt);
\node at (11,0)
{$\begin{cases}
1+(-1)^\beta(X_{\alpha}+X_{\beta'})+X_{\alpha'}X_{\beta'}\in\cM(\mu,\mu)\\
1+(-1)^\beta(X_{\alpha'}-X_\beta)-X_{\beta}X_{\beta'}\in\cM(\la,\la)
\end{cases}$};
\end{tikzpicture}
\end{equation}

The relations for $\la$-pairs not involving cups only as the relevant pieces are totally analogous, except that some of the generators of type $t_\gamma$ act by an idempotent and the corresponding variables $x$ should be set to zero.

({\text R}8): These are Lemmas~\ref{stupidcalcc}-\ref{stupidcalca} below.

Hence the assignments \eqref{isops} define an algebra homomorphism.
\end{proof}

\begin{lemma}
\label{stupidcalcc}
Let $k\geq 4$ and assume that $(\la_1,\la_2,\la_3)$ is a triple as in Definition~\ref{Bradenalgebra} \eqref{rel7} c). With $\Phi$ as in \eqref{isops} we have
  \begin{eqnarray}
  \label{zerooo}
    \Phi(p(\la_3,\la_2))\Phi(p(\la_2,\la_1))&=0=&\Phi(p(\la_1,\la_2))\Phi(p(\la_2,\la_3)).
  \end{eqnarray}
\end{lemma}

\begin{proof}
Let us first assume that the corresponding cup diagrams $\underline{\la_i}$ have no rays. The restriction $\alpha<0$ forces $\la_2\leftarrow \la_1$ to be locally of the form \eqref{parent4}. Let $C_1, C_2$ be the two cups (the outer and inner resp.) in the left picture of \eqref{parent4}. Since $(\gamma,\delta)$ is a $\la_2$-pair, but not a $\la_1$-pair, it must be given by one of the cups $C_1$ or $C_2$. The triple $(\la_1,\la_2,\la_3)$ look then locally like $(\la_1,\la_2,\la_3)$ resp. $(\mu_1,\mu_2,\mu_3)$ in Example~\ref{triples}. Hence in the case of $C_1$ the compositions \eqref{zerooo} are obviously zero, whereas in case of $C_2$ the triple is not of the required form. Obviously, (by adding $\up$'s to the left and $\down$'s to the right of the diagram), the general case can be deduced from the case where no rays occur.
\end{proof}

\begin{lemma}
\label{stupidcalb}
Let $k\geq 4$ and assume that $(\la_1,\la_2,\la_3)$ is a triple as in Definition~\ref{Bradenalgebra} \eqref{rel7} b). With $\Phi$ as in \eqref{isops} we have
  \begin{eqnarray}
    \Phi(p(\la_3,\la_2))\Phi(p(\la_2,\la_1))&=0=&\Phi(p(\la_1,\la_2))\Phi(p(\la_2,\la_3)).
  \end{eqnarray}
\end{lemma}

\begin{proof}
By assumption the triple can not be extended to a diamond, but can be enlarged to a diamond $(\tilde{\la}_1,\tilde{\la}_2,\tilde{\la}_3,\tilde{\la}_4)$. In particular, $\tilde{\la}_4$ must contains at least one $\down$ at position larger then $k$ or at least one $\up$ at position smaller than $1$ and only one of these two cases can occur.
Assume we have such an $\up$, then there is only one, since assuming there are two or more we have to find
\begin{eqnarray}
\label{leftright}
\tilde{\la}_1\leftrightarrow \tilde{\la}_4\leftrightarrow \tilde{\la}_3
\end{eqnarray}
with both $\leftrightarrow$ converting these $\up$'s into $\down$'s. But at most two $\up$'s can we switched and there is only one way to do so, which implies $\tilde{\la}_1=\tilde{\la}_3$, a contradiction.
Hence we need
two $\leftrightarrow$'s as in \eqref{leftright} which either remove our special (hence leftmost) $\up$ or at least move it to the right. Wlog. we may assume there are no rays. If we take one of the form \eqref{parent1}, then there is no second possibility. Hence they must be of the form \eqref{parent3} or \eqref{parent4} respectively. Then locally the enlarged diamond has to be of the form displayed on the left
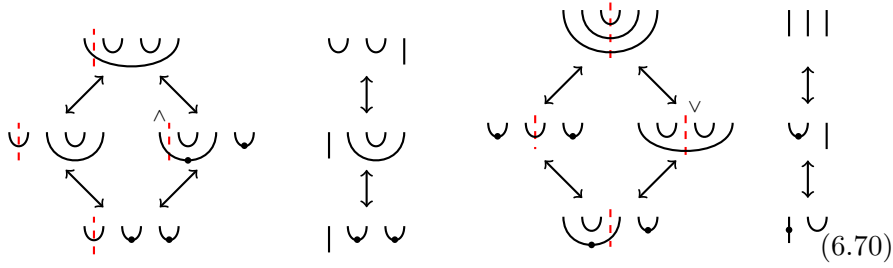
\begin{eqnarray}
\label{nodiamonds}
\begin{tikzpicture}[thick,,scale=.5]
\draw[<->] (2.5,-1) -- +(-1,-1);
\draw[<->] (4,-1) -- +(1,-1);
\draw[<->] (2.5,-4.5) -- +(-1,1);
\draw[<->] (4,-4.5) -- +(1,1);

\begin{scope}[shift={(2,0)}]
\draw [dashed,red] (.25,-.75) -- +(0,1);
\draw (0,0) .. controls +(0,-1) and +(0,-1) .. +(2.5,0);
\draw (.5,0) .. controls +(0,-.5) and +(0,-.5) .. +(.5,0);
\draw (1.5,0) .. controls +(0,-.5) and +(0,-.5) .. +(.5,0);
\end{scope}

\begin{scope}[shift={(0,-2.5)}]
\draw [dashed,red] (.25,-.75) -- +(0,1);
\draw (1,0) .. controls +(0,-1) and +(0,-1) .. +(1.5,0);
\draw (0,0) .. controls +(0,-.5) and +(0,-.5) .. +(.5,0);
\draw (1.5,0) .. controls +(0,-.5) and +(0,-.5) .. +(.5,0);
\end{scope}

\begin{scope}[shift={(2,-5)}]
\draw [dashed,red] (.25,-.75) -- +(0,1);
\draw (0,0) .. controls +(0,-.5) and +(0,-.5) .. +(.5,0);
\draw (1,0) .. controls +(0,-.5) and +(0,-.5) .. +(.5,0);
\draw (2,0) .. controls +(0,-.5) and +(0,-.5) .. +(.5,0);
\fill (1.25,-.36) circle(2.5pt);
\fill (2.25,-.36) circle(2.5pt);
\end{scope}

\begin{scope}[shift={(4,-2.5)}]
\draw [dashed,red] (.25,-.75) -- +(0,1);
\draw (0,0) node[above]{\tiny $\wedge$} .. controls +(0,-1) and +(0,-1) .. +(1.5,0);
\draw (.5,0) .. controls +(0,-.5) and +(0,-.5) .. +(.5,0);
\draw (2,0) .. controls +(0,-.5) and +(0,-.5) .. +(.5,0);
\fill (.75,-.75) circle(2.5pt);
\fill (2.25,-.36) circle(2.5pt);
\end{scope}

\begin{scope}[xshift=8cm]
\draw[<->] (1.5,-1) -- +(0,-1);
\draw[<->] (1.5,-4.5) -- +(0,1);

\begin{scope}[shift={(0,0)}]
\draw (.5,0) .. controls +(0,-.5) and +(0,-.5) .. +(.5,0);
\draw (1.5,0) .. controls +(0,-.5) and +(0,-.5) .. +(.5,0);
\draw (2.5,0) -- +(0,-.7);
\end{scope}

\begin{scope}[shift={(0,-2.5)}]
\draw (.5,0) -- +(0,-.7);
\draw (1,0) .. controls +(0,-1) and +(0,-1) .. +(1.5,0);
\draw (1.5,0) .. controls +(0,-.5) and +(0,-.5) .. +(.5,0);
\end{scope}

\begin{scope}[shift={(0,-5)}]
\draw (.5,0) -- +(0,-.7);
\draw (1,0) .. controls +(0,-.5) and +(0,-.5) .. +(.5,0);
\draw (2,0) .. controls +(0,-.5) and +(0,-.5) .. +(.5,0);
\fill (1.25,-.36) circle(2.5pt);
\fill (2.25,-.36) circle(2.5pt);
\end{scope}
\end{scope}
\end{tikzpicture}
&\quad&
\begin{tikzpicture}[thick,,scale=.5]
\draw[<->] (2.5,-1) -- +(-1,-1);
\draw[<->] (4,-1) -- +(1,-1);
\draw[<->] (2.5,-4.5) -- +(-1,1);
\draw[<->] (4,-4.5) -- +(1,1);

\begin{scope}[shift={(2,.5)}]
\draw [dashed,red] (1.25,.2) -- +(0,-1.5);
\draw (0,0) .. controls +(0,-1.5) and +(0,-1.5) .. +(2.5,0);
\draw (.5,0) .. controls +(0,-1) and +(0,-1) .. +(1.5,0);
\draw (1,0) .. controls +(0,-.5) and +(0,-.5) .. +(.5,0);
\end{scope}

\begin{scope}[shift={(0,-2.5)}]
\draw [dashed,red] (1.25,.2) -- +(0,-.9);
\draw (0,0) .. controls +(0,-.5) and +(0,-.5) .. +(.5,0);
\draw (1,0) .. controls +(0,-.5) and +(0,-.5) .. +(.5,0);
\draw (2,0) .. controls +(0,-.5) and +(0,-.5) .. +(.5,0);
\fill (.25,-.36) circle(2.5pt);
\fill (2.25,-.36) circle(2.5pt);
\end{scope}

\begin{scope}[shift={(4,-2.5)}]
\draw [dashed,red] (1.25,.2) -- +(0,-1.1);
\draw (0,0) .. controls +(0,-1) and +(0,-1) .. +(2.5,0);
\draw (.5,0) .. controls +(0,-.5) and +(0,-.5) .. +(.5,0);
\draw (1.5,0) node[above]{\tiny $\vee$} .. controls +(0,-.5) and +(0,-.5) .. +(.5,0);
\end{scope}

\begin{scope}[shift={(2,-5)}]
\draw [dashed,red] (1.25,.2) -- +(0,-1);
\draw (0,0) .. controls +(0,-1) and +(0,-1) .. +(1.5,0);
\draw (.5,0) .. controls +(0,-.5) and +(0,-.5) .. +(.5,0);
\draw (2,0) .. controls +(0,-.5) and +(0,-.5) .. +(.5,0);
\fill (.75,-.75) circle(2.5pt);
\fill (2.25,-.36) circle(2.5pt);
\end{scope}

\begin{scope}[xshift=8cm]
\draw[<->] (.5,-1) -- +(0,-1);
\draw[<->] (.5,-4.5) -- +(0,1);

\begin{scope}[shift={(0,.5)}]
\draw (0,0) -- +(0,-.7);
\draw (.5,0) -- +(0,-.7);
\draw (1,0) -- +(0,-.7);
\end{scope}

\begin{scope}[shift={(0,-2.5)}]
\draw (0,0) .. controls +(0,-.5) and +(0,-.5) .. +(.5,0);
\draw (1,0) -- +(0,-.7);
\fill (.25,-.36) circle(2.5pt);
\end{scope}

\begin{scope}[shift={(0,-5)}]
\draw (0,0) -- +(0,-.7);
\draw (.5,0) .. controls +(0,-.5) and +(0,-.5) .. +(.5,0);
\fill (0,-.35) circle(2.5pt);
\end{scope}
\end{scope}

\end{tikzpicture}
\end{eqnarray}
with the triple $(\underline{\tilde{\la}_1},\underline{\tilde{\la}_2},
\underline{\tilde{\la}_3})$ displayed next to it (the additional points are to the left of the dashed line). In particular, the circle diagrams $\underline{\tilde{\la}_1}\overline{\tilde{\la}_3}$ and $\underline{\tilde{\la}_3}\overline{\tilde{\la}_1}$ cannot be oriented, hence the claim of the Lemma follows in this case.

Now assume we have a special $\down$ (automatically the rightmost one) which either has to be removed or moved to the left. Then the possible moves (with the position of our special $\down$ indicated) are of the form
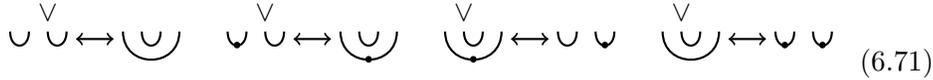
\begin{equation}
\label{downout}
\begin{tikzpicture}[thick,scale=0.5]
\draw (0,0) node[above]{$$} .. controls +(0,-.5) and +(0,-.5) ..
+(.5,0) node[above]{};
\draw (1,0) node[above]{$\down$} .. controls +(0,-.5) and +(0,-.5) ..
+(.5,0) node[above]{};
\draw[<->] (1.75,-.2) -- +(1,0);
\draw (3,0) node[above]{} .. controls +(0,-1) and +(0,-1) .. +(1.5,0)
node[above]{};
\draw (3.5,0) node[above]{} .. controls +(0,-.5) and +(0,-.5) ..
+(.5,0) node[above]{};
\end{tikzpicture}
\quad
\begin{tikzpicture}[thick,scale=0.5]
\draw (0,0) node[above]{} .. controls +(0,-.5) and +(0,-.5) .. +(.5,0)
node[above]{};
\fill (.25,-.36) circle(2.5pt);
\draw (1,0) node[above]{$\down$} .. controls +(0,-.5) and +(0,-.5) ..
+(.5,0) node[above]{};
\draw[<->] (1.75,-.2) -- +(1,0);
\draw (3,0) node[above]{$$} .. controls +(0,-1) and +(0,-1) .. +(1.5,0)
node[above]{};
\fill (3.75,-.74) circle(2.5pt);
\draw (3.5,0) node[above]{} .. controls +(0,-.5) and +(0,-.5) ..
+(.5,0) node[above]{};
\end{tikzpicture}
\quad
\begin{tikzpicture}[thick,scale=0.5]
\draw (0,0) node[above]{$$} .. controls +(0,-1) and +(0,-1) .. +(1.5,0)
node[above]{};
\draw (.5,0) node[above]{$\down$} .. controls +(0,-.5) and +(0,-.5) ..
+(.5,0) node[above]{};
\fill (0.75,-.74) circle(2.5pt);
\draw[<->] (1.75,-.2) -- +(1,0);
\draw (3,0) node[above]{} .. controls +(0,-.5) and +(0,-.5) ..
+(.5,0) node[above]{};
\draw (4,0) node[above]{} .. controls +(0,-.5) and +(0,-.5) .. +(.5,0)
node[above]{};
\fill (4.25,-.36) circle(2.5pt);
\end{tikzpicture}
\quad
\begin{tikzpicture}[thick,scale=0.5]
\draw (0,0) node[above]{} .. controls +(0,-1) and +(0,-1) .. +(1.5,0)
node[above]{};
\draw (.5,0) node[above]{$\down$} .. controls +(0,-.5) and +(0,-.5) ..
+(.5,0) node[above]{};
\draw[<->] (1.75,-.2) -- +(1,0);
\draw (3,0) node[above]{} .. controls +(0,-.5) and +(0,-.5) .. +(.5,0)
node[above]{};
\fill (3.25,-.36) circle(2.5pt);
\draw (4,0) node[above]{} .. controls +(0,-.5) and +(0,-.5) .. +(.5,0)
node[above]{};
\fill (4.25,-.36) circle(2.5pt);
\end{tikzpicture}
\end{equation}
One can easily verify that the possible extended diamonds must involve the first move. (For instance the third move in \eqref{downout} is only possible if the special cup is nested inside a dotted cup, but then the second move is impossible).
We are left with the following diamonds and the second diamond in \eqref{nodiamonds} (the additional points are now to the right of the dashed line).
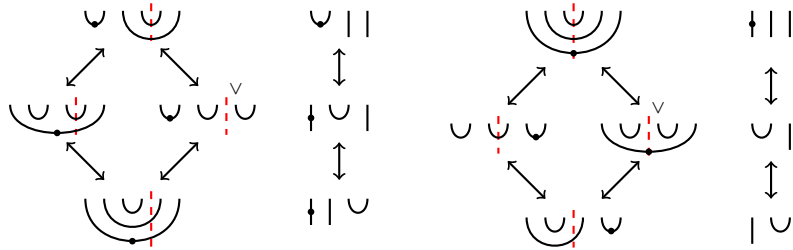
\begin{eqnarray*}
\begin{tikzpicture}[thick,,scale=.5]
\draw[<->] (2.5,-1) -- +(-1,-1);
\draw[<->] (4,-1) -- +(1,-1);
\draw[<->] (2.5,-4.5) -- +(-1,1);
\draw[<->] (4,-4.5) -- +(1,1);

\begin{scope}[shift={(2,0)}]
\draw [dashed,red] (1.75,.2) -- +(0,-1);
\draw (1,0) .. controls +(0,-1) and +(0,-1) .. +(1.5,0);
\draw (0,0) .. controls +(0,-.5) and +(0,-.5) .. +(.5,0);
\draw (1.5,0) .. controls +(0,-.5) and +(0,-.5) .. +(.5,0);
\fill (.25,-.36) circle(2.5pt);
\end{scope}

\begin{scope}[shift={(0,-2.5)}]
\draw [dashed,red] (1.75,.2) -- +(0,-1);
\draw (0,0) .. controls +(0,-1) and +(0,-1) .. +(2.5,0);
\draw (.5,0) .. controls +(0,-.5) and +(0,-.5) .. +(.5,0);
\draw (1.5,0) .. controls +(0,-.5) and +(0,-.5) .. +(.5,0);
\fill (1.25,-.75) circle(2.5pt);
\end{scope}

\begin{scope}[shift={(4,-2.5)}]
\draw [dashed,red] (1.75,.2) -- +(0,-1);
\draw (0,0) .. controls +(0,-.5) and +(0,-.5) .. +(.5,0);
\draw (1,0) .. controls +(0,-.5) and +(0,-.5) .. +(.5,0);
\draw (2,0) node[above]{\tiny $\vee$} .. controls +(0,-.5) and +(0,-.5) .. +(.5,0);
\fill (.25,-.36) circle(2.5pt);
\end{scope}

\begin{scope}[shift={(2,-5)}]
\draw [dashed,red] (1.75,.2) -- +(0,-1.5);
\draw (0,0) .. controls +(0,-1.5) and +(0,-1.5) .. +(2.5,0);
\draw (.5,0) .. controls +(0,-1) and +(0,-1) .. +(1.5,0);
\draw (1,0) .. controls +(0,-.5) and +(0,-.5) .. +(.5,0);
\fill (1.25,-1.13) circle(2.5pt);
\end{scope}

\begin{scope}[xshift=8cm]
\draw[<->] (.75,-1) -- +(0,-1);
\draw[<->] (.75,-4.5) -- +(0,1);

\begin{scope}[shift={(0,0)}]
\draw (0,0) .. controls +(0,-.5) and +(0,-.5) .. +(.5,0);
\draw (1,0) -- +(0,-.7);
\draw (1.5,0) -- +(0,-.7);
\fill (.25,-.36) circle(2.5pt);
\end{scope}

\begin{scope}[shift={(0,-2.5)}]
\draw (0,0) -- +(0,-.7);
\draw (.5,0) .. controls +(0,-.5) and +(0,-.5) .. +(.5,0);
\draw (1.5,0) -- +(0,-.7);
\fill (0,-.35) circle(2.5pt);
\end{scope}

\begin{scope}[shift={(0,-5)}]
\draw (0,0) -- +(0,-.7);
\draw (.5,0) -- +(0,-.7);
\draw (1,0) .. controls +(0,-.5) and +(0,-.5) .. +(.5,0);
\fill (0,-.35) circle(2.5pt);
\end{scope}
\end{scope}
\end{tikzpicture}
&\quad&
\begin{tikzpicture}[thick,,scale=.5]
\draw[<->] (2.5,-1) -- +(-1,-1);
\draw[<->] (4,-1) -- +(1,-1);
\draw[<->] (2.5,-4.5) -- +(-1,1);
\draw[<->] (4,-4.5) -- +(1,1);

\begin{scope}[shift={(2,.5)}]
\draw [dashed,red] (1.25,.2) -- +(0,-1.5);
\draw (0,0) .. controls +(0,-1.5) and +(0,-1.5) .. +(2.5,0);
\draw (.5,0) .. controls +(0,-1) and +(0,-1) .. +(1.5,0);
\draw (1,0) .. controls +(0,-.5) and +(0,-.5) .. +(.5,0);
\fill (1.25,-1.13) circle(2.5pt);
\end{scope}

\begin{scope}[shift={(0,-2.5)}]
\draw [dashed,red] (1.25,.2) -- +(0,-1);
\draw (0,0) .. controls +(0,-.5) and +(0,-.5) .. +(.5,0);
\draw (1,0) .. controls +(0,-.5) and +(0,-.5) .. +(.5,0);
\draw (2,0) .. controls +(0,-.5) and +(0,-.5) .. +(.5,0);
\fill (2.25,-.36) circle(2.5pt);
\end{scope}

\begin{scope}[shift={(4,-2.5)}]
\draw [dashed,red] (1.25,.2) -- +(0,-1.1);
\draw (0,0) .. controls +(0,-1) and +(0,-1) .. +(2.5,0);
\draw (.5,0) .. controls +(0,-.5) and +(0,-.5) .. +(.5,0);
\draw (1.5,0) node[above]{\tiny $\vee$} .. controls +(0,-.5) and +(0,-.5) .. +(.5,0);
\fill (1.25,-.75) circle(2.5pt);
\end{scope}

\begin{scope}[shift={(2,-5)}]
\draw [dashed,red] (1.25,.2) -- +(0,-1);
\draw (0,0) .. controls +(0,-1) and +(0,-1) .. +(1.5,0);
\draw (.5,0) .. controls +(0,-.5) and +(0,-.5) .. +(.5,0);
\draw (2,0) .. controls +(0,-.5) and +(0,-.5) .. +(.5,0);
\fill (2.25,-.36) circle(2.5pt);
\end{scope}

\begin{scope}[xshift=8cm]
\draw[<->] (.5,-1) -- +(0,-1);
\draw[<->] (.5,-4.5) -- +(0,1);

\begin{scope}[shift={(0,.5)}]
\draw (0,0) -- +(0,-.7);
\draw (.5,0) -- +(0,-.7);
\draw (1,0) -- +(0,-.7);
\fill (0,-.35) circle(2.5pt);
\end{scope}

\begin{scope}[shift={(0,-2.5)}]
\draw (0,0) .. controls +(0,-.5) and +(0,-.5) .. +(.5,0);
\draw (1,0) -- +(0,-.7);
\end{scope}

\begin{scope}[shift={(0,-5)}]
\draw (0,0) -- +(0,-.7);
\draw (.5,0) .. controls +(0,-.5) and +(0,-.5) .. +(.5,0);
\end{scope}
\end{scope}

\end{tikzpicture}
\end{eqnarray*}
In all cases the circle diagrams $\underline{\tilde{\la}_1}\overline{\tilde{\la}_3}$ and $\underline{\tilde{\la}_3}\overline{\tilde{\la}_1}$ cannot be oriented. Since the compositions ar zero if no rays are involved they are obviously also zero in the general case. The lemma follows.
\end{proof}

\begin{lemma}
\label{stupidcalca}
Let $k\geq 4$ and assume that $(\la_1,\la_2,\la_3,\la_4)$ is a diamond in $\Lambda_{k}^{\ov{0}}$. With $\Phi$ as in \eqref{isops} we have
  \begin{eqnarray*}
    \Phi(p(\la_3,\la_2))\Phi(p(\la_2,\la_1))&=&\Phi(p(\la_3,\la_4))\Phi(p(\la_4,\la_1)).
  \end{eqnarray*}
\end{lemma}

\begin{remark}
{\rm
When working over $\mathbb{F}_2$, we can ignore the dots and refer to the type $A$ treated in \cite{Str09}.
}
\end{remark}

\begin{proof}
 Given a diamond from \eqref{diamonds2} or \eqref{diamonds1}, say
\begin{eqnarray}
\label{diamondwrong}
\begin{tikzpicture}[thick,scale=0.7]
\draw[<->] (1,-1) node[above]{\quad $\lambda$} -- node[pos=0.5, left]{$\;(\alpha_1,\beta_1)$\;} +(-1,-1);
\draw[<->] (1.5,-1) --
node[pos=0.5, right]{$\;(\alpha_2,\beta_2)$} +(1,-1);
\draw[<->] (0,-2.7) node[above]{$\lambda'$} --
node[pos=0.5,left]{$(\alpha_3,\beta_3)$\;} +(1,-1)
node[below]{\quad $\lambda''$};
\draw[<->] (2.5,-2.7) node[above]{$\lambda'''$} --
node[pos=0.5,right]{\;$(\alpha_4,\beta_4)$} +(-1,-1);
\end{tikzpicture}
\end{eqnarray}
we consider the four cup diagrams  $\underline{\la},\underline{\la'},\underline{\la'''},\underline{\la''}$ and number the involved  points from left to right by $1$ to $6$ respectively $1$ to $8$.
Number in each of them the cups from left to right according to their left endpoint. Encode the dots via the four sets $D_1,D_2,D_3,D_4$ of dotted cups for the top, middle left, middle right and bottom cup diagram. For a $\la$-pair $(\alpha,\beta)$ we denote by $m$ the number of the rightmost vertex of the component which contains the cup-cap pair involved in the surgery. For instance, in the notation of \eqref{lapair}, the first diagram in \eqref{diamonds2} without dots has the following $\la$-pairs and maximal vertices $(\alpha_i,\beta_i,m_i)$ for $i=1,\ldots, 4$:
$$(2,3,4),(2,5,6),(4,5,6),(3,4,5).$$
Then the composition of the  maps on the left of \eqref{diamondwrong} equals $1+\frac{1}{2}(-1)^3X_4+\frac{1}{2}(-1)^5X_6=1-X_6$ as elements of $\cM(\underline{\la'}\overline{\la'''})$ from \eqref{MG}
by the relations given by $\underline{\la'''}$, whereas the composition of the  maps on the right equals $1+\frac{1}{2}(-1)^5X_6+\frac{1}{2}(-1)^4X_5=1-X_6$. The allowed composition of any other two maps in the diamond is also equal to $1-X_6$. Note that all the involved surgery moves in the diamonds \eqref{diamonds1} and \eqref{diamonds2} are merges, hence no additional signs appear and we do not have to care about the actual positions but only the relative position of the cups.

The following three tables list all the possible decorations with the corresponding resulting maps for the three cases in \eqref{diamonds2} where we abbreviate $P_i:=(\alpha_i,\beta_i,m_i)$.

\begin{eqnarray*}
\begin{array}[t]{ccccccccc}
D_1&D_2&D_3&D_4&P_1&P_2&P_3&P_4&\text{result}\\
\hline
\{\}&\{\}&\{\}&\{\}&(2,3,4)&(2,5,6)&(4,5,6)&(3,4,5)&1-X_6\\
\{1,3\}&\{1,3\}&\{\}&\{\}&(2,3,4)&(1,2,6)&(1,4,6)&(3,4,5)&1+X_6\\
\{1\}&\{1\}&\{1\}&\{1\}&(2,3,4)&(2,5,6)&(4,5,6)&(3,4,5)&1-X_6\\
\{3\}&\{3\}&\{1\}&\{1\}&(2,3,4)&(1,2,6)&(1,4,6)&(3,4,5)&1+X_6
\end{array}
\end{eqnarray*}

\begin{eqnarray*}
\begin{array}[t]{ccccccccc}
D_1&D_2&D_3&D_4&P_1&P_2&P_3&P_4&\text{result}\\
\hline
\{\}&\{\}&\{\}&\{\}&(4,5,6)&(2,5,6)&(2,3,6)&(3,4,5)&1-X_6\\
\{3\}&\{2\}&\{1\}&\{1\}&(3,4,6)&(1,2,6)&(1,2,6)&(3,4,5)&1+X_6\\
\{1\}&\{1\}&\{1\}&\{1\}&(4,5,6)&(2,5,6)&(2,3,6)&(3,4,5)&1-X_6\\
\{1,3\}&\{1,2\}&\{\}&\{\}&(3,4,6)&(1,2,6)&(1,2,6)&(3,4,5)&1+X_6\\
\end{array}
\end{eqnarray*}

\begin{eqnarray*}
\begin{array}[t]{ccccccccc}
D_1&D_2&D_3&D_4&P_1&P_2&P_3&P_4&\text{result}\\
\hline
\{\}&\{\}&\{\}&\{\}&(2,3,4)&(4,5,6)&(4,5,6)&(2,3,6)&1-X_6\\
\{1\}&\{1\}&\{1\}&\{1\}&(2,3,4)&(4,5,6)&(4,5,6)&(2,3,6)&1-X_6\\
\{2\}&\{1\}&\{2\}&\{1\}&(1,2,4)&(4,5,6)&(4,5,6)&(1,2,6)&1\\
\{1,2,3\}&\{3\}&\{1\}&\{1\}&(1,2,4)&(3,4,6)&(1,4,6)&(2,3,6)&1\\
\{3\}&\{3\}&\{2\}&\{1\}&(2,3,4)&(3,4,6)&(1,4,6)&(1,2,6)&1+X_6\\
\{2,3\}&\{1,3\}&\{\}&\{\}&(1,2,4)&(3,4,6)&(1,4,6)&(2,3,6)&1\\
\{1,3\}&\{1,3\}&\{1,2\}&\{\}&(2,3,4)&(3,4,6)&(1,4,6)&(1,2,6)&1+X_6\\
\{1,2\}&\{\}&\{1,2\}&\{\}&(1,2,4)&(4,5,6)&(4,5,6)&(1,2,6)&1\\
\end{array}
\end{eqnarray*}

The claim is obviously true for the first two diamonds in \eqref{diamonds1} with all possible decorations.
The following five tables list the possible decorations for the remaining five diamonds with the corresponding resulting maps:

\begin{eqnarray*}
\begin{array}[t]{ccccc|cccccc}
D_1&D_2&D_3&D_4&\text{result}&D_1&D_2&D_3&D_4&\text{result}\\
\hline
\{\}&\{\}&\{\}&\{\}&1-X_8&
\{1,2,4\}&\{4\}&\{1\}&\{1\}&1\\
\{1,2\}&\{\}&\{1,2\}&\{\}&1&
\{2,4\}&\{1,4\}&\{\}&\{\}&1\\
\{1,4\}&\{1,4\}&\{1,2\}&\{\}&1+X_8&
\{1\}&\{1\}&\{1\}&\{1\}&1-X_8\\
\{2\}&\{1\}&\{2\}&\{1\}&1\\
\end{array}
\end{eqnarray*}

\begin{eqnarray*}
\begin{array}[t]{ccccc|cccccc}
D_1&D_2&D_3&D_4&\text{result}&D_1&D_2&D_3&D_4&\text{result}\\
\hline
\{\}&\{\}&\{\}&\{\}&1-X_8&
\{1\}&\{1\}&\{1\}&\{1\}&1-X_8\\
\{1,3\}&\{\}&\{1,2\}&\{\}&1+X_8&
\{3\}&\{1\}&\{2\}&\{1\}&1+X_8\\
\end{array}
\end{eqnarray*}

\begin{eqnarray*}
\begin{array}[t]{ccccc|cccccc}
D_1&D_2&D_3&D_4&\text{result}&D_1&D_2&D_3&D_4&\text{result}\\
\hline
\{\}&\{\}&\{\}&\{\}&1-X_8&
\{1,4\}&\{\}&\{1,4\}&\{\}&1+X_8\\
\{1\}&\{1\}&\{1\}&\{1\}&1-X_8&
\{4\}&\{1\}&\{4\}&\{1\}&1+X_8\\
\end{array}
\end{eqnarray*}

\begin{eqnarray*}
\begin{array}[t]{ccccc|cccccc}
D_1&D_2&D_3&D_4&\text{result}&D_1&D_2&D_3&D_4&\text{result}\\
\hline
\{\}&\{\}&\{\}&\{\}&1+\frac{X_7-X_8}{2}&
\{1,2\}&\{\}&\{1,2\}&\{\}&1+\frac{X_7+X_8}{2}\\
\{1\}&\{1\}&\{1\}&\{1\}&1+\frac{X_7-X_8}{2}&
\{2\}&\{1\}&\{2\}&\{1\}&1+\frac{X_7+X_8}{2}\\
\end{array}
\end{eqnarray*}

\begin{eqnarray*}
\begin{array}[t]{ccccc|cccccc}
D_1&D_2&D_3&D_4&\text{result}&D_1&D_2&D_3&D_4&\text{result}\\
\hline
\{\}&\{\}&\{\}&\{\}&1+\frac{X_5-X_8}{2}&
\{1,4\}&\{\}&\{1,4\}&\{\}&1+\frac{X_5+X_8}{2}\\
\{1\}&\{1\}&\{1\}&\{1\}&1+\frac{X_5-X_8}{2}&
\{4\}&\{1\}&\{4\}&\{1\}&1+\frac{X_5+X_8}{2}\\
\end{array}
\end{eqnarray*}
The claim follows.
\end{proof}

\section{The orthosymplectic supergroup}
\label{super}

As a general convention, for a homogeneous vector $v$ in a vector superspace, ie. a $\mZ_2$-graded vector space $V=V_{\overline{0}}\oplus V_{\overline{1}}$, we will write $|v| \in \mZ_2$ for its parity. If $V=C^{a|b}$ is of superdimension $(a|b)$, ie. $\op{dim}V_{\overline{0}}=a$ and $V_{\overline{1}}=b$, the {\it general Lie superalgebra} $\mathfrak{gl}(a|b)$ is the
vector superspace $\End_\mC(V)$ of all (not necessarily homogeneous) linear endomorphisms of $V$, with superbracket $[x,y] := xy - (-1)^{|x| | y|} yx$.

\subsection{Finite dimensional representations}
Fix $m, n \geq 0$ let $r=2m$ or $r=2m+1$. Then $G=G(r|2n)$ denotes the algebraic supergroup $SOSP(r|2n)$ over $\mC$. Using scheme-theoretic language, $G$ can be regarded as a functor
from the category of commutative superalgebras over $\mC$
to the category of groups,
mapping a commutative superalgebra $A = A_{\overline{0}}\oplus A_{\overline{1}}$
to the group $G(A)$ of all  invertible $(r+2n) \times (2m+2n)$ orthosymplectic matrices.

We are interested here in the category of finite dimensional
representations of $G$ which can be viewed equivalently as
{\it integrable} supermodules over its
complex Lie superalgebra $\mg=\mathfrak{sosp}(r|2n,\mC)$. We allow
only {\em even} morphisms between $G$-modules, so that these
categories are abelian. In the following we will omit indicating the ground field $\mC$ in the notation. With the (skew)symmetric matrices $H=\begin{pmatrix}
0&\mathbf{1}_n\\
-\mathbf{1}_n&0
\end{pmatrix}$
and
$G=\begin{pmatrix}
1&0&0\\
0&0&\mathbf{1}_m\\
0&\mathbf{1}_m&0
\end{pmatrix}$
 (with the first column and row removed for $r$ even), $\mg$ can explicitly be realized as the Lie supersubalgebra of matrices in $\mathfrak{gl}(r|n)$ of the form
$
\begin{pmatrix}
A&B\\C&D
\end{pmatrix}$
where $A^tG+GA=B^tG-HC=D^tH+HD=0$. Then $\mg_0$ (resp. $\mg_1$) is the subset of all such matrices with $B=C=0$ (resp. $A=D=0$). In particular, $\mg_0\cong\mathfrak{o}(r)\oplus\mathfrak{sp}(2n)$ with its standard Cartan $\mathfrak{h}$ of diagonal matrices.
\def\eps{{\varepsilon}}
We denote by $\delta$ the super-trace of the vector representation, ie. $\delta=2m-2n$ for $\mathfrak{sosp}(2m|2n)$ and $\delta=2m-2n+1$ for $\mathfrak{sosp}(2m+1|2n)$.

Let $X=X(\mg)=\bigoplus_{i=1}^m\mathbb{Z}\eps_i\oplus\bigoplus_{j=1}^n\mathbb{Z}\delta_j$ be the integral weight lattice with even $\eps$'s and odd $\delta$'s and the standard symmetric bilinear form $(\eps_i,\eps_j)=\delta_{i,j}$, $(\eps_i,\delta_j)=0$, $(\eps_i,\delta_j)=-\delta_{i,j}$ for $1\leq i\leq m$, $1\leq j\leq n$. The even roots for $\mathfrak{sosp}(2m|2n)$ are the $\mp2\delta_i$ and $\mp\eps_i\mp\eps_j, \mp\delta_i\mp\delta_j$ for $i\not=j$ (where all signs can be chosen independently); the odd roots are the $\mp\epsilon_i\mp\delta_j$ (again with all combinations of signs) for $1\leq i\leq m$, $1\leq j\leq n$. For $\mathfrak{sosp}(2m+1|2n)$ we have additionally the even roots $\mp\epsilon_i$ and the odd roots $\mp\delta_j$.

The category $\cC$ of integrable $\mg$-modules is then the category of finite dimensional $\mathfrak{g}$-modules which are $\mh$-semisimple with integral weights. It has enough projectives and injectives, \cite[Lemma 3.1]{Vera}. Every simple object is a highest weight module and up to isomorphism and parity shift uniquely determined by its highest weight \cite[Theorem 9.9]{Vera}. The category $\cC$ decomposes into a direct sum  $\cC=\cF\oplus \Pi(\cF)$
of subcategories, $\mathcal{F}$ and its parity shift $\Pi\mathcal{F}$, where $\cF$ contains all objects such that the parity of any weight space agrees with parity of the corresponding weight. Therefore it is enough to study $\cF$.

With a fixed Borel, every irreducible module in $\cF$ is a quotient of a Verma module, in particular a highest weight module $L(\la)$ with highest weight $\la$. The occurring highest weights are precisely the dominant integral highest weights. The explicit dominance condition on the coefficients of $\la$ in our chosen basis depend on the choice of Borel we made, since in the orthosymplectic case Borels are not always pairwise conjugate. We follow closely \cite{GS2} and fix the slightly unusual choice of Borel with maximal possible number of odd simple roots, see \cite{GS1}. We only list here $\rho$, half of the sum of positive even roots minus the sum of positive odd roots:

\begin{itemize}
\item $\mg=\mathfrak{sosp}(2m|2n)$:
\begin{eqnarray*}
\rho&=&
\begin{cases}
\sum_{j=1}^{m-n}(m-n-j)\epsilon_j,&\text{ if $m>n$,}\\
\sum_{i=1}^{n-m}(n-m-i+1)\delta_i,&\text{ if $m\leq n$}.
\end{cases}
\end{eqnarray*}
\item $\mg=\mathfrak{sosp}(2m+1|2n)$:
\begin{eqnarray*}
\rho&=&-\frac{1}{2}\sum_{j=1}^{m}\epsilon_j+\frac{1}{2}\sum_{i=1}^{n}\delta_i+
\begin{cases}
\sum_{i=1}^{n-m}(n-m-i)\delta_i, &\text{if $m<n$,}\\
\sum_{j=1}^{m-n}(m-n-i+1)\epsilon_j, &\text{if $m\geq n$}.
\end{cases}
\end{eqnarray*}
\end{itemize}

For our choice of Borel, a weight $\la\in\Lambda$ is {\it dominant} if
\begin{eqnarray}
\label{laab}
\la+\rho=\sum_{i=1}^m a_i\eps_i+\sum_{j=1}^n b_j\delta_j
\end{eqnarray}
satisfies the following dominance condition, see \cite{GS1}:
\begin{itemize}
\label{dominance condition}
\item for $\mg=\mathfrak{sosp}(2m+1|2n)$:
\begin{enumerate}[(i)]
\item either $a_1>a_2>\cdots> a_m\geq \frac{1}{2}$ and  $b_1>b_2>\cdots> b_n\geq \frac{1}{2}$,
\item
or $a_1>a_2>\cdots> a_{m-l-1}> a_{m-l}=\cdots= a_m=-\frac{1}{2}$ and \hfill\\ $b_1>b_2>\cdots> b_{n-l-1}\geq b_{n-l}=\cdots= b_n=\frac{1}{2}$;
\end{enumerate}
\item
for $\mg=\mathfrak{sosp}(2m|2n)$:
\begin{enumerate}[(i)]
\item either $a_1>a_2>\cdots> a_{m-1}>|a_m|$ and  $b_1>b_2>\cdots> b_n>0$,
    \item
or $a_1>a_2>\cdots> a_{m-l-1}\geq a_{m-l}=\cdots= a_m=0$ and \hfill\\  $b_1>b_2>\cdots> b_{n-l-1}> b_{n-l}=\cdots= b_n=0$.
\end{enumerate}
\end{itemize}
Weights satisfying $(i)$ are called {\it tailless} and the number $l$ from \eqref{dominance condition} is the $\it{\op{d}-degree}$ of $\la$. The set  of dominant weights is denoted $X^+(\mg)$.

For $\la\in X^+(\mg)$ let $P(\la)$ be a fixed projective cover of $L(\la)$, see \cite{BKN} for an explicit construction. Then the $P(\la)$ form a complete set of representatives for the isomorphism classes of indecomposable projective objects in $\cF$. Note that they are also injective by \cite[Proposition 2.2.2]{BKN}.

\subsection{Hook partitions and weight diagrams}
A {\it partition of length $r$} is a non-increasing sequence $\la=(\la_1\geq\la_2,\ldots, \geq\la_{r}\geq0)$ of non-negative integers. We denote by $\la^t$ its transpose partition, ie. $\la^t_i=|\{k\mid \la_k\geq i\}|$. A partition $\la$ is called {\it $(n,m)$-hook} if $\la_{n+1}\leq m$. Given an $(n,m)$-hook partition $\la$ we associate $\op{w}(\la)=(a_1,a_2,\ldots, a_m, b_1,b_2,\ldots, b_n)-\rho\in\mZ^{m+n}$ and $\op{w}'(\la)=(a_1,a_2,\ldots, a_m, b_1,b_2,\ldots b_n)-\rho\in\frac{1}{2}+\mZ^{m+n}$ as follows
\begin{itemize}
\item For $\op{wt}(\la)$: $b_i=\op{max}\{\la_i-m+n-i+1,0\}$ and $a_j=\op{max}\{\la^t_j-n+m-j,0\}$, $1\leq i\leq n$, $1\leq j\leq m$,
\item For $ \op{wt}'(\la)$: $b_i=\op{max}\{\la_i-m+n-i+\frac{1}{2},\frac{1}{2}\}$ and $a_j=\op{max}\{\la^t_j-n+m-j+\frac{1}{2},-\frac{1}{2}\}$, $1\leq i\leq n$, $1\leq j\leq m$.
\end{itemize}

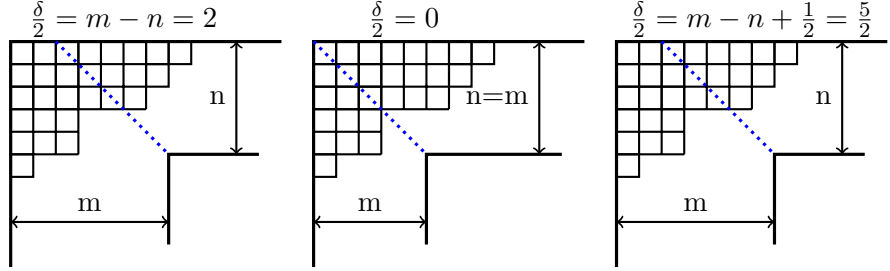
\begin{figure}
\begin{eqnarray*}
&
\begin{tikzpicture}[thick,scale=0.3]

\node at (5,1) {$\frac{\delta}{2}=m-n=2$};
\draw[very thick] (0,-10) -- (0,0) -- (12,0);
\draw[very thick] (7,-9) -- (7,-5) -- (11,-5);
\draw[<->] (0,-8) -- node[above]{m} (7,-8);
\draw[<->] (10,0) -- node[left]{n} (10,-5);

\draw[step=1] (0,-3) grid +(6,3);
\draw (6,0) rectangle +(1,-1);
\draw (7,0) rectangle +(1,-1);
\draw (6,-1) rectangle +(1,-1);
\draw[step=1] (0,-5) grid +(3,5);
\draw (0,-5) rectangle +(1,-1);

\begin{scope}[xshift=2cm]
\draw[blue,very thick,dotted] (0,0) -- (5,-5);
\end{scope}
\end{tikzpicture}
\quad
\begin{tikzpicture}[thick,scale=0.3]

\node at (4,1) {$\frac{\delta}{2}=0$};
\draw[very thick] (0,-10) -- (0,0) -- (12,0);
\draw[very thick] (5,-9) -- (5,-5) -- (11,-5);
\draw[<->] (0,-8) -- node[above]{m} (5,-8);
\draw[<->] (10,0) -- node[left]{n=m} (10,-5);

\draw[step=1] (0,-3) grid +(6,3);
\draw (6,0) rectangle +(1,-1);
\draw (7,0) rectangle +(1,-1);
\draw (6,-1) rectangle +(1,-1);
\draw[step=1] (0,-5) grid +(3,5);
\draw (0,-5) rectangle +(1,-1);

\draw[blue,very thick, dotted] (0,0) -- (5,-5);
\end{tikzpicture}
\quad
\begin{tikzpicture}[thick,scale=0.3]
\node at (6,1) {$\frac{\delta}{2}=m-n+\frac{1}{2}=\frac{5}{2}$};
\draw[very thick] (0,-10) -- (0,0) -- (12,0);
\draw[very thick] (7,-9) -- (7,-5) -- (11,-5);
\draw[<->] (0,-8) -- node[above]{m} (7,-8);
\draw[<->] (10,0) -- node[left]{n} (10,-5);

\draw[step=1] (0,-3) grid +(6,3);
\draw (6,0) rectangle +(1,-1);
\draw (7,0) rectangle +(1,-1);
\draw (6,-1) rectangle +(1,-1);
\draw[step=1] (0,-5) grid +(3,5);
\draw (0,-5) rectangle +(1,-1);

\begin{scope}[xshift=2cm]
\draw[blue,very thick,dotted] (0,0) -- (5,-5);
\end{scope}
\end{tikzpicture}
&
\end{eqnarray*}
\caption{The lower and upper part of a partition separated by the twisted diagonal.}
\label{shifteddiag}
\end{figure}

The $a_j$ and $b_i$ give a different way to describe $(n,m)$-hook partitions by encoding the number of boxes below and to the right of the $\lfloor\frac{\delta}{2}\rfloor$-shifted diagonal; for instance the pictures in Figure~\ref{shifteddiag} corresponds to $\mathfrak{sosp}(14|10)$ with $\mathbf{a}=(7,5,4,1,0,0,0)$, $\mathbf{b}=(6,4,2,0,0)$ and to $\mathfrak{sosp}(10|10)$ with
$\mathbf{a}=(5,3,2,0,0)$, and $\mathbf{b}=(8,6,4,0,0)$; and finally
$\mathfrak{sosp}(15|10)$ with $\mathbf{a}=(\frac{15}{2},\frac{11}{2},\frac{9}{2},\frac{3}{2},
\frac{1}{2},-\frac{1}{2},-\frac{1}{2})$ and $\mathbf{b}=(\frac{11}{2},\frac{7}{2},\frac{3}{2},\frac{1}{2},\frac{1}{2})$.

\begin{lemma}
\label{AB}
The assignments $\la\mapsto \op{wt}(\la)$ and $\la\mapsto  \op{wt}'(\la)$ define a bijection resp. inclusion
\begin{eqnarray*}
\Psi_{2m+1,2n}:\quad\left\{(n,m)-\text{hook partitions}\right\}
&\stackrel{1:1}{\leftrightarrow}&
X^+(\mathfrak{sosp}(2m+1|2n))\\
\Psi_{2m,2n}:\quad \left\{(n,m)-\text{hook partitions} \right\}
&{\hookrightarrow}&X^+(\mathfrak{sosp}(2m|2n))
\end{eqnarray*}
where the image of $\Psi_{2m,2n}$ misses exactly all tailless weights with $a_m<0$.
\end{lemma}
From now on we will abuse notation and write $\la$ instead of $w(\la)$ or $w'(\la)$.
\begin{remark}
{\rm
The $\op{d}$-degree of $\la\in X^+(\mg)$ equals $\op{min}\{m,n\}-d$, where $d$ is the number of boxes on the diagonal of the corresponding Young diagram.
}
\end{remark}

\begin{proof}[Proof of Lemma \ref{AB}]
Let $\la$ be a hook partition and consider $\op{wt}(\la)$. Let us first consider the case $\Psi_{2m,2n}$. Since $\la$ is a partition we have $a_{i+1}<a_i$ and $b_{i+1}<b_i$ whenever they are defined and nonzero.
For the map to be well-defined it remains to show that the number of zero $a$'s is equal or one larger then the number of zero $b$'s.
We first claim that $a_{m-s}>0$ implies $b_{n-s}$. Let first $m\geq n$, $s<n$. If $b_{n-s}= 0$ then $\la_{n-s}-m+n-n+s+1\leq 0$, hence $\la_{n-s}\leq m-s-1$ and so $\la$ has at most $n-s-1$ rows of length $m-s$. This means $\la^t_{m-s}\leq n-s-1$. Hence $a_{m-s}=\la^t_{m-s}-n+m-m+s\leq n-s-1-n+s=-1$ which is a contradiction. Let now $m<n$. If $b_{n-s}=0$ for $0\leq s\leq m-1$ then $a_{m-s}=0$ as above, but $b_{n-m}=\la_{n-m}-m+n-n+m+1=\la_{n-m}+1>0$. Altogether this shows that there are at least as many zero $a$'s as $b$'s. Hence it suffices now to see that $a_{m-n-1}>0$ and $a_{m-r}=0$ implies $b_{n-r+1}=0$ for $1\leq r\leq m$, $r\leq n$. Clearly, $a_{m-n-1}=\la^t_{m-n-1}-n+m-m+n-1=\la^t_{m-n-1}+1>0$, so assume $a_{m-r}=0$ and $b_{n-r+1}>0$. Then $\la_{n-r+1}-m+n-n+r-1+1>0$, hence $\la_{n-r+1}>m-r$ which implies $\la^t_{m-r}\leq n-r+1$ and therefore $a_{m-r}=\la^t_{m-r}-n+m-m+r\geq n-r+1-n+r\geq 1$ which is a contradiction.
Hence the map is well-defined and obviously injective. It remains to show that if $({\bf a},{\bf b})\in X^+(\mg)$ satisfies the dominance condition \eqref{dominance condition} with $a_m>0$ in case (i) then  it comes from a hook partition. It is enough to see it defines a partition, since $b_i$ is only defined for $1\leq i\leq n$ and $a_j$ for $1\leq j\leq m$, hence if it is a partition it must be $(n,m)$-hook. It suffices to see that $a_j\not=0$ implies $b_{j-m+n}\geq 1$. Setting $j:=m-s$ this is equivalent to $a_{m-l}\not=0$ implies $b_{m-l-m+n}=b_{n-l}\geq 1$. But we have seen this implication already above. The arguments for $w'(\la)$ are analogous.
\end{proof}

To formulate our main conjecture we connect the representation theory of $\mathfrak{g}$ now with our diagram calculus.

\subsection{The case $\mathfrak{sosp}(2m+1|2n)$}
Let $\mg=\mathfrak{sosp}(2m+1|2n)$ for the whole remaining section. Given a $(n,m)$-hook partition $\la$, we assign an {\it infinite diagrammatical weight $w(\la)$} to $\la$ by first considering the sequence of half-integers
\begin{eqnarray}
\label{Sla}
\cS(\la)&:=&\left(-\frac{\delta}{2}+i-\la_i^t\right)_{i\geq 1}
\end{eqnarray}

then number the vertices on the nonnegative number line by $\frac{1}{2}+\mZ_{\geq0}$ and finally attach to the vertex $i>0$ the label

\begin{eqnarray}\label{dict}
\left\{
\begin{array}{ll}
\circ&\text{if $i$ nor $-i$ occurs in $\cS(\la)$},\\
{\scriptstyle\down} &\text{if $-i$, but not $i$, occurs in $\cS(\la)$,}\\
{\scriptstyle\up} &\text{if $i$, but not $-i$, occurs in $\cS(\la)$,}\\
\times&\text{if both, $-i$ and $i$ occurs in $\cS(\la)$.}\\
\end{array}\right.
\end{eqnarray}

\begin{lemma}
\label{abovebelow}
We have $\cS(\la)_i>0$ (resp.  $\cS(\la)_i<0$) in  \eqref{Sla} iff the column $i$ in the corresponding Young diagram ends above (resp. below or on) the $\lfloor\frac{\delta}{2}\rfloor$-shifted diagonal, see Figure~\ref{shifteddiag}.
\end{lemma}

\begin{proof}
$\cS(\la)_i>0$ is equivalent to $-\frac{\delta}{2}+i-\la_i^t>0$, hence $\la_i^t<-\frac{\delta}{2}+i$.
\end{proof}

Formally we can define $\underline{w(\la)}$ as in Definition~\ref{decoratedcups}; the resulting diagram has however infinitely many dotted cups.

Note that $w(\la)$ only depends on $\la$, but not on $m,n$. Therefore we introduce a dependence on $m,n$:

\begin{definition}
\label{fakecups}
Given an $(n,m)$-hook partition $\la$ with infinite cup diagram $\underline{w(\la)}$, a cup $C$ is a \emph{ fake cup} if $C$ is dotted and there are at least $\op{d}$-degree dotted cups to the left of $C$. The vertices attached to fake cups are called \emph{frozen} vertices in $w(\la)$.
\end{definition}

For instance, the empty partition corresponds to the zero weight and then to the infinite weight diagram $w(\varnothing)$ of the following form with $\op{d}$-degree equal to $\op{min}\{m,n\}$ and infinitely many frozen vertices indicated by $\owedge$.

\begin{eqnarray}
\label{empty}
&&
\begin{cases}
\begin{picture}(260,12)
\put(30,14){$_{\frac{1}{2}}$}
\put(100.5,14){$_{-\frac{\delta}{2}+1}$}
\put(105,-5){$\underbrace{\phantom{hellow worleeeuuuuu uuuu}}_{2m}$}
\put(8,-.3){}
\put(227,-.3){}
\put(30,-5){$\underbrace{\phantom{hellow worleeee}}_{n-m}$}
\put(25,2.3){\line(1,0){210}}
\put(30,-.4){$\circ$}
\put(50,-.4){$\circ$}
\put(70,-.4){$\circ$}
\put(90,-.4){$\circ$}
\put(130,-3.4){$\up$}
\put(110,-3.4){$\up$}
\put(150,-3.4){$\up$}
\put(170,-3.4){$\up$}
\put(190,-3.4){$\up$}
\put(210,-3.4){$\up$}
\put(230,-0.4){$\owedge\quad\cdots\quad$}
\end{picture}&\text{if $\delta<0$;}\\
\\
\\
\begin{picture}(260,12)
\put(30,14){$_{\frac{1}{2}}$}
\put(110.5,14){$_{\frac{\delta}{2}}$}
\put(105,-5){$\underbrace{\phantom{hellow worleeeuuuuuuuuu}}_{2n}$}
\put(8,-.3){}
\put(227,-.3){}
\put(30,-5){$\underbrace{\phantom{hellow worleeee}}_{m-n}$}
\put(25,2.3){\line(1,0){210}}
\put(30,-.4){$\times$}
\put(50,-.4){$\times$}
\put(70,-.4){$\times$}
\put(90,-.4){$\times$}
\put(130,-3){$\up$}
\put(110,-3){$\up$}
\put(150,-3){$\up$}
\put(170,-3){$\up$}
\put(190,-3){$\up$}
\put(210,-3){$\up$}
\put(230,-.4){$\owedge\quad\cdots\quad$}
\end{picture}&\text{if $\delta>0$;}
\end{cases}
\nonumber\\
\nonumber\\
\end{eqnarray}
Obviously fake cups are never nested insider another cup.

\begin{definition}
\label{superweight}
Given a hook partition $\la$ or the corresponding dominant weight $\Psi_{2m+1,2n}(\la)\in X^+(\mg)$, we define the \emph {super weight diagram $\la^\circ$} as the one obtained from $w(\la)$ by replacing the frozen labels by $\down$'s if the $\op{d}$-degree is even and  replacing all of them except of the leftmost (which is kept as an $\up$) by $\down$'s if the $\op{d}$-degree is odd.
\end{definition}

Again, we have formally a cup diagram $\underline {\la^\circ}$ which has an infinite number of undotted rays,  $\op{d}$-degree dotted cups and no dotted ray if the  $\op{d}$-degree is even and exactly one dotted ray if it is odd. Observe that this diagram coincides with $\underline{w(\la)}$, except that each fake cup is replaced by two vertical rays (possibly the leftmost one dotted). Figure~\ref{huge} shows an example how the cup diagrams $\un\la^\circ$ change if we add boxes to the partitions.

\begin{definition}
A \emph{ super weight diagram} is a labelling $\la$ of the vertices such that vertex $i$ has an $\down$ for $i\gg 0$ and the associated infinite cup diagram $\underline{\la}$ via the rules from Definition~\ref{decoratedcups} has an even number of dots.
Two super weight diagrams are \emph{ super-linked} if the positions of their $\circ$'s and $\times$ agree and the total number of cups in their cup diagrams are the same.
\end{definition}

\begin{figure}
\label{huge}
\usetikzlibrary{arrows}
\begin{tikzpicture}[thick,>=angle 60,circ/.style={insert path={circle[radius=3pt]}},scale=.6]
\draw (0,.2) .. controls +(0,-1) and +(0,-1) .. +(1,0);\fill (.5,-.55) circle(2.5pt);
\draw (2,.2) .. controls +(0,-1) and +(0,-1) .. +(1,0);\fill (2.5,-.55) circle(2.5pt);
\draw (4,.2) .. controls +(0,-1) and +(0,-1) .. +(1,0);\fill (4.5,-.55) circle(2.5pt);
\draw (6,-.6) -- +(0,1);\fill (6,-.25) circle(2.5pt);\draw (7,-.6) -- +(0,1);
\draw[thin] (-3,0) -- +(2.5,0);
\node at (-3.85,.2) {$\emptyset$};

\draw[thin] (-3,1.3) -- +(2.5,0);
\draw (-4,1.3) rectangle +(.25,.25);
\draw [>-<] (0,0) -- (0,1.5);
\fill (.02,.75) circle(2.5pt);
\foreach \x in {1,2,3,4,5} \draw [>->] (\x,0) -- +(0,1.5);
\draw [>-] (6,0) -- +(0,1.5);\fill (6,.75) circle(2.5pt);\draw [<-] (7,-.05) -- +(0,1.5);
\draw [red] (-.2,-.05) rectangle +(.4,.35);
\draw [red] (1.8,-.05) rectangle +(.4,.35);
\draw [red] (3.8,-.05) rectangle +(.4,.35);

\draw[thin] (-3,2.75) -- +(2.5,0);
\foreach \p in {0,.25} \draw (-4,2.75+\p) rectangle +(.25,.25);
\draw (0,1.3) .. controls +(0,1) and +(0,1) .. +(1,0);
\draw (0,2.85) [circ];\draw (.9,2.75) -- +(.2,.2) +(0,.2) -- +(.2,0);
\foreach \x in {2,3,4,5} \draw [->] (\x,1.5) -- +(0,1.5);
\draw [<-] (6,1.15) -- +(0,1.6);\draw [<-] (7,1.15) -- +(0,1.6);
\draw [red] (1.8,1.15) rectangle +(.4,.35);\draw [red] (3.8,1.15) rectangle +(.4,.35);

\draw[thin] (-3,4.2) -- +(2.5,0);
\foreach \p in {0,.25,.5} \draw (-4,4.2+\p) rectangle +(.25,.25);
\draw (0,4.35) [circ];\draw (1.9,4.25) -- +(.2,.2) +(0,.2) -- +(.2,0);
\draw [->] (2,3) to [out=90,in=-90] +(-1,1.5);
\foreach \x in {3,4,5} \draw [->] (\x,3) -- +(0,1.5);
\draw [<-] (6,2.65) -- +(0,1.6);\draw [<-] (7,2.65) -- +(0,1.6);
\draw [red] (1.8,2.65) rectangle +(.4,.35);\draw [red] (3.8,2.65) rectangle +(.4,.35);

\draw[thin] (-3,5.7) -- +(2.5,0);
\foreach \p in {0,.25,.5} \draw (-4,5.7+\p) rectangle +(.25,.25);
\draw (-3.75,6.2) rectangle +(.25,.25);
\draw [->] (1,4.5) to [out=90,in=-90] +(-1,1.5);
\draw (1,5.85) [circ];\draw (1.9,5.75) -- +(.2,.2) +(0,.2) -- +(.2,0);
\foreach \x in {3,4,5} \draw [->] (\x,4.5) -- +(0,1.5);
\draw [<-] (6,4.15) -- +(0,1.6);\draw [<-] (7,4.15) -- +(0,1.6);
\draw [red] (.8,4.15) rectangle +(.4,.35);\draw [red] (3.8,4.15) rectangle +(.4,.35);

\draw[thin] (-3,7.2) -- +(2.5,0);
\foreach \p in {0,.25,.5} \draw (-4,7.2+\p) rectangle +(.25,.25);
\draw (-3.75,7.7) rectangle +(.25,.25);
\draw (-3.5,7.7) rectangle +(.25,.25);
\draw [->] (3,6) to [out=90,in=-90] +(-2,1.5);
\draw [>->] (2,7.5) .. controls +(0,-.8) and +(0,-.8) .. +(1,0);
\foreach \x in {0,4,5} \draw [->] (\x,6) -- +(0,1.5);
\draw [<-] (6,5.65) -- +(0,1.6);\draw [<-] (7,5.65) -- +(0,1.6);
\draw [red] (-.2,5.65) rectangle +(.4,.35);\draw [red] (3.8,5.65) rectangle +(.4,.35);

\draw[thin] (-3,8.7) -- +(2.5,0);
\foreach \p in {0,.25,.5} \draw (-4,8.7+\p) rectangle +(.25,.25);
\foreach \p in {.25,.5} \draw (-3.75,8.7+\p) rectangle +(.25,.25);
\draw (-3.5,9.2) rectangle +(.25,.25);
\draw [-<] (0,7.5) -- +(0,1.5);\fill (.02,8.15) circle(2.5pt);
\draw [-<] (2,7.3) -- +(0,1.7);
\foreach \x in {1,3,4,5} \draw [->] (\x,7.5) -- +(0,1.5);
\draw [<-] (6,7.15) -- +(0,1.8);\fill (6,8.1) circle(2.5pt);\draw [<-] (7,7.15) -- +(0,1.6);
\draw [red] (-.2,7.15) rectangle +(.4,.35);\draw [red] (3.8,7.15) rectangle +(.4,.35);

\draw[thin] (-3,10.2) -- +(2.5,0);
\foreach \p in {0,.25,.5} \draw (-4,10.2+\p) rectangle +(.25,.25);
\foreach \p in {.25,.5} \draw (-3.75,10.2+\p) rectangle +(.25,.25);
\draw (-3.5,10.7) rectangle +(.25,.25);
\draw (0,8.8) .. controls +(0,1) and +(0,1) .. +(1,0);
\draw (0,10.35) [circ];\draw (.9,10.25) -- +(.2,.2) +(0,.2) -- +(.2,0);
\draw [-<] (2,8.8) -- +(0,1.7);
\foreach \x in {3,4,5} \draw [->] (\x,9) -- +(0,1.5);
\draw [>-] (6,8.7) -- +(0,1.7);\draw [<-] (7,8.65) -- +(0,1.6);
\draw [red] (3.8,8.65) rectangle +(.4,.35);

\draw[thin] (-3,11.7) -- +(2.5,0);
\foreach \p in {0,.25,.5} \draw (-4,11.7+\p) rectangle +(.25,.25);
\foreach \p in {0,.25,.5} \draw (-3.75,11.7+\p) rectangle +(.25,.25);
\foreach \p in {.25,.5} \draw (-3.5,11.7+\p) rectangle +(.25,.25);
\draw [->] (4,10.5) to [out=90,in=-90] +(-4,1.5);
\draw [-<] (2,10.3) to [out=90,in=-90] +(-1,1.7);
\draw [->] (3,10.5) to [out=90,in=-90] +(1,1.5);
\draw [>->] (2,12) .. controls +(0,-.7) and +(0,-.7) .. +(1,0);
\draw [->] (5,10.5) -- +(0,1.5);
\draw [>-] (6,10.2) -- +(0,1.7);\draw [<-] (7,10.15) -- +(0,1.6);
\draw [red] (3.8,10.15) rectangle +(.4,.35);

\draw[thin] (-3,13.3) -- +(2.5,0);
\foreach \p in {0,.25,.5} \draw (-4,13.3+\p) rectangle +(.25,.25);
\foreach \p in {0,.25,.5} \draw (-3.75,13.3+\p) rectangle +(.25,.25);
\foreach \p in {0,.25,.5} \draw (-3.5,13.3+\p) rectangle +(.25,.25);
\draw [-<] (0,12) -- +(0,1.5);\fill (.02,12.6) circle(2.5pt);
\foreach \x in {1,2} \draw [-<] (\x,11.8) -- +(0,1.7);
\foreach \x in {3,4,5} \draw [->] (\x,12) -- +(0,1.5);
\draw [>-<] (6,11.7) -- +(0,1.75);\fill (6,12.6) circle(2.5pt);\draw [<-<] (7,11.65) -- +(0,1.8);
\draw [red] (-.2,11.65) rectangle +(.4,.35);
\end{tikzpicture}
\caption{Super cup diagrams $\un{\la^\circ}$ associated to hook partitions. The red boxes count the $\op{d}$-degree.}
\end{figure}
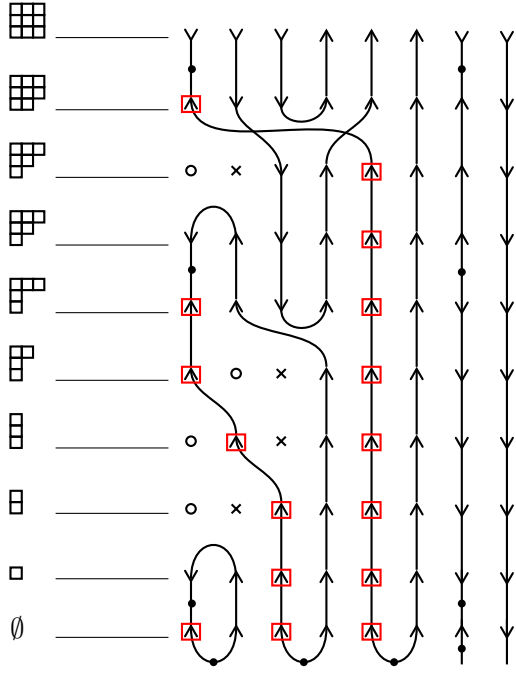

Note that for  a super weight diagram $\la$, the number $\op{def}(\la)$ of cups (dotted or undotted) in $\underline{\la}$ is always finite and called its {\it atypicality} or {\it defect}. Moreover, two super weights are super linked if they have the same atypicality and differ otherwise by a finite sequence of basic linkage moves.

Given a super weight diagram $\la$, an orientation of $\un\la$ is obtained by putting a super weight diagram $\nu$ on top to obtain $\un\la\nu$ such that the symbols $\circ$ and $\times $ appear in $\nu$ and $\la$ at the same places and every cup and ray is oriented in the sense of \eqref{oriented}. We call then $\la\supset\nu$ oriented and write as before  $\un\la\nu$.
Of course all definitions carry over to cap diagrams $\overline{\la}$ as well.

\begin{definition}
Given two super-weight diagrams $\la,\mu$ we say $\un\la\ov\mu\in\mathbb{I}$ if it contains at least one line which is not propagating.
\end{definition}
To avoid too much notation we will from now on denote both,  a hook partition $\la$ and its dominant weight $\Psi_{2m+1,2n}(\la)\in X^+(\mg)$, just by $\la$.

\begin{theorem}
\label{mainprop}
Consider $\mathfrak{g}=\mathfrak{sosp}(2m+1|2n)$ for fixed $m,n$.
\begin{enumerate}
\item The Cartan matrix of $\cF$ is symmetric, ie. for any $\la,\mu\in X^+(\mg)$
\begin{eqnarray*}
[P(\la)\::\:L(\mu)]=[P(\mu)\::\:L(\la)],
\end{eqnarray*}
and therefore
    $\op{dim}\op{Hom}_\cF(P(\la),P(\mu))
    =\op{dim}\op{Hom}_\cF(P(\mu),P(\la))$.
\item Moreover, there are isomorphisms of vector spaces
\begin{equation}
\op{Hom}_\cF(P(\la),P(\mu))\cong
\begin{cases}
\mathbb{B}(\la,\mu)&\text{if $\la$ is super-linked to $\mu$;} \\
0&\text{otherwise.}
\end{cases}
\end{equation}
where $\mathbb{B}(\la,\mu)$ is the vector space on basis
\begin{eqnarray}
\left\{\underline{\la^\circ}\nu\overline{\mu^\circ}\mid
\mu\supset\nu\subset\la\text{ and }\underline{\la^\circ}\overline{\mu^\circ}\not\in \mathbb{I}\}
 \right\}.
\end{eqnarray}
\end{enumerate}
\end{theorem}

\begin{proof}
The first part is by BGG-reciprocity, \cite[Theorem 1]{GS2} and the fact that $\op{End}(L(\la))\cong\mC$, since $L(\la)$ is a highest weight module. Alternatively one could use that $P(\la)\cong I(\la)$ and apply duality. The second part follows directly from Proposition~ \ref{cancel} below.
\end{proof}

\subsection{A positive version of the Gruson-Serganova combinatorics}
To prove Proposition \ref{mainprop} we have to connect the diagram calculus developed in \cite{GS2} to our calculus. For later reference we give an explicit dictionary, although we could prove the result more directly. The GS-weight diagram $\op{GS}(\la)$ associated with $\la\in X^+(\mg)$ or its $(n,m)$-hook partition is
a certain labelling of the positive number line (again indexed by $\frac{1}{2}+\mZ_{\geq0}$) with the symbols $<,>,\times ,\circ,\color{red}{\otimes}$ with almost all vertices labelled $\circ$ and an extra indicator $\mathbf{[+]}$ or $\mathbf{[-]}$ at vertex $\frac{1}{2}$ if there are only $\times$ at this vertex. It is obtained as the image of a composite of two maps
\begin{eqnarray}
\op{GS}:\quad X^+(\mg)&\longrightarrow& \left\{(m,n)-\text{diagrams with tail} \right\}\nonumber\\
&\longrightarrow &\left\{\text{coloured } (m,n)-\text{diagrams without tail} \right\}\label{GSmap}
\end{eqnarray}
We refer to the proof of Lemma~\ref{GSES} and \cite{GS2} for details.
\begin{lemma}
\label{GSES}
With the assignment $\op{T}$ from \eqref{T} we have $$\op{T}(\op{GS}(\la))=w(\la).$$ The associated cup diagram in the sense of \cite{GS2} agrees with $\underline{w(\la)}$ when forgetting the decoration and fake cups and with $\underline{\la^\circ}$ when forgetting the decoration and all rays.
\begin{eqnarray}
\label{T}
\begin{array}{c||c|c|c|c|c|c|c}
\op{GS}-\text{label }$s$&<&>&\times&\circ&\color{red}{\otimes}&
\text{at ${\frac{1}{2}}$: }{[-]\color{red}{\otimes}}&
\text{at ${\frac{1}{2}}$: }{[+]\color{red}{\otimes}}\\
\hline
$T(s)$&\circ&\times&\down&\up&\up&\up&\down
\end{array}
\end{eqnarray}
The cups attached to $\color{red}{\otimes}$'s except of the cups attached to the vertex $\frac{1}{2}$ with the sign $[+]$
correspond precisely to the dotted non-fake cups in $\underline{w(\la)}$.
\end{lemma}

\begin{proof}
We start by recalling the construction of $\op{GS}$. The first map in \eqref{GSmap} takes a weight $\la\in X^+(\mg)$ writes $\la+\rho$ in the form \eqref{laab} and puts $\alpha_p$ symbols $>$ at the $p$th vertex, where $\alpha_p=|\{1\leq j\leq m\mid a_j=\mp p\}|$ and $\beta_p$ symbols $<$ at the $p$th vertex, where $\beta_p=|\{1\leq i\leq n\mid b_i=\mp p\}|$ and a symbol $\circ$ if $\alpha_p=\beta_p=0$. We use the abbreviation $\times$ for a pair $>$ and $<$ at a common vertex. Then the dominance condition is equivalent to the statement that there is at most one symbol, $>$, $<$, $\times$ or $\circ$ at each vertex $p>\frac{1}{2}$ and at $\frac{1}{2}$ at most one $<$ or $>$, but possibly many $\times$. If there are only $\times$'s at $\frac{1}{2}$ we have to put an indicator which is $\mathbf{[+]}$ if $a_j=\frac{1}{2}$ for some $j$ and $\mathbf{[-]}$ otherwise. For instance, the diagram for the weight $\rho$ are the following for $n<m$, $m=n$, $m>n$ respectively.
\begin{eqnarray}
\label{withtail}
m\left\{
\begin{picture}(100,15)(0,5)
\put(0,-9){$<<\cdots<<\circ\circ\circ\circ\cdots$}
\put(0,-11){$\underbrace{\phantom{hellokikiki}}_{n-m}$}
\put(0,-3){$\times$}
\put(0,3){$\times$}
\put(0,9){$\times$}
\put(0,15){$\times$}
\end{picture}\right.
n\left\{
\begin{picture}(80,15)(0,5)
\put(0,-9){$\text{\tiny[-]}\circ\circ\circ\circ\cdots$}
\put(0,-3){$\times$}
\put(0,3){$\times$}
\put(0,9){$\times$}
\put(0,15){$\times$}
\end{picture}\right.
n\left\{
\begin{picture}(100,15)(0,5)
\put(0,-9){$>>\cdots>>\circ\circ\circ\circ\cdots$}
\put(0,-11){$\underbrace{\phantom{hellokikiki}}_{m-n}$}
\put(0,-3){$\times$}
\put(0,3){$\times$}
\put(0,9){$\times$}
\put(0,15){$\times$}
\end{picture}\right.
\nonumber\\
\nonumber\\
\end{eqnarray}

Note that the $\op{d}$-degree is always the number of $\times$ at $\frac{1}{2}$. Ignoring the position $\frac{1}{2}$ there is an associated cup diagram obtained as follows: first ignore the symbols $<$ and $>$ and connect neighboured pairs $\times$ and $\circ$ successively by a cup. Then number the vertices not connected to a cup and not containing $<$ or $>$ from the left by $1,2$ etc and relabel those with number $1,3,5, \ldots, 2k-1$ etc. by $\color{red}{\otimes}$ where $k=\op{d}$-degree of $\la$. (The symbol $\color{red}{\otimes}$ should indicate that at least apart from the special case $\frac{1}{2}$ a $\times$ was actually moved and placed on top of a $\circ$). The resulting labelling is $\op{GS}(\la)$. Finally connect neighboured pairs $\color{red}{\otimes}$ and $\circ$ successively by a cup.
In \cite{GS2} these new labels $\color{red}{\otimes}$ are called {\it coloured} and we call the attached cups {\it coloured}; note they are never nested inside other cups. For instance, the weights from \eqref{withtail} are transferred into $m$, respectively $n$ in the last case, coloured cups placed next to each other starting at position $-\frac{\delta}{2}+1$ and $\frac{\delta}{2}$ respectively. On the other hand, our diagrammatics assigns to the empty partition the weight diagrams \eqref{empty} and then gets transferred into $n$, respectively $m$ in the last case, dotted cups placed next to each other starting at position $0,\frac{\delta}{2}$, and  $-\frac{\delta}{2}+1$  respectively. Hence the claim of the lemma is true for the weight corresponding to the empty partition via Lemma \ref{AB} and so we proceed by induction on the number of boxes in the hook partition.

Assume that the partition for $\la+\rho$ is obtained from the one for $\mu+\rho$ by adding an extra box and the statements are true for $\mu$. Then we distinguish the three cases
\subsubsection*{The additional box is added above the diagonal} say in row $i$ and column $j$. Then $b_i>\frac{1}{2}$ gets increased by $1$ and all other $a$'s and $b$'s are preserved. That means (ignoring the colourings) a symbol $<$ is moved to the right from position $b_i$ to $b_i+1$ with $b_i=-\frac{\delta}{2}+1+\la_i-i$. The only possibilities for configurations at $b_i$ and $b_{i+1}$ are the following (where the translation into our diagrammatics is displayed in the second row and the symbols in brackets are added to indicate the shape of the cup diagram).
\begin{eqnarray*}
 \usetikzlibrary{arrows}
\def\down{\vee}
\def\up{\wedge}
\begin{tikzpicture}[scale=.5]
\draw (-2,2.3) -- +(24.5,0);
\draw (-2,-1.3) -- +(24.5,0);
\foreach \x in {-1,1.2,4.7,7.1,10.8,14,18.2,20.4} {\draw (\x,-1.3) -- (\x,2.3);}

\node at (-1.5,-.15) {$\mu$};
\node at (-1.5,1.6) {$\lambda$};

\node at (0,0) {$< \circ$};
\draw (.45,.3) to [out=90,in=-90] +(-.85,.9);
\node at (0,1.5) {$\circ <$};

\node at (3,0) {$< > (\circ)$};
\draw (3.75,.3) to [out=90,in=-90] +(-1.8,.9);
\draw (2.7,1.25) .. controls +(0,-.5) and +(0,-.5) .. +(1,0);
\node at (3,1.5) {$\circ \times (\circ)$};

\draw (5.7,-.25) .. controls +(0,-.5) and +(0,-.5) .. +(.7,0);
\node at (6,0) {$\times \circ$};
\draw (5.7,.25) .. controls +(0,.5) and +(0,.5) .. +(.7,0);
\node at (6,1.5) {$> <$};

\draw (7.9,-.25) .. controls +(0,-.5) and +(0,-.5) .. +(2,0);
\node at (9,0) {$\times > (\circ)$};
\draw (7.9,.3) to [out=90,in=-90] +(1,.9);
\draw (9.85,.3) -- +(-.1,.9);
\node at (9,1.5) {$> \times (\circ)$};

\node at (12.5,0) {$< {\color{red} \otimes} \circ$};
\draw (12.75,.3) to [out=90,in=-90] +(-1,.9);
\draw (12.6,-.3) .. controls +(0,-.5) and +(0,-.5) .. +(.7,0);
\draw (13.25,.3) -- +(.1,.9);
\node at (12.5,1.5) {${\color{red} \otimes} < \circ$};

\node at (16.3,0) {$< > ({\color{red} \otimes} \circ)$};
\draw (16.75,-.3) .. controls +(0,-.5) and +(0,-.5) .. +(.6,0);
\draw (17.35,.3) -- +(-.05,.9);
\draw (16.8,.35) to [out=90,in=-90] +(-2.1,.8);
\node at (16.1,1.5) {${\color{red} \otimes} \times (\circ \; \circ)$};
\draw (15.55,1.25) .. controls +(0,-.5) and +(0,-.5) .. +(1.05,0);

\draw (19.25,-.25) .. controls +(0,-.5) and +(0,-.5) .. +(.6,0);
\node at (19.2,-.8) {\tiny $\frac{1}{2}$};
\node at (18.7,0) {\tiny{[+]}};
\node at (19.5,0) {${\color{red}\otimes} \circ$};
\draw (19.25,.25) .. controls +(0,.5) and +(0,.5) .. +(.6,0);
\node at (19.5,1.5) {$> <$};

\node at (21.5,0) {$\times {\color{red} \otimes}$};
\node at (21.6,1.5) {$> <$};
\node at (21.5,-1) {\tiny imposs.};

\begin{scope}[yshift=-5cm]
\draw (-2,2.3) -- +(22.4,0);
\draw (-2,-1.1) -- +(22.4,0);
\foreach \x in {-1,1.2,4.7,7.1,10.8,14,18.2} {\draw (\x,-1.1) -- (\x,2.3);}

\node at (-1.5,-.1) {$\mu$};
\node at (-1.5,1.6) {$\lambda$};

\node at (0,0) {$\circ \up$};
\draw (.2,.4) to [out=90,in=-90] +(-.45,.9);
\node at (0,1.5) {$\up \circ$};

\node at (3,0) {$\circ \times (\up)$};
\draw (3.7,.35) to [out=90,in=-90] +(-1.8,.9);
\draw (2.7,1.25) .. controls +(0,-.5) and +(0,-.5) .. +(1,0);
\node at (3,1.5) {$\up \down (\up)$};

\draw (5.7,-.35) .. controls +(0,-.5) and +(0,-.5) .. +(.6,0);
\node at (6,0) {$\down \up$};
\draw (5.7,.35) .. controls +(0,.5) and +(0,.5) .. +(.6,0);
\node at (6,1.5) {$\times \circ$};

\draw (8,-.3) .. controls +(0,-.5) and +(0,-.5) .. +(1.7,0);
\node at (9,0) {$\down \times (\up)$};
\draw (7.85,.3) to [out=90,in=-90] +(.85,.9);
\draw (9.75,.45) -- +(0,.8);
\node at (9,1.5) {$\times \down (\up)$};

\draw (12.4,-.3) .. controls +(0,-.5) and +(0,-.5) .. +(.75,0);
\fill (12.75,-.65) circle(2.5pt);
\node at (12.5, 0) {$\circ \up \up$};
\draw (12.4,.4) to [out=90,in=-90] +(-.6,.9);
\draw (13.16,.4) -- +(0,.8);
\node at (12.5,1.5) {$\up \circ \up$};

\fill (17.05,-.65) circle(2.5pt);
\draw (16.7,-.3) .. controls +(0,-.5) and +(0,-.5) .. +(.7,0);
\node at (16.3,0) {$\circ \times (\up \;\up)$};
\draw (16.6,.35) to [out=90,in=-90] +(-1.9,.8);
\draw (15.45,1.25) .. controls +(0,-.5) and +(0,-.5) .. +(1.1,0);
\draw (17.4,.4) -- +(-.05,.8);
\node at (16.2,1.5) {$\up \down (\up \; \up)$};

\draw (19,-.3) .. controls +(0,-.5) and +(0,-.5) .. +(.6,0);
\node at (18.9,-.7) {\tiny $\frac{1}{2}$};
\node at (19.3,0) {$\down \up$};
\draw (19,.35) .. controls +(0,.5) and +(0,.5) .. +(.6,0);
\node at (19.3,1.5) {$\times \circ$};

\end{scope}

\end{tikzpicture}
\end{eqnarray*}

Note that via the translation $T$ the $\up$ at position $b_{i+1}$ would be moved to the left in each case. On the other hand $w(\la)_j>0$ by Lemma \ref{abovebelow} and hence adding a box moves in our diagrammatics indeed an $\up$ to the left. Since $i=\la_j^t+1$ and $j=\la_i+1$ we have  $w(\la)_j=-\frac{\delta}{2}+j-\la_j^t
=-\frac{\delta}{2}+\la_i+1-(i-1)
=(-\frac{\delta}{2}+\la_i-i+1)+1=b_i+1>0$ it really involves $b_i$ and $b_i+1$. Moreover, adding the extra box the $\op{d}$-degree is preserved, hence in case any of the involved $\up$'s is frozen in $w(\mu)$, then the corresponding $\up$ in $w(\la)$ connected by a line in the diagram is also frozen. Hence the claim about the weights is clear by induction and the rest follows from the two constructions of the cup diagrams.

\subsubsection*{The additional box is added below the diagonal} In this case one $a_j$ is increased, all other $a$'s and $b$'s are preserved. Hence the $j$th symbol $>$ to moved to the right corresponds in our diagrammatics to a $\down$ moved to the right:
 \begin{eqnarray*}
\usetikzlibrary{arrows}
\def\down{\vee}
\def\up{\wedge}
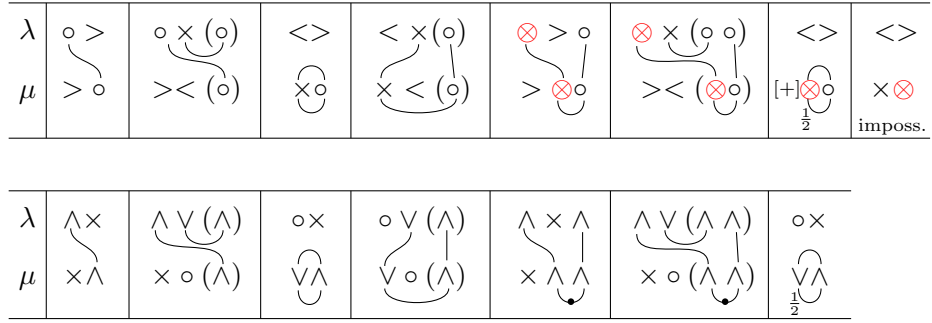
\begin{tikzpicture}[scale=.5]
\draw (-2,2.3) -- +(24.5,0);
\draw (-2,-1.3) -- +(24.5,0);
\foreach \x in {-1,1.2,4.7,7.1,10.8,14,18.2,20.4} {\draw (\x,-1.3) -- (\x,2.3);}

\node at (-1.5,-.15) {$\mu$};
\node at (-1.5,1.6) {$\lambda$};

\node at (0,0) {$> \circ$};
\draw (.45,.3) to [out=90,in=-90] +(-.85,.9);
\node at (0,1.5) {$\circ >$};

\node at (3,0) {$> < (\circ)$};
\draw (3.75,.3) to [out=90,in=-90] +(-1.5,.9);
\draw (2.7,1.25) .. controls +(0,-.5) and +(0,-.5) .. +(1,0);
\node at (3,1.5) {$\circ \times (\circ)$};

\draw (5.7,-.25) .. controls +(0,-.5) and +(0,-.5) .. +(.7,0);
\node at (6,0) {$\times \circ$};
\draw (5.7,.25) .. controls +(0,.5) and +(0,.5) .. +(.7,0);
\node at (6,1.5) {$< >$};

\draw (7.9,-.25) .. controls +(0,-.5) and +(0,-.5) .. +(2,0);
\node at (9,0) {$\times < (\circ)$};
\draw (7.9,.3) to [out=90,in=-90] +(1,.9);
\draw (9.85,.3) -- +(-.1,.9);
\node at (9,1.5) {$< \times (\circ)$};

\node at (12.5,0) {$> {\color{red} \otimes} \circ$};
\draw (12.75,.3) to [out=90,in=-90] +(-1,.9);
\draw (12.6,-.3) .. controls +(0,-.5) and +(0,-.5) .. +(.7,0);
\draw (13.25,.3) -- +(.1,.9);
\node at (12.5,1.5) {${\color{red} \otimes} > \circ$};

\node at (16.3,0) {$> < ({\color{red} \otimes} \circ)$};
\draw (16.75,-.3) .. controls +(0,-.5) and +(0,-.5) .. +(.6,0);
\draw (17.35,.3) -- +(-.05,.9);
\draw (16.8,.35) to [out=90,in=-90] +(-2.1,.8);
\node at (16.1,1.5) {${\color{red} \otimes} \times (\circ \; \circ)$};
\draw (15.55,1.25) .. controls +(0,-.5) and +(0,-.5) .. +(1.05,0);

\draw (19.25,-.25) .. controls +(0,-.5) and +(0,-.5) .. +(.6,0);
\node at (19.2,-.8) {\tiny $\frac{1}{2}$};
\node at (18.7,0) {\tiny{[+]}};
\node at (19.5,0) {${\color{red} \otimes} \circ$};
\draw (19.25,.25) .. controls +(0,.5) and +(0,.5) .. +(.6,0);
\node at (19.5,1.5) {$< >$};

\node at (21.5,0) {$\times {\color{red} \otimes}$};
\node at (21.6,1.5) {$< >$};
\node at (21.5,-1) {\tiny imposs.};

\begin{scope}[yshift=-5cm]
\draw (-2,2.3) -- +(22.4,0);
\draw (-2,-1.1) -- +(22.4,0);
\foreach \x in {-1,1.2,4.7,7.1,10.8,14,18.2} {\draw (\x,-1.1) -- (\x,2.3);}

\node at (-1.5,-.1) {$\mu$};
\node at (-1.5,1.6) {$\lambda$};

\node at (0,0) {$\times \up$};
\draw (.35,.4) to [out=90,in=-90] +(-.7,.9);
\node at (0,1.5) {$\up \times$};

\node at (3,0) {$\times \circ (\up)$};
\draw (3.7,.35) to [out=90,in=-90] +(-1.8,.9);
\draw (2.7,1.25) .. controls +(0,-.5) and +(0,-.5) .. +(1,0);
\node at (3,1.5) {$\up \down (\up)$};

\draw (5.7,-.35) .. controls +(0,-.5) and +(0,-.5) .. +(.6,0);
\node at (6,0) {$\down \up$};
\draw (5.7,.35) .. controls +(0,.5) and +(0,.5) .. +(.6,0);
\node at (6,1.5) {$\circ \times$};

\draw (8,-.3) .. controls +(0,-.5) and +(0,-.5) .. +(1.7,0);
\node at (9,0) {$\down \circ (\up)$};
\draw (8,.3) to [out=90,in=-90] +(.7,.9);
\draw (9.65,.45) -- +(0,.8);
\node at (9,1.5) {$\circ \down (\up)$};

\fill (12.95,-.65) circle(2.5pt);
\node at (12.5, 0) {$\times \up \up$};
\draw (12.5,.4) to [out=90,in=-90] +(-.8,.9);
\draw (12.6,-.3) .. controls +(0,-.5) and +(0,-.5) .. +(.7,0);
\draw (13.26,.4) -- +(0,.8);
\node at (12.5,1.5) {$\up \times \up$};

\fill (17.05,-.65) circle(2.5pt);
\draw (16.7,-.3) .. controls +(0,-.5) and +(0,-.5) .. +(.7,0);
\node at (16.3,0) {$\times \circ (\up \;\up)$};
\draw (16.6,.35) to [out=90,in=-90] +(-1.9,.8);
\draw (15.45,1.25) .. controls +(0,-.5) and +(0,-.5) .. +(1.1,0);
\draw (17.4,.4) -- +(-.05,.8);
\node at (16.2,1.5) {$\up \down (\up \; \up)$};

\draw (19,-.3) .. controls +(0,-.5) and +(0,-.5) .. +(.6,0);
\node at (18.9,-.7) {\tiny $\frac{1}{2}$};
\node at (19.3,0) {$\down \up$};
\draw (19,.35) .. controls +(0,.5) and +(0,.5) .. +(.6,0);
\node at (19.3,1.5) {$\circ \times$};

\end{scope}

\end{tikzpicture}
    \end{eqnarray*}
    Here the $j$th $>$ is moved from position $a_j=\frac{\delta}{2}+\la_j^t-j$ to $a_{j}+1$, but $a_j=-w(\la)_j$ and the claim follows. Again, the $\op{d}$-dot is preserved and the statement follows by induction.

\subsubsection*{The additional box is added on the diagonal} Then we have just a switch from $[-]$ to $[+]$ not changing the cup diagram. In our diagrammatics it swaps $\up$ to $\down$, ie. removes a dot on the cup attached to vertex $\frac{1}{2}$. On the other hand, the $\op{d}$-degree is also decreased by $1$ and so the frozen vertices are preserved.

Since the cup diagrams agree, their leftmost label determines if they are coloured (in the sense of \cite{GS2} or dotted in our sense, hence the statement follows from the definition of \eqref{T}.
\end{proof}

\begin{lemma}
\label{superlink}
The super linkage defines the blocks for the category $\cF$, ie. $\la$ and $\mu$ are super-linked if and only if $L(\la)$ and $L(\mu)$ are in the same block.
\end{lemma}

\begin{proof}
Two weights are super linked, if the labels $\circ$ and $\times$ appear at the same places and the weights have the same atypicality. Via \eqref{T}, $\circ$ and $\times$ correspond precisely the so-called core symbols in \cite{GS2} and the claim follows from the description of blocks there.
\end{proof}

Let $\la,\mu\in X^+(\mg)$. Then recall from \cite{GS2} that $\op{dim}\op{Hom}_\cF(P(\la),P(\mu))=\sum a(\la,\nu)a(\mu,\nu)$
where $\nu$ runs through all tailless dominant weights and $a(\la,\nu)$ expresses the Euler characteristics $\cE(\nu)$ in terms of simple modules, ie. we have in the Grothendieck group
$[\cE(\nu)]=\sum_\nu a(\la,\nu)[L(\la)]$.  By \cite{GS2} $a(\la,\nu)\in\{0,1,-1\}$.

\begin{prop}
\label{cancel}
Let $\la,\mu\in X^+(\mg)$. Then the following are equivalent
\begin{enumerate}[(I)]
\item \label{equI} $\op{Hom}_\cF(P(\la),P(\mu))\not=0$,
\item \label{equII} $\underline{w(\la^\circ)}\overline{w(\mu^\circ)}$ has no non-propagating line and every component has an  even number of dots.
\end{enumerate}
In these cases moreover, the following holds
\begin{enumerate}
\item $a(\la,\nu)a(\mu,\nu)\in\{0,1\}$ for any $\nu\in X^+(\mg)$, and
\item $\op{dim}\op{Hom}_\cF(P(\la),P(\mu))=2^c$, where $c$ is the number of closed components.
\end{enumerate}
\end{prop}

\begin{proof}
 For $\eqref{equII}\Rightarrow \eqref{equI}$ it is enough to show that $\op{Hom}_\cF(P(\la),P(\mu))\not=0$ implies we have no non-propagating line and that each closed component has an even number of dots, since then only the leftmost line is allowed to carry dots and by definition of $\la^\circ$ and $\mu^\circ$ the total number of dots is even.
By \cite{GS2}, $\op{dim}\op{Hom}_\cF(P(\la),P(\mu))
=\sum_\nu a(\la,\nu)a(\mu,\nu)$ and $a(\la,\nu)$ is non-zero if putting the GS-weight $\op{GS}(\nu)$ on top of the cup diagram $D$ associated with $\op{GS}(\la)$ results in a picture where the labels $>$ and $<$ in $\la$ and $\nu$ agree and each cup has the two symbols $\circ$, $\times$ in any order at its two endpoints. Clearly it is zero if there is a non-propagating line, since the line must have $\circ$'s at the end, but has an odd total number of cups and caps.
If non-zero, then it is $(-1)^{x+z+y}$, where $x$ is the number of coloured cups in $C$ and $y=y(\la,\nu)$ is the number of coloured cups with $\times$ $\circ$ in this order at the endpoints when putting $\nu$ on top, and $z=1$ if $\la$ has the indicator $[+]$ and $\nu$ has a $\times$ at position $\frac{1}{2}$. In particular, $x$ doesn't depend on $\nu$. Let $K$ be a closed component of  $\underline{w(\la^\circ)}\overline{w(\mu^\circ)}$. If $a(\la,\nu)a(\mu\nu)\not=0$ then we can find a weight $\nu'$ such that $\op{GK}(\nu')$  agrees with $\op{GK}(\nu)$ at all vertices not contained in $K$, but the symbols $\times$ and $\circ$ swapped for the vertices contained in $K$. Assume now that $K$ has an odd total number of dots. If $K$ does not contain the vertex $\frac{1}{2}$ then $(-1)^{y(\la,\nu)+y(\mu,\nu)}=-(-1)^{y(\la,\nu')+y(\mu,\nu')}$, hence $a(\la,\nu)a(\mu\nu)=-a(\la,\nu')a(\mu\nu')$ and so the two contributions cancel. The same holds if $K$ does contain the vertex $\frac{1}{2}$ but with the same indicator $[+]$ or $[-]$ in $\la$ and $\mu$. In case the symbols differ then we have an even number of coloured cups and caps in $K$, hence
$(-1)^{y(\la,\nu)+y(\mu,\nu)}=(-1)^{y(\la,\nu')+y(\mu,\nu')}$ and
$(-1)^{z(\la,\nu)+z(\mu,\nu)}=-(-1)^{z(\la,\nu')+z(\mu,\nu')}$ hence again $a(\la,\nu)a(\mu\nu)=-a(\la,\nu')a(\mu\nu')$. Hence each closed component requires an odd number of dots and so \eqref{equI} implies \eqref{equII}.

For the converse note that \eqref{equII} implies that the diagram is orientable, each line in a unique way and each closed component in exactly two ways. The same holds if we remove the dots. After applying $\op{T}$, any such orientation gives an allowed labelling $\nu$ in the sense of \cite{GS2}. We claim that the corresponding value $A:=a(\la,\nu)a(\mu\nu)$ is equal to $1$. By definition $A:=(-1)^{x(\la)+x(\mu)+z(\la,\nu)+z(\mu,\nu)}
(-1)^{y(\la,\nu)+y(\mu,\nu)}$. If $X:=\underline{w(\la^\circ)}\overline{w(\mu^\circ)}$ is a small circle then it has either no dots, hence no coloured cups and there is nothing to check. Or two dots and two coloured cups (and the same indicator) and the statement is clear as well. Otherwise, if $X$ contains a kink without coloured cups and caps then we can remove the kink to obtain a new $\la$ and $\mu$ with the same value $A$ attached. So we assume there is no such kink, but then it contains a configuration of the form (dashed lines indicate the colouring)
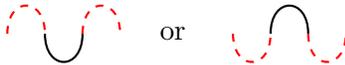
\begin{eqnarray*}
\begin{tikzpicture}[thick,scale=1]
\draw[color=red,dashed] (0,0) .. controls +(0,.5) and +(0,.5) .. +(.5,0);
\draw (.5,0) .. controls +(0,-.5) and +(0,-.5) .. +(.5,0);
\draw[color=red,dashed] (1,0) .. controls +(0,.5) and +(0,.5) .. +(.5,0);
\node at (2.2,0) {or};
\draw[color=red,dashed] (3,0) .. controls +(0,-.5) and +(0,-.5) .. +(.5,0);
\draw (3.5,0) .. controls +(0,.5) and +(0,.5) .. +(.5,0);
\draw[color=red,dashed] (4,0) .. controls +(0,-.5) and +(0,-.5) .. +(.5,0);
\end{tikzpicture}
\end{eqnarray*}
Removing the colouring and then also the newly created uncoloured kink changes $\la$ and $\mu$, but not the corresponding value $A$. Hence it must be equal to $1$.

Altogether every orientation of the circle diagram $\underline{w(\la^\circ)}\overline{w(\mu^\circ)}$ gives a contribution of $1$ to $\op{dim}\op{Hom}_\cF(P(\la),P(\mu))
=\sum_\nu a(\la,\nu)a(\mu,\nu)$. But on the other hand the number of possible $\nu$'s is precisely the number of orientations. Therefore all the remaining statements follow.
\end{proof}

\subsection{The conjecture}
Let $S\subset X^+(\mg)$ be a finite subset of pairwise super-linked dominant weights. Then we find $N>0$ such that the $N$-truncations $(\la_i)_{1\leq i\leq N}$ are good for all $\la\in S$, where good means that all $\la_i$ for $i>N$ correspond to frozen vertices. By assumption all the truncated weights from $S^\circ_N=\{\la^\circ_N\mid, \la\in S\}$ are linked, hence belong to a single block $\La(N)$ in the sense of Section~\ref{linkage} and we have the generalized Khovanov algebra $\mathbb{D}_{\La(N)}$. Define the idempotent $e_S:=\sum_{\la\in S^\circ_N} e_{\la}$ and consider the idempotent truncation subalgebra $e_S\mathbb{D}_{\La(N)} e_S$ with its distinguished basis
\begin{eqnarray}
\mathbb{B}_S:=\left\{\underline{\la}\nu\overline{\mu}\mid
\la,\mu\in  S^\circ_N, \nu\in\La
 \right\}.
\end{eqnarray}
Note that all $\la\in S^\circ_N$ have the same atypicality.

\begin{lemma}
\label{isideal}
The subspace $I=I_S$ of $e_S\mathbb{D}_{\La(N)} e_S$ spanned by all $b\in \mathbb{B}_S$ such that $b$ contains at least one non-propagating line is an ideal.
\end{lemma}

\begin{proof}
Let $b\in I_S\cap \mathbb{B}_S$. It is enough to show that $cb, bc\in I_S$ for any $c\in \mathbb{B}_S$. By Lemma~\ref{antiaut} it is enough to show $bc\in I$. Consider the non-propagating lines in $b$. Then the number of those ending at the top equals the number of those ending at the bottom since the weights in $S^\circ_N$ are linked and have the same defect. Hence assume there is at least one such line $L$ ending at the bottom. From \eqref{circline} and the diagrams below we see that any surgery involving such a line and a circle either preserves this property or produces zero. By \eqref{yuk} or \eqref{lineline} any surgery with a second line produces zero.
\end{proof}

Let $\Gamma$ be a block of $\cF$ and $P:=\oplus_{\la\in\Gamma}P(\la)$ be a minimal projective generator. For any finite subset $S$ of $\Gamma$ choose $N=N(\Gamma)$ as above and consider the algebra  $e_S\mathbb{D}_{\La(N)} e_S$.

\begin{prop}
Let $\Lambda$ be a block of $\cF$ and $P:=\oplus_{\la\in\La}P(\la)$ be a minimal projective generator. For any finite subset $S$ of $\Lambda$ set $P_S=\oplus_{\la\in S}P(\la)$ and choose $N$ such that the $N$-truncations of all $s\in S$ are good. Then there is an isomorphism of vector spaces
\begin{eqnarray}
\label{isoPS}
End_\mg(P_S)\cong e_S\mathbb{D}_{\La(N)} e_S/\mathbb{I}
\end{eqnarray}
and the right hand side has an induced algebra structure
\end{prop}

\begin{proof}
This follows directly from Proposition \ref{mainprop} and Lemma~\ref{isideal}.
\end{proof}

\begin{conjecture}
\label{conjnaiv}
The isomorphisms \eqref{isoPS} are isomorphisms of algebras.
\end{conjecture}

The above construction is an ad hoc construction avoiding working with infinite weights and weight diagrams, but does not cover the whole endomorphism ring of $P$. Hence we prefer the following slightly different perspective

\begin{prop}
\label{better}
The assignment $\la\mapsto \la^\circ$ from Definition~\ref{superweight} defines a bijection between $X^+(\mathfrak{g})$ (or equivalently $(n,m)$-hook partitions) and the set $D$ of super weight diagrams $\mu$ satisfying

\begin{eqnarray}
\label{timescirc}
\op{def}(\mu)+\#\!\!\times(\mu)=m,&&\#\!\!\circ(\mu)-\#\!\!\times(\mu)=n-m.
\end{eqnarray}
where $\#\!\!\times(\mu)$ resp. $\#\!\circ(\mu)$ denotes the number of $\times$'s resp. $\circ$'s appearing in $\mu$.
\end{prop}
Observe that \eqref{timescirc} is equivalent to $\op{def}(\mu)+\#\!\circ(\mu)=n$, $\#\!\!\times(\mu)-\#\!\circ(\mu)=m-n$.

\begin{proof}
First note that (see the tables in the proof of Lemma~\ref{GSES}), the number $\#\!\circ(\mu)-\#\!\!\times(\mu)$ does not change when adding an additional box and the same holds for $\op{def}(\mu)+\#\!\!\times(\mu)$. Hence $\la\mapsto \la^\circ$ has image in $D$ because the image, call it $\emptyset^\circ$, of the empty partition is in $D$ by \eqref{empty}. The map is obviously injective. To see that is surjective let $\mu\in D$. Set $m=\op{def}(\mu)+\#\!\!\times(\mu)$ and $n=\#\!\circ(\mu)-\#\!\!\times(\mu)+m$. Consider $\underline{\mu}$ and let $s_r<s_{r-1}<\ldots <s_1$ be the positions where the $\down$'s defining a cup in $\underline{\mu}$ and $\times$'s occur. Clearly $r\leq m$. Then set $\la^t_i=n-m+i+s_i-\frac{1}{2}$ for $1\leq i\leq r$. Let $t_{r+1}<t_{r+2},\ldots$ be the positions of all the $\up$'s and the $\down$'s not defining cups. Then set $\la^t_i=n-m+i-t_{i}-\frac{1}{2}$ for $i>r$. Since $t_i\leq t_j-j+i$ we have $\la^t_i\geq \la^t_j$ for any $j>i$. Moreover, $\la^t_{m+1}=n-m+m+1-t_{m+1}-\frac{1}{2}=n+\frac{1}{2}-t_{m+1}\leq n$, since $t_{m+1}\geq \frac{1}{2}$.
Choose $N$ such that the $N$-truncations of $\emptyset^\circ$ and $\mu$ are good. Then $N$ is the number of symbols appearing in the truncation,
$$\#_{N}\up(\nu)+\#_{N}\down(\nu)+\#_{N}\times(\nu)+\#_{N}\circ(\nu)=N$$
where $\#_N$ is the number of symbols appearing at positions smaller than $N+1$. Hence, the number  $k=\#_{N}\!\up(\nu)+\#_{N}\!\down(\nu)+2\#_{N}\!\times(\nu)$ is the same for $\nu=\mu$ and $\nu=\emptyset^\circ$, by our assumption \eqref{timescirc}. Therefore, $\la^t_j=0$ for $j>k$. Altogether we defined a $(n,m)$-hook partition $\la$ which by construction defines a preimage of $\mu$.
\end{proof}

Note that when fixing $m,n$ as well as the positions of $\circ$ and $\times$ in \eqref{timescirc}, the atypicality, say $k$, is fixed as well. The resulting weights form a set of pairwise super-linked weights defining a block $\La$ for $\cF$. Moreover, the generalized surgery procedure with the signed rules ({Surg}-1)-({Surg}-3) defines a graded (non-unitary) associative algebra structure $\mathbb{H}_k^\infty$ on the infinite dimensional graded vector space
\begin{eqnarray}
\B&:=&\left\{ \underline{\la}\nu\ov{\mu}\mid \la,\mu,\nu\in\Lambda,\quad\ov{\mu}\nu \text{ and $\underline{\la}\nu$ are
oriented}\right\},
\end{eqnarray}

with finite dimensional graded pieces.

Since for such a weight $\la\in\La$ for large enough $N$, the $N$-truncation and $N+1$-truncation only differ by adding an extra $\down$ at the end of the weights, the multiplication doesn't really depend on $N$. For a fixed atypicality $k$ we can therefore take the union $\La(k,n)$ of all blocks of atypicality $k$ supported on the first $n$ vertices and its corresponding algebra $\mathbb{D}_{\La(k,n)}$. Then we have a directed system of algebras $\mathbb{D}_{\La(k,n)}\subset \mathbb{D}_{\La(k,n+1)}$ and the limiting algebra $\mathbb{H}^\infty_k=\varinjlim_n\mathbb{D}_{\La(k,n)}$ similar to the construction in \cite[Section 4]{BS1}. The arguments from there show that this limiting algebra agrees with the algebra $\mathbb{H}^\infty_k$ from above.

\begin{conjecture}
Let $m,n$ be natural numbers and fix a block $\cB$ of atypicality $k$ in $\cF(\op{SOSP}(m|2n))$. Then
\begin{center}
$\cB$ is equivalent to the category of finite dimensional $\mathbb{H}_k^\infty/\mathbb{I}$-modules.
\end{center}
The equivalence sends $P(\la)\in\cF$ to the indecomposable projective module $P(\la^\circ)$ and the irreducible highest weight module $L(\la)$ to $L(\la^\circ)$.
\end{conjecture}

In particular, $\cF$ inherits a grading from $\mathbb{H}_k^\infty/\mathbb{I}$ induced from $\mathbb{H}_k^\infty/\mathbb{I}$.

\subsection{The case $\mathfrak{sosp}(2m|2n)$}
The case $\mathfrak{sosp}(2m|2n)$ is analogous to the previous one, but technically more difficult and with a slightly different dictionary. It will appear in an extra note.

\bibliographystyle{alpha}
\bibliography{references}

\end{document}